\documentclass[reqno,12pt,twoside]{smfbook}

\usepackage[francais]{babel} 
\usepackage[T1]{fontenc}
\usepackage[latin1]{inputenc}

\usepackage{amssymb,amsmath,amsfonts}
\usepackage{bbm,mathrsfs}
\usepackage{smfthm,smfenum}
\usepackage[all]{xy}
\usepackage{amscd}
\usepackage{graphics,color,graphicx,pstricks}

\usepackage[small,nohug,heads=vee]{diagrams}
\diagramstyle[labelstyle=\scriptstyle]

\renewcommand{\headsep}{25pt} 

\theoremstyle{plain}
\newtheorem{theoIntro}{Théorème}

\newtheorem{coroIntro}[theoIntro]{Corollaire}

\theoremstyle{remark}
\newtheorem{nota}{Notation}
\newtheorem{exple}{Exemple}[chapter]

\author{Guillaume Deltour}
\email{guillaume.deltour@math.univ-montp2.fr}
\date{\today}

\newcommand{\N}{\mathbb{N}}
\newcommand{\Z}{\mathbb{Z}}
\newcommand{\Q}{\mathbb{Q}}
\newcommand{\R}{\mathbb{R}}
\newcommand{\C}{\mathbb{C}}

\newcommand{\got}[1]{\mathfrak{#1}}
\newcommand{\Ad}{\mathop{\mathrm{Ad}}}
\newcommand{\ad}{\mathop{\mathrm{ad}}}
\newcommand{\tr}{\mathop{\mathrm{Tr}}}
\newcommand{\diag}{\mathrm{diag}}
\newcommand{\re}{\mathop{\mathrm{Re}}}

\newcommand{\imag}{\mathop{\mathrm{Im}}}
\newcommand{\Orb}{\mathcal{O}}

\newcommand{\M}{\mathcal{M}}

\newcommand{\Pn}{\mathcal{P}_n}
\newcommand{\Chol}{\mathcal{C}_{\mathrm{hol}}}
\newcommand{\Pic}{\mathrm{Pic}}
\newcommand{\PicG}{\mathrm{Pic}^{\KC}}
\newcommand{\PicQ}{\mathrm{Pic}_{\Q}}
\newcommand{\PicGQ}{\mathrm{Pic}_{\Q}^{\KC}}
\newcommand{\Xss}{X^{\mathrm{ss}}}
\newcommand{\XssL}{X^{\mathrm{ss}}(\mathcal{L})}

\newcommand{\id}{\mathrm{id}}
\newcommand{\KC}{K_{\C}}
\newcommand{\TC}{T_{\C}}
\newcommand{\hKC}{\hat{K}_{\C}}
\newcommand{\hTC}{\hat{T}_{\C}}
\newcommand{\tKC}{\tilde{K}_{\C}}
\newcommand{\tTC}{\tilde{T}_{\C}}
\newcommand{\pt}{[\mathrm{pt}]}

\newcommand{\halpha}{\hat{\alpha}}

\newcommand{\WT}{\mathcal{W}_{\TC}}
\newcommand{\DGIT}[1][K\cdot\Lambda\times E]{\Delta_K^{\mathrm{GIT}}(#1)}
\newcommand{\bmin}{\beta_{\mathrm{min}}}
\newcommand{\bmax}{\beta_{\mathrm{max}}}
\newcommand{\quotient}{/\!/}
\newcommand{\Csa}{C_{\Q}(X)^+}
\newcommand{\Ca}{C_{\Q}(X)^{++}}
\newcommand{\car}{\mathcal{X}}
\newcommand{\XEC}{X_{E\oplus\C}}

\newcommand{\CQ}{\mathcal{C}_{\mathbb{Q}}}
\newcommand{\CR}{\mathcal{C}_{\mathbb{R}}}
\newcommand{\Lsym}{L_{\mathrm{sym}}^+}
\newcommand{\Lasym}{L_{\mathrm{asym}}}
\newcommand{\glr}{\got{gl}_r(\C)}

\newcommand{\Glambda}{G\cdot\Lambda}
\newcommand{\Klambda}{K\cdot\Lambda}

\newcommand{\coh}[1]{\mathrm{H}^*(#1,\Z)}
\newcommand{\homol}[1]{\mathrm{H}_*(#1,\Z)}

\DeclareMathAlphabet{\mathpzc}{OT1}{pzc}{m}{it}

\numberwithin{equation}{section}

\makeindex

\begin{document}


\thispagestyle{empty}

\begin{center}
{\bf TH\`ESE}

pour obtenir le titre de

{\bf \large Docteur en Sciences}

de l'Université Montpellier II

\vspace{16pt}

Spécialité : {MATH\'EMATIQUES}

\vspace{24pt}

présentée par

\vspace{16pt}

{\bf \large Guillaume DELTOUR}

\vspace{80pt}

{\large \bf PROPRI\'ET\'ES SYMPLECTIQUES ET HAMILTONIENNES DES ORBITES COADJOINTES HOLOMORPHES}

\vspace{80pt}

Thèse dirigée par {\bf Paul-\'Emile PARADAN}
\end{center}

\vspace{32pt}
\noindent Soutenue le 10 décembre 2010 devant le jury composé de :

\vspace{16pt}

\begin{center}
\begin{tabular}{ll}
Michel BRION & Directeur de recherche à l'Université de Grenoble I \\
Michel DUFLO & Professeur émérite à l'Université Paris VII \\
Gert HECKMAN & Professeur à l'Université de Nijmegen, Pays-Bas\\
Paul-\'Emile PARADAN & Professeur à l'Université Montpellier II \\
Nicolas RESSAYRE & Maître de conférences à l'Université Montpellier II  \\
Michèle VERGNE & Membre de l'Académie des Sciences
\end{tabular}
\end{center}

\vspace{16pt}

\noindent Au vu des rapports de Michel BRION et Reyer SJAMAAR.

\vfill

\begin{center}
Université Montpellier II

\vspace{12pt}

Institut de Mathématiques et de Modélisation de Montpellier

Place Eugène Bataillon, 34095 MONTPELLIER Cedex, France
\end{center}

\frontmatter

\thispagestyle{empty}

\begin{center}
{\bf TH\`ESE}

pour obtenir le titre de

{\bf \large Docteur en Sciences}

de l'Université Montpellier II

\vspace{16pt}

Spécialité : {MATH\'EMATIQUES}

\vspace{24pt}

présentée par

\vspace{16pt}

{\bf \large Guillaume DELTOUR}

\vspace{80pt}

{\large \bf PROPRI\'ET\'ES SYMPLECTIQUES ET HAMILTONIENNES DES ORBITES COADJOINTES HOLOMORPHES}

\vspace{80pt}

Thèse dirigée par {\bf Paul-\'Emile PARADAN}
\end{center}

\vspace{32pt}
\noindent Soutenue le 10 décembre 2010 devant le jury composé de :

\vspace{16pt}

\begin{center}
\begin{tabular}{ll}
Michel BRION & Directeur de recherche à l'Université de Grenoble I \\
Michel DUFLO & Professeur émérite à l'Université Paris VII \\
Gert HECKMAN & Professeur à l'Université de Nijmegen, Pays-Bas\\
Paul-\'Emile PARADAN & Professeur à l'Université Montpellier II \\
Nicolas RESSAYRE & Maître de conférences à l'Université Montpellier II  \\
Michèle VERGNE & Membre de l'Académie des Sciences
\end{tabular}
\end{center}

\vspace{16pt}

\noindent Au vu des rapports de Michel BRION et Reyer SJAMAAR.

\vfill

\begin{center}
Université Montpellier II

\vspace{12pt}

Institut de Mathématiques et de Modélisation de Montpellier

Place Eugène Bataillon, 34095 MONTPELLIER Cedex, France
\end{center}

\newpage
~
\thispagestyle{empty}
\newpage
~
\thispagestyle{empty}

\vfill

\begin{center}
{\large PROPRI\'ET\'ES SYMPLECTIQUES ET HAMILTONIENNES DES ORBITES COADJOINTES HOLOMORPHES}

\vspace{24pt}

{\large Guillaume DELTOUR}
\end{center}

\vfill


\newpage
~
\thispagestyle{empty}


\begin{center}
{\bf \large REMERCIEMENTS}
\end{center}

\thispagestyle{empty}

\vspace{36pt}

Je tiens à remercier en tout premier lieu mon directeur de thèse P-E. Paradan, qui a accepté d'encadrer cette thèse et m'a initié à la géométrie hamiltonienne. Il a su me laisser une grande liberté dans mon travail tout en restant présent et disponible lorsque j'en avais besoin, et en cela je lui suis très reconnaissant. Merci également à lui pour ses conseils dans la rédaction de ce manuscrit, dont la qualité et la clarté auraient été bien moindres sans ceux-ci.

Je remercie M. Brion et R. Sjamaar d'avoir accepté de rapporter cette thèse et pour l'intérêt qu'ils ont porté à mon travail. Je leur suis très reconnaissant pour leurs commentaires et suggestions qui m'ont permis d'améliorer la qualité de ce manuscrit. Je remercie également les autres membres du jury, M. Duflo, G. Heckman, N. Ressayre et M. Vergne, qui ont accepté d'examiner ce mémoire. Je voudrais remercier en particulier N. Ressayre de m'avoir expliqué ses travaux et suggéré de les utiliser dans ma thèse.

Cette thèse a pu se dérouler dans des conditions favorables grâce à la bonne ambiance régnant au laboratoire I3M, entretenue par les permanents mais également par le personnel administratif. Je remercie I. Badulescu et M. Herzlich de m'avoir fait confiance en me permettant d'enseigner dans la préparation à l'agrégation de Montpellier.

Ces trois années auraient été bien difficiles sans la présence et le soutien des thésards du bâtiment 9. J'ai passé des moments inoubliables avec eux, que ce soit au labo ou à l'extérieur. Je remercie tous les thésards, les anciens comme les nouveaux, pour tous ces excellents moments.

Parmi les thésards, ma première pensée va à ma co-bureau Chloé (alias M) et à nos discussions sur les mathématiques et sur bien d'autres sujets. Tous ces moments partagés en 007 et après m'ont été très agréables et m'auront beaucoup appris humainement. Je voudrais également remercier Soffana avec qui j'ai partagé la gestion du séminaire des doctorants de l'I3M pendant ma deuxième année. Sa bonne humeur et son sourire m'ont beaucoup aidé dans les moments difficiles. Merci à Olivier pour ses jeux de mots légendaires et son humour, les sorties vélo et toutes nos discussions sur la musique, la guitare et le cinéma (pour ne citer que ça). Pour compléter l'équipe des \og{}Mini-Pouces\fg{}, je tiens à remercier Benoît d'avoir été un très bon colocataire pendant deux années de cette thèse, et pour toute son aide et toutes ses explications mathématiques qui ont toujours été au \og{}taupe\fg{}. Je remercie également Afaf et Julien, mes dignes successeurs à l'organisation du séminaire des doctorants, et tous les autres doctorants : Paul, Florence, Jonathan, Bruno, Frédéric, Junior, Anthony, Matthieu S., Pierre, et les nouveaux arrivants Claudia, Mathieu C., Christophe et Vincent. J'ajoute une pensée aux anciens thésards et ATER Thomas L., Hilde, Rémi Servien, Rémy Sart, Chady, Nadia, Romain et Elamine. Je tiens également à remercier Gwladys, Boris, \'Etienne, Vanessa et Simon qui ont su nous intégrer dans la jeune équipe de recherche de l'I3M.

Je voudrais également remercier mes amis Coralie et Benjamin, Didier, Sylvain et les autres, sur qui j'ai toujours pu compter.

Enfin, je remercie ma famille, et tout particulièrement ma mère et ma soeur, pour leur soutien indéfectible et leur affection.


\tableofcontents




\chapter*{Introduction}
\label{chap:introduction}

Cette thèse étudie deux aspects différents des orbites coadjointes holomorphes de groupes de Lie réels simples non compacts : la structure symplectique canonique d'orbite coadjointe d'une part; les propriétés de convexité de la projection d'orbite d'autre part.

\bigskip

Dans la théorie des actions de groupes de Lie, les orbites coadjointes se sont révélées être, au cours du 20\ieme{} siècle, des objets fondamentaux et disposant de propriétés géométriques très intéressantes.

L'étude des orbites coadjointes est un véritable atout en particulier en théorie des représentations. En effet, lorsque le groupe considéré est compact, elles permettent tout simplement de classifier les représentations irréductibles du groupe de Lie. Plus exactement, ce sont certaines orbites, les orbites coadjointes entières ou dites également \og{}préquantifiables\fg{}, qui paramètrent les représentations irréductibles du groupe, à isomorphisme près. Ce résultat est connu sous le nom de Théorème de Borel-Weil.

\smallskip

Une des propriétés importantes des orbites coadjointes est qu'elles sont naturellement munies d'une structure symplectique invariante. Cette structure symplectique invariante canonique se trouve être également une structure hamiltonienne, dont l'application moment est l'injection de l'orbite coadjointe dans le dual de l'algèbre de Lie du groupe.

L'étude de la structure hamiltonienne d'une orbite coadjointe prend tout son sens lorsque l'on regarde l'action d'un sous-groupe sur cette même orbite et la structure hamiltonienne qui en est induite. L'application moment standard qui est associée à l'action hamiltonienne de ce sous-groupe est appelée \emph{projection d'orbite}. 

Pour fixer les idées, notons $G$ un groupe de Lie connexe et $K$ un sous-groupe de Lie fermé connexe. Notons $\got{g}$ et $\got{k}$ leurs algèbres de Lie respectives. Fixons $\Lambda\in\got{g}^*$ et considérons l'orbite coadjointe $\Orb_{\Lambda}:=G\cdot\Lambda$ de $G$. La projection d'orbite est la projection $\Phi_{\Orb_{\Lambda}}:\Orb_{\Lambda}\subset\got{g}^*\rightarrow\got{k}^*$.

L'intérêt de cette projection d'orbite est qu'elle permet, en quelque sorte, d'obtenir des informations sur les représentations irréductibles du sous-groupe $K$ apparaissant dans une représentation irréductible de $G$, tout du moins lorsque les groupes $G$ et $K$ sont compacts.

En effet, lorsque $G$ (et donc $K$) est compact et $\Lambda$ est un poids dominant de $G$, la représentation irréductible de $G$ de plus haut poids $\Lambda$, notée $V_{\Lambda}^G$, se décompose en somme de représentations irréductibles de $K$. La question qui se pose alors est:
\smallskip
\begin{center}
\begin{minipage}{13cm}
Peut-on déterminer les représentations irréductibles $V_{\mu}^K$ de $K$ qui apparaissent dans la représentation irréductible $V_{\Lambda}^G$ de $G$?
\end{minipage}
\end{center}
\smallskip
La réponse à cette question est partielle. Les poids dominants $\mu$ de $K$ dont la représentation irréductible associée $V_{\mu}^K$ apparaît dans $V_{\Lambda}^G$ ont leur orbite coadjoite $K\cdot\mu$ qui est dans la projection de l'orbite $\Orb_{\Lambda}$. En revanche, la réciproque est une question bien plus difficile, et est par ailleurs fausse en général.

Cependant, nous avons tout de même une information asymptotique: si l'orbite coadjointe $K\cdot\mu$ de $K$ est dans la projection de l'orbite $\Orb_{\Lambda}$ du groupe $G$, alors pour un entier $N>0$ suffisamment grand, on a $V_{N\mu}^K\subset V_{N\Lambda}^G$.

Il existe malgré tout de bons cas pour lesquels la projection d'orbite détermine précisément les représentations irréductibles de $K$ qui apparaissent dans $V_{\Lambda}^G$. C'est par exemple le cas pour le groupe unitaire $K=U(n)$ qui s'injecte diagonalement dans le produit de $s$ copies $G=U(n)^s$ de $U(n)$. Ici, on dit que $U(n)$ vérifie une propriété de \emph{saturation} (pour le produit tensoriel de représentations irréductibles), prouvée par Knutson-Tao \cite{knutson_tao}.

\bigskip

La projection d'orbite s'avère donc un outil fondamental dans la théorie des représentations. En particulier, c'est la connaissance des orbites coadjointes de $K$ appartenant à l'image de la projection d'orbite qui est importante.

\bigskip

L'étude de la projection d'orbites coadjointes, initiée par Kostant \cite{kostant} dans les années 1970 et développée par Heckman \cite{heckman} au début des années 1980, a été également le point de départ de nombreuses recherches dans le domaine de la géométrie hamiltonienne. Le résultat le plus remarquable qui a pu ainsi être découvert est la propriété de convexité de l'image de l'application moment.

Les précurseurs dans ce domaine sont Atiyah \cite{atiyah} ainsi que Guillemin et Sternberg \cite{guillemin82}. Au début des années 1980, ils ont prouvé que l'image d'une application moment pour l'action d'un tore sur une variété symplectique compacte est un polytope convexe dont les sommets sont les images des points de la variété fixés par l'action du tore.

Cette propriété de convexité a été généralisée peu de temps après par Kirwan \cite{kirwan} dans le cas d'un groupe compact, connexe, non forcément abélien, agissant sur une variété compacte. Dans le cas non abélien, c'est plus exactement l'intersection de l'image de l'application moment avec une chambre de Weyl fixée qui est un polytope convexe. Ce polytope est appelé \emph{polytope moment} ou \emph{polytope de Kirwan}.

Cependant, dans le cas non abélien, la perte de l'hypothèse de commutativité du groupe entraîne la disparition de l'information sur les sommets du polytope moment. Il est alors en général difficile de déterminer les sommets, ou bien les équations affines, du polytope moment. Or, connaître les équations du polytope moment a un très grand intérêt pour les projections d'orbites, puisque ce polytope permet de déterminer asymptotiquement les représentations irréductibles du sous-groupe $K$ apparaissant dans une représentation irréductible du groupe $G$. L'exemple de la projection d'orbite pour les groupes $K=U(n)$ et $G=U(n)^2$, qui est connu sous le nom de \emph{problème de Horn}, s'intéresse au lien entre les spectres de deux matrices hermitiennes avec le spectre de la somme de ces deux matrices. L'énoncé de ce problème a été formulé au dix-neuvième siècle et Horn \cite{horn62} y a apporté une réponse en petites dimensions (dimension inférieure ou égale à $5$). Il a fallu attendre 2001 pour obtenir une réponse complète par Knutson et Tao \cite{knutsontao2001}.

\bigskip

L'étude des équations du polytope moment dans le cadre de variétés non compactes se révèle être beaucoup plus marginale. Cependant, la propriété de convexité hamiltonienne reste valide pour l'action d'un groupe compact connexe sur une variété symplectique avec application moment propre \cite{lerman98}. L'image de l'application moment n'étant plus compacte, on parle alors de polyèdre moment. Cependant, il n'est généralement que localement polyédral.

Hilgert, Neeb et Plank \cite{hilgert} ont déterminé le polyèdre moment de la projection de certaines orbites coadjointes elliptiques sur une sous-algèbre de Cartan compacte d'une algèbre de Lie réelle. Ce polyèdre est décrit comme la somme de la projection de l'orbite compacte sous-jacente, dont le polytope moment a été étudié par Kostant \cite{kostant}, et du cône convexe engendré par certaines racines non compactes de l'algèbre de Lie.

Eshmatov et Foth \cite{eshmatov_foth} ont, quant à eux, étudié la somme de deux orbites coadjointes elliptiques non compactes, admissibles au sens de \cite{hilgert}, pour le cas d'algèbres de Lie réelles semi-simples quasi-hermitiennes. Ce résultat est une version non compacte du problème de Horn.

Duflo, Heckman et Vergne \cite{duflo84} proposent une autre étude de projection d'orbites coadjointes non compactes, même si celle-ci n'est pas destinée, à proprement parler, à obtenir des équations d'un polyèdre moment. En effet, ils ont calculé l'image de la mesure de Liouville par la projection de certaines orbites coadjointes elliptiques. La formule de Duflo-Heckman-Vergne peut alors permettre de déterminer le polyèdre moment de la projection d'orbite, puisque ce polyèdre constitue le support de la mesure donnée par cette même formule. Malheureusement, cette formule ressemble à la formule de Kostant, c'est-à-dire qu'il s'agit d'une somme alternée de mesures positives avec une combinatoire mal comprise. Cela entraîne une utilisation délicate lorsqu'on souhaite l'appliquer à des cas pratiques.

\bigskip

Nous avons mentionné plus tôt le lien entre projection d'orbites coadjointes du groupe de Lie compact $G$, par rapport au sous-groupe $K$, et décomposition de représentations irréductibles de $G$ en somme directe de représentations irréductibles de $K$. Qu'en est-il pour le cas où $G$ est non compact, $K$ restant quant à lui compact?

Malheureusement, ce lien n'existe pas nécessairement dans le cas où $G$ n'est pas compact. Cependant, dans les années 1950, lorsque Harish-Chandra a voulu étendre la méthode des orbites au cadre non compact, il s'est rendu compte qu'un analogue au Théorème de Borel-Weil existait pour certaines orbites coadjointes de groupes de Lie très particuliers. Ces orbites coadjointes correspondent aux séries discrètes holomorphes de Harish-Chandra. On les appelle les \emph{orbites coadjointes holomorphes}.

Précisons un peu le cadre de définition de ces orbites coadjointes. Nous devons considérer ici un groupe de Lie $G$ réel, semi-simple, connexe, non compact et à centre fini, et $K$ un sous-groupe compact maximal. Nous demandons, de plus, que l'espace symétrique $G/K$ soit hermitien, c'est-à-dire, tel que la variété $G/K$ soit munie d'une structure complexe pour laquelle l'action à gauche de $G$ soit holomorphe.

%
%
%

L'espace symétrique hermitien $G/K$ a de nombreuses propriétés, dont celle d'être une variété kählerienne à pôle (c'est-à-dire qu'il existe un point $p\in G/K$ au-dessus duquel l'application exponentielle est un difféomorphisme de $T_p G/K$ sur $G/K$) et à courbure sectionnelle négative. Ces variétés kählériennes ont pour particularité d'avoir une structure symplectique simple. En effet, McDuff \cite{mcduff} a prouvé qu'une telle variété est symplectomorphe à l'espace vectoriel symplectique $(\got{p},\Omega_{\got{p}})$. Ici, $\got{p}$ désigne la partie non compacte de la décomposition de Cartan $\got{g} = \got{k}\oplus\got{p}$. La forme symplectique $\Omega_{\got{p}}$ est la forme symplectique linéaire $K$-invariante standard (voir paragraphe \ref{section:symplecto_mcduff}). La variété symplectique $(\got{p},\Omega_{\got{p}})$ est également hamiltonienne, avec pour polyèdre moment associé un cône convexe polyédral engendré par des sommes de certaines racines non compactes positives (cf \cite[Section 5]{Paradan}).



Les orbites coadjointes holomorphes du groupe de Lie $G$ généralisent l'espace symétrique hermitien $G/K$. Ce sont les orbites coadjointes elliptiques qui sont munies d'une structure kählerienne naturelle. Ces orbites vérifient des propriétés symplectiques analogues à l'espace symétrique hermitien $G/K$, nous permettant de déterminer les équations de leurs projections par rapport à l'action de $K$.

\bigskip

Cette thèse débute par deux chapitres dont l'objectif principal est de regrouper des références bibliographiques et de poser certaines notations.

\bigskip

Le Chapitre 1 propose de brefs rappels sur la notion d'action hamiltonienne, le Théorème de Convexité hamiltonienne et le lien entre polytope moment hamiltonien et polytope moment algébrique. Ce chapitre se termine sur la définition de la projection d'orbite et sur quelques exemples fondamentaux, dont le fameux problème de Horn.

\bigskip

Dans le Chapitre \ref{chap:espaces_symétriques_hermitiens}, nous rappelons la définition d'espace symétrique hermitien et introduisons la notion d'orbite coadjointe holomorphe, dont la projection d'orbite est le thème central de cette thèse. Le Théorème de Schmid est ensuite énoncé pour les espaces symétriques hermitiens irréductibles. Nous posons également les notations pour les systèmes de racines des groupes $Sp(2n,\R)$, $SO^*(2n)$, $SU(p,q)$ et $SO(p,2)$, qui sont les groupes classiques de la classification des espaces symétriques hermitiens irréductibles.


\bigskip

Nous appliquons ensuite, dans le Chapitre 3, deux méthodes différentes permettant de donner plusieurs exemples de calculs de projections d'orbites coadjointes holomorphes.

La première méthode consiste à utiliser la formule de Duflo-Heckman-Vergne afin de déterminer la projection d'orbite par le calcul du support de l'image de la mesure de Liouville. 
Cette étude n'est réalisée qu'en petit rang, pour les groupes $SU(2,1)$ et $Sp(4,\R)$, car les calculs sont très vite difficiles.

La seconde méthode revient à déterminer le polyèdre de la projection d'orbite par ses points rationnels. Ceux-ci sont décrits en termes de produits tensoriels de représentations irréductibles du sous-groupe compact maximal $K$. On calcule les projections d'orbites pour les groupes classiques simples $Sp(2n,\R)$ et $SU(n,1)$ pour $n\geqslant 2$, ainsi que $SO^*(6)$, en appliquant les résultats de Knutson et Tao \cite{knutsontao2001} résolvant le problème de Horn, et de Klyachko \cite{klyachko98} pour sa version en termes de produits tensoriels de représentations irréductibles.

\bigskip

Le Chapitre \ref{chap:symplecto_mcduff} apporte une généralisation du symplectomorphisme de McDuff au cas des orbites coadjointes holomorphes. Pour toute orbite coadjointe elliptique $\Orb_{\Lambda} := G\cdot\Lambda$, avec $\Lambda\in\got{k}^*$, la décomposition de Cartan donne un difféomorphisme $K$-équivariant entre l'orbite $\Orb_{\Lambda}$ et la variété produit $K\cdot\Lambda\times\got{p}$. Ces deux variétés possèdent chacune une structure symplectique privilégiée. Celle de l'orbite coadjointe non compacte $\Orb_{\Lambda}$ est la forme symplectique de Kirillov-Kostant-Souriau $\Omega_{\Orb_{\Lambda}}$. Quant à $K\cdot\Lambda\times\got{p}$, on la munit du produit direct de la forme symplectique de Kirillov-Kostant-Souriau sur l'orbite compacte $K\cdot\Lambda$ et de la forme symplectique linéaire $\Omega_{\got{p}}$ sur $\got{p}$. Ces deux variétés symplectiques sont liées par le théorème suivant.

\begin{theoIntro}
\label{theoIntro:symplecto_mcduff}
Soit $\Lambda\in\got{k}^*$ tel que l'orbite $\Orb_{\Lambda}$ soit holomorphe. Alors il existe un symplectomorphisme $K$-équivariant entre les variétés symplectiques $(\Glambda,\Omega_{G\cdot\Lambda})$ et $(\Klambda\times\got{p},\Omega_{K\cdot\Lambda\times\got{p}})$ qui envoie $(k\Lambda,0)\in\Klambda\times\got{p}$ sur $k\Lambda\in G\cdot\Lambda$, pour tout $k\in K$.
\end{theoIntro}

L'espace symétrique hermitien peut être identifié à une certaine orbite holomorphe $\Orb_{\Lambda_0}$, où $\Lambda_0$ est un élément de $\got{k}^*$ fixé par l'action de $K$ (il correspond donc à un élément du centre de $\got{k}$). Par cette identification, on retrouve bien le Théorème de McDuff \cite{mcduff}.

Le Théorème \ref{theoIntro:symplecto_mcduff} permet alors de donner une autre description du polyèdre moment de la projection de l'orbite $\Orb_{\Lambda}$. La notation $\Delta_K(\Glambda)$ (resp. $\Delta_K(\Klambda\times\got{p})$) désigne le polyèdre moment de la projection de l'orbite $\Orb_{\Lambda}$ sur $\got{k}^*$ (resp. le polyèdre moment associé à la variété hamiltonienne $K\cdot\Lambda\times\got{p}$).

\begin{coroIntro}
\label{coroIntro:égalité_polyèdresmoments}
Les polyèdres moments $\Delta_K(\Orb_{\Lambda})$ et $\Delta_K(\Klambda\times\got{p})$ sont égaux.
\end{coroIntro}

Ce dernier résultat a été également prouvé par Paradan \cite{Paradan}, dans le cas où l'élément $\Lambda$ est un poids d'un tore maximal de $K$, par des méthodes de quantification géométrique, faisant intervenir la théorie de l'indice.

\bigskip

Le Corollaire \ref{coroIntro:égalité_polyèdresmoments} nous amène à étudier, dans le Chapitre \ref{chap:projectiondorbitecoadjointe+GIT}, le polytope moment des variétés hamiltoniennes de la forme $(K\cdot\Lambda\times E, \Omega_{K\cdot\Lambda\times E}, \Phi_{K\cdot\Lambda\times E})$, pour $E$ une représentation complexe de $K$. La forme symplectique $\Omega_{K\cdot\Lambda\times E}$ est obtenue comme produit de la forme symplectique de Kirillov-Kostant-Souriau sur l'orbite coadjointe $K\cdot\Lambda$, et d'une forme symplectique linéaire $K$-invariante $\Omega_E$ sur $E$. L'application $\Phi_{K\cdot\Lambda\times E}:K\cdot\Lambda\times E\rightarrow\got{k}^*$ désigne ici l'application moment canonique qui en découle et $\Phi_E:E\rightarrow\got{k}^*$ désigne l'application moment canonique pour $(E,\Omega_E)$ (voir section \ref{section:polydremoment_KLambdaE}).

Lorsque $\Phi_E$ est propre, le Théorème de Convexité hamiltonienne implique que l'ensemble, noté $\DGIT$, des points rationnels de $\Delta_K(\Klambda\times E)$ est dense dans le polyèdre moment $\Delta_K(\Klambda\times E)$. Par des propriétés de quantification géométrique sur $K\cdot\Lambda\times E$, on obtient la description suivante de $\DGIT$. Dans l'énoncé ci-dessous, $\wedge^*$ désigne le réseau des poids d'un tore maximal $T$ fixé dans $K$ et $\wedge^*_{\Q,+}$ désigne l'ensemble des poids rationnels dominants pour le choix d'une chambre de Weyl $\got{t}_+^*$ dans $\got{t}^*$, dual de l'algèbre de Lie de $T$.

\begin{theoIntro}
\label{theoIntro:theoC}
Supposons que $\Phi_E: E\rightarrow \got{k}^*$ est propre. Pour tout $\Lambda\in\wedge^*_{\Q,+}$, on a
\[
\Delta_K^{\mathrm{GIT}}(K\cdot\Lambda\times E) = \left\{\mu\in\wedge^*_{\Q,+} \left|
\begin{array}{l}
\exists N>0 \text{ t.q. } (N\mu,N\Lambda)\in(\wedge^*)^2\mbox{ et}  \\
\left(V^{K*}_{N\mu}\otimes V^K_{N\Lambda}\otimes \C[E]\right)^K \neq 0
\end{array}\right.\right\}.
\]
\end{theoIntro}


Cette nouvelle présentation de l'ensemble $\DGIT$ nous conduit à étudier le cône
\[
\Pi_{\Q}(E) = \left\{ (\mu,\nu)\in(\wedge^*_{\Q,+})^2 \left|
\begin{array}{l}
\exists N>0 \text{ t.q. } (N\mu,N\nu)\in(\wedge^*)^2\mbox{ et} \\
\left(V_{N\mu}^{\KC*}\otimes V_{N\nu}^{\KC*}\otimes \C[E]\right)^{\KC} \neq 0
\end{array}\right.\right\}
\]
où $\KC$ est le complexifié de $K$.

Dans ce cadre, l'ensemble $\Pi_{\Q}(E)$ est la projection linéaire d'un objet provenant de la Théorie Géométrique des Invariants (abrégé par l'acronyme \emph{GIT} en anglais). Cet objet est le cône semi-ample $C_{\Q}(\XEC)^+$ de la variété projective $\XEC := \KC/B\times\KC/B\times\mathbb{P}(E\oplus\C)$, où $B$ est un sous-groupe de Borel contenant le complexifié $\TC$ du tore maximal $T$. De manière plus générale, si $X$ est une variété algébrique projective, $\Csa$ est le cône engendré par les classes d'isomorphisme des fibrés semi-amples $\KC$-linéarisés $\mathcal{L}$ dont l'ensemble $\Xss(\mathcal{L})$ des points semi-stables, associé à $\mathcal{L}$, est non vide.
L'ensemble $C_{\Q}(\XEC)^+$ est un cône convexe polyédral de l'espace vectoriel rationnel $\PicG(X)_{\Q}$, nous permettant d'affirmer que $\Pi_{\Q}(E)$ et $\DGIT$ sont des polyèdres convexes.

Les résultats de Ressayre \cite{ressayre08} s'appliquent ici, nous donnant un ensemble d'équations déterminant complètement le cône semi-ample $C_{\Q}(\XEC)^+$. Ses équations sont indexées par les paires bien couvrantes de la variété $\XEC$ définies dans \cite{ressayre08}. Il s'agit des paires de la forme $(C(w,w',m),\lambda)$, où $\lambda$ est un sous-groupe à un paramètre de $\TC$, $(w,w',m)\in W\times W\times\Z$ avec $W$ le groupe de Weyl de $\KC$ relativement au tore maximal $\TC$ et $C(w,w',m)$ est une composante irréductible des points de $\XEC$ fixés par $\lambda$. On en déduit un ensemble d'équations du cône convexe polyédral $\Pi_{\Q}(E)$.

On note $\mathcal{P}_0(E)$ le sous-ensemble des paires bien couvrantes $(C(w,w',m),\lambda)$ de $\XEC$ telles que $m=0$ et $\lambda$ est dominant indivisible et $E$-admissible. La $E$-admissibilité est une propriété liée aux poids de l'action de $\TC$ sur $E$.

\begin{theoIntro}
\label{theo:équations_PiQ}
Soit $\zeta:\KC\rightarrow GL(E)$ une représentation de noyau fini, telle que l'application moment associée $\Phi_E:E\rightarrow\got{k}^*$ soit propre. Un couple $(\mu,\nu)\in(\wedge^*_{\Q,+})^2$ appartient à $\Pi_{\Q}(E)$ si et seulement si, pour toute paire $(C(w,w',0),\lambda)$ de $\mathcal{P}_0(E)$, on a
\[
\langle w\lambda,\mu\rangle + \langle w'\lambda,\nu\rangle \leqslant 0.
\]
\end{theoIntro}

Du Théorème \ref{theo:équations_PiQ}, on obtient directement un ensemble d'équations pour le polyèdre $\DGIT$.

\begin{coroIntro}
\label{coroIntro:équationsgénérales_DGIT}
Soit $\zeta:\KC\rightarrow GL(E)$ une représentation de noyau fini, telle que l'application moment associée $\Phi_E:E\rightarrow\got{k}^*$ soit propre. Si $\Lambda\in\wedge^*_{\Q,+}$, alors on a
\[
\DGIT = \left\{\xi\in\wedge_{\Q,+}^*;\ \langle w\lambda,\xi\rangle \leqslant \langle w_0w'\lambda,\Lambda\rangle, \forall (C(w,w',0),\lambda)\in\mathcal{P}_0(E)\right\}.
\]
\end{coroIntro}

Ces équations donnent aussi les équations de $\Delta_K(K\cdot\Lambda\times E)$.

%

\bigskip

Le Chapitre \ref{chap:PairesBienCouvrantes} est consacré à l'amélioration de la condition nécessaire et suffisante, donnée dans \cite{ressayre08}, pour qu'une paire $(C(w,w',m),\lambda)$ de $X_M$ soit bien couvrante, où $M$ est un $\KC$-module et $X_M=\KC/B\times\KC/B\times\mathbb{P}(M)$. Cette nouvelle équivalence se décompose en deux conditions : une condition cohomologique, en terme de produits cup de classes de Schubert de la variété des drapeaux $\KC/B$, et une condition linéaire sur $\lambda$ faisant intervenir $w$, $w'$ et $m$.

Fixons quelques notations. On note $W_{\lambda}$ le stabilisateur de $\lambda$ dans $W$, et $W^{\lambda}$ l'ensemble des représentants de longueur maximale des classes de $W/W_{\lambda}$. Le plus long élément de $W$ (resp. $W_{\lambda}$) sera noté $w_0$ (resp. $w_{\lambda}$). \`A chaque $w$ appartenant à $W$, on pourra associer une sous-variété $\overline{BwB/B}$ de $\KC/B$, appelée variété de Schubert. Les classes fondamentales des variétés de Schubert forment une base du $\Z$-module libre $\homol{\KC/B}$, permettant de définir une base duale $\{\sigma_{w}^B;w\in W\}$ dans $\coh{\KC/B}$. On pose $\Theta:\wedge^*\rightarrow\mathrm{H}^2(\KC/B,\Z)$ le morphisme qui envoie un poids $\mu$ de $\TC$ sur la première classe de Chern du fibré en droites $\mathcal{L}_{\mu}$ sur $\KC/B$ de poids $\mu$. Pour tout entier relatif $m$, on pose
\[
M_{<m} := \bigoplus_{k<m}\{v\in M; \forall t\in\C^*, \lambda(t)\cdot v = t^k v\}.
\]
De plus, $\WT(M_{<m})$ désigne l'ensemble des poids de l'action de $\TC$ sur $M_{<m}$. Enfin, $\rho$ dénote la demi-somme des racines positives de $\got{k}_{\C}$.

Nous prouvons le théorème suivant.

\begin{theoIntro}
\label{theoIntro:cns_pairebiencouvrante}
Soit $(w,w',m)\in W^{\lambda}\times W^{\lambda}\times\Z$ tel que $C(w,w',m)$ est non vide. La paire $(C(w,w',m),\lambda)$ de $X_{M}$ est bien couvrante si et seulement si, soit $w' = w_0ww_{\lambda}$ et $M_{< m} = 0$, soit les deux assertions suivantes sont simultanément vérifiées:
\begin{enumerate}
\item[(\emph{i})] $\sigma_{w_0w}^B\,.\,\sigma_{w_0w'}^B\,.\,\prod_{\beta\in\WT(M_{< m})}\Theta(-\beta)^{n_{\beta}}= \sigma_{w_0w_{\lambda}}^B$,
\item[(\emph{ii})] $\langle w\lambda + w'\lambda,\rho\rangle + \sum_{k<m}(m-k)\dim_{\C}(M_{\lambda,k}) = 0$.
\end{enumerate}
\end{theoIntro}

Par ailleurs, un corollaire du Théorème \ref{theoIntro:cns_pairebiencouvrante} est utilisé pour prouver le Théorème \ref{theo:équations_PiQ}. Cet argument est indispensable pour permettre un bon comportement des équations lors de la projection linéaire, pour passer des équations du cône semi-ample à celles de $\Pi_{\Q}(E)$.
%

\bigskip

Le Chapitre \ref{chap:calcul_exemples_par_pairesbiencouvrantes} termine l'étude des équations du polyèdre moment associé à la projection d'orbites coadjointes holomorphes. Nous y considérons toujours un groupe de Lie $G$ réel, semi-simple, connexe, non compact et à centre fini, et $K$ un sous-groupe compact maximal, donné par la décomposition de Cartan au niveau des algèbres de Lie $\got{g} = \got{k}\oplus\got{p}$.

Dans le cas où l'espace symétrique hermitien $G/K$ est irréductible, c'est-à-dire que $G$ est simple, le Corollaire \ref{coroIntro:équationsgénérales_DGIT} s'applique à la projection d'orbites coadjointes holomorphes de $G$ et les équations obtenues sont décrites dans l'énoncé du théorème suivant.

En effet, l'espace vectoriel réel $\got{p}$ possède alors des structures complexes $K$-invariantes données par la représentation adjointe sur $\got{p}$ de certains éléments du centre de $\got{k}$. Rappelons qu'ici, l'action de $K$ sur $\got{p}$ est induite par l'action adjointe. On note $\got{p}^+$ l'espace vectoriel complexe $\got{p}$ pour le choix d'une telle structure complexe $K$-invariante. On désigne alors par $\got{p}^-$ l'espace vectoriel $\got{p}$ muni de la structure complexe opposée à celle de $\got{p}^+$. Dans l'énoncé ci-dessous, $\mathcal{P}_0(\got{p}^-)$ désigne l'ensemble des paires bien couvrantes $\mathcal{P}_0(E)$ pour $E=\got{p}^-$.

\begin{theoIntro}[\'Equations de $\Delta_K(\Orb_{\Lambda})$]
\label{theoIntro:équations_DeltaK_G0Lambda}
On suppose que $G$ est un groupe de Lie réel, simple, connexe, non compact, à centre fini et que $G/K$ est hermitien. Soit $\Lambda\in\wedge^*_{\Q}$ tel que l'orbite $\Orb_{\Lambda}$ est holomorphe. Alors, un élément $\xi$ de $\got{t}_+^*$ appartient à $\Delta_K(\Orb_{\Lambda})$ si et seulement s'il vérifie les équations
\[
\langle w\lambda,\xi\rangle \leqslant \langle w_0w'\lambda,\Lambda\rangle
\]
pour toute paire $(C(w,w',0),\lambda)$ de $\mathcal{P}_0(\got{p}^-)$.
\end{theoIntro}

La structure complexe sur $\got{p}$, considérée dans le Théorème \ref{theoIntro:équations_DeltaK_G0Lambda} pour pouvoir appliquer le Corollaire \ref{coroIntro:équationsgénérales_DGIT}, est intimement liée à la forme symplectique $\Omega_{\got{p}}$ intervenant dans le Théorème \ref{theoIntro:symplecto_mcduff} et le Corollaire \ref{coroIntro:égalité_polyèdresmoments}. Il est à noter que les poids de l'action de $\TC$ sur $\got{p}^-$ forment, dans cette situation, un système de racines non compactes négatives de $\got{g}$.

Le Théorème \ref{theoIntro:équations_DeltaK_G0Lambda} donne une méthode générale pour trouver un ensemble d'équations caractérisant la projection d'une orbite holomorphe lorsque $G$ est simple. Cependant, le calcul effectif de cet ensemble d'équations sur un exemple précis exige de déterminer les paires bien couvrantes associées. Cela consiste en deux étapes :
\begin{enumerate}
\item Déterminer tous les sous-groupes à un paramètre $\lambda$ de $\TC$ qui sont dominants, indivisibles et $\got{p}$-admissibles. Nous réalisons le calcul systématique de ces $\lambda$ pour tous les groupes classiques de la classification, c'est-à-dire pour les groupes $Sp(2n,\R)$, $SO^*(2n)$, $SU(p,q)$ et $SO(2p,2)$.
\item Déterminer, pour chacun de ces sous-groupes à un paramètre $\lambda$ obtenus ci-dessus, les paires associées qui sont bien couvrantes. Nous effectuons ce calcul pour les groupes $Sp(2n,\R)$, $SU(n,1)$, $SO^*(6)$, $SO^*(8)$ et aussi pour $SU(2,2)$.
\end{enumerate}
Ces calculs concluent la partie principale de la thèse.

\bigskip

Le travail d'amélioration du critère de Ressayre, pour le compte du Théorème \ref{theoIntro:cns_pairebiencouvrante}, nécessite de connaître certains produits cup entre des classes de Schubert bien spécifiques, dans la variété des drapeaux complète de $GL_n(\C)$. L'Annexe \ref{chap:combinatoire_gpdeWeyl_de_GL_r(C)} recueille tous les calculs de combinatoire dans le groupe de Weyl de $GL_n(\C)$ nécessaires et introduit la formule de Chevalley, qui sera ensuite appliquée pour déterminer les produits cup utilisés au Chapitre \ref{chap:PairesBienCouvrantes}.

L'Annexe \ref{chap:classefondamentale} présente la définition de la classe fondamentale d'une variété projective éventuellement singulière. Elle est utilisée pour définir proprement les classes de Schubert, qui sont largement utilisées dans le Chapitre \ref{chap:PairesBienCouvrantes} et le Chapitre \ref{chap:calcul_exemples_par_pairesbiencouvrantes}.

Cette thèse se clôt sur l'Annexe \ref{chap:exemples_de_triplets_pour_Horn}, prouvant que des triplets $(I,J,L)$ très particuliers apparaissent comme triplets d'indices des équations de Horn. Ces triplets interviennent dans le calcul de la projection d'orbites coadjointes holomorphes de $SU(n,1)$, effectué au Chapitre \ref{chap:premières_tentatives}.

\mainmatter


\chapter*{Notations}
\label{chap:notations}

Sauf mention du contraire, $G$ désignera un groupe de Lie réel connexe généralement non compact, $K$ un groupe de Lie compact connexe et $T$ un tore maximal de $K$.

On notera $\KC$ le groupe de Lie (et algébrique) réductif obtenu par complexification de $K$ et $\TC$ le tore maximal de $\KC$ tel que $\TC\cap K = T$. La lettre $B$ désignera toujours un sous-groupe de Borel de $\KC$.

On dénotera par $\got{g}$, $\got{k}$, $\got{t}$, $\got{k}_{\C}$, $\got{t}_{\C}$ et $\got{b}$ leurs algèbres de Lie respectives.

Le groupe $W$ sera le groupe de Weyl $W:=N_K(T)/T = N_{\KC}(\TC)/\TC$ de $K$ relativement à son tore maximal $T$. On fixe une chambre de Weyl $\got{t}_+^*$ de $K$ dans $\got{t}^*$.

On désignera par $\wedge^*$ le réseau des poids de $\got{t}^*$. On prend la convention suivante: c'est l'ensemble des éléments $\frac{1}{i}\alpha$, où $\alpha$ est la différentielle en l'identité d'un caractère du tore $T$. Les éléments de $\wedge^*$ seront appelés les éléments entiers (ou poids) de $\got{t}^*$. On notera $\wedge^*_{\Q}:=\wedge^*\otimes_{\Z}\Q$ le sous-espace vectoriel rationnel engendré par le réseau $\wedge^*$. Les éléments dominants seront les poids (resp. poids rationnels) de $\wedge_+^*:=\wedge^*\cap\got{t}_+^*$ (resp. $\wedge_{\Q,+}^*:=\wedge_{\Q}^*\cap\got{t}_+^*$).

Pour tout $\mu\in\wedge_+^*$, la représentation irréductible de $K$ de plus haut poids $\mu$ sera notée $V_{\mu}^K$, voire tout simplement $V_{\mu}$ s'il n'y a pas d'ambiguïté sur le groupe.

Le groupe de Weyl $W$ a un unique plus long élément, noté $w_0$. Pour tout poids dominant $\mu$, $\mu^*$ désignera le poids dual de $\mu$, c'est-à-dire, l'unique poids dominant qui vérifie $V_{\mu^*}^K \cong V_{\mu}^{K*}$ en tant que $K$-modules. De manière générale, on notera $\mu^* := -w_0\mu\in\got{t}_+^*$ l'élément dual de $\mu\in\got{t}_+^*$.



\chapter{Polytope moment et projection d'orbite}

Ce premier chapitre rappelle les notions d'action hamiltonienne, d'application moment et de polytope moment associée. Il introduit également la version algébrique du polytope moment et fait le lien entre les deux cadres. Nous terminons le chapitre en donnant l'exemple fondamental de la projection d'orbite coadjointe, qui sera le sujet central de cette thèse.

\section{Polytope moment: le cadre hamiltonien}

\subsection{Actions hamiltoniennes et Théorème de convexité}

Soient $G$ un groupe de Lie réel, et $(M,\Omega)$ une variété symplectique munie d'une action de $G$. Tout élément $g$ du groupe $G$ définit un difféomorphisme $m\in M\mapsto g\cdot m\in M$. La forme symplectique est dite $G$-invariante si elle est préservée par les difféomorphismes de $M$ dans $M$ définis par les éléments de $G$, c'est-à-dire :
\[
g^*\Omega = \Omega, \quad\text{pour tout $g\in G$}.
\]
Le groupe $G$ agit alors par symplectomorphismes, et on dit que l'action de $G$ sur $(M,\Omega)$ est \emph{symplectique}\index{Action symplectique}.

\begin{defi}
Une \emph{application moment}\index{Application moment} pour une variété symplectique $(M,\Omega)$ munie d'une action symplectique du groupe de Lie $G$ est une application $\Phi : M\rightarrow\got{g}^*$ lisse, $G$-équivariante pour l'action coadjointe sur $\got{g}^*$ et qui vérifie
\[
d\langle\Phi,X\rangle = -\imath(X_M)\Omega \quad\text{pour tout $X\in\got{g}$},
\]
où $\langle\Phi,X\rangle$ désigne la fonction $m\in M\mapsto \langle\Phi(m),X\rangle\in\R$ et $X_M$ est le champ de vecteurs fondamental associé à $X\in\got{g}$.
\end{defi}

\begin{defi}
Une action symplectique du groupe de Lie $G$ sur une variété symplectique $(M,\Omega)$ est dite \emph{hamiltonienne}\index{Action hamiltonienne} s'il existe une application moment $\Phi: M\rightarrow\got{g}^*$ pour cette action.
\end{defi}

\begin{exple}
Soit $(V,\Omega)$ un $\R$-espace vectoriel symplectique. Soit $G$ un groupe de Lie réel agissant linéairement sur $V$ tel que $\Omega$ soit $G$-invariante. Cette action est alors hamiltonienne. Une application moment pour cette action est
\[
\begin{array}{cccl}
\Phi_V : & V & \longrightarrow & \got{g}^* \\
& v& \longmapsto & (X\in\got{g}\mapsto \frac{1}{2}\Omega(v,Xv)\in\R),
\end{array}
\]
où $Xv$ désigne l'action infinitésimale de $X\in\got{g}$ sur l'élément $v$ de $V$.
\end{exple}

\begin{exple}
\label{exmple:espacevectorielhermitien_est_variétéhamiltonienne}
Soit maintenant $(V,h)$ un espace vectoriel hermitien. La partie imaginaire $\imag h$ définit une forme symplectique sur $V$ et sa partie réelle $\re h$ un produit scalaire. Cette forme symplectique sur $V$ induit une structure symplectique sur l'espace projectif $\mathbb{P}(V)$ des droites vectorielles complexes de $V$, notée $\omega_{\mathbb{P}(V)}$, définie comme suit : l'espace tangent à $\mathbb{P}(V)$ en $[v]$, $v\in V\setminus\{0\}$ s'identifie naturellement avec l'orthogonal hermitien à $[v]$. La forme symplectique est alors donnée par
\[
\omega_{\mathbb{P}(V)}|_{[v]}(v_1,v_2) = \frac{\imag h(v_1,v_2)}{h(v,v)}.
\]
Si $G$ est un groupe de Lie qui agit linéairement sur $V$ en préservant la forme hermitienne $h$, alors l'action induite sur $\mathbb{P}(V)$ est aussi hamiltonienne. Elle bénéficie d'une application moment canonique
\[
\langle\Phi_{\mathbb{P}(V)}([v]),X\rangle = \frac{\imag h(v,Xv)}{h(v,v)}.
\]
De plus, la structure de $G$-variété hamiltonienne sur $\mathbb{P}(V)$ induit par restriction une structure de $G$-variété hamiltonienne sur toute sous-variété complexe $G$-invariante $X\subset\mathbb{P}(V)$.
\end{exple}

\begin{exple}
Soit $S^2=\{x^2+y^2+z^2=1\}$ la sphère unité de $\R^3$. Notons $\Omega$ la $2$-forme sur $\R^3$ $xdy\wedge dz - ydx\wedge dz + zdx\wedge dz = \det\bigl((x,y,z),\cdot,\cdot\bigr)$, qui se restreint en une forme symplectique sur $S^2$ invariante par rotation dans $\R^3$. Cette action induit une action de $S^1$ sur $S^2$ par rotation autour de l'axe des $z$. Elle est hamiltonienne, d'application moment $(x,y,z)\in S^2\mapsto z\in\R$.
\end{exple}

Dans différentes situations, la structure hamiltonienne a une propriété géométrique très forte : l'image de l'application moment, ou plus exactement, un certain domaine fondamental canonique de l'image de l'application moment, est convexe, voire, dans certains bons cas, un polyèdre convexe.

La situation que l'on rencontrera le plus souvent fait intervenir un groupe de Lie compact connexe $G=K$ et une variété symplectique connexe. On fixe $T$ un tore maximal de $K$ et $\got{t}_+^*$ une chambre de Weyl du dual $\got{t}^*$ de l'algèbre de Lie de $T$.

\begin{theo}[Atiyah, Guillemin-Sternberg, Kirwan, Lerman et al.]
\label{theo:convexitéhamiltonienne}
\index{Théorème de convexité hamiltonienne}
Soit $(M, \Omega)$ une variété connexe munie d'une action hamiltonienne d'un groupe de Lie compact $K$, avec une application moment propre $\Phi: M \rightarrow\got{k}^*$.
\begin{enumerate}
\item L'ensemble $\Delta_K(M) := \Phi(M)\cap\got{t}^*$ est un ensemble convexe localement polyédral. En particulier, si $M$ est compact, $\Delta_K(M)$ est un polytope convexe.
\item Chaque fibre de l'application moment $\Phi$ est connexe.
\end{enumerate}
\end{theo}

La version abélienne, c'est-à-dire lorsque $K=T$ est un tore, avec l'hypothèse de compacité sur la variété, a été prouvée simultanément par Atiyah \cite{atiyah} d'un côté, Guillemin et Sternberg \cite{guillemin82} de l'autre. Ensuite, Kirwan \cite{kirwan} a généralisé ce résultat au cas non abélien, mais toujours avec une variété compact. D'autres versions on été proposées ensuite, comme \cite{hilgert,weinstein} par exemple. L'énoncé ci-dessus a été prouvé par Lerman, Meinrenken, Tolman et Woodward dans \cite{lerman98}.

\subsection{Orbites coadjointes}
\label{subsection:orbites_coadjointes}

Soit $G$ un groupe de Lie réel connexe. Le groupe $G$ agit sur $\got{g}^*$, dual de son algèbre de Lie, par l'action coadjointe. Ses orbites sont appelées \emph{orbites coadjointes}.

Soit $\Lambda\in\got{g}^*$. Notons $G_{\Lambda}$ le stabilisateur de $\Lambda$ dans $G$ et $\got{g}_{\Lambda}$ son algèbre de Lie. L'application $g\in G\mapsto g\cdot\Lambda\in\got{g}^*$ induit un difféomorphisme $G$-équivariant de $G/G_{\Lambda}$ sur l'orbite coadjointe $\Orb_{\Lambda} := G\cdot\Lambda$.

L'espace vectoriel quotient $\got{g}/\got{g}_{\Lambda}$ est muni d'une action de $G_{\Lambda}$ induite par l'action adjointe. Il est connu que nous avons un isomorphisme de $G$-fibrés vectoriels sur $G/G_{\Lambda}\simeq G\cdot\Lambda$
\begin{equation}
\label{eq:difféo_fibrétangent_espacehomogène}
\begin{array}{cccl}
& G\times_{G_{\Lambda}}\got{g}/\got{g}_{\Lambda} & \longrightarrow &  T\Orb_{\Lambda} \\
& [g,\overline{X}] & \longmapsto & \frac{d}{dt}(g\exp(tX)\cdot\Lambda)|_{t=0},
\end{array}
\end{equation}
où $\overline{X}$ désigne la classe de $X\in\got{g}$ modulo $\got{g}_{\Lambda}$. Pour tout $g\in G$, nous noterons $l_g:h\cdot\Lambda\in\Orb_{\Lambda}\rightarrow (gh)\cdot\Lambda\in\Orb_{\Lambda}$ le difféomorphisme induit par la multiplication à gauche par $g$ dans $G$. On peut aisément vérifier que la différentielle de cette application satisfait la propriété suivante,
\begin{equation}
\label{eq:propriété_différentielle_de_lg_générale}
dl_g(h\cdot\Lambda)[h,\overline{X}] = [gh,\overline{X}],
\end{equation}
pour tous $g,h\in G$ et tout $X\in\got{g}$. En particulier, on a la relation
\begin{equation}
\label{eq:propriété_différentielle_de_lg}
dl_g(\Lambda)[e,\overline{X}] = [g,\overline{X}], 
\end{equation}
pour tout $g\in G$ et tout $X\in\got{g}$.

La forme symplectique de Kirillov-Kostant-Souriau est une forme symplectique définie canoniquement sur une orbite coadjointe $\Orb_{\Lambda}$ de $\got{g}^*$.

Dans l'énoncé suivant, $\Omega^p(\Orb_{\Lambda})$ désigne l'espace vectoriel des $p$-formes différentielles sur la variété $\Orb_{\Lambda}$, et $\Omega^p(\got{g}/\got{g}_{\Lambda})$ l'ensemble des formes $p$-linéaires alternées sur l'espace vectoriel $\got{g}/\got{g}_{\Lambda}$.

\begin{prop}
Pour tout $p\geqslant 0$, l'application restriction $\omega\in\Omega^p(\Orb_{\Lambda})\mapsto \omega|_{\Lambda}\in\Omega^p(\got{g}/\got{g}_{\Lambda})$ induit un isomorphisme entre l'espace des $p$-formes différentielles $G$-invariantes sur $\Orb_{\Lambda}$ et l'espace des formes $p$-linéaires alternées $G_{\Lambda}$-invariantes sur $\got{g}/\got{g}_{\Lambda}$.
\end{prop}

\begin{proof}
Clairement, $\omega|_{\Lambda}$ est $G_{\Lambda}$-invariante si $\omega$ est $G$-invariante, puisque $\Lambda$ est fixé par l'action induite de $G_{\Lambda}$.

Inversement, si une forme $p$-linéaire alternée $b:(\got{g}/\got{g}_{\Lambda})^p\rightarrow\R$ est $G_{\Lambda}$-invariante, on peut définir une $p$-forme différentielle $\omega$ sur $\Orb_{\Lambda}$, grâce au difféomorphisme \eqref{eq:difféo_fibrétangent_espacehomogène}, en posant
\[
\omega|_{g\Lambda}([g,\overline{X}_1],\ldots,[g,\overline{X}_p]) := b(X_1,\ldots,X_p), \quad \text{pour tous $X_1,\ldots,X_p\in\got{g}$}.
\]
Or, la relation \eqref{eq:propriété_différentielle_de_lg} nous donne alors
\[
(l_g^*\omega)|_{\Lambda}([e,\overline{X}_1],\ldots,[e,\overline{X}_p]) = \omega|_{g\Lambda}([g,\overline{X}_1],\ldots,[g,\overline{X}_p]) = b(X_1,\ldots,X_p),
\]
pour tous $X_1,\ldots,X_p\in\got{g}$. La forme différentielle $\omega$ est donc bien $G$-invariante grâce à la relation générale \eqref{eq:propriété_différentielle_de_lg_générale}.
\end{proof}

\begin{rema}
Ceci est encore vrai quand on regarde plus généralement les espaces de $(p,q)$-tenseurs. Par exemple, toute métrique riemannienne $G$-invariante sur $\Orb_{\Lambda}$ provient d'un unique produit scalaire $G_{\Lambda}$-invariant sur $\got{g}/\got{g}_{\Lambda}$.
\end{rema}

\begin{defi}
\label{defi:formesymplectique_kirillovkostantsouriau}
Soit $\Lambda$ un élément de $\got{g}^*$. La \emph{forme symplectique de Kirillov-Kostant-Souriau}\index{Forme symplectique de Kirillov-Kostant-Souriau} est l'unique $2$-forme différentielle sur l'orbite coadjointe $\Orb_{\Lambda}$, notée $\Omega_{\Orb_{\Lambda}}$, telle que
\[
(\Omega_{\Orb_{\Lambda}})|_{\Lambda}([e,\overline{X}],[e,\overline{Y}]) = \langle\Lambda,[X,Y]\rangle
\]
pour tous $X,Y\in\got{g}$. La $2$-forme $\Omega_{\Orb_{\Lambda}}$ est fermée, cette propriété se vérifiant grâce à l'identité de Jacobi sur l'algèbre de Lie $\got{g}$, cf \cite{audin} par exemple.
\end{defi}

La forme de Kirillov-Kostant-Souriau étant $G$-invariante, l'action induite de $G$ sur l'orbite coadjointe $(\Orb_{\Lambda},\Omega_{\Orb_{\Lambda}})$ est symplectique. Cette action est aussi hamiltonienne, avec pour application moment l'inclusion de l'orbite $\Orb_{\Lambda}$ dans $\got{g}^*$:
\[
\Phi_G : g\cdot\Lambda \in \Orb_{\Lambda} \longmapsto g\cdot\Lambda\in \got{g}^*.
\]

\begin{rema}
\label{rema:inclusionorbitepropre_ssi_orbitefermée}
L'application moment $\Phi_G$ est propre si et seulement si l'orbite coadjointe $\Orb_{\Lambda}$ est fermée dans $\got{g}^*$.
\end{rema}

\subsection{Préquantification}

Afin d'établir mathématiquement la notion de quantification utilisée par les physiciens, Kostant et Souriau ont introduit la préquantification des variétés hamiltoniennes.

Soit $G$ un groupe de Lie connexe et $(M,\Omega,\Phi)$ une $G$-variété hamiltonienne.

\begin{defi}
La $G$-variété hamiltonienne $(M,\Omega,\Phi)$ est préquantifiée si elle est pourvue d'un fibré en droites complexes hermitien $G$-équivariant $\mathcal{L}$ sur $M$ et d'une connexion $G$-invariante $\nabla$ dont la courbure est $-i\Omega$, tels que la condition de Kostant
\[
L_X - \nabla_{X_M} = i\langle\Phi,X\rangle
\]
soit vérifiée pour tout $X\in\got{g}$. On dit alors que $\mathcal{L}$ est \emph{un fibré de Kostant-Souriau}.
\end{defi}

Si $(M,\Omega,\Phi)$ est préquantifiée par le fibré de Kostant-Souriau $\mathcal{L}$, alors la première classe de Chern de $\mathcal{L}$ est la classe de $\frac{\Omega}{2\pi}$. De plus, l'existence d'un tel fibré implique que la classe de cohomologie $[\frac{\Omega}{2\pi}]$ est entière, c'est-à-dire, $[\frac{\Omega}{2\pi}]$ est dans l'image du morphisme $H_{\text{Cech}}(M,\Z)\rightarrow H_{\text{DeRham}}(M,\R)$ \cite{kostant70}.

\begin{exple}
\label{exple:fibré_kostantsouriau_pour_orbitecoadjointe}
Soient $\Lambda\in\got{g}^*$ et $\Orb_{\Lambda}\subset\got{g}^*$ l'orbite coadjointe de $G$ correspondante. On dit que la forme linéaire $\Lambda$ est \emph{entière} si $i\Lambda$ est la différentielle d'un caractère $\chi$ de $G_{\Lambda}$. Si ce caractère $\chi$ existe, il n'est pas forcément unique. Cependant, lorsque $G_{\Lambda}$ est connexe, on a unicité de $\chi$. Remarquons que, dans la suite, les éléments $\Lambda$ de $\got{g}^*$ considérés vérifieront toujours $G_{\Lambda}$ connexe. C'est le cas lorsque par exemple $G$ est compact connexe, ainsi que pour les orbites coadjointes holomorphes qui seront introduites dans le Chapitre \ref{chap:espaces_symétriques_hermitiens}.

Définissons $\C_{\chi}$ la représentation de $G_{\Lambda}$ dans $\C$ par multiplication par $\chi$. Le fibré en droites $G\times_{G_{\Lambda}}\C_{\chi}$ sur $G/G_{\Lambda}\cong\Orb_{\Lambda}$ est un fibré de Kostant-Souriau. En fait, par \cite{kostant70}, il s'agit de l'unique fibré de Kostant-Souriau sur $\Orb_{\Lambda}$. De plus, seules les orbites coadjointes $\Orb_{\Lambda}$ associées à une forme entière $\Lambda$ de $\got{g}^*$ sont munies d'un fibré de Kostant-Souriau.
\end{exple}

Remarquons que, lorsque $G$ est compact connexe et $T$ un tore maximal de $G$, les orbites préquantifiables (c'est-à-dire entières) paramètrent les représentations irréductibles de $G$. C'est le Théorème de Borel-Weil.

\section{Polytope moment: le cadre algébrique}
\label{section:polytopemoment_algébrique}

Nous allons maintenant décrire brièvement la version algébrique du polytope moment, introduite par Brion dans \cite{brion87}.

%

Soit $K$ un groupe de Lie compact connexe et $T$ un tore maximal de $K$. Prenons $X$ une variété algébrique projective complexe sur laquelle le groupe algébrique réductif $\KC$ agit rationnellement. Soit $\mathcal{L}$ un fibré en droites ample $\KC$-linéarisé sur $X$, c'est-à-dire, on a une action de $\KC$ sur l'espace total de $\mathcal{L}$ qui relève l'action de $\KC$ sur la base $X$ du fibré (la notion de fibré $\KC$-linéarisé sera plus détaillée dans le Chapitre \ref{chap:projectiondorbitecoadjointe+GIT}). Ainsi, l'espace des sections globales $H^0(X,\mathcal{L})$ de $\mathcal{L}$ est un $\KC$-module. Le but est de déterminer comment se décompose ce $\KC$-module en somme de représentations irréductibles de $\KC$.


\begin{defi}
On appelle \emph{polytope moment algébrique} associé à $(X,\mathcal{L})$ l'ensemble
\[
\Delta_K^{\text{alg}}(X,\mathcal{L}) := \{\mu\in\wedge^*_{\Q,+}; V_{N\mu}^{\KC}\subset H^0(X,\mathcal{L}^{\otimes N}) \text{ pour un certain entier } N\geqslant 1\}.
\]
\end{defi}

La principale propriété géométrique de cet ensemble est donnée par le résultat suivant.

\begin{prop}[\cite{brion87}]
L'ensemble $\Delta_K^{\mathrm{alg}}(X,\mathcal{L})$ est un polytope convexe de $\wedge^*_{\Q}$.
\end{prop}

La terminologie de \og{}polytope moment algébrique\fg{} n'est pas innocente. En effet, le polytope convexe rationnel $\Delta_K^{\text{alg}}(X,\mathcal{L})$ est fortement lié au polytope moment associé à une certaine $K$-variété hamiltonienne.

Soit $V$ un $\KC$-module rationnel de dimension finie et $X$ une sous-variété fermée, irréductible et $\KC$-stable de $\mathbb{P}(V)$. D'après l'Exemple \ref{exmple:espacevectorielhermitien_est_variétéhamiltonienne}, le choix d'un produit scalaire hermitien $h$ sur $V$, invariant par $K$, induit une structure de $K$-variété hamiltonienne sur $X$, d'application moment $\Phi_{X}$ donnée par
\[
\langle\Phi_{X}([v]),Y\rangle = \frac{\imag h(v,Yv)}{h(v,v)}\quad\text{pour tout $[v]\in X\subset\mathbb{P}(V)$ et tout $Y\in\got{k}$}.
\]

\begin{theo}[\cite{brion87}]
\label{theo:égalité_polymomenthamiltonien_et_polymomentalgébrique}
Le polytope moment algébrique $\Delta_K^{\mathrm{alg}}(X,\mathcal{O}_{X}(1))$ est l'ensemble des points rationnels de $\Delta_K(X)=\Phi_{X}(X)\cap\got{t}_+^*$. En particulier, $\Delta_K(X)$ est l'enveloppe convexe d'un nombre fini de points rationnels de $\wedge_{\Q,+}^*$ et
\[
\Delta_K(X) = \overline{\Delta_K^{\mathrm{alg}}(X,\mathcal{O}_X(1))}.
\]
\end{theo}

\section{Projection d'orbite}
\label{section:proj_orbites_defi_et_exemples}

L'étude des structures symplectique et hamiltonienne des orbites coadjointes, définies dans les paragraphes précédents, est très intéressante et apporte de nombreuses informations en particulier pour la théorie des représentations.

Lorsqu'on considère l'action d'un sous-groupe de $G$ sur les orbites coadjointes de $\got{g}^*$, la structure hamiltonienne restreinte à ce sous-groupe se décrit en terme de projection d'orbites, notion que nous allons expliquer dans cette section, en proposant quelques exemples fondamentaux.

\subsection{Définition générale}

Considérons toujours un groupe de Lie réel connexe $G$ et prenons maintenant $H$ un sous-groupe de Lie fermé connexe de $G$. Soit $\got{h}$ son algèbre de Lie.

Pour tout $\Lambda\in\got{g}^*$, l'action de $G$ sur $\Orb_{\Lambda}$ induit une action de son sous-groupe $H$ sur cette même orbite coadjointe. L'action de $H$ laisse évidemment invariante la forme symplectique $\Omega_{\Orb_{\Lambda}}$. On note $\Phi_H$ la composée de l'application moment $\Phi_G:\Orb_{\Lambda}\hookrightarrow\got{g}^*$ pour l'action hamiltonienne de $G$ sur $\Orb_{\Lambda}$ et de la projection linéaire $\got{g}^*\rightarrow\got{h}^*$ obtenue en transposant l'injection canonique $\got{h}\hookrightarrow\got{g}$. L'application $\Phi_H$ est une application moment de $(\Orb_{\Lambda},\Omega_{\Orb_{\Lambda}})$ pour l'action de $H$, c'est la \emph{projection de l'orbite coadjointe}\index{Projection d'orbite coadjointe} $\Orb_{\Lambda}$ relativement au sous-groupe de Lie $H$. Cette dernière application moment sera aussi notée $\Phi_{\Orb_{\Lambda}}$ s'il n'y a pas d'ambiguïté sur le groupe qui agit.

\subsection{$G$ compact}
\label{subsection:projection_orbite_compacte}

Lorsque $G=K$ est pris compact connexe, le sous-groupe $H$ est alors lui aussi compact connexe. La projection d'une orbite $\Phi_{\Orb_{\Lambda}}:\Orb_{\Lambda}\rightarrow\got{h}^*$, pour un $\Lambda\in\got{k}^*$ quelconque, rentre donc ici dans le cadre d'application du Théorème de Convexité Hamiltonienne de Kirwan. Ainsi, l'image de la projection d'orbite $\Phi_{\Orb_{\Lambda}}$ intersecte la chambre de Weyl $\got{t}_+^*$ de $H$ en un polytope convexe $\Delta_H(\Orb_{\Lambda}) := \Phi_H(\Orb_{\Lambda})\cap\got{t}_+^*$.

Historiquement, ce résultat a été prouvé par Kostant, une décennie avant la preuve du cadre général par Kirwan, dans le cas où le sous-groupe compact $H=T$ est un tore maximal de $K$ \cite{kostant}. Le résultat démontré par Kostant, dont l'énoncé est donné ci-dessous, a l'avantage de donner explicitement et de manière simple une description du polyèdre moment $\Delta_T(\Orb_{\Lambda})$.

\begin{theo}[Kostant]
Soit $K$ un groupe compact connexe et $T$ un tore maximal de $K$. La projection de l'orbite coadjointe $K\cdot\Lambda$ d'un élément $\Lambda\in\got{t}^*$ est l'enveloppe convexe de l'orbite $W\cdot\Lambda$ de $\Lambda$ sous l'action du groupe de Weyl $W$.
\end{theo}

Il est également possible de connaître la description exacte du polytope moment d'une projection d'orbite, même lorsque le sous-groupe n'est pas abélien. Berenstein-Sjamaar proposent une formule pour déterminer les équations de ce polytope convexe dans le cas général \cite{berenstein_sjamaar}. Cependant, ces formules font apparaître des équations en cohomologie où interviennent des classes de Schubert de variétés des drapeaux, rendant les calculs difficiles et fastidieux en règle générale.

Il est également important de remarquer qu'avec $K$ compact, les orbites coadjointes de $\got{k}^*$ sont des variétés projectives complexes. Plus précisément, on a une identification $K$-équivariante de $\Orb_{\Lambda}=K/K_{\Lambda}$ avec l'espace homogène complexe $\KC/P_{\Lambda}$, où $P_{\Lambda}$ est un sous-groupe parabolique de $K_{\C}$, qui induit une structure de variété kählérienne $K$-invariante sur la variété symplectique $(\Orb_{\Lambda},\Omega_{\Orb_{\Lambda}})$.

Supposons que $\Lambda$ est une forme entière, c'est-à-dire $\Lambda\in\wedge^*$, et notons $\chi$ le caractère de $K_{\Lambda}$ associé, comme indiqué dans l'Exemple \ref{exple:fibré_kostantsouriau_pour_orbitecoadjointe}. Le fibré de Kostant-Souriau $\mathcal{L}_{\Lambda}:=K\times_{K_{\Lambda}}\C_{\chi}$ sur $\Orb_{\Lambda}$ peut également s'écrire comme le fibré en droites holomorphe $K_{\C}\times_{P_{\Lambda}}\C_{\chi}$. Dans ce dernier fibré, on a confondu le caractère $\chi$ de $K_{\Lambda}$ avec l'unique caractère de $P_{\Lambda}$ qui coïncide avec $\chi$ sur $G_{\Lambda}$. Les polytopes moments hamiltonien et algébrique associés sont alors liés par le résultat du Théorème \ref{theo:égalité_polymomenthamiltonien_et_polymomentalgébrique}. En effet, on aura ici
\[
\Delta_H(\Orb_{\Lambda}) = \overline{\Delta_{H}^{\mathrm{alg}}(\Orb_{\Lambda},\mathcal{L}_{\Lambda})} = \overline{\{\mu\in\wedge_{\Q,+}^*; \exists N>0, V_{N\mu}^H \subset H^0(\Orb_{\Lambda}, \mathcal{L}_{\Lambda}^{\otimes N})|_H\}}.
\]
Or, d'après le théorème de Borel-Weil, cela revient à écrire
\begin{equation}
\label{eq:polymomenthamiltonien=polymomentalgébrique_casprojorbitecompacte}
\Delta_H(\Orb_{\Lambda}) = \overline{\{\mu\in\wedge_{\Q,+}^*; \exists N>0, V_{N\mu}^H \subset V_{N\Lambda}^{K}|_H\}}.
\end{equation}
Cette approche du polytope moment hamiltonien de la projection d'orbite par les représentations irréductibles a été initiée par Heckman \cite{heckman}. Voir aussi \cite{guillemin82, berenstein_sjamaar}. Ceci est une conséquence d'une propriété de type \og{}la quantification commute avec la réduction\fg{}.

\subsection{Le problème de Horn}
\label{subsection:problème_de_Horn}

Le problème de Horn a vu le jour au dix-neuvième siècle sous la forme suivante : que peut-on dire des valeurs propres de la somme de deux matrices hermitiennes, par rapport aux valeurs propres de ces deux matrices? Cette question, qui, au premier abord, ne fait intervenir que des objets d'algèbre linéaire, a de nombreuses interprétations géométriques et algébriques, par exemple en combinatoire, en géométrie algébrique ou en géométrie hamiltonienne. L'article de Fulton \cite{fulton00}, donne un aperçu du problème de Horn et des interprétations connues à l'heure actuelle.

La formulation moderne du problème de Horn peut se décrire en termes de projection d'orbites coadjointes. Soit $s>2$ un entier. On fixe cette fois-ci un groupe de Lie $H=K$ compact connexe, $T$ un tore maximal de $K$ et on considère $G=K^s = K\times\cdots\times K$ le produit direct de $s$ copies du groupe $K$. Le groupe $K$ se voit comme sous-groupe de $G$ par l'injection diagonale. Une orbite coadjointe pour $G$ est un produit direct $\Orb_{\Lambda_1}\times\ldots\times\Orb_{\Lambda_s}$ des orbites coadjointes pour $K$ de $s$ éléments $\Lambda_1,\ldots,\Lambda_s\in\got{k}^*$. Dans ce contexte, l'application moment est tout simplement la somme d'éléments des orbites coadjointes, c'est-à-dire,
\[
\Phi_{K}(A_1,\ldots,A_s) = A_1 + \ldots + A_s
\]
pour tout $(A_1,\ldots,A_s)\in\Orb_{\Lambda_1}\times\ldots\times\Orb_{\Lambda_s}$.

Comme conséquence du Théorème \ref{theo:convexitéhamiltonienne}, l'intersection $\Delta_K(\Orb_{\Lambda_1}\times\ldots\times\Orb_{\Lambda_s})$ de la projection d'orbite $\Phi_{K}(\Orb_{\Lambda_1}\times\ldots\times\Orb_{\Lambda_s})$ avec une chambre de Weyl de $\got{t}^*$ est un polytope convexe. Le problème de Horn généralisé consiste à trouver des conditions sur $\mu\in\got{t}_+^*$ de sorte que $\mu$ soit somme des éléments des orbites de $\Lambda_1,\ldots,\Lambda_s$. Remarquons que $\mu$ est une somme d'éléments de ces $s$ orbites coadjointes si et seulement si $0\in \Orb_{\mu^*}+\Orb_{\Lambda_1}+\ldots+\Orb_{\Lambda_s}$, c'est-à-dire, $0$ est élément du polytope moment $\Delta_K(\Orb_{\mu^*}\times\Orb_{\Lambda_1}\times\ldots\times\Orb_{\Lambda_s})$ de la somme de $s+1$ orbites.

Le problème de Horn généralisé n'est qu'un cas particulier de projections d'orbites compactes expliqué dans le paragraphe \ref{subsection:projection_orbite_compacte}. Une réécriture de \eqref{eq:polymomenthamiltonien=polymomentalgébrique_casprojorbitecompacte}, lorsque $s=2$, donne le résultat suivant.

\begin{prop}
\label{prop:équivalences_polytopemomentproduitorbites}
Soient $\mu,\nu,\lambda\in\wedge^*_+$. Les assertions suivantes sont équivalentes:
\begin{enumerate}
\item $\lambda\in\Delta_K(\Orb_{\mu}\times\Orb_{\nu})$;
\item $0\in\Delta_K(\Orb_{\lambda^*}\times\Orb_{\mu}\times\Orb_{\nu})$;
\item Il existe un entier $N>0$ tel que $\left(V_{N\lambda}^{K*}\otimes V_{N\mu}^K\otimes V_{N\nu}^K\right)^K \neq \{0\}$;
\item Il existe un entier $N>0$ tel que $V_{N\lambda}^K\subset V_{N\mu}^K\otimes V_{N\nu}^K$.
\end{enumerate}
\end{prop}

Lorsque $K$ est le groupe unitaire $U(n)$ et $s=2$, on retrouve le problème de Horn originel. 
Il est bien connu qu'une matrice hermitienne de taille $n\times n$ est diagonalisable et ses valeurs propres sont toutes réelles. Ceci nous permet de lister ses valeurs propres par ordre décroissant, en indiquant chacune des valeurs propres autant de fois que sa multiplicité. Ainsi, pour trois matrices hermitiennes $A$, $B$ et $C$, on notera $\alpha$ (resp. $\beta$, resp. $\gamma$) le spectre ordonné des valeurs propres de $A$ (resp. $B$, resp. $C$),
\[
\alpha = (\alpha_1\geqslant\ldots\geqslant\alpha_n),
\]
et similairement pour $\beta$ et $\gamma$. La question devient alors :

\medskip

\begin{equation}
\label{eq:question_Horn}\tag{Horn}
\begin{minipage}{10cm}
Quels sont les $n$-uplets ordonnés $\alpha$, $\beta$ et $\gamma$ qui sont les spectres de matrices hermitiennes $A$, $B$ et $C$ de taille $n\times n$, telles que $C = A+B$?
\end{minipage}
\end{equation}

\bigskip

De nombreux mathématiciens ont étudié ce problème, dont Horn, qui donna une réponse dans \cite{horn62} pour les cas $n\leqslant 5$ et conjectura le résultat pour $n>5$. Knutson et Tao prouvèrent finalement la conjecture en 2001 dans l'article \cite{knutsontao2001}, grâce aux percées obtenues dans \cite{klyachko98,helmke_rosenthal,knutson_tao,knutson_tao_woodward}. Ils ont prouvé que l'ensemble des $(\alpha,\beta,\gamma)\in(\R^n)^3$ répondant à la question \eqref{eq:question_Horn} forme un cône convexe polyédral de $(\R^n)^3$ et ont donné ses équations. Ces équations sont indexées par les ensembles $T^n_r$ de triplets $(I,J,L)$ de parties de $\{1,\ldots,n\}$ de même cardinal, pour $1\leqslant r<n$, définis récursivement.

Tout d'abord, définissons, pour $r\in\{1,\ldots,n-1\}$, l'ensemble
\[
U^n_r = \left\{ (I,J,L); \ |I| = |J| = |L| = r, \mbox{ et } \sum_{i\in I}i + \sum_{j\in J} j = \sum_{\ell\in L}\ell + r(r+1)/2\right\}.
\]
Ensuite, pour $r=1$, posons $T^n_1 = U^n_1$. Puis, en général, définissons
\[
T^n_r = \left\{(I,J,L)\in U^n_r \left| \begin{array}{l}
\mbox{pour tout } p<r \mbox{ et tout } (F,G,H)\in T^r_p, \\
\displaystyle\sum_{f\in F} i_f + \sum_{g\in G} j_g \leqslant \sum_{h\in H}\ell_h + p(p+1)/2
\end{array}\right.\right\}.
\]

\begin{exple}
\label{exple:calcul_de_T_1^2}
Le cas le plus simple est pour $n=2$. On ne peut prendre que $r=1$, et
\[
T_1^2 = U_1^2 = \bigl\{(\{1\},\{1\},\{1\}),\ (\{2\},\{1\},\{2\}),\ (\{1\},\{2\},\{2\})\bigr\}.
\]
\end{exple}

\begin{exple}
\label{exple:calcul_de_T_1^3_et_T_2^3}
Nous aurons besoin plus tard du cas $n=3$. Le calcul de $T_1^3=U_1^3$ est aisé,
\[
T_1^3 = \left\{\begin{array}{l}
(\{1\},\{1\},\{1\}),\ (\{2\},\{1\},\{2\}),\ (\{1\},\{2\},\{2\}),\\
(\{3\},\{1\},\{3\}),\ (\{1\},\{3\},\{3\}),\ (\{2\},\{2\},\{3\})
\end{array}\right\}.
\]
On peut ensuite vérifier par le calcul que, dans le cas présent, on aura aussi $T_2^3=U_2^3$, c'est-à-dire,
\[
T_2^3 = \left\{\begin{array}{l}
(\{1,2\},\{1,2\},\{1,2\}),\ (\{1,3\},\{1,2\},\{1,3\}),\ (\{1,2\},\{1,3\},\{1,3\}),\\
(\{1,3\},\{1,3\},\{2,3\}),\ (\{2,3\},\{1,2\},\{2,3\}),\ (\{1,2\},\{2,3\},\{2,3\})
\end{array}\right\}.
\]
\end{exple}

Le théorème qui suit est l'énoncé de la Conjecture de Horn.

\begin{theo}[Horn \cite{horn62}, Knutson-Tao \cite{knutsontao2001}]
\label{theo:horn_klyachko}
Un triplet $(\alpha,\beta,\gamma)$ de $n$-uplets ordonnés apparaît comme spectre de trois matrices hermitiennes $A$, $B$ et $C$ de taille $n\times n$ avec $A+B=C$ si et seulement si
\[
\sum_{i=1}^n\alpha_i+\sum_{j=1}^n\beta_j = \sum_{\ell=1}^n\gamma_{\ell},
\]
et, pour tout $r<n$ et tout triplet $(I,J,L)$ de $T_r^n$,
\[
\sum_{i\in I}\alpha_i+\sum_{j\in J}\beta_j \geqslant \sum_{\ell\in L}\gamma_{\ell}.
\]
\end{theo}

Nous avons dit au début de ce paragraphe que le problème de Horn avait plusieurs interprétations. En particulier, il peut s'interpréter en termes de représentations irréductibles de $U(n)$ (ou, de manière équivalente, du complexifié $GL_n(\C)$ de $U(n)$). Le plus haut poids joue alors le rôle du spectre. Un poids dominant de $U(n)$ est représenté par un $n$-uplet d'entiers ordonnés de manière décroissante, $\mu = (\mu_1\geqslant\ldots\geqslant\mu_n)\in\Z^n$. L'ensemble des poids de $U(n)$ s'identifie au réseau $\wedge^*$ des éléments entiers du tore maximal des matrices diagonales de $U(n)$. Quant à l'ensemble des poids dominants, il s'identifie à l'ensemble $\wedge_+^*$ des éléments de $\wedge^*$ ordonnés de manière décroissante. On notera $V_{\mu}^{U(n)}$ la représentation irréductibles de $U(n)$, de plus haut poids de $\mu$.


\begin{theo}[Klyachko \cite{klyachko98}, Knutson-Tao \cite{knutson_tao}]
\label{theo:klyachko_knutson-tao}
Soient trois poids dominants $\lambda$, $\mu$ et $\nu$ de $U(n)$. Les assertions suivantes sont équivalentes:
\begin{enumerate}

\item $\lambda$, $\mu$ et $\nu$ sont les spectres de matrices hermitiennes $A$, $B$ et $C=A+B$;
\item Il existe un entier $N>0$ tel que la représentation irréductible $V_{N\nu}^{U(n)}$ apparaît dans la décomposition du produit tensoriel $V_{N\lambda}^{U(n)}\otimes V_{N\mu}^{U(n)}$;
\item La représentation irréductible $V_{\nu}^{U(n)}$ apparaît dans la décomposition du produit tensoriel $V_{\lambda}^{U(n)}\otimes V_{\mu}^{U(n)}$.
\end{enumerate}
\end{theo}

Les Théorèmes \ref{theo:horn_klyachko} et \ref{theo:klyachko_knutson-tao} permettent par conséquent de déterminer quelles sont les représentations irréductibles de $U(n)$ qui apparaissent dans un produit tensoriel de deux représentations irréductibles fixées de ce même groupe.

\begin{exple}
Soient $\lambda=(\lambda_1\geqslant\lambda_2)$, $\mu=(\mu_1\geqslant\mu_2)$ et $\nu=(\nu_1\geqslant\nu_2)$ trois poids de $U(2)$. De l'Exemple \ref{exple:calcul_de_T_1^2}, on obtient $V_{\nu}^{U(2)}\subset V_{\lambda}^{U(2)}\otimes V_{\mu}^{U(2)}$ si et seulement si le triplet $(\lambda,\mu,\nu)$ vérifie les trois inégalités suivantes,
\begin{align*}
\lambda_1+\mu_1&\geqslant\nu_1\\
\lambda_2+\mu_1&\geqslant\nu_2\\
\lambda_1+\mu_2&\geqslant\nu_2
\end{align*}
et l'égalité $\lambda_1+\lambda_2+\mu_1+\mu_2=\nu_1+\nu_2$.
\end{exple}

\begin{exple}
\label{exple:équationsgénérales_Horn_pourU(3)}
Prenons maintenant $\lambda=(\lambda_1\geqslant\lambda_2\geqslant\lambda_3)$, $\mu=(\mu_1\geqslant\mu_2\geqslant\mu_3)$ et $\nu=(\nu_1\geqslant\nu_2\geqslant\nu_3)$ trois poids de $U(3)$. Le Théorème \ref{theo:horn_klyachko} et l'Exemple \ref{exple:calcul_de_T_1^3_et_T_2^3} nous donnent $V_{\nu}^{U(3)}\subset V_{\lambda}^{U(3)}\otimes V_{\mu}^{U(3)}$ si et seulement si le triplet $(\lambda,\mu,\nu)$ de $\wedge_+^*$ vérifie les douze inégalités suivantes,
\[
\begin{array}{lr}
\left\{\begin{aligned}
\lambda_1+\mu_1&\geqslant\nu_1\\
\lambda_2+\mu_1&\geqslant\nu_2\\
\lambda_1+\mu_2&\geqslant\nu_2\\
\lambda_3+\mu_1&\geqslant\nu_3\\
\lambda_2+\mu_2&\geqslant\nu_3\\
\lambda_1+\mu_3&\geqslant\nu_3\\
\end{aligned}\right. & \qquad\text{et}\qquad
\left\{\begin{aligned}
\lambda_1+\lambda_2+\mu_1+\mu_2&\geqslant\nu_1+\nu_2\\
\lambda_1+\lambda_3+\mu_1+\mu_2&\geqslant\nu_1+\nu_3\\
\lambda_1+\lambda_2+\mu_1+\mu_3&\geqslant\nu_1+\nu_3\\
\lambda_1+\lambda_3+\mu_1+\mu_3&\geqslant\nu_2+\nu_3\\
\lambda_2+\lambda_3+\mu_1+\mu_2&\geqslant\nu_2+\nu_3\\
\lambda_1+\lambda_2+\mu_2+\mu_3&\geqslant\nu_2+\nu_3\\
\end{aligned}\right.
\end{array}
\]
et l'égalité $\lambda_1+\lambda_2+\lambda_3+\mu_1+\mu_2+\mu_3=\nu_1+\nu_2+\nu_3$.
\end{exple}

\bigskip

Le problème de Horn possède une propriété surprenante: il est possible de déterminer les ensembles $T_r^n$ en résolvant le problème de Horn pour tous les entiers $p\leqslant r$. Cette propriété est décrite par le Théorème $2$ de \cite{fulton00}, dont l'énoncé est donné ci-dessous.

Pour une partie $I= \{i_1<i_2<\ldots<i_r\}$ de $\{1,\ldots,n\}$ de cardinal $r$, on note la partition $\lambda(I) = (i_r-r\geqslant i_{r-1}-(r-1)\geqslant\ldots\geqslant i_2-2\geqslant i_1-1)$ associée à $I$.

\begin{theo}
\label{theo:Fultonthm2}
Un triplet $(I,J,L)$ est dans $T^n_r$ si et seulement si le triplet correspondant $(\lambda(I),\lambda(J),\lambda(L))$ apparaît comme spectre d'un triplet de matrices hermitiennes de taille $r\times r$ telles que la troisième est somme des deux premières.
\end{theo}

L'utilisation de ce théorème permet de donner des exemples de triplets non triviaux $(I,J,L)$ qui appartiennet à $T_r^n$, cf Annexe \ref{chap:exemples_de_triplets_pour_Horn}.

\subsection{$G$ semi-simple et $H=K$ sous-groupe compact maximal}
\label{subsection:Gsemi-simple}

Soit $G$ un groupe de Lie réel semi-simple connexe non compact à centre fini. On note $\got{g}=\got{k}\oplus\got{p}$ la décomposition de Cartan de son algèbre de Lie, associée à une involution de Cartan $\theta$ de $\got{g}$, et soit $K$ le sous-groupe connexe de $G$ d'algèbre de Lie $\got{k}$. Alors $K$ est un sous-groupe compact maximal de $G$.

Soient $\Lambda\in\got{g}^*$ et $\Orb_{\Lambda}$ son orbite coadjointe. Rappelons que l'application moment $\Phi_{G}$ est propre si et seulement si $\Orb_{\Lambda}$ est fermée dans $\got{g}^*$. L'action induite de $K$ sur $\Orb_{\Lambda}$ est hamiltonienne, d'application moment la projection d'orbite $\Phi_{K}:\Orb_{\Lambda}\rightarrow\got{k}^*$.

La forme de Killing $B_{\got{g}}$ sur $\got{g}$ est définie négative (resp. définie positive) sur $\got{k}$ (resp. sur $\got{p}$) et $G$-invariante. Elle définit donc des structures euclidiennes $K$-invariantes sur $\got{k}$ et $\got{p}$ de sorte que
\[
B_{\got{g}}(X,X) = -\|X_{\got{k}}\|^2+\|X_{\got{p}}\|^2, \qquad \forall\ X\in\got{g}.
\]
Cette forme de Killing permet également de réaliser les identifications $\got{g}\cong\got{g}^*$, ainsi que $\got{k}\cong\got{k}^*$ et $\got{p}\cong\got{p}^*$, et on peut voir $\got{k}^*$ comme un sous-espace de $\got{g}^*$.

D'un autre côté, l'involution de Cartan induit un produit scalaire $K$-invariant 
\[
B_{\theta}(X,X)=\|X_{\got{k}}\|^2+\|X_{\got{p}}\|^2
\]
sur $\got{g}$. Une vérification directe nous donne $\|\Phi_{K}\|^2 = \frac{1}{2}\|\Phi_{G}\|^2-\frac{1}{2}B_{\got{g}}(\Lambda,\Lambda)$. On en déduit que $\Phi_{K}$ est propre si et seulement si $\Phi_{G}$ est propre si et seulement si l'orbite $\Orb_{\Lambda}$ est fermée dans $\got{g}^*$, la deuxième équivalence provenant de la Remarque \ref{rema:inclusionorbitepropre_ssi_orbitefermée}.

Il est clair qu'une orbite coadjointe $\Orb_{\Lambda}$ est fermée si et seulement si l'orbite adjointe correspondante $\Ad(G) X_{\Lambda}$ est fermée dans $\got{g}$, où $X_{\Lambda}$ est l'unique élément de $\got{g}$ vérifiant $\langle\Lambda,X\rangle = B_{\theta}(X_{\Lambda},X)$ pour tout $X\in\got{g}$. Or, d'après le Théorème de Borel-Harish-Chandra \cite[Proposition 1.3.5.5]{warner72}, comme $\got{g}$ est semi-simple, l'orbite adjointe $\Ad(G)X_{\Lambda}$ est fermée si et seulement si l'endomorphisme $\ad(X_{\Lambda})$ de $\got{g}$ est semi-simple, c'est-à-dire diagonalisable sur $\C$. On en déduit que tout élément d'une sous-algèbre de Cartan de $\got{g}$ est dans une orbite fermée, mais également que les orbites elliptiques (c'est-à-dire les orbites intersectant $\got{k}^*$) sont toutes fermées.

\begin{rema}
La propriété que tout élément d'une sous-algèbre de Cartan de $\got{g}$ est dans une orbite fermée, a aussi été prouvée par Neeb \cite[Théorème I.13]{neeb} plus généralement pour toute algèbre de Lie réelle.
\end{rema}

Ainsi, la projection d'orbite associée à une orbite coadjointe elliptique $\Orb_{\Lambda}$ est propre. La version non compacte de Lerman et al. \cite{lerman98} du Théorème de Convexité hamiltonienne implique que la projection de l'orbite $\Orb_{\Lambda}$ intersecte une chambre de Weyl de $K$ en un ensemble convexe localement polyédral $\Delta_K(\Orb_{\Lambda})$.

Il est cependant compliqué de déterminer cet ensemble. Dans la suite de cette thèse, nous étudierons ce polyèdre moment pour un type particulier d'orbites coadjointes. Pour ce faire, nous chercherons à trouver une description de $\Delta_K(\Orb_{\Lambda})$ similaire au cas compact, c'est-à-dire, comme adhérence d'un ensemble de points rationnels du type
\[
\{\mu\in\wedge_{\Q,+}^*; \exists N>0, V_{N\mu}^K\subset V_{N\Lambda}^G|_K\}.
\]


\chapter{Espaces symétriques hermitiens}
\label{chap:espaces_symétriques_hermitiens}

Ce chapitre présente les espaces symétriques hermitiens et les propriétés que nous utiliserons par la suite. On décrit certaines orbites coadjointes associées: les orbites coadjointes holomorphes, dont la structure symplectique et la projection d'orbite seront étudiées le long des chapitres suivants. En fin de chapitre, nous fixons les notations pour les groupes simples et leurs algèbres de Lie qui apparaissent dans la classification des espaces symétriques hermitiens.

\section{Description des espaces symétriques hermitiens $G/K$}
\label{section:description_espace_symetrique_hermitien}

Soit $G$ un groupe de Lie réel semi-simple connexe non compact à centre fini et $\got{g}$ son algèbre de Lie. Puisque $\got{g}$ est semi-simple réelle, il existe une décomposition de Cartan associée à une involution de Cartan, au niveau de l'algèbre de Lie,
\[
\got{g} = \got{k}\oplus\got{p},
\]
où $\got{k}$ (resp. $\got{p}$) est appelé la \emph{partie compacte} (resp. la \emph{partie non compacte}) de la décomposition de Cartan de $\got{g}$. Le sous-espace vectoriel $\got{k}$ se révèle être également une sous-algèbre de Lie de $\got{g}$. Le sous-groupe connexe $K$ de $G$ d'algèbre de Lie $\got{k}$ est un sous-groupe compact maximal de $G$, lorsque $G$ vérifie les hypothèses de semi-simplicité, de connexité et de finitude de son centre. Nous avons un difféomorphisme, la décomposition de Cartan au niveau du groupe,
\begin{equation}
\label{eq:premier_difféo_décomp_cartan}
\begin{array}{ccl}
K\times\got{p} & \longrightarrow & G \\
(k,X) & \longmapsto & \exp(X)k,
\end{array}
\end{equation}
justifiant la terminologie de partie compacte et partie non compacte de la décomposition de Cartan.

\begin{rema}
Le difféomorphisme généralement donné dans la littérature est l'application $(k,X)\in K\times\got{p} \longmapsto k\exp(X)\in G$. Cependant, le difféomorphisme \eqref{eq:premier_difféo_décomp_cartan} a l'avantage d'être $K$-équivariant pour l'action de multiplication à gauche sur $G$ et pour une action diagonale de $K$ sur $K\times\got{p}$ qui sera introduite plus tard.
\end{rema}

Pour plus de détails sur la décomposition de Cartan et les définitions de $\got{k}$ et $\got{p}$, nous invitons le lecteur à regarder la Section \ref{subsection:notations_décomp_cartan} ou \cite{knapp}.


\begin{defi}
Nous dirons que $G/K$ est \emph{hermitien} si $G/K$ admet une structure de variété complexe pour laquelle $G$ agisse sur $G/K$ par transformations holomorphes.
\end{defi}

\begin{prop}
\label{prop:equivalence_G/Khermitien}
Soient $G$ et $K$ vérifiant les hypothèses décrites ci-dessus. Alors les assertions suivantes sont équivalentes :
\begin{enumerate}
\item $G/K$ est hermitien;
\item Il existe un élément $z_0\in\got{z}(\got{k})$ tel que $\ad(z_0)|_{\got{p}}^2 = -\id|_{\got{p}}$.
\end{enumerate}
\end{prop}

\begin{proof}
Voir les Théorèmes 7.177 et 7.119 de \cite{knapp}.
%
%
\end{proof}

L'étude des espaces symétriques hermitiens $G/K$ où $G$ est semi-simple connexe non-compact à centre fini peut être ramenée à celle des espaces symétriques hermitiens ayant pour groupe d'isométrie $G$ simple connexe non-compact à centre fini. Or, la classification de ces groupes simples est bien connue : un tel groupe simple est isomorphe à l'un des groupes suivants :
\begin{itemize}
\item $Sp(2n,\R)$, pour $n\geqslant 1$;
\item $SO^*(2n)$, pour $n\geqslant 3$;
\item $SU(p,q)$, pour $p\geqslant q\geqslant 1$;
\item $SO_0(p,2)$, pour $p\geqslant 1$;
\item $E{\mathrm{III}}$ (de type $E_6$);
\item $E{\mathrm{VII}}$ (de type $E_7$).
\end{itemize}
On pourra retrouver cette classification dans \cite{knapp,johnson80} par exemple.

 
\bigskip

Nous ferons désormais l'hypothèse suivante : $G/K$ est un espace symétrique \textbf{hermitien}. Il existe donc un élément $z_0$ de $\got{z}(\got{k})$ tel que $\ad(z_0)|_{\got{p}}^2=-\id_{\got{p}}$, c'est-à-dire que $\ad(z_0)|_{\got{p}}$ définit une structure complexe sur $\got{p}$.

Notons $\got{p}_{\C}$ le complexifié de $\got{p}$. L'action de $K$ sur $\got{p}$ passe en une action sur son complexifié $\got{p}_{\C}$, qui fait de $\got{p}_{\C}$ un sous-espace $K$-stable de $\got{g}_{\C}$. Nous avons la décomposition en sous-espaces propres $\got{p}_{\C} = \got{p}^{+,z_0}\oplus\got{p}^{-,z_0}$, où $\got{p}^{\pm,z_0}$ est le $K$-module $\ker(\ad(z_o)|_{\got{p}_{\C}}\mp i)$. Ces deux sous-espaces de $\got{p}_{\C}$ sont stables par l'action de $K$, car $z_0$ est fixé par l'action adjointe de $K$.

\begin{rema}
Il n'y a jamais unicité d'un tel élément $z_0$ du centre de $\got{k}$. En effet, pour un tel $z_0$, son vecteur opposé $-z_0$ définira de même une structure complexe $K$-invariante sur $\got{p}$. En revanche, deux choix différents de $z_0$ peuvent faire intervenir deux décompositions $\got{p}_{\C} = \got{p}^{+,z_0}\oplus\got{p}^{-,z_0}$ différentes. Par exemple, on a toujours $\got{p}^{+,-z_0} = \got{p}^{-,z_0}$.

On peut noter que lorsque $\got{g}$ est simple, le centre de $\got{k}$ est nécessairement de dimension $1$. Il n'existe alors que deux vecteurs $\pm z_0\in\got{z}(\got{k})$ qui induisent une telle structure complexe $K$-invariante sur $\got{p}$.
\end{rema}

Soit $T$ un tore maximal du groupe compact $K$. Notons $\got{t}$ son algèbre de Lie. Remarquons que le centre $Z(K)$ de $K$ est contenu dans $T$.

\begin{prop}
\label{prop:gott_sousalgèbredecartandegotg}
La sous-agèbre de Cartan $\got{t}_{\C}$ de $\got{k}_{\C}$ est aussi une sous-algèbre de Cartan de $\got{g}_{\C}$. Les algèbres de Lie $\got{g}$ et $\got{k}$ sont donc de même rang.
\end{prop}

\begin{proof}
D'après la Proposition \ref{prop:equivalence_G/Khermitien}, puisque $G/K$ est hermitien, si on note $\got{c}:=\got{z}(\got{k})$, on aura
\[
\got{z}_{\got{g}}(\got{c}) = \got{k}.
\]
Or, comme $Z(K)$ est contenu dans $T$, son algèbre de Lie $\got{c}$ est une sous-algèbre de $\got{t}$. On en déduit que $\got{z}_{\got{g}}(\got{t}) \subseteq \got{z}_{\got{g}}(\got{c}) = \got{k}$. En passant aux complexifiés, on en conclut que $\got{t}_{\C}$ est une sous-algèbre de Cartan de $\got{g}_{\C}$.
\end{proof}

En particulier, si le groupe $G$ vérifie les hypothèses énoncées au début de ce paragraphe, ainsi que l'hypothèse $G/K$ hermitien, alors toutes les conditions sont réunies pour pouvoir appliquer la Formule de Duflo-Heckman-Vergne qui sera donnée au Théorème \ref{theo:fomule_DHV}.

\section{La chambre holomorphe}
\label{section:la_chambre_holomorphe}

On prendra la convention suivante: un élément $\alpha\in\got{t}^*$ est une racine de $\got{g}$ s'il existe un $X\in\got{g}_{\C}$ non nul tel que $[H,X]=i\alpha(H)X$ pour tout $H\in\got{t}$. Ceci permet de rester cohérent avec la convention prise pour les poids, cf Notations page \pageref{chap:notations}.

Nous gardons les mêmes hypothèses sur $G$ introduites dans le paragraphe \ref{section:description_espace_symetrique_hermitien}, avec $G/K$ hermitien. Nous pouvons donc fixer un élément de $z_0\in\got{z}(\got{k})$ vérifiant l'assertion ($2$) de la Proposition \ref{prop:equivalence_G/Khermitien}.

Notons $\got{R}_c$\index{$\got{R}_c$} l'ensemble des racines de $\got{k}$ pour sa sous-algèbre de Cartan $\got{t}$. Fixons un système de racines positives $\got{R}_c^+$ de $\got{k}$ relativement à $\got{t}$. D'après la Proposition \ref{prop:gott_sousalgèbredecartandegotg}, les poids de $\got{t}_{\C}$ sur $\got{p}_{\C}$ sont des racines de $\got{g}_{\C}$ relativement à $\got{t}_{\C}$ vue comme sous-algèbre de Cartan de $\got{g}_{\C}$. Nous noterons alors $\got{R}_n$ \index{$\got{R}_n$}(resp. $\got{R}_n^{+,z_0}$, resp. $\got{R}_n^{-,z_0}$) l'ensemble des poids de $\got{t}_{\C}$ sur $\got{p}_{\C}$ (resp. $\got{p}^{+,z_0}$, resp. $\got{p}^{-,z_0}$), ses éléments seront appelés les racines non-compactes (resp. non-compactes positives, resp. non-compactes négatives) de $\got{g}$. Par la décomposition $\got{g}=\got{k}\oplus\got{p}$, le sous-ensemble $\got{R}_{\mathrm{hol}}:=\got{R}_c\cup\got{R}_n$ de $\got{t}^*$, est égal à l'ensemble des racines de $\got{g}$.

En outre, l'ensemble $\got{R}_{\mathrm{hol}}^{+,z_0}:=\got{R}_c^+\cup\got{R}_n^{+,z_0}$ forme un système de racines positives de $\got{g}$. En effet, on peut facilement voir que pour toute racine compacte $\alpha$, on a $\alpha(z_0) = 0$, car $z_0$ est central dans $\got{k}$. Quant aux racines non compactes, elles vérifient $\beta(z_0) = \pm 1$ pour tout $\beta\in\got{R}_n^{\pm,z_0}$, puisque $\got{R}_n^{\pm,z_0} = \{\beta\in\got{t}^*; \got{g}_{\beta}\subset\got{p}^{\pm,z_0}\}$, ce qui induit
\[
[z_0,X] = i\beta(z_0)X = \pm iX \qquad \forall X\in\got{g}_{\beta}\subset\got{p}^{\pm,z_0},\index{$\Chol$}
\]
car par définition, $\ad(z_0)_{\got{p}^{\pm,z_0}} = \pm i\,\id_{\got{p}^{+,z_0}}$. On peut donc trouver un ordre lexicographique tel que $\got{R}_c^+\cup\got{R}_n^{+,z_0}$ soit le système de racines positives associé.

Notons $\got{t}^*_+$ la chambre de Weyl de $\got{t}^*$ associée au système de racines compactes positives $\got{R}^+_c$. Nous définissons la \emph{chambre holomorphe} de $\got{t}^*$ par
\[
\Chol^{z_0} := \{\xi\in\got{t}^*_+ | (\beta,\xi)>0, \forall\beta\in\got{R}_n^{+,z_0}\} \subset\got{t}^*_+,
\]
où $(\cdot,\cdot)$ désigne le produit scalaire sur $\got{t}^*$ induit par la forme de Killing sur $\got{g}$. Ainsi, $\overline{\Chol^{z_0}}$ est la chambre de Weyl définie par $\got{R}^{+,z_0}_{\mathrm{hol}}$.

\begin{rema}
De nouveau, un choix différent d'élément $z_0$ apportera une décomposition $\got{p}_{\C} = \got{p}^{+,z_0}\oplus\got{p}^{-,z_0}$ différente et, donc, un autre système de racines non compactes positives $\got{R}_n^{+,z_0}$ et une autre chambre holomorphe $\Chol^{z_0}$. En particulier, pour $z_0$ fixé, on peut voir $\Chol^{-z_0}$ comme la \og{}chambre anti-holomorphe\fg{} associée à $\Chol^{z_0}$.

\`A partir de maintenant, une fois l'élément $z_0$ fixé, nous n'aurons plus besoin de souligner les différences apportées par des choix distincts de $z_0$. Par conséquent, nous désignerons les objets définis ci-dessus par $\got{p}^{\pm}$\index{$\got{p}^{\pm}$}, $\got{R}_n^{\pm}$\index{$\got{R}_n^{\pm}$}, $\got{R}_{\mathrm{hol}}^{\pm}$\index{$\got{R}_{\mathrm{hol}}^{\pm}$} et $\Chol$\index{$\Chol$}.
\end{rema}

Soit $\Lambda\in\got{g}^*$. Notons $\Orb_{\Lambda}:=G\cdot\Lambda$ l'orbite coadjointe de $\Lambda$ dans $\got{g}^*$. Si $\Orb_{\Lambda}$ intersecte $\got{t}^*$, nous dirons que l'orbite $\Orb_{\Lambda}$ est elliptique\index{Orbite elliptique}. Si, de plus, $\Orb_{\Lambda}$ intersecte $\Chol$, on dira que l'orbite $\Orb_{\Lambda}$ est holomorphe\index{Orbite holomorphe}. La terminologie d'orbite holomorphe est justifiée par le fait que, pour tout $\Lambda\in\Chol$, l'orbite $\Orb_{\Lambda}$ peut être munie d'une structure de variété complexe, compatible avec la structure symplectique canonique d'orbite coadjointe. C'est-à-dire, $\Orb_{\Lambda}$ est une variété kählérienne pour la structure symplectique de Kostant-Kirillov-Souriau. En effet, sur l'orbite coadjointe $\Orb_{\Lambda}$, nous avons les structures suivantes :
\begin{enumerate}
\item La forme symplectique de Kirillov-Kostant-Souriau $\Omega_{\Orb_{\Lambda}}$, définie dans le paragraphe \ref{subsection:orbites_coadjointes}.
\item L'inclusion $\Orb_{\Lambda}\hookrightarrow\got{g}^*$, qui est une application moment pour l'action hamiltonienne de $G$ sur $(\Orb_{\Lambda},\Omega_{\Orb_{\Lambda}})$. Ce point est expliqué également en \ref{subsection:orbites_coadjointes}.
\item Une structure complexe $G$-invariante $J_{\Orb_{\Lambda}}$ caractérisée par le fait que son fibré tangent holomorphe $T^{1,0}\Orb_{\Lambda}\rightarrow\Orb_{\Lambda}$ est égal, au-dessus de $\Lambda\in\Orb_{\Lambda}$, au $T$-module
\[
\sum_{\stackrel{\alpha\in\got{R}_c^+}{\langle\alpha,\Lambda\rangle\neq 0}}\got{g}_{\alpha} + \underbrace{\sum_{\beta\in\got{R}_n^-}\got{g}_{\beta}}_{\got{p}^-}.
\]
\end{enumerate}

Nous avons vu dans le paragraphe \ref{section:proj_orbites_defi_et_exemples} que l'action de $K$ sur $\Orb_{\Lambda}$ induite de l'action de $G$ est, elle aussi, hamiltonienne. L'application moment canonique associée est la projection d'orbite
\[
 \Phi_{\Orb_{\Lambda}}:\Orb_{\Lambda}\rightarrow\got{k}^*,
\]
composée de l'injection de $\Orb_{\Lambda}$ dans $\got{g}^*$ et de la projection linéaire canonique $\got{g}^*\rightarrow\got{k}^*$. Lorsque l'orbite est elliptique, l'application $\Phi_{\Orb_{\Lambda}}$ est propre, d'après le paragraphe \ref{subsection:Gsemi-simple}. On peut alors appliquer les résultats de Lerman et al. \cite{lerman98}, nous assurant que $\Delta_K(\Orb_{\Lambda}) := \Phi_{\Orb_{\Lambda}}(\Orb_{\Lambda})\cap\got{t}_+^*$\index{$\Delta_K(\Orb_{\Lambda})$} est un ensemble convexe localement polyédral dans $\got{t}^*$.

\begin{enonce*}{Objectif}
Prouver que $\Delta_K(\Orb_{\Lambda})$ est convexe polyédral, c'est-à-dire, est défini comme intersection finie de demi-espaces de $\got{t}^*$. Trouver un ensemble d'équations affines décrivant complètement le polyèdre convexe $\Delta_K(\Orb_{\Lambda})$.
\end{enonce*}

\section{Théorème de Schmid}
\label{section:theo_schmid}

Lorsque $G$ vérifie toutes les hypothèses du paragraphe \ref{section:description_espace_symetrique_hermitien} et qu'il est de plus \textbf{simple} (l'espace symétrique hermitien $G/K$ est alors dit \emph{irréductible}), nous savons également que le polyèdre moment $\Delta_K(\Orb_{\Lambda})$ est en étroite relation avec l'action de $K$ sur $\got{p}$. En effet, $\got{p}$ admet une structure complexe $K$-invariante canonique, qui fait de $\got{p}$ une représentation irréductible complexe de $K$, que l'on identifie à $\got{p}^-$. Pour les cas classiques, nous avons la liste suivante, que l'on peut trouver dans \cite{knapp} ou \cite{johnson80} par exemple.
\begin{enumerate}
\item Quand $G=Sp(2n,\R)$, alors $K=U(n)$ et $\got{p}^+\cong \mathrm{S}^2(\C^n)$, avec action standard de $U(n)$.
\item Quand $G=SO^*(2n)$, alors $K=U(n)$ et $\got{p}^+\cong \wedge^2\C^n$, avec action standard de $U(n)$.
\item Quand $G=SU(p,q)$, alors $K=S(U(p)\times U(q))$ et $\got{p}^+\cong M_{p,q}(\C)$, avec $U(p)$ (resp. $U(q)$) agissant par multiplication à gauche (resp. à droite) sur l'espace des matrices complexes $M_{p,q}(\C)$.
\item Quand $G=SO_0(p,2)$, alors $K=SO(p)\times SO(2)$ et $\got{p}^+\cong \C^p$, avec actions standard de $SO(p)$ et $SO(2)=S^1$ sur $\C^p$.
\end{enumerate}
Les identifications du $K$-module $\got{p}^+$ données ci-dessus seront explicitées pour chacun des quatre cas dans la section \ref{section:systèmes_de_racines}.

Or, nous avons ici une propriété bien plus forte : l'algèbre $\C[\got{p}^-]$ des polynômes complexes sur $\got{p}^-$ est sans multiplicité en tant que $K$-module. 
De plus, on connaît exactement les représentations irréductibles qui apparaissent dans la décomposition de $\C[\got{p}^-]$ en somme directe de représentations irréductibles de $K$.

On dira que deux racines $\alpha$ et $\beta$ sont \emph{fortement orthogonales} si ni $\alpha+\beta$, ni $\alpha-\beta$ ne sont des racines. Soit alors $\{\gamma_1,\ldots,\gamma_r\}$ l'ensemble maximal construit par récurrence de la manière suivante :
\begin{itemize}
 \item $\gamma_1$ est la plus grande racine non compacte positive,
 \item pour tout $i=1,\ldots,r-1$, $\gamma_{i+1}$ est la plus grande racine non compacte positive fortement orthogonale à $\gamma_1,\ldots,\gamma_i$.
\end{itemize}
Le Théorème suivant est dû à Schmid, cf \cite{schmid, johnson80}.

\begin{theo}[Schmid]
\label{theo:schmid}
Toute représentation de $K$ apparaît avec une multiplicité $1$ ou $0$ dans $\C[\got{p}^-]$. Plus précisément, on a
\begin{equation}
\label{eq:decomp_repirr_C[p]}
\C[\got{p}^-]=\sum_{p_1\geqslant\ldots\geqslant p_r\geqslant 0}V_{p_1\gamma_1+\ldots+p_r\gamma_r}^K,
\end{equation}
\end{theo}

La décomposition \eqref{eq:decomp_repirr_C[p]} de $\C[\got{p}^-]$ en somme directe de représentations irréductibles de $K$ est un des outils principaux que nous utiliserons dans le Chapitre \ref{chap:premières_tentatives} pour trouver les équations de certains polyèdres moments $\Delta_K(\Orb_{\Lambda})$.

\section[Systèmes de racines]{Systèmes de racines des groupes simples classiques de la classification des espaces symétriques hermitiens}
\label{section:systèmes_de_racines}

Les calculs de projections d'orbites qui seront effectués dans les derniers chapitres de cette thèse nécessiteront de connaître la description des algèbres de Lie et des systèmes de racines associés aux groupes de Lie classiques simples parmi les espaces symétriques hermitiens. Cette section est l'occasion de fixer les notations qui seront utilisées par la suite.

\subsection{Le groupe $Sp(2n,\R)$, $n\geqslant 2$}
\label{subsection:Notations_Sp(R2n)}

Dans ce paragraphe, on pose les notations et les paramétrages de l'algèbre de Lie de $Sp(2n,\R)$  nécessaires dans l'étude du polyèdre de Kirwan $\Delta_K(Sp(2n,\R)\cdot\Lambda)$. On pose
\[
J_n=\left(\begin{array}{cc}
0 & I_n \\
-I_n & 0 
\end{array}\right)\in GL_{2n}(\R),
\]
et on considère le sous-groupe fermé $Sp(2n,\R)$ de $GL_{2n}(\R)$ défini de la manière suivante,
\[
G = Sp(2n,\R) = \{M\in GL_{2n}(\R); \ ^tMJ_nM = J_n\}.
\]
On peut vérifier que $Sp(2n,\R)$ est en fait contenu dans $SL_{2n}(\R)$. Son algèbre de Lie est 
\[
\begin{array}{ccl}
\got{sp}(\R^{2n}) & = & \{M\in \mathcal{M}_{2n}(\R); \ ^tMJ_n + J_nM = 0\} \\
& = & \left\{\left(\begin{array}{cc}
A & B \\
C & -^t A
\end{array}\right); \ A,B,C\in\mathcal{M}_n(\R), \ ^tB=B, \ ^tC=C \right\}.
\end{array}
\]
Les éléments de $\got{sp}(\R^{2n})$ sont effectivement de trace nulle. Le groupe $Sp(2n,\R)$ est de dimension $n(2n+1)$. On considère l'involution de Cartan de $\got{sp}(\R^{2n})$ définie par : $M\mapsto -^t M$. Cela fixe une décomposition de Cartan $\got{sp}(\R^{2n}) = \got{k} \oplus \got{p}$, avec
\[
\got{k} = \left\{\left(\begin{array}{cc}
A & B \\
-B & A
\end{array}\right); \ A,B\in\mathcal{M}_n(\R), \ ^tA=-A, \ ^tB=B \right\}
\]
et
\[
\got{p} = \left\{\left(\begin{array}{cc}
A & B \\
B & -A
\end{array}\right); \ A,B\in\mathcal{M}_n(\R), \ ^tA=A, \ ^tB=B \right\}.
\]
Les dimensions de ces deux espaces vectoriels sont $\dim\got{k} = n^2$ et $\dim\got{p} = n(n+1)$. La matrice $J_n$ vérifiant l'équation $J_n^2 = -I_{2n}$, elle permet de définir une structure d'espace vectoriel complexe de dimension $n$ sur $\R^{2n}$, donnant une identification canonique de $U(n)$ à un sous-groupe $K$ de $Sp(2n,\R)$.

On peut préciser cette identification. Toute matrice $M$ de $\M_n(\C)$ s'écrit de manière unique sous la forme $M = A + iB$, avec $A$ et $B$ deux matrices de $\M_n(\R)$. Ceci nous permet de définir une application
\[
\begin{array}{cccc}
\Phi : & GL_n(\C) & \longrightarrow & GL_{2n}(\R) \\
& M = A + iB & \longmapsto & \left(\begin{array}{cc}
A & B \\
-B & A
\end{array}\right).
\end{array}
\]
On peut facilement vérifier que $\Phi$ est bien un morphisme de groupes à valeurs dans $GL_{2n}(\R)$. 
Le groupe $U(n)$ s'injecte dans $GL_{2n}(\R)$ et son image est contenue dans $Sp(2n,\R)$. De plus, $U(n)$ s'identifie au sous-groupe compact maximal $K$ de $Sp(2n,\R)$ par l'intermédiaire du morphisme de groupes $\Phi$.

On a une sous-algèbre de Cartan de $\got{k}$ canonique, c'est l'algèbre de Lie de l'image par $\Phi$ du tore maximal
\[
T:=\left\{\left(\begin{array}{ccc}e^{i\theta_1} & & 0 \\ & \ddots & \\ 0 & & e^{i\theta_n}\end{array}\right);\ \theta_1,\ldots,\theta_n\in\R\right\}
\]
de $U(n)$. Cette sous-algèbre de Cartan de $\got{k}$ est l'ensemble de matrices suivant:
\[
\got{t} = \left\{\left(\begin{array}{cccccc}
0 & \ldots & 0 & \theta_1 & & 0 \\
\vdots & \ddots & \vdots & & \ddots &  \\
0 & \ldots & 0  & 0 & & \theta_n \\
-\theta_1 & & 0 & 0 & \ldots & 0 \\
& \ddots & & \vdots & \ddots & \vdots \\
0 & & -\theta_n & 0 & \ldots & 0
\end{array}\right); \ \theta_1,\ldots,\theta_n\in\R\right\} \subset \got{k}.
\]
Nous choisissons maintenant une base de l'algèbre de Lie $\got{t}$. Nous notons $e_i$, pour $i=1,\ldots,n$, la matrice de $\got{t}$ avec $\theta_i = 1$ et $\theta_j = 0$ pour $j\neq i$, pour l'écriture donnée ci-dessus. On notera $(e_1^*,\ldots,e_n^*)$ la base duale dans $\got{t}^*$ de la base $(e_1,\ldots,e_n)$.

Le centre de $U(n)$ est de dimension $1$. Il s'agit du groupe $\left\{e^{i\theta}I_n; \ \theta\in\R\right\}\cong S^1$. L'algèbre de Lie qui lui est associée, dans $\got{sp}(\R^{2n})$, est la sous-algèbre de dimension $1$, $\got{z}(K) = \{\theta J_n; \ \theta\in\R\}$.

Posons $z_0 = \frac{J}{2}$, un élément de $\got{z}(K)$. Cet élément est particulier, car il est l'un des deux éléments de $\got{z}(K)$ qui définissent une structure complexe $K$-invariante sur $\got{p}$. En effet, pour tout $M\in\got{p}$, nous avons $\ad(z_0)^2 M = -M$. Le complexifié $\got{p}_{\C} = \got{p}\otimes\C$ de $\got{p}$ s'identifie à
\[
\left\{\left(\begin{array}{cc}
A & B \\
B & -A
\end{array}\right); \ A,B\in\mathcal{M}_n(\C), \ ^tA=A, \ ^tB=B \right\},
\]
et se décompose comme somme directe des deux sous-espaces propres de $\ad(z_0)_{\C}$ pour les valeurs propres $\pm i$. Ces deux sous-espaces propres sont
\[
\got{p}^{\pm} = \left\{\left(\begin{array}{cc}
A & \pm iA \\
\pm iA & -A
\end{array}\right); \ A\in\mathcal{M}_n(\C), \ ^tA=A\right\}.
\]
Par un calcul direct, on arrive à décrire l'action du sous-groupe compact $K$, par l'intermédiaire de son identification à $U(n)$, par
\[
\Phi(U)\cdot\left(\begin{array}{cc}A & iA\\ iA& -A\end{array}\right) = \Phi(U)\left(\begin{array}{cc}A & iA\\ iA& -A\end{array}\right)\Phi(U)^{-1} = \left(\begin{array}{cc}UA^tU & iUA^tU\\ iUA^tU& -UA^tU\end{array}\right)
\]
pour toutes matrices $U\in U(n)$ et $A\in\mathcal{M}_n(\C)$, $A$ symétrique. On retrouve ici que l'action de $K$ sur $\got{p}^+$ correspond à l'action par conjugaison de $U(n)$ sur l'espace des matrices symétriques complexes, ce qui est la même chose que l'action standard de $U(n)$ sur $\mathrm{S}^2(\C^n)$.

\bigskip

Les racines sont ici au nombre de $2n^2$. On a $n(n-1)$ racines compactes et $n(n+1)$ racines non compactes. Les racines compactes sont de la forme $\alpha_{i,j} = e_i^* - e_j^*$, avec $i\neq j$. Les racines non compactes sont les éléments $\pm\beta_{i,j} = \pm(e_i^*+e_j^*)$, $1\leqslant i\leqslant j\leqslant n$, de $\got{t}^*$. Un système de racines positives est
\[
\got{R}^+ = \underbrace{\{\alpha_{i,j} ; \ 1\leqslant i < j \leqslant n\}}_{\got{R}_c^+} \cup \underbrace{\{\beta_{i,j} ; \ 1\leqslant i \leqslant j \leqslant n\}}_{\got{R}_n^+}.
\]
Parmi les racines non compactes positives, il y a des racines remarquables, ce sont les racines $\gamma_i := \beta_{i,i} = 2e_i^*$, $i=1,\ldots,n$.

Pour ce système de racines positives, on obtient l'ordre suivant parmi les racines non compactes positives:
\[
\beta_{1,1} > \beta_{1,2} > \ldots > \beta_{1,n} > \beta_{2,2} > \beta_{2,3} > \ldots  > \beta_{n-1,n-1} > \beta_{n-1,n} > \beta_{n,n}.
\]
Ceci peut se trouver en choisissant un bon ordre correspondant au choix du système de racines positives. Un ordre convenable est donné par la famille ordonnée des vecteurs $H_i = 2 e_i + \sum_{j\neq i} e_j$ de $\got{t}$, pour $i$ variant de $1$ à $n$.

La famille de racines non compactes positives $(\gamma_1,\ldots,\gamma_n)=(2e_1^*,\ldots,2e_n^*)$ est fortement orthogonale et elle est maximale pour cette propriété. 
On remarque que les éléments du monoïde $\sum_{k=1}^n \N(\gamma_1+\ldots+\gamma_k)$ sont des poids de la forme $\delta=(2 p_1,\ldots, 2p_n)$, où $p_1\geqslant \ldots\geqslant p_n\geqslant 0$.

\begin{defi}
Un poids $\delta=(\delta_1,\ldots,\delta_n)$ de $U(n)$ est dit \emph{pair} si tous les $\delta_i$, pour $i=1,\ldots,n$, sont des entiers pairs.

Un poids dominant $\delta=(\delta_1\geqslant\ldots\geqslant\delta_n)$ de $U(n)$ est dit \emph{positif} si $\delta_n\geqslant 0$.
\end{defi}

D'après le Théorème de Schmid, énoncé dans le Théorème \ref{theo:schmid}, on a la décomposition suivante du $U(n)$-module $\C[\got{p}^-]$ pour le groupe $Sp(2n,\R)$:
\begin{equation}
\label{theo:Sp(R2n)_decompositiondeC[p-]}
\text{Pour $G = Sp(2n,\R)$, on a $\C[\got{p}^-] = \sum_{\beta\in\wedge^*_+ \text{ pair et positif}} V_{\beta}^{U(n)}$.}
\end{equation}

\subsection{Le groupe $SU(p,q)$, $p\geqslant q\geqslant 2$}
\label{subsection:Notations_SU(p,q)}

Donnons ici la définition de $SU(p,q)$, de son algèbre de Lie et les notations que nous utiliserons dans la suite. Fixons deux entiers $p\geqslant q\geqslant 1$. On pose
\[
I_{p,q}=\left(\begin{array}{cc}
I_p & 0\\
0 & -I_q
\end{array}\right)\in GL_{p+q}(\C)
\]
et on considère le sous-groupe fermé $SU(p,q)$ de $SL_{p+q}(\C)$ défini de la manière suivante,
\[
G = SU(p,q) = \{M\in SL_{p+q}(\C); \ ^t\overline{M}I_{p,q}M = I_{p,q}\}.
\]
Son algèbre de Lie est 
\[
\begin{array}{ccl}
\got{su}(p,q) & = & \{M\in \got{sl}_{p+q}(\C); \ ^t\overline{M}I_{p,q} + I_{p,q}M = 0\} \\
& = & \left\{\left(\begin{array}{cc}
A & B \\
^t\overline{B} & D
\end{array}\right); \ A\in\got{u}(p), D\in\got{u}(q), B\in\M_{p,q}(\C), \tr(D)=-\tr(A) \right\},
\end{array}
\]
où $\got{u}(n)$ est l'ensemble des matrices anti-hermitiennes de taille $n\times n$. Le groupe $SU(p,q)$ est de dimension $p^2+q^2+2pq-1$. On considère toujours l'involution de Cartan de $\got{su}(p,q)$ définie par : $M\mapsto -^t M$. Cela fixe une décomposition de Cartan $\got{su}(p,q) = \got{k} \oplus \got{p}$, avec
\[
\got{k} = \left\{\left(\begin{array}{cc}
A & 0 \\
0 & D
\end{array}\right); \ A\in\got{u}(p),D\in\got{u}(q) \text{ telles que } \tr(D)=-\tr(A)\right\}
\]
et
\[
\got{p} = \left\{\left(\begin{array}{cc}
0 & B \\
^t\overline{B} & 0
\end{array}\right); \ B\in\M_{p,q}(\C)\right\}.
\]
Les dimensions de ces deux espaces vectoriels sont $\dim\got{k} = p^2+q^2-1$ et $\dim\got{p} = 2pq$. Notons $K$ le sous-groupe compact maximal de $SU(p,q)$ d'algèbre de Lie $\got{k}$. Ce groupe est tout simplement l'ensemble de matrices suivant :
\[
K = \left\{\left(\begin{array}{cc}
A & 0 \\
0 & D
\end{array}\right); \ A\in U(p), D\in U(q) \text{ telles que } \det(D)=\det(A)^{-1}\right\},
\]
et on a un isomorphisme évident entre $K$ et le groupe de Lie $S(U(p)\times U(q))$.

On a une sous-algèbre de Cartan de $\got{k}$ canonique, c'est l'algèbre de Lie du tore maximal de $K$ des matrices diagonales de $SU(p,q)$
\[
T:=\left\{\left(\begin{array}{ccc}e^{i\theta_1} & & 0 \\ & \ddots & \\ 0 & & e^{i\theta_{p+q}}\end{array}\right);\ \theta_1,\ldots,\theta_{p+q}\in\R,\ \sum_{i=1}^{p+q}\theta_i=0\right\}.
\]
Cette sous-algèbre de Cartan de $\got{k}$ est l'ensemble de matrices suivant:
\[
\got{t} = \left\{\left(\begin{array}{ccc}
i\theta_1 & & \\
& \ddots & \\
& & i\theta_{p+q}
\end{array}\right); \ \theta_1,\ldots,\theta_{p+q}\in\R,\ \sum_{j=1}^{p+q} \theta_j = 0\right\},
\]
elle est de dimension $p+q-1$. On notera, de manière similaire à $U(n)$, la forme linéaire sur $\got{t}$
\[
e_i^*\left(\begin{array}{ccc}
ih_1 & & \\
& \ddots & \\
& & ih_{p+q}
\end{array}\right) = h_i,
\]
pour tout $i=1,\ldots,p+q$. Remarquons qu'ici, la famille $(e_1^*,\ldots,e_{p+q}^*)$ n'est pas libre, mais engendre $\got{t}^*$.


\bigskip

Posons $z_0 = \frac{i}{p+q}\tilde{I}_{n,1}$, un élément de $\got{z}(K)$, où
\[
\tilde{I}_{p,q}=\left(\begin{array}{cc}
qI_p & 0\\
0 & -pI_q
\end{array}\right).
\]
Cet élément est un des deux éléments de $\got{z}(K)$ qui définissent une structure complexe $K$-invariante sur $\got{p}$.

Le complexifié de $\got{su}(p,q)$ peut être identifié à $\got{sl}_{p+q}(\C)$. Avec cette identification, on a
\[
\got{p}_{\C} = \left\{
\left(\begin{array}{cc}
0 & B \\
^t C & 0
\end{array}\right); \ B,C\in\M_{p,q}(\C)
\right\} \subset \got{sl}_{p+q}(\C).
\]
Un calcul direct nous donne la décomposition en sous-espaces propres de $\ad(z_0)|_{\got{p}_{\C}}$ suivante :
\[
\got{p}^+ = \left\{
\left(\begin{array}{cc}
0 & B \\
0 & 0
\end{array}\right); \ B\in\M_{p,q}(\C)
\right\} \quad\text{et}\quad \got{p}^- = \left\{
\left(\begin{array}{cc}
0 & 0 \\
^t C & 0
\end{array}\right); \ C\in\M_{p,q}(\C)
\right\}.
\]
Il est clair que $\got{p}^+$ s'identifie à l'espace vectoriel $\mathcal{M}_{p,q}(\C)$. De plus, l'action de $S(U(p)\times U(q))$ sur $\got{p}^+$ passe à $\mathcal{M}_{p,q}(\C)$ en l'action
\[
(U,V)\cdot M := UMV^{-1}
\]
pour tout $M\in\mathcal{M}_{p,q}(\C)$ et tout $(U,V)\in S(U(p)\times U(q))$.

\bigskip

Les racines compactes sont les formes linéaires $\alpha_{i,j} = e_i^*-e_j^*$ sur $\got{t}$, pour $1\leqslant i,j\leqslant p$ et $p+1\leqslant i,j\leqslant p+q$, $i\neq j$. On prend pour système de racines compactes positives l'ensemble $\{\alpha_{i,j}\,;\ 1\leqslant i<j\leqslant p\} \cup \{\alpha_{i,j}\,;\ p+1\leqslant i<j\leqslant p+q\}$.

Les racines non compactes positives, racines provenant de $\got{p}^+$, sont les formes linéaires $\beta_{i,j} = e_i^*-e_j^*$ pour $1\leqslant i\leqslant p$ et $p+1\leqslant j\leqslant p+q$. Remarquons que l'on a $\beta_{i,j} = \beta_{i,l}+\alpha_{l,j} = \beta_{k,j}+\alpha_{i,k} = \beta_{k,l}+\alpha_{i,k}+\alpha_{l,j}$, pour des entiers $i,j,k,l$ \emph{ad hoc}. Par conséquent, la racine non compacte positive maximale (resp. minimale) est $\beta_{1,p+q}$ (resp. $\beta_{p,p+1}$). De plus, deux racines non compactes positives $\beta_{i,j}$ et $\beta_{k,l}$ sont fortement orthogonales si et seulement si $i\neq k$ et $j\neq l$. Ainsi, la famille $(\beta_{1,p+q},\beta_{2,p+q-1},\ldots,\beta_{q,p+1})$ est la famille fortement orthogonale maximale obtenue dans l'algorithme accompagnant le Théorème de Schmid.

De manière identique au cas du groupe $Sp(2n,\R)$, le Théorème de Schmid nous donne la décomposition du $S(U(p)\times U(q))$-module $\C[\got{p}^-]$:
\[
\C[\got{p}^-] = \sum_{m_1\geqslant\ldots\geqslant m_q\geqslant 0} V^{S(U(p)\times U(q))}_{(m_1,\ldots,m_q,0,\ldots,0;-m_q,\ldots,-m_1)}.
\]
On peut facilement voir que la représentation irréductible $V^{S(U(p)\times U(q))}_{(m_1,\ldots,m_q,0,\ldots,0;-m_q,\ldots,-m_1)}$ est isomorphe au produit tensoriel $V^{U(p)}_{(m_1,\ldots,m_q,0,\ldots,0)}\otimes(V^{U(q)}_{(m_1,\ldots,m_q)})^*$, pour tout poids dominant positif $(\mu_1\geqslant\ldots\geqslant\mu_q)$ de $U(q)$, lorque $q\geqslant 2$.

\begin{equation}
\label{theo:SU(p,q)_décompositiondeC[p-]}
\begin{minipage}{12cm}
\text{Pour $G = SU(p,q)$ avec $q\geqslant 2$, on a}
\[
\C[\got{p}^-] = \sum_{m_1\geqslant\ldots\geqslant m_q\geqslant 0} V^{U(p)}_{(m_1,\ldots,m_q,0,\ldots,0)}\otimes(V^{U(q)}_{(m_1,\ldots,m_q)})^*.
\]
\end{minipage}
\end{equation}

\paragraph*{Cas de $SU(n,1)$}

Tout ce qui a été écrit ci-dessus est bien évidemment applicable à $SU(n,1)$. Cependant, le cas de $SU(n,1)$ est bien plus simple. Tout d'abord, le sous-groupe compact maximal est ici identifiable à $U(n)$. Quant à la représentation $\got{p}^+$ de $U(n)$, elle correspond à l'action de $U(n)$ sur $\C^n$ tordu par la représentation déterminant de $U(n)$, c'est-à-dire,
\[
M\cdot v := \det(M)Mv
\]
pour tout $v\in\C^n$ et tout $M\in U(n)$. Il y a alors $n$ racines non compactes positives, notées $\beta_i := \beta_{i,n+1} = e_i^*-e_{n+1}^* = e_i^*+\sum_{k=1}^ne_k^*$. La racine non compacte positive maximale est $\beta_1$. Elle forme par ailleurs un système maximal de racines non compactes positives fortement orhogonales. On en déduit la décomposition suivante du $U(n)$-module $\C[\got{p}^-]$, donnée par le Théorème de Schmid, mais qui peut être vue directement aussi :

\begin{equation}
\label{theo:SU(n,1)_décompositiondeC[p-]}
\text{Pour $G = SU(n,1)$, on a $\C[\got{p}^-] = \sum_{k\geqslant 0} (\det)^k\otimes \mathrm{S}^k(\C^n)$.}
\end{equation}

\subsection{Le groupe $SO^*(2n)$, $n\geqslant 3$}
\label{subsection:Notations_SO*(2n)}

Le groupe $G = SO^*(2n)$, pour $n\geqslant 3$, est le sous-groupe de $SU(n,n)$ défini de la manière suivante,
\[
SO^*(2n) = \{M\in SU(n,n);\ ^tMI_{n,n}J_{n,n} M = I_{n,n}J_{n,n}\},
\]
avec
\[
I_{n,n} = \left(\begin{array}{cc}
I_n & 0 \\
0 & -I_n
\end{array}\right), \quad
J_{n,n} = \left(\begin{array}{cc}
0 & I_n \\
-I_n & 0
\end{array}\right) \quad \text{et}\quad
I_{n,n}J_{n,n} = \left(\begin{array}{cc}
0 & I_n \\
I_n & 0
\end{array}\right).
\]
Son algèbre de Lie est $\got{so}^*(2n) = \{M\in \got{su}(n,n);\ ^tMI_{n,n}J_{n,n} +I_{n,n}J_{n,n}M = 0\}$, de décomposition de Cartan $\got{so}^*(2n) = \got{k} \oplus \got{p}$, avec
\[
\got{k} = \left\{\left(\begin{array}{cc}
A & 0\\
0 & -^tA
\end{array}\right); \ A\in\got{u}(n)\right\} \cong \got{u}(n),
\]
et
\[
\got{p} = \left\{\left(\begin{array}{cc}
0 & B \\
^t\overline{B} & 0
\end{array}\right); \ B\in\got{o}(n,\C)\right\}.
\]Le sous-groupe compact maximal $K$ associé à $\got{k}$ est donc isomorphe à $U(n)$, par le morphisme de groupes
\[
U\in U(n) \longmapsto \left(\begin{array}{cc}
U & 0 \\
0 & \overline{U}
\end{array}\right)\in SO^*(2n),
\]
ce qui permet d'identifier le tore maximal des matrices diagonales de $U(n)$ au tore maximal
\[
T = \left\{\left(\begin{array}{cccccc}
e^{i\theta_1} & & & & & \\
& \ddots & & & & \\
& & e^{i\theta_n} & & & \\
& & & e^{-i\theta_1} & & \\
& & & & \ddots & \\
& & & & & e^{-i\theta_n}
\end{array}\right); \ \theta_1,\ldots,\theta_n\in\R\right\}.
\]
L'algèbre de Lie de $T$ est 
\[
\got{t} = \left\{\left(\begin{array}{cccccc}
i\theta_1 & & & & & \\
& \ddots & & & & \\
& & i\theta_n & & & \\
& & & -i\theta_1 & & \\
& & & & \ddots & \\
& & & & & -i\theta_n
\end{array}\right); \ \theta_1,\ldots,\theta_n\in\R\right\},
\]
et on note $e_j$ la matrice diagonale de $\got{t}$ avec $\theta_j=1$ et $\theta_k=0$ si $k\neq j$, pour $j=1,\ldots,n$. On pose $(e_1^*,\ldots,e_n^*)$ la base duale associée. Il est clair que dans cette base, les racines compactes de $\got{so}^*(2n)$ sont les formes linéaires $\alpha_{i,j}= e_i^*-e_j^*$, $1\leqslant i<j\leqslant n$.

On peut identifier le complexifié de $\got{p}$ au sous-espace vectoriel complexe
\[
\got{p}_{\C} = \left\{\left(\begin{array}{cc} 0 & B \\ C & 0\end{array}\right); B,C\in\got{o}(n,\C)\right\}.
\]
Le centre de $\got{k}$ est égal à $\{\theta I_{n,n}; \theta\in\R\}$, il est bien de dimension réelle $1$. On peut vérifier par le calcul que l'élément $z_0=\frac{I_{n,n}}{2}$ de $\got{z}(\got{k})$ donne par l'intermédiaire de $\ad(z_0)$ une structure complexe $K$-invariante sur $\got{p}$. En effet, on obtient $\ad(z_0)|_{\got{p}}^2 = -\id|_{\got{p}}$. On a
\[
\got{p}^+ = \left\{\left(\begin{array}{cc} 0 & B \\ 0 & 0\end{array}\right); B\in\got{o}(n,\C)\right\}
\quad\text{et}\quad\got{p}^- = \left\{\left(\begin{array}{cc} 0 & 0 \\ C & 0\end{array}\right); C\in\got{o}(n,\C)\right\}.
\]
Le sous-espace $\got{p}^+$ s'identifie clairement à l'ensemble des matrices anti-symétriques complexes $\got{o}(n,\C)$. L'action de $U(n)$ sur $\got{p}^+$ passe à $\got{o}(n,\C)$ par cette identification en l'action définie par
\[
U\cdot B := UB^tU \in \got{o}(n,\C)
\]
pour tout $B\in\got{o}(n,\C)$ et tout $U\in U(n)$. Il s'agit bien sûr de l'action de $U(n)$ sur $\got{o}(n,\C)$ induite par l'action standard de $GL_n(\C)$ sur ce même espace de matrices anti-symétriques. Elle peut se voir également comme l'action standard de $U(n)$ sur $\wedge^2\C^n$, induite de $GL_n(\C)$. On a donc bien $\got{p}^+\cong\wedge^2\C^n$ en tant que $U(n)$-modules.

\bigskip

Les racines non compactes positives de $\got{so}^*(2n)$ sont les formes linéaires $\beta_{i,j} = e_i^*+e_j^*$, pour $1\leqslant i< j \leqslant n$; ce sont les racines apparaissant dans $\got{p}^+$. Les racines non compactes négatives sont les $-\beta_{i,j}$, $1\leqslant i<j\leqslant n$.

Prenons pour système de racines compactes positives l'ensemble $\got{R}_c^+ = \{\alpha_{i,j}; 1\leqslant i<j\leqslant n\}$. Si maintenant on prend pour convention de noter indifféremment $\beta_{i,j}=\beta_{j,i}$ quel que soit l'ordre des entiers $i$ et $j$ de $\{1,\ldots,n\}$, alors $\beta_{i,j} + \alpha_{k,i} = \beta_{k,j}$, pour tout $k\in\{1,\ldots,n\}\setminus\{i,j\}$. Par conséquent, la plus grande racine non compacte positive est $\bmax = \beta_{1,2}$ et la plus petite est $\bmin = \beta_{n-1,n}$.

De plus, si $\beta_{i,j}$ est une racine non compacte positive et $k\notin\{i,j\}$, alors $\beta_{i,j}-\beta_{i,k} = \alpha_{j,k}$ est une racine compacte. Ce sera d'ailleurs la seule possibilité pour obtenir une racine compacte à partir de la somme de deux racines non compactes. On pose donc $\gamma_j = \beta_{2j-1,2j}$ pour tout $j\in\{1,\ldots,\mathrm{E}(n/2)\}$, où $\mathrm{E}(x)$ désigne la partie entière du réel $x$. La famille de racines non compactes positives $(\gamma_1,\ldots,\gamma_{\mathrm{E}(n/2)})$ est donc fortement orthogonale et maximale pour cette propriété. On obtient ainsi la décomposition du $U(n)$-module $\C[\got{p}^-]$ suivante.

\begin{equation}
\label{theo:SO*(4p)_décompositiondeC[p-]}
\text{Pour $G = SO^*(4p)$, on a $\C[\got{p}^-] = \sum_{m_1\geqslant\ldots\geqslant m_{p}\geqslant0} V^{U(2p)}_{(m_1,m_1,m_2,m_2,\ldots,m_p,m_p)}$}.
\end{equation}

\begin{equation}
\label{theo:SO*(4p+2)_décompositiondeC[p-]}
\text{Pour $G = SO^*(4p+2)$, on a $\C[\got{p}^-] = \sum_{m_1\geqslant\ldots\geqslant m_{p}\geqslant0} V^{U(2p+1)}_{(m_1,m_1,m_2,m_2,\ldots,m_p,m_p,0)}$}.
\end{equation}


\subsection{Le groupe $SO(2p+1,2)$, $p\geqslant 1$}
\label{subsection:Notations_SO(2p+1,2)}

Le groupe de Lie $G = SO(2p+1,2)$ est le sous-groupe de $SL_{2p+3}(\R)$ défini de la manière suivante,
\[
SO(2p+1,2) = \{M\in SL_{2p+3}(\R);\ ^tMI_{2p+1,2}M = I_{2p+1,2}\},
\]
avec
\[
I_{2p+1,2} = \left(\begin{array}{cc}
I_{2p+1} & 0 \\
0 & -I_2
\end{array}\right).
\]
Le groupe $SO(2p+1,2)$ a deux composantes connexes. Son algèbre de Lie est $\got{so}(2p+1,2) = \{M\in \got{sl}_{2p+3}(\R);\ ^tMI_{2p+1,2} + I_{2p+1,2}M = 0\}$, de décomposition de Cartan $\got{so}(2p+1,2) = \got{k} \oplus \got{p}$, avec
\[
\got{k} = \left\{\left(\begin{array}{cc}
A & 0\\
0 & D
\end{array}\right); \ A\in\got{so}(2p+1,\R), D\in\got{so}(2,\R)\right\} \cong \got{so}(2p+1,\R)\oplus\got{so}(2,\R),
\]
et
\[
\got{p} = \left\{\left(\begin{array}{cc}
0 & B \\
^tB & 0
\end{array}\right); \ B\in\mathcal{M}_{2p+1,2}(\R)\right\}.
\]
Le sous-groupe compact maximal $K$ associé à $\got{k}$ est donc isomorphe à $SO(2p+1,\R)\times SO(2,\R)$, par l'intermédiaire du morphisme de groupes
\[
(M,M')\in SO(2p+1,\R)\times SO(2,\R) \longmapsto \left(\begin{array}{cc}
M & 0 \\
0 & M'
\end{array}\right)\in SO(2p+1,2).
\]
On fixe la sous-algèbre de Cartan de $K$ suivante,
\[
\got{t} = \left\{\left(\begin{array}{ccccc}
\left(\begin{array}{cc}0 & h_1\\ -h_1&0\end{array}\right)& & & & \\
& \ddots & & & \\
& & \left(\begin{array}{cc}0 & h_p\\ -h_p&0\end{array}\right)& & \\
& & & 0& \\
& & & & \left(\begin{array}{cc}0 & h_{p+1}\\ -h_{p+1}&0\end{array}\right)
\end{array}\right); \ h_i\in\R
\right\}.
\]
et on note $(e_1,\ldots,e_{p+1})$ la base canonique de $\got{t}$ associée à cette décomposition. On notera également $(e_1^*,\ldots,e_{p+1}^*)$ la base duale associée. Il est clair que le centre de $K$ est isomorphe à $\{1\}\times SO(2,\R)$ et son algèbre de Lie est $\R e_{p+1}$. Les racines compactes sont les formes linéaires $\pm e_i^*\pm e_j^*$, pour $i\neq j$ dans $\{1,\ldots,p\}$, et $\pm e_k$, pour tout $k=1,\ldots,p$. On choisit le système de racines compactes positives $\got{R}_c^+=\{e_i^*\pm e_j^*; 1\leqslant i<j\leqslant p\}\cup\{e_k^*; k=1,\ldots,p\}$.

On peut identifier le complexifié de $\got{p}$ à l'algèbre de Lie complexe
\[
\got{p}_{\C} = \left\{\left(\begin{array}{cc} 0 & B \\ ^tB & 0\end{array}\right); B\in\mathcal{M}_{2p+1,2}(\C)\right\}.
\]
On peut vérifier par le calcul que l'élément $z_0=e_{p+1}$ de $\got{z}(K)$ donne par l'intermédiaire de $\ad(z_0)$ une structure complexe $K$-invariante sur $\got{p}$. En effet, si on note
\[
D:=\left(\begin{array}{cc}
0 & 1 \\
-1 & 0
\end{array}\right),
\]
alors, comme matrice de $\mathcal{M}_{2p+3}(\R)$, on a
\[
e_{p+1} = \left(\begin{array}{cc}
0 & 0\\
0 & D
\end{array}\right),
\]
et cela nous donne
\begin{align*}
\ad(e_{p+1})\left(\begin{array}{cc}
0 & B\\
^tB & 0
\end{array}\right) & = \left(\begin{array}{cc}
0 & 0\\
0 & D
\end{array}\right)\left(\begin{array}{cc}
0 & B\\
^tB & 0
\end{array}\right) - \left(\begin{array}{cc}
0 & B\\
^tB & 0
\end{array}\right)\left(\begin{array}{cc}
0 & 0\\
0 & D
\end{array}\right) \\
& = -\left(\begin{array}{cc}
0 & BD\\
^t(BD) & 0
\end{array}\right)
\end{align*}
En remarquant que $D^2 = -I_2$, on obtient donc $\ad(z_0)|_{\got{p}}^2 = -\id|_{\got{p}}$. On a
\[
\got{p}^+ = \left\{\left(\begin{array}{ccc} 0 & C & -iC \\
^tC & 0 & 0\\
-i^tC & 0 & 0
\end{array}\right); C\in\mathcal{M}_{2p+1,1}(\C)\right\}
\]
et
\[
\got{p}^- = \left\{\left(\begin{array}{ccc} 0 & C & iC \\
^tC & 0 & 0\\
i^tC & 0 & 0
\end{array}\right); C\in\mathcal{M}_{2p+1,1}(\C)\right\}.
\]
En particulier, $\got{p}^+$ s'identifie au $SO(2p+1,\R)\times SO(2,\R)$-module $\C^{2p+1}$, où l'action de $SO(2p+1,\R)\times SO(2,\R)$ sur $\C^{2p+1}$ est définie par
\[
(M,e^{i\theta})\cdot v := e^{i\theta}Mv,
\]
pour tout $v\in\C^{2p+1}$ et tout $(M,e^{i\theta})\in SO(2p+1,\R)\times SO(2,\R)$.

\bigskip

Les racines non compactes de $\got{so}(2p+1,2)$ sont les $4p+2$ formes linéaires $\pm e_i^*\pm e_{p+1}^*$, pour $i=1,\ldots p$, et $\pm e_{p+1}^*$. Les racines non compactes positives sont les racines $\got{R}_n^+=\{\pm e_i^*+e_{p+1}^*; i=1,\ldots,p\}\cup\{e_{p+1}^*\}$. Du choix du système de racines compactes positives, on obtient l'ordre suivant des racines non compactes positives:
\[
-e_1^*+e_{p+1}^* < \cdots < -e_p^*+e_{p+1}^* < e_{p+1}^* < e_p^*+e_{p+1}^* < \cdots < e_1^*+e_{p+1}^*.
\]
De plus, $e_{p+1}^*$ n'est fortement orthogonale à aucune autre racine non compacte, alors que la racine $\pm e_i^*+e_{p+1}^*$ n'est fortement orthogonale qu'à la racine non compacte positive $\mp e_i^*+e_{p+1}^*$. Par conséquent, la famille fortement orthogonale obtenue par l'algorithme précédant le Théorème \ref{theo:schmid} est la famile $(e_1^*+e_{p+1}^*,-e_1^*+e_{p+1}^*)$.

\begin{equation}
\label{eq:SO*(2n)_décompositiondeC[p-]}
\text{Pour $G = SO(2p+1,2)$, on a $\C[\got{p}^-] = \sum_{m_1\geqslant m_2\geqslant0} V^{SO(2p+1)\times SO(2)}_{(m_1-m_2,0,\ldots,0,m_1+m_2)}$}.
\end{equation}

\subsection{Le groupe $SO(2p,2)$, $p\geqslant 2$}
\label{subsection:Notations_SO(2p,2)}

Les notations pour le groupe $G=SO(2p,2)$ sont quasiment identiques à celles du groupe $SO(2p+1,2)$ données dans le paragraphe \ref{subsection:Notations_SO(2p+1,2)}. En particulier, on a la décomposition de Cartan $\got{so}(2p,2) = \got{k}\oplus\got{p}$, où
\[
\got{k} = \left\{\left(\begin{array}{cc}
A & 0\\
0 & D
\end{array}\right); \ A\in\got{so}(2p,\R), D\in\got{so}(2,\R)\right\}
\]
et
\[
\got{p} = \left\{\left(\begin{array}{cc}
0 & B \\
^tB & 0
\end{array}\right); \ B\in\mathcal{M}_{2p,2}(\R)\right\}.
\]
Le sous-groupe compact maximal de $SO(2p,2)$ d'algèbre de Lie $\got{k}$ est donc isomorphe $SO(2p,\R)\times SO(2,\R)$. La sous-algèbre de Cartan que l'on fixe ici diffère légèrement de celle de $SO(2p+1,2)$, c'est l'algèbre suivante,
\[
\got{t} = \left\{\left(\begin{array}{cccc}
\left(\begin{array}{cc}0 & h_1\\ -h_1&0\end{array}\right)& & & \\
& \ddots & & \\
& & \left(\begin{array}{cc}0 & h_p\\ -h_p&0\end{array}\right)& \\
& & & \left(\begin{array}{cc}0 & h_{p+1}\\ -h_{p+1}&0\end{array}\right)
\end{array}\right); \ h_i\in\R\right\}.
\]
On note $(e_1,\ldots,e_{p+1})$ la base canonique de $\got{t}$ et $(e_1^*,\ldots,e_{p+1}^*)$ la base duale. Les racines compactes sont les formes linéaires $\pm e_i^*\pm e_j^*$, pour $i\neq j$ dans $\{1,\ldots,p\}$. On choisit le système standard de racines compactes positives $\got{R}_c^+=\{e_i^*\pm e_j^*; 1\leqslant i<j\leqslant p\}$.

Les racines non compactes sont les formes linéaires $\pm e_i^*\pm e_{p+1}^*$, pour $i=1,\ldots,p$. Les racines non compactes positives sont les racines de $\got{R}_n^+=\{\pm e_i^*+e_{p+1}^*; i=1,\ldots,p\}$. On notera ces racines
\[
\beta^{\pm}_i := \pm e_i^* + e_{p+1}^*, \quad \forall i\in\{1,\ldots,p\}.
\]
\`A nouveau, le choix effectué sur le système de racines compactes positives nous donne l'ordre des racines non compactes positives  suivant:
\[
-e_1^*+e_{p+1}^* < \cdots < -e_p^*+e_{p+1}^* < e_p^*+e_{p+1}^* < \cdots < e_1^*+e_{p+1}^*.
\]
La famille fortement orthogonale obtenue par l'algorithme précédant le Théorème \ref{theo:schmid} est encore la famile $(e_1^*+e_{p+1}^*,-e_1^*+e_{p+1}^*)$.

\begin{equation}
\label{eq:SO_0(2p,2)_décompositiondeC[p-]}
\text{Pour $G = SO(2p,2)$, on a $\C[\got{p}^-] = \sum_{m_1\geqslant m_2\geqslant0} V^{SO(2p)\times SO(2)}_{(m_1-m_2,0,\ldots,0,m_1+m_2)}$}.
\end{equation}


\chapter[Premières tentatives de calculs de projections d'orbites]{Premières tentatives de calcul de projection d'orbites coadjointes holomorphes}
\label{chap:premières_tentatives}

Avant d'entrer véritablement dans l'étude théorique des orbites coadjointes holomorphes et de leurs projections, nous proposons dans ce chapitre d'expliciter le polyèdre moment d'une telle projection d'orbite pour certains groupes de la classification des espaces symétriques hermitiens irréductibles.

Pour ce faire, nous suivons deux méthodes différentes. La première utilise la Formule de Duflo-Heckman-Vergne, en calculant le polyèdre moment comme support d'une certaine mesure. La seconde s'appuie sur les résultats obtenus autour du problème de Horn, vus dans le paragraphe \ref{subsection:problème_de_Horn}, permettant de déterminer les points rationnels du polyèdre en question. Ceci nous permet de donner les équations des projections d'orbites coadjointes holomorphes pour les groupes $Sp(2n,\R)$, $SU(n,1)$ et $SO^*(6)$.

\section{Formule de Duflo-Heckman-Vergne}

Dans cette section, nous décrivons tout d'abord brièvement la Formule de Duflo-Heckman-Vergne. Pour plus de détails, le lecteur peut se référer à l'article \cite{duflo84}. Nous en déduisons ensuite les projections d'orbites coadjointes holomorphes associées aux groupes $SU(2,1)$ et $Sp(4,\R)$.

\subsection{\'Enoncé de la formule}

Soit $G$ un groupe de Lie réel semi-simple connexe non compact à centre fini et soit $\got{g}$ son algèbre de Lie. Fixons $\got{g}=\got{k}\oplus\got{p}$ une décomposition de Cartan de $\got{g}$ et supposons que $\got{g}$ et $\got{k}$ sont de même rang. Soit $K$ le sous-groupe de Lie connexe de $G$ d'algèbre de Lie $\got{k}$. Le sous-groupe $K$ est compact maximal dans $G$. On fixe $T$ un tore maximal de $K$ et $\got{t}$ son algèbre de Lie. L'algèbre de Lie $\got{t}$ est également une sous-algèbre de Cartan de $G$.

Soit $\Lambda$ un élément $G$-régulier de $\got{t}^*$, c'est-à-dire que son stabilisateur $G_{\Lambda}$ est égal à $T$. D'après le paragraphe \ref{subsection:Gsemi-simple}, $\Orb_{\Lambda}$ est fermée dans $\got{g}^*$ et la projection d'orbite $\Phi_{K}:\Orb_{\Lambda}\rightarrow\got{k}^*$ est propre.

La structure symplectique $\Omega_{\Orb_{\Lambda}}$ induit une orientation et une mesure
\[
\beta_{\Orb_{\Lambda}} := \frac{1}{(2\pi)^n n!}\Omega_{\Orb_{\Lambda}}^n
\]
sur $\Orb_{\Lambda}$. La mesure $\beta_{\Orb_{\Lambda}}$ est communément appelée \emph{mesure de Liouville}\index{Mesure de Liouville} de la variété symplectique $(\Orb_{\Lambda},\Omega_{\Orb_{\Lambda}})$.

On note $W$ le groupe de Weyl de $T$ dans $K$. Pour tout élément $n\in N_K(T)$, l'application $\Ad(n)$ sur $\got{g}$ définit par restriction un automorphisme de $\got{t}$, qui ne dépend que de la classe de $n$ dans $W$. On note $\varepsilon(w)$ le déterminant de l'élément de $GL(\got{t})$ donné par un représentant quelconque de la classe $w\in W$.

Soit $\got{R}$ l'ensemble des racines de $\got{t}$ dans $\got{g}_{\C}$. On désigne par $\got{R}_c$ (resp. $\got{R}_n$) les racines de $\got{t}$ dans $\got{k}_{\C}$ (resp. dans $\got{p}_{\C}$). Pour tout $\mu\in\got{t}^*$, on définit
\[
\got{R}^+(\mu) := \{\alpha\in\got{R}; (\alpha,\mu)>0\},
\]
où $(\cdot,\cdot)$ est le produit scalaire sur $\got{t}^*$ induit par la forme de Killing $B_{\got{g}}$. On notera alors $\got{R}_n^+(\mu):=\got{R}^+(\mu)\cap\got{R}_n$.

Pour tout $\mu\in\got{t}^*$, on définit également $\delta_{\mu}$ la mesure de Dirac au point $\mu$ sur $\got{k}^*$ et $\mathcal{H}_{\mu}$ la mesure de Heavyside\index{Mesure de Heavyside}, c'est-à-dire la mesure de Radon sur $\got{k}^*$ définie par
\[
\langle\mathcal{H}_{\mu},\varphi\rangle := \int_{0}^{\infty}\varphi(t\mu)dt \qquad \forall\varphi\in C^{\infty}(\got{k}^*).
\]

Si $\got{R}_n^+(\mu) = \{\alpha_1,\ldots,\alpha_p\}$, on notera
\[
Y^+_{\mu} = \mathcal{H}_{\alpha_1} * \mathcal{H}_{\alpha_2} * \dots * \mathcal{H}_{\alpha_p}
\]
le produit de convolution des mesures de Heavyside associées aux racines non compactes positives pour $\mu$. Enfin, on posera, pour tout $\nu\in\got{t}^*$,
\[
\pi^+(\nu) = \prod_{\alpha\in\got{R}_n^+(\Lambda)}(\alpha,\nu).
\]
La fonction $\pi^+$ est polynômiale et $W$-anti-invariante sur $\got{t}^*$.

Notons $\got{t}_{reg}^*$ l'ensemble des éléments réguliers de $\got{t}^*$ pour l'action de $K$. La variété $M := (\Phi_K)^{-1}(\got{t}_{reg}^*)$ est une sous-variété localement fermée $T$-stable de $\Orb_{\Lambda}$, non vide car $M$ contient nécessairement $\Lambda$, et
\[
\Phi_M := \Phi_K|_{M}:M\rightarrow\got{t}^*
\]
est une application moment pour l'action de $T$ sur la variété symplectique $(M,\Omega_{\Orb_{\Lambda}}|_{M})$. Ce résultat est bien connu et est utilisé dans la démonstration du Théorème de convexité hamiltonienne non abélien, cf \cite[Théorème 6.4]{guillemin82} pour le cas compact, \cite[Théorème 3.1]{lerman98} pour un cadre plus général. 
L'ensemble $K\cdot M\cong K\times_T M$ est un ouvert dense de $\Orb_{\Lambda}$. Par conséquent, $\Orb_{\Lambda}\setminus(K\cdot M)$ est de mesure nulle pour $\beta_{\Orb_{\Lambda}}$. Pour déterminer la mesure image $(\Phi_K)_*(\beta_{\Orb_{\Lambda}})$, il suffit donc de connaître la mesure $(\Phi_M)_*(\beta_M)$.

\begin{theo}[Formule de Duflo-Heckman-Vergne]
\label{theo:fomule_DHV}
\index{Formule de Duflo-Heckman-Vergne}
L'image de la mesure de Liouville par l'application moment $\Phi_{M}$ est donnée par la formule
\[
(\Phi_{M})_*(\beta_{M}) = \frac{\pi^+}{|\pi^+|}\sum_{w\in W}\varepsilon(w)(\delta_{w\Lambda}* Y^+_{w\Lambda}).
\]
\end{theo}

\begin{rema}
Duflo, Heckman et Vergne ont prouvé cette formule dans le cadre plus général d'une algèbre de Lie réductive réelle $\got{g}$ et un groupe de Lie $G$ connexe à centre compact d'algèbre de Lie $\got{g}$.
\end{rema}

Le Théorème 3.1 de \cite{lerman98} affirme également que $\Delta_K(\Orb_{\Lambda})$ est l'adhérence de $\Delta_T(M)\cap\got{t}_+^*$. 
On en déduit que le calcul du support de la mesure image $(\Phi_M)_*(\beta_M)$ permet de déterminer le polyèdre moment $\Delta_K(\Orb_{\Lambda}) := \Phi_K(\Orb_{\Lambda})\cap\got{t}_+^*$.

\subsection{Le cas de $SU(2,1)$}
\label{subsection:exemple_utilisation_formule_DHV}

Nous terminons ce chapitre en proposant une application de la formule de Duflo-Heckman-Vergne au calcul du polyèdre moment de la projection d'une orbite coadjointe elliptique régulière de $SU(2,1)$.

Les notations pour le groupe $SU(2,1)$ et son algèbre de Lie ont été fixées en \ref{subsection:Notations_SU(p,q)}. Nous choisissons pour chambre de Weyl de $\got{t}^*$, la chambre $\got{t}_+^* := \{\xi\in\got{t}^*; (\alpha,\xi)\geqslant 0\}$. Le groupe de Weyl $W$ de $T$ dans $K$ possède deux éléments, l'élément non trivial étant la symétrie orthogonale pour $B_{\theta}$ relativement à la droite vectorielle $\ker\alpha$. Si on note $s$ cette symétrie, on a alors
\[
s\alpha = -\alpha, \qquad s\beta_1 = \beta_2, \qquad s\beta_2 = \beta_1.
\]

Prenons maintenant un élément $\Lambda\in\got{t}_+^*$ qui est $SU(2,1)$-régulier. Trois cas distincts apparaissent ici.

\subsubsection*{L'élément $\Lambda$ est dans $\{\xi\in\got{t}^*_+; (\beta_1,\xi)>0, (\beta_2,\xi)>0\}$}

Les racines non compactes positives associées sont donc celles de $\got{R}_n^+(\Lambda) = \{\beta_1,\beta_2\}$ et on a également $\got{R}_n^+(s\Lambda) = \{\beta_1,\beta_2\}$. Par le calcul, la Formule de Duflo-Heckman-Vergne donne la mesure
\[
(\Phi_{M})_*(\beta_{M}) = \frac{\pi^+}{|\pi^+|}c(\mathbf{1}_{\Lambda+\R_+\beta_1+\R_+\beta_2}-\mathbf{1}_{s\Lambda+\R_+\beta_1+\R_+\beta_2})d\mu.
\]
On en déduit que le support de la mesure $(\Phi_{M})_*(\beta_{M})$ est égal à l'ensemble
\[
\Bigl((\Lambda+\R_+\beta_1+\R_+\beta_2)\cup(s\Lambda+\R_+\beta_1+\R_+\beta_2)\Bigr) \setminus \Bigl((\Lambda+\R_+\beta_1+\R_+\beta_2)\cap(s\Lambda+\R_+\beta_1+\R_+\beta_2)\Bigr),
\]
dont l'intersection avec la chambre de Weyl $\got{t}_+^*$ est bien un polyèdre convexe.

\begin{figure}[ht]
\begin{center}
\setlength{\unitlength}{1cm}
\resizebox{8cm}{8cm}{
\begin{pspicture}(10,10)
\newgray{lightlightgray}{0.90}

\pspolygon[linecolor=lightlightgray,fillstyle=solid,linewidth=0pt,fillcolor=lightlightgray](5,0)(5,10)(10.5,10)(10,5)(9.5,3)(9,2)(8,0.8)(7,0.2)
\pspolygon[linecolor=lightgray,fillstyle=solid,linewidth=0pt,fillcolor=lightgray](7,4.5)(5,7.98)(6.16,10)(10.16,10)
\pspolygon[linecolor=gray,fillstyle=solid,linewidth=0pt,fillcolor=gray](3,4.5)(5,7.98)(6.16,10)(-0.16,10)

\psdot(3,4.5)
\psdot(7,4.5)

\psline[linewidth=1pt](5,0)(5,10)
\psline[linewidth=1pt](2.52,0.5)(10,4.95)
\psline[linewidth=1pt](7.48,0.5)(0,4.95)
\psline[linewidth=0.5pt](10.16,10)(7,4.5)(5,7.98)(6.16,10)

\psline[linewidth=1pt]{->}(5,2)(6,2)
\psline[linewidth=1pt]{->}(5,2)(5.5,2.87)
\psline[linewidth=1pt]{->}(5,2)(4.5,2.87)

\psline[linewidth=1pt]{->}(7,4.5)(7.5,5.37)
\psline[linewidth=1pt]{->}(7,4.5)(6.5,5.37)

\rput(7,4.2){$\Lambda$}
\rput(3,4.2){$s\Lambda$}
\rput(6,1.8){$\alpha$}
\rput(4.2,3){$\beta_{1}$}
\rput(5.8,3){$\beta_{2}$}
\rput(5.5,0.5){$\got{t}^*_+$}
\rput(7,8){$\Delta_K(G\cdot\Lambda)$}
\rput(6.4,5){$\beta_{1}$}
\rput(7.6,5){$\beta_{2}$}

\end{pspicture}
}
\end{center}
\caption{Polyèdre de la projection d'une orbite $SU(2,1)$-régulière avec $(\beta_1,\Lambda)>0$ et $(\beta_2,\Lambda)>0$}
\end{figure}

\subsubsection*{L'élément $\Lambda$ est dans $\{\xi\in\got{t}^*_+; (\beta_1,\xi)<0, (\beta_2,\xi)>0\}$}

Les calculs sont similaires aux paragraphes précédents. Les éléments qui diffèrent sont les ensembles des racines non compactes positives associées à $\Lambda$ et $s\Lambda$, qui sont ici $\got{R}_n^+(\Lambda) = \{-\beta_1,\beta_2\}$ et $\got{R}_n^+(s\Lambda) = \{\beta_1,-\beta_2\}$, ainsi que la mesure image, qui vaut
\[
(\Phi_{M})_*(\beta_{M}) = \frac{\pi^+}{|\pi^+|}c(\mathbf{1}_{\Lambda+\R_+(-\beta_1)+\R_+\beta_2}-\mathbf{1}_{s\Lambda+\R_+\beta_1+\R_+(-\beta_2}))d\mu,
\]
pour une certaine constant $c>0$. Remarquons qu'ici, le cône $s\Lambda+\R_+\beta_1+\R_+(-\beta_2)$ est contenu dans le demi-plan $\{(\alpha,\cdot)< 0\}$, alors que le cône $\Lambda+\R_+(-\beta_1)+\R_+\beta_2$ est contenu dans le demi-plan opposé $\{(\alpha,\cdot)> 0\}$. Par conséquent, ces deux cônes ne s'intersectent pas et la mesure $(\Phi_{M})_*(\beta_{M})$ a pour support sur $\got{t}^*$ l'ensemble
\[
\bigl(\Lambda+\R_+(-\beta_1)+\R_+\beta_2\bigr)\cup\bigl(s\Lambda+\R_+\beta_1+\R_+(-\beta_2)\bigr).
\]

\begin{figure}[ht]
\begin{center}
\setlength{\unitlength}{1cm}
\resizebox{8cm}{8cm}{
\begin{pspicture}(10,10)
\newgray{lightlightgray}{0.90}

\pspolygon[linecolor=lightlightgray,fillstyle=solid,linewidth=0pt,fillcolor=lightlightgray](5,0)(5,10)(10,10)(10,0)
\pspolygon[linecolor=lightgray,fillstyle=solid,linewidth=0pt,fillcolor=lightgray](7.5,4.5)(10,8.85)(10,0.15)
\pspolygon[linecolor=gray,fillstyle=solid,linewidth=0pt,fillcolor=gray](2.5,4.5)(0,8.85)(0,0.15)

\psdot(2.5,4.5)
\psdot(7.5,4.5)

\psline[linewidth=1pt](5,0)(5,10)
\psline[linewidth=1pt](0,2.13)(10,7.87)
\psline[linewidth=1pt](0,7.87)(10,2.13)
\psline[linewidth=0.5pt](10,8.85)(7.5,4.5)(10,0.15)

\psline[linewidth=1pt]{->}(5,5)(6,5)
\psline[linewidth=1pt]{->}(5,5)(5.5,5.87)
\psline[linewidth=1pt]{->}(5,5)(4.5,5.87)

\rput(2.9,4.5){$s\Lambda$}
\rput(7.2,4.5){$\Lambda$}
\rput(6,4.8){$\alpha$}
\rput(4.2,6){$\beta_{1}$}
\rput(5.8,6){$\beta_{2}$}
\rput(5.5,0.5){$\got{t}^*_+$}
\rput(9,5){$\Delta_K(G\cdot\Lambda)$}

\end{pspicture}
}
\end{center}
\caption{Polyèdre de la projection d'une orbite $SU(2,1)$-régulière avec $(\beta_1,\Lambda)<0$ et $(\beta_2,\Lambda)>0$}
\end{figure}

\subsubsection*{L'élément $\Lambda$ est dans $\{\xi\in\got{t}^*_+; (\beta_1,\xi)>0, (\beta_2,\xi)<0\}$}

Il s'agit ici du cas \og{}opposé\fg{} au premier cas. La mesure obtenue par la Formule de Duflo-Heckman-Vergne est
\[
(\Phi_{M})_*(\beta_{M}) = \frac{\pi^+}{|\pi^+|}c(\mathbf{1}_{\Lambda+\R_+(-\beta_1)+\R_+(-\beta_2)}-\mathbf{1}_{s\Lambda+\R_+(-\beta_1)+\R_+(-\beta_2)})d\mu.
\]

\begin{rema}
Ces calculs, très simples pour le groupe $SU(2,1)$, se compliquent dès que le nombre de racines non compactes positives est supérieur au rang de $\got{g}$. C'est le cas de $Sp(2n,\R)$. Même en ne considérant que les groupes de type $SU(n,1)$, les difficultés apparaissent rapidement lorsque $n$ grandit.
\end{rema}

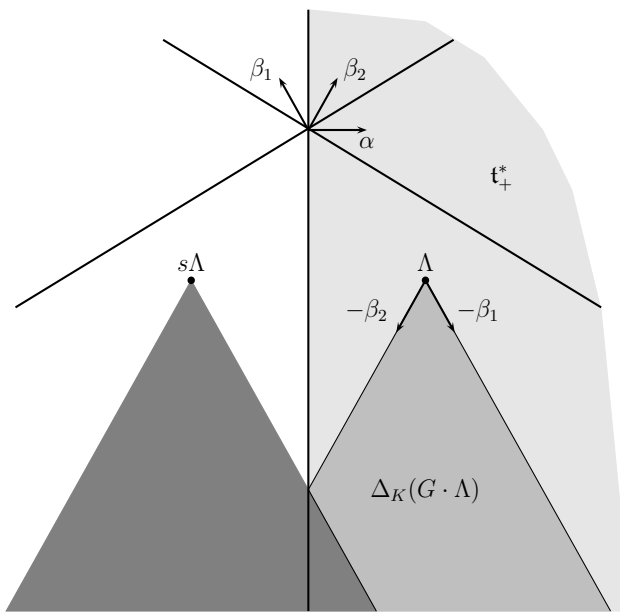
\begin{figure}[ht]
\begin{center}
\setlength{\unitlength}{1cm}
\resizebox{8cm}{8cm}{
\begin{pspicture}(10,10)
\newgray{lightlightgray}{0.90}

\pspolygon[linecolor=lightlightgray,fillstyle=solid,linewidth=0pt,fillcolor=lightlightgray](5,10)(5,0)(10.5,0)(10,5)(9.5,7)(9,8)(8,9.2)(7,9.8)
\pspolygon[linecolor=lightgray,fillstyle=solid,linewidth=0pt,fillcolor=lightgray](7,5.5)(5,2.02)(6.16,0)(10.16,0)
\pspolygon[linecolor=gray,fillstyle=solid,linewidth=0pt,fillcolor=gray](3,5.5)(5,2.02)(6.16,0)(-0.16,0)

\psdot(3,5.5)
\psdot(7,5.5)

\psline[linewidth=1pt](5,0)(5,10)
\psline[linewidth=1pt](2.52,9.5)(10,5.05)
\psline[linewidth=1pt](7.48,9.5)(0,5.05)
\psline[linewidth=0.5pt](10.16,0)(7,5.5)(5,2.02)(6.16,0)

\psline[linewidth=1pt]{->}(5,8)(6,8)
\psline[linewidth=1pt]{->}(5,8)(5.5,8.87)
\psline[linewidth=1pt]{->}(5,8)(4.5,8.87)

\psline[linewidth=1pt]{->}(7,5.5)(6.5,4.63)
\psline[linewidth=1pt]{->}(7,5.5)(7.5,4.63)

\rput(7,5.8){$\Lambda$}
\rput(3,5.8){$s\Lambda$}
\rput(6,7.8){$\alpha$}
\rput(4.2,9){$\beta_{1}$}
\rput(5.8,9){$\beta_{2}$}
\rput(6,5){$-\beta_{2}$}
\rput(7.9,5){$-\beta_{1}$}
\rput(8.3,7.2){$\got{t}^*_+$}
\rput(7,2){$\Delta_K(G\cdot\Lambda)$}

\end{pspicture}
}
\end{center}
\caption{Polyèdre de la projection d'une orbite $SU(2,1)$-régulière avec $(\beta_1,\Lambda)<0$ et $(\beta_2,\Lambda)<0$}
\end{figure}

Pour $SU(2,1)$, lorsque $\Lambda$ est dans la chambre holomorphe, on voit que les cônes affines engendrés par les racines non compactes positives et de sommets respectifs $\Lambda$ et $s\Lambda$ ont une intersection non vide dans $\got{t}_+^*$. Les deux mesures portées par ses cônes s'y annihilent, ce qui donne une sorte de réflexion du polyèdre moment lorsque celui-ci touche le mur de la chambre de Weyl.

Pour d'autres groupes, les mesures obtenues sur les différents cônes peuvent ne pas se neutraliser sur les parties communes de leurs supports. Ceci est le cas par exemple pour le groupe $Sp(4,\R)$, schématisé par la Figure 4.

\subsection{Le cas de $Sp(4,\R)$}

Disons quelques mots pour le cas du groupe $Sp(4,\R)$. Les notations de ce groupe ont été introduites en \ref{subsection:Notations_Sp(R2n)}. La chambre de Weyl choisie est la chambre $\got{t}_+^* := \{\xi\in\got{t}^*; (\alpha,\xi)\geqslant 0\}$. Le groupe de Weyl a deux éléments, dont l'élément non trivial est la symétrie $s$ relativement au mur $\ker\alpha$ de la chambre de Weyl. La symétrie $s$ agit sur l'ensemble des racines non compactes positives de $\got{sp}(\R^4)$ par
\[
s\beta_{1,1} = \beta_{2,2}, \quad s\beta_{2,2} = \beta_{1,1} \quad \text{et}\quad s\beta_{1,2} = \beta_{1,2}.
\]

Soit $\Lambda\in\Chol$ qui est $Sp(4,\R)$-régulier (nous ne parlerons pas ici des éléments des autres sous-chambres de $\got{t}_+^*$). Dans ce cas, on peut voir par le calcul que la mesure $(\Phi_{M})_*(\beta_{M})$ a pour support dans $\got{t}_+^*$ le polyèdre convexe
\[
\got{t}_+^* \cap (\Lambda + \R_+\beta_{1,1} + \R_+\beta_{2,2}).
\]

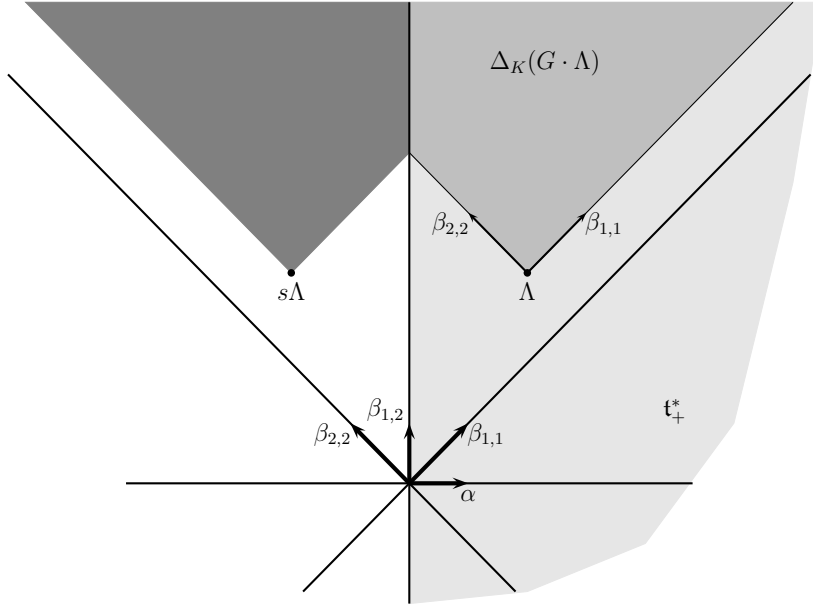
\begin{figure}[ht]
\label{figure:polyèdremoment_Sp(R4)}
\begin{center}
\setlength{\unitlength}{1cm}
\resizebox{11.2cm}{8cm}{
\begin{pspicture}(14,10)
\newgray{lightlightgray}{0.90}

\pspolygon[linecolor=lightlightgray,fillstyle=solid,linewidth=0pt,fillcolor=lightlightgray](7,0)(7,10)(14,10)(13.5,7)(12.5,3)(11,1)(9,0.2)
\pspolygon[linecolor=lightgray,fillstyle=solid,linewidth=0pt,fillcolor=lightgray](9,5.5)(7,7.5)(7,10)(13.5,10)
\pspolygon[linecolor=gray,fillstyle=solid,linewidth=0pt,fillcolor=gray](5,5.5)(7,7.5)(7,10)(0.5,10)

\psdot(5,5.5)
\psdot(9,5.5)

\psline[linewidth=1pt](7,0)(7,10)
\psline[linewidth=1pt](2.2,2)(11.8,2)
\psline[linewidth=1pt](5.2,0.2)(13.8,8.8)
\psline[linewidth=1pt](0.2,8.8)(8.8,0.2)
\psline[linewidth=0.5pt](7,7.5)(9,5.5)(13.5,10)

\psline[linewidth=2pt]{->}(7,2)(8,2)
\psline[linewidth=2pt]{->}(7,2)(8,3)
\psline[linewidth=2pt]{->}(7,2)(6,3)
\psline[linewidth=2pt]{->}(7,2)(7,3)
\psline[linewidth=1pt]{->}(9,5.5)(10,6.5)
\psline[linewidth=1pt]{->}(9,5.5)(8,6.5)

\rput(9,5.2){$\Lambda$}
\rput(5,5.2){$s\Lambda$}
\rput(8,1.8){$\alpha$}
\rput(8.3,2.8){$\beta_{1,1}$}
\rput(5.7,2.8){$\beta_{2,2}$}
\rput(6.6,3.2){$\beta_{1,2}$}
\rput(11.5,3.2){$\got{t}^*_+$}
\rput(9.3,9){$\Delta_K(G\cdot\Lambda)$}
\rput(10.3,6.3){$\beta_{1,1}$}
\rput(7.7,6.3){$\beta_{2,2}$}

\end{pspicture}
}
\end{center}
\caption{Polyèdre de la projection d'une orbite $Sp(4,\R)$-régulière holomorphe}
\end{figure}

\section{Caractérisation de $\Delta_K(\Orb_{\Lambda})$ par ses points rationnels}

Dans la cas d'une orbite compacte $G\cdot\Lambda$ avec $\Lambda\in\wedge_+^*$, on a vu que la projection de l'orbite $G\cdot\Lambda$ sur $\got{k}^*$ était complètement déterminée par l'étude des représentations irréductibles de $K$ qui apparaissent dans la décomposition des représentations irréductibles $V_{N\Lambda}^G$, où $N$ parcourt l'ensemble des entiers strictement positifs. Plus exactement, les éléments rationnels $\mu\in\wedge_{\Q,+}^*$ pour lesquels il existe $N>0$ tel que $V_{N\mu}^K \subset V_{N\Lambda}^{G}|_K$ sont exactement les points rationnels de la projection d'orbite $\Delta_K(G\cdot\Lambda)$. Mais surtout, $\Delta_K(G\cdot\Lambda)$ est égal à l'adhérence de ses points rationnels, ce qui permet de décrire $\Delta_K(G\cdot\Lambda)$ à partir des équations du polyèdre rationnel $\Delta_{\KC}^{\mathrm{alg}}(K\cdot\Lambda,\mathcal{L}_{\Lambda}) = \{\mu\in\wedge_{\Q,+}^*; \exists N>0, V_{N\mu}^K \subset H^0(G\cdot\Lambda, \mathcal{L}_{\Lambda}^{\otimes N})|_K\}$.

\bigskip

Si maintenant $G$ n'est pas compact, on ne peut plus appliquer les résultats sur les polytopes moments algébriques. Cependant, on conserve la propriété suivante: lorsque $\Lambda$ est dans $\got{t}^*$, le polyèdre de Kirwan $\Delta_K(\Orb_{\Lambda})$ est rationnel. Ceci provient de \cite{lerman98}. Il est donc égal à l'adhérence (ou, de manière équivalente, l'enveloppe convexe) de ses points rationnels. En particulier, on peut donc se ramener à l'étude de ses points rationnels, comme dans le cas compact.

Or, lorsque $G/K$ est un espace symétrique hermitien, on a un moyen de déterminer l'ensemble de ces points rationnels. Et ceux-ci se décrivent de manière proche du cas compact, en termes de représentations de $K$. Ce résultat utilise une version non compacte de la \og{}quantification commute avec la réduction\fg{}, en introduisant la quantification formelle.

\begin{theo}
\label{theo:DeltaK(OrbLambda)_par_représentationsirréductibles}
Soient $G/K$ un espace symétrique hermitien et $\Lambda\in\wedge^*_{\Q,+}\cap\Chol$. Alors, pour tout $\mu\in\wedge^*_{\Q,+}$, on a $\mu\in\Delta_K(\Orb_{\Lambda})$ si et seulement s'il existe un entier $N\geqslant 1$ tel que $(N\mu,N\Lambda)\in(\wedge^*)^2$ et $\left(V^*_{N\mu}\otimes V_{N\Lambda}\otimes \C[\got{p}^-]\right)^K \neq 0$.
\end{theo}


La preuve du Théorème \ref{theo:DeltaK(OrbLambda)_par_représentationsirréductibles} nécessite deux résultats qui seront prouvés respectivement aux quatrième et cinquième chapitres.

\bigskip

Le Théorème \ref{theo:DeltaK(OrbLambda)_par_représentationsirréductibles} donne une condition nécessaire et suffisante pour qu'un élément rationnel de $\got{t}^*$ soit dans le polyèdre moment $\Delta_K(\Orb_{\Lambda})$. Cette condition, en termes de produits tensoriels de représentations du groupe compact connexe $K$ fait intervenir la représentation $\C[\got{p}^-]$ de $K$. Or, lorsque $G/K$ est symétrique hermitien irréductible, le Théorème de Schmid donne la description exacte de la décomposition de $\C[\got{p}^-]$ en somme de représentations irréductibles de $K$.

Notons $(\gamma_1,\ldots,\gamma_r)$ la famille de racines non compactes positives obtenue par l'algorithme précédant l'énoncé du Théorème \ref{theo:schmid}. Définissons le $\Q$-cône convexe
\[
\mathscr{C}(\gamma_1,\ldots,\gamma_r) := \sum_{m_1\geqslant\ldots\geqslant m_r\geqslant0,\ m_i\in\Q}m_1\gamma_1+\cdots+m_r\gamma_r
\]
obtenu à partir de ces $r$ vecteurs de $\got{t}^*$.

\begin{prop}
\label{prop:DeltaK(OrbLambda)_par_représentationsirréductibles}
Soient $G/K$ un espace hermitien irréductible, $\Lambda\in\wedge^*_{\Q,+}\cap\Chol$. Alors, pour tout $\mu\in\wedge^*_{\Q,+}$, on a $\mu\in\Delta_K(\Orb_{\Lambda})$ si et seulement s'il existe un entier $N\geqslant 1$ et $\gamma\in\mathscr{C}(\gamma_1,\ldots,\gamma_r)$ tels que $(N\mu,N\Lambda,N\gamma)\in(\wedge^*)^3$ et $V_{N\gamma}\subset V_{N\mu}\otimes V_{N\Lambda^*}$. Autrement dit,
\begin{equation}
\label{eq:équivalence_DeltaK(OrbLambda)_et_intersection_deuxpolyèdres}
\mu\in\Delta_K(\Orb_{\Lambda}) \quad \Longleftrightarrow \quad \mathscr{C}(\gamma_1,\ldots,\gamma_r)\cap\Delta_K(K\cdot\mu\times K\cdot\Lambda^*) \neq \emptyset.
\end{equation}
\end{prop}

\begin{proof}
C'est une conséquence directe du Théorème \ref{theo:DeltaK(OrbLambda)_par_représentationsirréductibles}, du Théorème de Schmid et du fait que la condition $V_{N\mu}\subset V_{N\Lambda}\otimes V_{N\gamma}$ est équivalente à $V_{N\gamma}\subset V_{N\mu}\otimes V_{N\Lambda^*}$. La seconde équivalence est obtenue grâce à la Proposition \ref{prop:équivalences_polytopemomentproduitorbites}.
\end{proof}


\begin{rema}
De l'équivalence \eqref{eq:équivalence_DeltaK(OrbLambda)_et_intersection_deuxpolyèdres}, il est naturel d'essayer de résoudre le problème de la projection de l'orbite $\Orb_{\Lambda}$ en déterminant l'intersection de deux polyèdres convexes. Ces deux polyèdres convexes sont $\wedge_{\Q,+}^*\times\wedge_{\Q,+}^*\times\mathscr{C}(\gamma_1,\ldots,\gamma_r)$, qui est un cône convexe polyédral, et
\begin{equation}
\label{eq:polyèdre_convexe_dont_on_cherche_linteresction_avec_un_autre}
\{(\mu,\nu,\gamma)\in(\wedge_{\Q,+}^*)^3; \ \gamma\in\Delta_K(K\cdot\mu\times K\cdot\nu^*)\}.
\end{equation}
Le polyèdre \eqref{eq:polyèdre_convexe_dont_on_cherche_linteresction_avec_un_autre} peut également se décrire, grâce à la Proposition \ref{prop:équivalences_polytopemomentproduitorbites}, par
\[
\{(\mu,\nu,\gamma)\in(\wedge_{\Q,+}^*)^3; \exists N>0 \text{ t.q. } (V_{N\mu}^K\otimes V_{N\nu^*}^K\otimes V_{N\gamma^*}^K)^K\neq 0\}.
\]
Il s'agit bien d'un polyèdre, puisqu'il est, à peu de choses près, ce que l'on appelle le cône semi-ample de la variété projective $K/T\times K/T\times K/T$, cf Chapitre \ref{chap:projectiondorbitecoadjointe+GIT}.

Cette idée va être utilisée dans la section \ref{section:exemples_projections_orbites} pour calculer quelques exemples de projections d'orbites.

Déterminer l'intersection de deux polyèdres convexes est en général compliqué, mais ici cela reste envisageable lorsque le cône $\mathscr{C}(\gamma_1,\ldots,\gamma_r)$ a une géométrie simple (par exemple pour $Sp(2n,\R)$ et $SU(n,1)$) ou lorsque le rang de l'algèbre de Lie est petit (calcul pour $SO^*(6)$).
\end{rema}

\section{Exemples de calculs de projections d'orbites}
\label{section:exemples_projections_orbites}

La Proposition \ref{prop:DeltaK(OrbLambda)_par_représentationsirréductibles} donne une description de la projection d'une orbite holomorphe en termes de représentations irréductibles du sous-groupe compact maximal. En alliant nos connaissances des représentations irréductibles de $K$ apparaissant dans le $K$-module $\C[\got{p}^-]$, grâce au Théorème de Schmid, et les résultats du problème de Horn dus à Klyachko et Knutson-Tao, 
nous allons être en mesure de calculer certains polyèdres moments associés aux projections d'orbites holomorphes.

Nous ne présenterons ici que les calculs pour les groupes classiques $Sp(2n,\R)$ et $SU(n,1)$ pour $n\geqslant 2$, et $SO^*(6)$, qui font intervenir le problème de Horn classique, c'est-à-dire, lorsque le sous-groupe compact maximal $K$ est $U(n)$.

\subsection{Projections d'orbites coadjointes holomorphes de $Sp(2n,\R)$}
\label{subsection:polyèdremoment_Sp(R2n)}

Le Théorème de Schmid appliqué au groupe $Sp(2n,\R)$ nous permet de donner la décomposition de $\C[\got{p}^-]$ en somme directe de représentations irréductibles, cf \eqref{theo:Sp(R2n)_decompositiondeC[p-]}.
%

\begin{prop}
\label{prop:Sp(R2n)_repr_et_Polyèdre}
Soient $\Lambda\in\wedge_{\Q,+}^*\cap\Chol$ et $\mu\in\wedge_{\Q,+}^*$. Alors $\mu\in\Delta_K(\Orb_{\Lambda})$ si et seulement s'il existe un entier $k\geqslant 1$ et un poids dominant positif $\delta$ tels que $(k\mu,k\Lambda)\in(\wedge^*)^2$ et $V_{\delta} \subset V_{k\mu} \otimes V_{k\Lambda^*}$.
\end{prop}

\begin{proof}
Ceci découle de la Proposition \ref{prop:DeltaK(OrbLambda)_par_représentationsirréductibles}, appliquée à \eqref{theo:Sp(R2n)_decompositiondeC[p-]} qui donne l'énoncé du Théorème de Schmid dans le cas du groupe $Sp(2n,\R)$. Remarquons que l'on ne peut appliquer la Proposition \ref{prop:DeltaK(OrbLambda)_par_représentationsirréductibles} que pour les $\delta=(\delta_1\geqslant\ldots\geqslant\delta_n)$ où les composantes $\delta_i$ sont paires, d'après \eqref{theo:Sp(R2n)_decompositiondeC[p-]}. Mais si $\delta$ est dominant positif, alors $2\delta$ a ses composantes paires et vérifie $V_{2\delta} \subset V_{2k\mu} \otimes V_{2k\Lambda^*}$ pour un certain entier $k>0$. Autrement dit, la condition de parité des composantes de $\delta$ n'est pas nécessaire ici.
\end{proof}

%

Rappelons que, pour tout $\lambda=(\lambda_1\geqslant\ldots\geqslant\lambda_n)\in\got{t}_+^*$, la notation $\lambda^*$ désigne l'élément dual $\lambda^* = (-\lambda_n\geqslant\ldots\geqslant-\lambda_1)$ de $\lambda$ dans $\got{t}_+^*$.

\begin{theo}
\label{theo:Sp(R2n)_polyèdremomentSpn}
Soient $\Lambda\in\wedge_{\Q,+}^*\cap\Chol$ et $\mu\in\wedge^*_{\Q,+}$. Les assertions suivantes sont équivalentes:
\begin{enumerate}
\item[(\emph{i})] $\mu\in\Delta_K(\Orb_{\Lambda})$;
\item[(\emph{ii})] $\forall(I,J,L)\in T_r^n$, on a $\sum_{i\in I}\mu_i + \sum_{j\in J}\Lambda_j^* \geqslant 0$;
\item[(\emph{iii})] pour tout $i=1,\ldots,n$, on a $\mu_i\geqslant \Lambda_i$.
\end{enumerate}
En particulier, les polyèdres convexes $\Delta_K(\Orb_{\Lambda})$ et $(\Lambda+\sum_{\beta\in\got{R}_n^+}\R_+\beta)\cap\got{t}_+^*$ sont égaux.
\end{theo}

\begin{proof}
L'implication (\emph{i})$\Rightarrow$(\emph{ii}) est une conséquence des Théorèmes \ref{theo:horn_klyachko} et \ref{theo:klyachko_knutson-tao} (tirés du problème de Horn) et de la Proposition \ref{prop:Sp(R2n)_repr_et_Polyèdre}. La réciproque est donnée par la solution triviale $\delta = 0$. En effet, si $\mu$ vérifie (\emph{ii}), alors les Théorèmes \ref{theo:horn_klyachko} et \ref{theo:klyachko_knutson-tao} impliquent que le poids trivial $\delta=0$ satisfait $V_0 \subset V_{\mu}\otimes V_{\Lambda^*}$. Or le poids trivial est dominant positif, donc la Proposition \ref{prop:Sp(R2n)_repr_et_Polyèdre} donne que $\mu$ appartient à $\Delta_K(\Orb_{\Lambda})$. Ainsi, (\emph{i})$\Leftrightarrow$(\emph{ii}).

Prouvons (\emph{ii})$\Rightarrow$(\emph{iii}). Quitte à remplacer $\mu$ et $\Lambda$ par respectivement $k\mu$ et $k\Lambda$, pour un certain entier $k\geqslant 1$, on peut supposer $\mu,\Lambda\in\wedge^*$. En effet, on a toujours $(k\Lambda)^* = k \Lambda^*$, car $k\geqslant 0$. Il suffit alors de voir que, pour tout $i$ appartenant à $\{1,\ldots,n\}$, le triplet $(\{i\},\{n+1-i\},\{n\})$ est dans $U^n_1 = T^n_1$. C'est le cas puisque $i + (n+1-i) = n + 1$. Par conséquent, les coordonnées de $\mu$ et $\Lambda$ doivent vérifier les conditions suivantes :
\[
\mu_i + \Lambda^*_{n+1-i} \geqslant 0, \mbox{ pour tout } i=1,\ldots,n.
\]
Revenant à la définition de $\Lambda^*$ en terme de coordonnées, c'est-à-dire l'écriture $\Lambda^* = (-\Lambda_n,\ldots,-\Lambda_1)$, on obtient $\mu_i \geqslant \Lambda_i$ pour tout $i=1,\ldots,n$.

Réciproquement, si $\mu$ vérifie (\emph{iii}), alors le poids dominant $\delta = (\mu_{i_1}-\Lambda_{i_1}\geqslant\ldots,\mu_{i_n}-\Lambda_{i_n})$ est positif. De plus, les poids $\mu$, $\Lambda$ et $\delta$ sont les spectres respectifs des matrices hermitiennes suivantes:
\[
\begin{array}{c}
\begin{array}{ccc}
A = \left(\begin{array}{ccccc}
k\mu_1 & & 0 \\
& \ddots & \\
0 & & k\mu_n
\end{array}\right), & &

B = \left(\begin{array}{ccccc}
-k\Lambda_1 & & 0 \\
& \ddots & \\
0& & -k\Lambda_n
\end{array}\right),
\end{array} \\
\\
C = \left(\begin{array}{ccccc}
k(\mu_1-\Lambda_1) & & 0 \\
& \ddots & \\
0 & & k(\mu_n-\Lambda_n)
\end{array}\right).
\end{array}
\]
Or ces trois matrices vérifient $A+B=C$. Le Théorème \ref{theo:horn_klyachko} implique par conséquent que $\mu$ vérifie l'assertion (\emph{ii}).

Les assertions (\emph{i}), (\emph{ii}) et (\emph{iii}) sont donc équivalentes. De plus, étant donné que les racines non compactes positives de $\got{sp}(\R^{2n})$ sont les formes linéaires $\beta_{i,j}=e_i^*+e_j^*$, pour $1\leqslant i\leqslant j\leqslant n$, il est clair que le cône $\sum_{\beta\in\got{R}_n^+}\R_+\beta$ est tout simplement le cône $\sum_{i=1}^n\R_+e_i^*$. On en déduit alors les égalités $\Delta_K(\Orb_{\Lambda}) = (\Lambda+\sum_{i=1}^n\R_+e_i^*)\cap\got{t}_+^* = (\Lambda+\sum_{\beta\in\got{R}_n^+}\R_+\beta)\cap\got{t}_+^*$.
\end{proof}

\begin{figure}[ht]
\begin{center}
\setlength{\unitlength}{1cm}
\begin{pspicture}(10,10)
\definecolor{Couleur1}{rgb}{0,0.5,0.7}
\definecolor{Couleur2}{rgb}{0.6,0.4,0.2}

\pspolygon[linecolor=lightgray,fillstyle=solid,linewidth=0pt,fillcolor=lightgray](2,0)(2,10)(3.5,10)(4.5,9.9)(5,9.8)(6,9.6)(7,9.2)(8,8.6)(8.8,7.8)(9.3,7)(9.7,6.3)(10,5)(9.7,3.7)(9.3,3)(8.8,2.2)(8,1.4)(7,0.8)(6,0.4)(5,0.2)(4.5,0.1)(3.5,0)
\pspolygon[linecolor=gray,fillstyle=solid,linewidth=0pt,fillcolor=gray](2,1)(2,9.8)(3.5,9.8)(4.5,9.7)(5.5,9.4)(6.5,9)(7.6,8.4)(8.5,7.5)
\pspolygon[linecolor=lightgray,fillstyle=solid,linewidth=0pt,fillcolor=lightgray](3.5,5)(2,6.5)(2,9.7)(3.5,9.7)(4.5,9.6)(5.5,9.3)(6.3,8.95)(7,8.5)

\psdot(3.5,5)

\put(2, 0){\line(0, 1){10}}
\put(1, 0){\line(1, 1){7.8}}
\put(3.5,5){\line(-1, 1){1.5}}
\put(3.5,5){\line(1, 1){3.5}}

\put(2, 1){\vector(0, 1){1}}
\put(2, 1){\vector(1, 0){1}}
\put(2, 1){\vector(1, 1){1}}
\put(2, 1){\vector(-1, 1){1}}

\put(3.4,4.5){$\Lambda$}
\put(2.9,0.7){$\alpha$}
\put(2.8,1.5){$\beta_{1,1}$}
\put(2.1,1.9){$\beta_{1,2}$}
\put(0.8,1.5){$\beta_{2,2}$}
\put(8,4){$\got{t}^*_+$}
\put(6.2,6.8){$\Chol$}
\put(3,7.5){$\Delta_K(\Orb_{\Lambda})$}

\end{pspicture}
\end{center}
\caption{Polyèdre de la projection d'une orbite coadjointe holomorphe de $Sp(\R^2)$}
\end{figure}
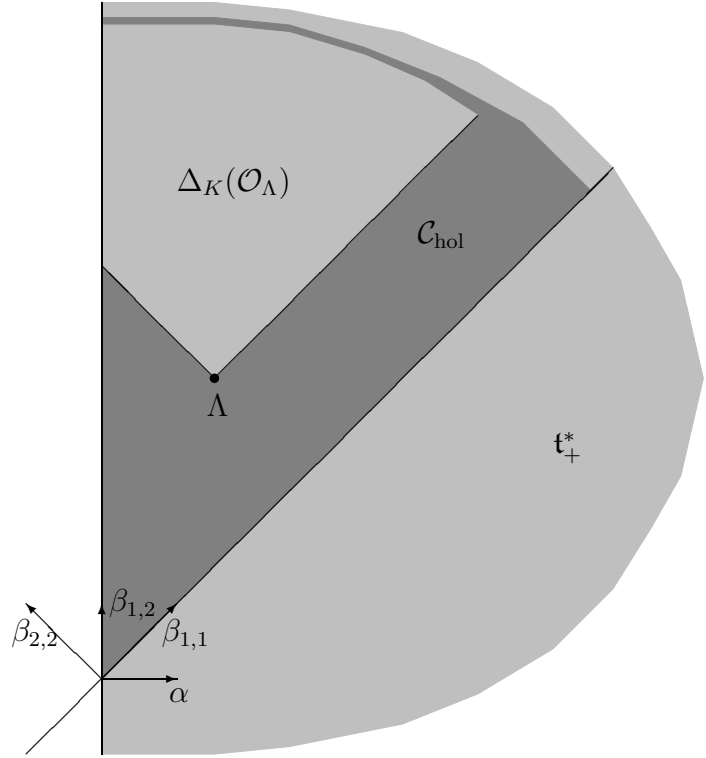

\subsection{Projections d'orbites coadjointes holomorphes de $SU(n,1)$}
\label{subsection:polyèdremoment_SU(n,1)}

Rappelons que les notations sur les différents éléments de $\got{su}(n,1)$ ont été introduites dans le paragraphe \ref{subsection:Notations_SU(p,q)}.

Grâce au Théorème de Schmid, on connait la décomposition de $\C[\got{p}^-]$ en somme directe de représentations irréductibles de $U(n)$ : ce sont les représentations irréductibles $V_{\delta}$, avec $\delta$ un poids dominant de la forme $\delta = m \beta_1$, où $m\in\N$, cf \eqref{theo:SU(n,1)_décompositiondeC[p-]}. En appliquant le Théorème \ref{theo:horn_klyachko}, on obtient aisément l'énoncé suivant.

\begin{prop}
\label{prop:SU(n,1)_CNS_polyèdre}
Soient $\Lambda\in\wedge_{\Q,+}^*\cap\Chol$ et $\mu\in\wedge^*_{\Q,+}$. On a $\mu\in\Delta_K(\Orb_{\Lambda})$ si et seulement s'il existe $R\in\Q_+$ tel que pour tout $r\in\{1,\ldots,n-1\}$, pour tout $(I,J,K)\in T^n_r$,
\begin{equation}
\label{eq:SU(n,1)_CNS_Polyèdre}
\sum_{i\in I} \mu_i + \sum_{j\in J} \Lambda_j^* \geqslant R\sum_{\ell\in L}(\beta_1)_{\ell},
\end{equation}
et
\[
\sum_{i = 1}^n \mu_i + \sum_{j=1}^n \Lambda_j^* = (n+1)R.
\]
\end{prop}

Nous allons maintenant présenter les équations qui déterminent le polyèdre de Kirwan $\Delta_K(\Orb_{\Lambda})$, pour $\Lambda\in\wedge_{\Q,+}^*\cap\Chol$. Ces équations font intervenir uniquement les directions des faces du cône engendré par les racines non compactes positives de $\got{su}(n,1)$, c'est-à-dire, par les $n$ vecteurs $\beta_1,\ldots,\beta_n$. Ces $n$ vecteurs formant une base de $\got{t}^*$, chaque $(n-1)$-uplet que l'on pourra obtenir à partir de cette famille de vecteurs, va nous donner un hyperplan contenant la face engendrée par ces $n-1$ racines. On peut par exemple chercher, pour chaque $k\in\{1,\ldots,n\}$, un vecteur $\chi_k\in\got{t}$ tel que $\langle\chi_k,\beta_i\rangle = 0 $ si $i\neq k$, et $\langle\chi_k,\beta_k\rangle> 0$. On aura alors :
\[
Cone(\beta_1,\ldots,\beta_n) = \{\xi\in\got{t}^*; \ \langle\xi,\chi_k\rangle\geqslant 0, \forall\ 1\leqslant k\leqslant n\}. 
\]
On prendra pour cela
\[
\chi_k = (n+1)e_k - \sum_{i=1}^n e_i.
\]
Par le calcul, on voit que l'on a bien $\langle\beta_i,\chi_k\rangle = 0$ si $i\neq k$, et $\langle\beta_k,\chi_k\rangle = n+1>0$.

Le but de cette section est de prouver le théorème suivant.

\begin{theo}
\label{theo:SU(n,1)_PolyèdreMoment}
Soit $\Lambda$ dans $\wedge_{\Q,+}^*\cap\Chol$. Alors, pour $G = SU(n,1)$, on  a
\[
\Delta_K(\Orb_{\Lambda}) = \{\xi\in\got{t}^*; \langle\xi,\chi_1\rangle\geqslant\langle\Lambda,\chi_1\rangle\geqslant\langle\xi,\chi_2\rangle\geqslant\langle\Lambda,\chi_2\rangle\geqslant\ldots\geqslant\langle\xi,\chi_n\rangle\geqslant\langle\Lambda,\chi_n\rangle\}.
\]
\end{theo}

Pour simplifier les notations dans le preuve de ce théorème, notons $\Pn = \{\xi\in\got{t}^*; \langle\xi,\chi_1\rangle\geqslant\langle\Lambda,\chi_1\rangle\geqslant\langle\xi,\chi_2\rangle\geqslant\langle\Lambda,\chi_2\rangle\geqslant\ldots\geqslant\langle\xi,\chi_n\rangle\geqslant\langle\Lambda,\chi_n\rangle\}$. Dans la suite de ce paragraphe, nous désignerons les équations apparaissant ci-dessus de la manière suivante:
\[
\begin{array}{cl}
(A_k): & \langle\xi,\chi_k\rangle\geqslant\langle\Lambda,\chi_k\rangle, \ \mbox{ pour } k=1,\ldots,n; \\
&\\
(B_k): & \langle\Lambda,\chi_k\rangle\geqslant\langle\xi,\chi_{k+1}\rangle, \ \mbox{ pour } k=1,\ldots,n-1.
\end{array}
\]

\begin{rema}
Pour $k=1,\ldots,n-1$, en sommant les équations $(A_k)$ et $(B_k)$, on obtient:
\[
(A_k) + (B_k) : \ \xi_k - \xi_{k+1} \geqslant 0.
\]
Par conséquent, un élément de $\Pn$ sera dans $\got{t}^*_+$. Cela est cohérent avec le résultat que l'on veut obtenir, c'est-à-dire $\Pn=\Delta_K(\Orb_{\Lambda})$ est un polyèdre convexe contenu dans $\got{t}^*_+$.
\end{rema}

Le polyèdre $\Pn$ est rationnel, donc pour montrer que $\Delta_K(\Orb_{\Lambda})$ est égal à $\Pn$, nous allons prouver l'égalité $\Delta_K(\Orb_{\Lambda})\cap\wedge^*_{\Q} = \Pn\cap\wedge^*_{\Q}$.

\paragraph*{Première étape : montrons que $\Delta_K(\Orb_{\Lambda})\cap\wedge^*_{\Q} \subset \Pn\cap\wedge^*_{\Q}$}

Les équations de type $(A_k)$ sont les équations qui définissent le cône engendré par les racines non compactes positives.

\begin{prop}
\label{prop:SU(n,1)_Equations_type1}
Soit $\mu$ un élément de $\Delta_K(\Orb_{\Lambda})\cap \wedge^*_{\Q,+}$. Alors, pour tout entier $k\in\{1,\ldots,n\}$, on a $\langle\mu - \Lambda,\chi_k\rangle \geqslant 0$.
\end{prop}

\begin{proof}
Soit $\mu$ un élément de $\Delta_K(\Orb_{\Lambda})\cap \wedge^*_{\Q,+}$. Pour $k\in\{1,\ldots,n\}$ fixé, on a $\langle\mu-\Lambda,\chi_k\rangle \geqslant 0$ si et seulement si
\[
(n+1)\mu_k - \sum_{i=1}^n \mu_i \geqslant (n+1)\Lambda_k -\sum_{i=1}^n\Lambda_i.
\]
D'après la Proposition \ref{prop:SU(n,1)_CNS_polyèdre}, comme $\mu$ appartient à $\Delta_K(\Orb_{\Lambda})$, $\mu$ vérifie la \og condition de trace\fg{} $\sum_{i=1}^n \mu_i = \sum_{j=1}^n \Lambda_j + R(n+1)$, ainsi que les $n$ conditions $\mu_k \geqslant \Lambda_k + R$, pour $k=1,\ldots,n$ et $R$ un certain rationnel positif. Ces dernières conditions proviennent du fait que le triplet $(\{k\},\{n-k+1\},\{n\})$ appartient à $U^n_1 = T^n_1$, avec $\Lambda^*_{n-k+1} = -\Lambda_k$, et que la $n$-ième coordonnée $(\beta_1)_n$ de $\beta_1$ vaut $1$. Pour $k$ fixé dans $\{1,\ldots,n\}$, on associe ces deux équations pour obtenir
\[
\begin{array}{ccl}
(n+1)\mu_k - \sum_{i=1}^n\mu_i & \geqslant & -\sum_{j=1}^n \Lambda_j - R(n+1) + (n+1)(\Lambda_k + R) \\
& \geqslant & (n+1)\Lambda_k - \sum_{j=1}^n \Lambda_j.
\end{array}
\]
Et l'équivalence précédente nous donne $\langle\mu - \Lambda,\chi_k\rangle \geqslant 0$.
\end{proof}

Les équations de type $(B_k)$ font à nouveau intervenir les équations des hyperplans vectoriels orthogonaux aux $\chi_k$, mais cette fois pour deux entiers $k$ successifs.

\begin{prop}
\label{prop:SU(n,1)_Equations_type2}
Soit $\mu$ un élément de $\Delta_K(\Orb_{\Lambda})\cap \wedge^*_{\Q,+}$. Alors, pour tout entier $k\in\{1,\ldots,n-1\}$, on a
\[
\langle\mu,\chi_{k+1}\rangle \leqslant \langle\Lambda,\chi_k\rangle.
\]
\end{prop}

\begin{lemm}
\label{lemm:SU(n,1)_IJL_k_n-k+1_n}
Soit $k\in\{2,\ldots,n\}$. Notons les trois ensembles $I =\{1,\ldots,n\}\backslash\{k\}$, $J=\{1,\ldots,n\}\backslash\{n-k+2\}$, et $L = \{1,\ldots,n\}\backslash\{2\}$. On a $(I,J,L) \in T^n_{n-1}$.
\end{lemm}

\begin{proof}[Preuve du lemme]
C'est un cas particulier (pour $r=n-1$) de la Proposition \ref{prop:SU(n,1)_IJL_1dansL}.
\end{proof}

\begin{proof}[Preuve de la Proposition \ref{prop:SU(n,1)_Equations_type2}]
Comme dans la preuve de la Proposition \ref{prop:SU(n,1)_Equations_type1}, on a équivalence entre $\langle\mu,\chi_{k+1}\rangle \leqslant \langle\Lambda,\chi_k\rangle$ et l'équation $(\sum_{i=1}^n \mu_i) - (n+1)\mu_{k+1} \geqslant (\sum_{i=1}^n\Lambda_i) - (n+1)\Lambda_k$. Supposons que $\mu$ est un élément de $\Delta_K(\Orb_{\Lambda})\cap \wedge^*_{\Q,+}$. D'après la Proposition \ref{prop:SU(n,1)_CNS_polyèdre} et le Lemme \ref{lemm:SU(n,1)_IJL_k_n-k+1_n}, il existe $R\in\Q_+$ tel que $\sum_{i=1}^n\mu_i + \sum_{j=1}^n\Lambda^*_j = R(n+1)$ et, pour tout $k\in\{1,\ldots,n-1\}$,  $\sum_{i\neq k+1}\mu_i + \sum_{j\neq n-k+1}\Lambda^*_j \geqslant R\sum_{\ell\neq 2}(\beta_1)_{\ell}$. Pour $1\leqslant k\leqslant n-1$ fixé, on a
\[
-n\mu_{k+1} = n\sum_{i\neq k+1}\mu_i - n(n+1)R - n\sum_{j=1}^n \Lambda_j.
\]
Ceci nous donne :
\[
\begin{array}{ccl}
(\sum_{i=1}^n\mu_i) - (n+1)\mu_{k+1} & = & (\sum_{i\neq k+1}\mu_i) - n\mu_{k+1} \\
& = & (n+1)(\sum_{i\neq k+1}\mu_i) - n(n+1)R - n\sum_{j=1}^n \Lambda_j \\
& \geqslant & (n+1)(nR + \sum_{j\neq k}\Lambda_j) - n(n+1)R - n\sum_{j=1}^n \Lambda_j \\
& \geqslant & \sum_{j=1}^n \Lambda_j - (n+1)\Lambda_k.
\end{array}
\]
Dans les inégalités ci-dessus, on a utilisé le fait que $\sum_{\ell\neq 2}(\beta_1)_{\ell} = n$ et que $\Lambda^*_{n-k+1} = \Lambda_k$. On en déduit que $\langle\mu,\chi_{k+1}\rangle \leqslant \langle\Lambda,\chi_k\rangle$.
\end{proof}

Les Propositions \ref{prop:SU(n,1)_Equations_type1} et \ref{prop:SU(n,1)_Equations_type2} prouvent bien que l'on a l'inclusion $\Delta_K(\Orb_{\Lambda})\cap\wedge^*_{\Q} \subset \Pn\cap\wedge^*_{\Q}$.

\paragraph*{Seconde partie : montrons que $\Pn\cap\wedge^*_{\Q}\subset\Delta_K(\Orb_{\Lambda})\cap\wedge^*_{\Q}$}

Fixons, pour la suite de cette section, un élément $\mu$ de $\Pn\cap\wedge^*_{\Q}$. Un tel élément vérifiera donc les équations $(A_k)$ et $(B_k)$. En particulier, si on somme toutes les équations $(A_k)$, on obtient:
\[
\sum_{k=1}^n(A_k) : \ \sum_{k=1}^n \mu_k \geqslant \sum_{k=1}^n\Lambda_k.
\]
Puisque $\mu$ et $\Lambda$ appartiennent à $\wedge^*_{\Q}$, il existe un rationnel $R$ positif tel que
\[
\begin{array}{cc}
(T) : & \sum_{i=1}^n\mu_i = \sum_{j=1}^n\Lambda_j + (n+1)R,
\end{array}
\]
et ceci est équivalent à $\sum_{i=1}^n\mu_i + \sum_{j=1}^n\Lambda_j^* = (n+1)R$.

Le but est d'appliquer la Proposition \ref{prop:SU(n,1)_CNS_polyèdre} avec ce nombre rationnel $R$ obtenu. Nous devons donc montrer que pour tout $r=1,\ldots,n-1$ et tout triplet $(I,J,L)\in T^n_r$, on a $\sum_{i\in I}\mu_i + \sum_{j\in J}\Lambda^*_j \geqslant R\sum_{\ell\in L}(\beta_1)_{\ell}$.

\bigskip

Deux types de triplets de $T_r^n$ ont une importance particulière dans ce contexte.

Pour une partie $I \subset\{1,\ldots,n\}$ avec $|I| = r\leqslant n-1$, nous noterons $\overline{I} = \{n-i+1; \ i\in I\}$ et $L^n_r = \{n-r+1< n-r+2 < \ldots < n-1 < n\}$. Les trois sous-ensembles $I$, $\overline{I}$ et $L^n_r$ de $\{1,\ldots,n\}$ sont de même cardinal $r$.

Si, de plus, $1\in I$, nous noterons $I^* = \left\{i-1; \ i\in I\backslash\{1\}\right\} \cup \{n\}$ et $\tilde{L}^n_r = \{1<n-r+2<n-r+3<\ldots<n-1<n\}$. On indexe les éléments de $I$ par ordre croissant, $I=\{1<i_2<\ldots<i_r\}$, ce qui nous donne pour les autres ensembles, $I^* = \{i_2-1<\ldots<i_r-1<n\}$, et $\overline{I^*} = \{1<n-i_r+2<\ldots<n-i_2+2\}$.

\begin{lemm}
\label{lemm:SU(n,1)_IJL_types_spéciaux}
Soit $I \subset\{1,\ldots,n\}$ avec $|I| = r\leqslant n-1$. Alors, on a
\begin{equation}
\label{eq:SU(n,1)_IJL_type1}
\sum_{i\in I} \mu_i + \sum_{i\in \overline{I}}\Lambda_i^* \geqslant R\sum_{\ell\in L^n_r}(\beta_1)_{\ell} \ .
\end{equation}
Si, de plus, $1\in I$, alors
\begin{equation}
\label{eq:SU(n,1)_IJL_type2}
\sum_{i\in I} \mu_i + \sum_{i\in \overline{I^*}}\Lambda^*_i \geqslant R\sum_{\ell\in \tilde{L}^n_r}(\beta_1)_{\ell}\ .
\end{equation}
\end{lemm}


\begin{proof}[Démonstration du Lemme \ref{lemm:SU(n,1)_IJL_types_spéciaux}]
Soit $I \subset\{1,\ldots,n\}$ avec $|I| = r\leqslant n-1$. Puisque $\mu$ est dans $\Pn$, il vérifie $(T)$, et $(A_k)$ pour tout $k\in I$. Donc, il vérifie aussi l'équation $r(T) + \sum_{i\in I} (A_i)$, c'est-à-dire
\[
\sum_{j=1}^n r\mu_j + \sum_{i\in I}\left((n+1)\mu_i - \sum_{j=1}^n\mu_j\right)\geqslant \sum_{j=1}^n r\Lambda_j + \sum_{i\in I}\left((n+1)\Lambda_i - \sum_{j=1}^n\Lambda_j\right)+ r(n+1)R,
\]
ce qui est équivalent à $\sum_{i\in I} \mu_i \geqslant \sum_{i\in I}\Lambda_i + rR$, qui est également équivalente à \eqref{eq:SU(n,1)_IJL_type1} en changeant de côté les termes en $\Lambda$. 

Supposons maintenant que $1\in I$. Nous allons voir que $\mu$ vérifie l'équation \eqref{eq:SU(n,1)_IJL_type2}. Puisque $1\in I$, on a $1\notin\{1,\ldots,n\}\backslash I$, et l'équation $(r+1)(T) + \sum_{k\in\{1,\ldots,n\}\backslash I}(B_{k-1})$ donne:
\begin{multline*}
\displaystyle\sum_{j=1}^n (r+1)\mu_j + \sum_{k\in\{1,\ldots,n\}\backslash I}\left(\sum_{j=1}^n\mu_j - (n+1)\mu_k\right) \\
\geqslant \sum_{j=1}^n (r+1)\Lambda_j +  \displaystyle\sum_{k\in\{1,\ldots,n\}\backslash I}\left(\sum_{j=1}^n\Lambda_j - (n+1)\Lambda_{k-1}\right)+ (r+1)(n+1)R.
\end{multline*}
Cette inégalité est équivalente à la suivante:
\[
(n+1)\sum_{i\in I}\mu_i \geqslant (n+1)\left(\sum_{\stackrel{i\in\{1,\ldots,n-1\}}{i+1\in I}}\Lambda_i + \Lambda_n + (r+1)R\right),
\]
et en divisant par $(n+1)$, on obtient
\[
\sum_{i\in I}\mu_i \geqslant \sum_{i\in I^*}\Lambda_i + (r+1)R,
\]
ou, ce qui est équivalent,
\[
\sum_{i\in I} \mu_i + \sum_{i\in \overline{I^*}}\Lambda_i^* \geqslant R\sum_{\ell\in \tilde{L}^n_r}(\beta_1)_{\ell}\ ,
\]
car $1\in \tilde{L}^n_r$ et donc $\sum_{\ell\in \tilde{L}^n_r}(\beta_1)_{\ell} = r+1$. Donc l'équation \eqref{eq:SU(n,1)_IJL_type2} est bien vérifiée par $\mu$.
\end{proof}

Les Propositions \ref{prop:SU(n,1)_I_Ibar_Lnr} et \ref{prop:SU(n,1)_IJL_1dansL} montrent que les triplets $(I,\overline{I},L^n_r)$ et $(I,\overline{I^*},\tilde{L}^n_r)$ appartiennent à $T_r^n$. Par conséquent, les équations \eqref{eq:SU(n,1)_IJL_type1} et \eqref{eq:SU(n,1)_IJL_type2} nous donnent une partie des équations que nous recherchons. Nous allons tout de suite voir que l'on peut s'y ramener à chaque fois.

Partons donc de $(I,J,L)$ quelconque dans $T^n_r$, pour un $r\in\{1,\ldots,n-1\}$. Nous voulons obtenir \eqref{eq:SU(n,1)_CNS_Polyèdre}. Deux cas se présentent à nous.

Tout d'abord, lorsque $1\notin L$, on a l'inégalité $j_k \leqslant n-i_{r-k+1}+1$ pour tout $1\leqslant k \leqslant r$. En effet, $(I,J,L)\in T^n_r$ et $(\{r-k+1\},\{k\},\{r\})\in T^r_1 = U^r_1$. Donc $i_{r-k+1} + j_k \leqslant l_r+1 \leqslant n+1$. D'où l'inégalité $j_k \leqslant n-i_{r-k+1}+1 =: \overline{i}_k$. Par conséquent, on a $\Lambda^*_{j_k} \geqslant \Lambda^*_{\overline{i}_k}$. Nous obtenons l'inégalité attendue :
\[
\begin{array}{rcl}
\displaystyle\sum_{i\in I}\mu_i + \sum_{j\in J}\Lambda^*_j = \sum_{i\in I}\mu_i + \sum_{k=1}^r\Lambda^*_{j_k} & \geqslant & \displaystyle\sum_{i\in I}\mu_i + \sum_{k=1}^r\Lambda^*_{\overline{i}_k} = \sum_{i\in I}\mu_i + \sum_{\overline{i}\in \overline{I}}\Lambda^*_{\overline{i}} \\
& \geqslant & R\displaystyle\sum_{\ell\in L^n_r}(\beta_1)_{\ell} = rR = R\displaystyle\sum_{\ell\in L}(\beta_1)_{\ell}.
\end{array}
\]
Nous avons appliqué l'inégalité \eqref{eq:SU(n,1)_IJL_type1} démontrée ci-dessus.

Enfin, si $1\in L$, prenons $2\leqslant k\leqslant r$. Nous voulons montrer que $j_k \leqslant (\overline{i^*})_k = n-i_{r-k+2}+2$. Pour ce faire, on doit d'abord remarquer que $(\{1,r-k+2\},\{1,k\},\{1,r\})\in T^r_2$. Donc $i_1+i_{r-k+2} +j_1+j_k\leqslant \ell_1+\ell_r + 3$. Or  $\ell_1 = 1$ et $(\{1\},\{1\},\{1\})\in T^r_1$, ce qui implique que $i_1+j_1\leqslant\ell_1+1 = 2$ et, comme $i_1,j_1\in\{1,\ldots,n\}$, on a nécessairement $i_1=j_1=\ell_1 = 1$. Puisque $(I,J,L)\in T_r^n$, on obtient bien $j_k\leqslant n-i_{r-k+2}+2$. Et pour $k=1$, on a $j_1 = 1 \leqslant (\overline{i^*})_1 = 1$. D'où,
\[
\begin{array}{rcl}
\displaystyle\sum_{i\in I}\mu_i + \sum_{j\in J}\Lambda^*_j = \sum_{i\in I}\mu_i + \sum_{k=1}^r\Lambda^*_{j_k} & \geqslant & \displaystyle\sum_{i\in I}\mu_i + \sum_{k=1}^r\Lambda^*_{(\overline{i^*})_k} = \sum_{i\in I}\mu_i + \sum_{j\in \overline{I^*}}\Lambda^*_{j} \\
& \geqslant & R\displaystyle\sum_{\ell\in \tilde{L}^n_r}(\beta_1)_{\ell} = (r+1)R = R\displaystyle\sum_{\ell\in L}(\beta_1)_{\ell},
\end{array}
\]
l'équation que nous voulions. Nous avons utilisé ici le fait que $1\in I$, car $1\in L$, et avons appliqué l'inégalité \eqref{eq:SU(n,1)_IJL_type2} obtenue précédemment pour $(I,\overline{I^*},\tilde{L}^n_r)$.

Pour récapituler, nous avons montré que si $\mu$ appartient à $\Pn\cap\wedge^*_{\Q}$, alors il existe un rationnel positif $R$ tel que $\mu$ vérifie $\sum_{i=1}^n \mu_i + \sum_{j=1}^n\Lambda^*_j = (n+1)R$ et, pour tout $r\in\{1,\ldots,n-1\}$ et tout $(I,J,L)\in T^n_r$,
\[
\sum_{i\in I} \mu_i + \sum_{j\in J} \Lambda^*_j \geqslant R\sum_{\ell\in L}(\beta_1)_{\ell}.
\]
La Proposition \ref{prop:SU(n,1)_CNS_polyèdre} permet maintenant de conclure que $\Pn\cap\wedge^*_{\Q}\subset\Delta_K(\Orb_{\Lambda})\cap\wedge^*_{\Q}$. Cela prouve le Théorème \ref{theo:SU(n,1)_PolyèdreMoment}.

\begin{figure}[ht]
\begin{center}
\setlength{\unitlength}{1cm}
\begin{pspicture}(10,10)
\definecolor{Couleur1}{rgb}{0,0.5,0.7}
\definecolor{Couleur2}{rgb}{0.6,0.4,0.2}

\pspolygon[linecolor=lightgray,fillstyle=solid,linewidth=0pt,fillcolor=lightgray](2,0)(2,10)(3.5,10)(4.5,9.9)(5,9.8)(6,9.6)(7,9.2)(8,8.6)(8.8,7.8)(9.3,7)(9.7,6.3)(10,5)(9.7,3.7)(9.3,3)(8.8,2.2)(8,1.4)(7,0.8)(6,0.4)(5,0.2)(4.5,0.1)(3.5,0)
\pspolygon[linecolor=gray,fillstyle=solid,linewidth=0pt,fillcolor=gray](2,1)(2,9.8)(3.5,9.8)(4.5,9.7)(5.5,9.4)(6.5,9)(7.6,8.4)(8.5,7.5)(9.1,6.4)(9.4,5.272)
\pspolygon[linecolor=lightgray,fillstyle=solid,linewidth=0pt,fillcolor=lightgray](4,3.5)(2,6.964)(3.6,9.735)(4.4,9.65)(5,9.5)(5.8,9.25)(6.3,9.05)(6.7,8.85)(7,8.696)


\psdot(4,3.5)

\psline[linewidth=1pt](2,0)(2,10)
\psline[linewidth=1pt](0.268,0)(9.9,5.561)
\psline[linewidth=0.5pt](7,8.69)(4,3.5)(2,6.96)(3.6,9.73)

\psline[linewidth=1pt]{->}(2,1)(3,1)
\psline[linewidth=1pt]{->}(2,1)(2.5,1.87)
\psline[linewidth=1pt]{->}(2,1)(1.5,1.87)

\rput(4,3.2){$\Lambda$}
\rput(2.9,0.7){$\alpha$}
\rput(2.8,2){$\beta_{1}$}
\rput(1.2,2){$\beta_{2}$}
\rput(7,2){$\got{t}^*_+$}
\rput(8,5.5){$\Chol$}
\rput(4,7){$\Delta_K(\Orb_{\Lambda})$}

\end{pspicture}
\end{center}
\caption{Polyèdre de la projection d'une orbite coadjointe holomorphe de $SU(2,1)$}
\end{figure}
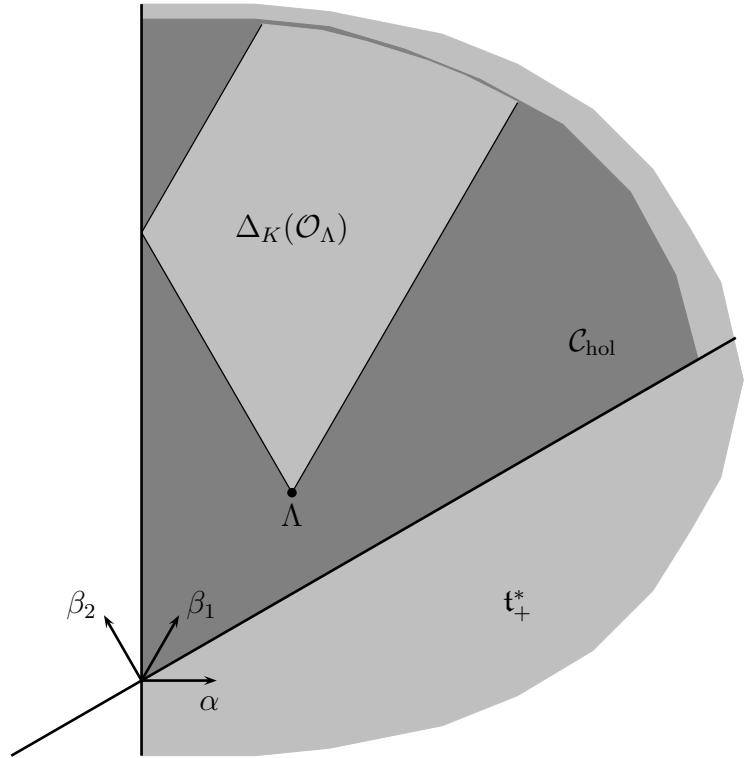

\subsection{Projections d'orbites coadjointes holomorphes de $SO^*(6)$}
\label{subsection:polyèdremoment_SO*(6)}

De manière analogue aux groupes $Sp(2n,\R)$ et $SU(n,1)$, on applique le Théorème de Schmid (Théorème \ref{theo:schmid}), spécialisé à $SO^*(6)$ dans l'assertion \eqref{theo:SO*(4p+2)_décompositiondeC[p-]}, pour obtenir une condition nécessaire et suffisante pour qu'un élément de $\wedge_{\Q,+}^*$ soit dans le polyèdre $\Delta_K(\Orb_{\Lambda})$.

%

\begin{prop}
\label{prop:SO*(2n)_CNS_polyèdre}
Soit $\Lambda\in\wedge_{\Q,+}^*\cap\Chol$ et $\mu\in\wedge_{\Q,+}^*$. Alors $\mu\in\Delta_K(\Orb_{\Lambda})$ si et seulement s'il existe $R\in\Q_+$ tel que
\begin{align}
\mu_1+\mu_2+\mu_3& =\Lambda_1+\Lambda_2+\Lambda_3+2R, \label{eq:SO*(2n)_équationsparHorn_égalité}\tag{E}\\
\mu_2+\mu_3&\geqslant\Lambda_2+\Lambda_3+R, \label{eq:SO*(2n)_équationsparHorn_eq1}\tag{A1}\\
\mu_1+\mu_3&\geqslant\Lambda_1+\Lambda_3+R, \label{eq:SO*(2n)_équationsparHorn_eq2}\tag{A2}\\
\mu_1+\mu_2&\geqslant\Lambda_1+\Lambda_2+R, \label{eq:SO*(2n)_équationsparHorn_eq3}\tag{A3}\\
\mu_1+\mu_2&\geqslant\Lambda_2+\Lambda_3+2R, \label{eq:SO*(2n)_équationsparHorn_eq4}\tag{A4}\\
\mu_1&\geqslant\Lambda_2+R, \label{eq:SO*(2n)_équationsparHorn_eq5}\tag{A5}\\
\mu_2&\geqslant\Lambda_3+R, \label{eq:SO*(2n)_équationsparHorn_eq6}\tag{A6}\\
\mu_3&\geqslant\Lambda_3, \label{eq:SO*(2n)_équationsparHorn_eq7}\tag{A7}\\
\mu_2&\geqslant\Lambda_2, \label{eq:SO*(2n)_équationsparHorn_eq8}\tag{A8}\\
\mu_1&\geqslant\Lambda_1. \label{eq:SO*(2n)_équationsparHorn_eq9}\tag{A9}
\end{align}
\end{prop}

\begin{proof}
Ce résultat découle directement du Théorème \ref{theo:horn_klyachko} et du Théorème de Schmid et de l'Exemple \ref{exple:équationsgénérales_Horn_pourU(3)}, donnant les équations pour le cas de $U(3)$.
\end{proof}

Il faut noter que certaines équations de l'Exemple \ref{exple:équationsgénérales_Horn_pourU(3)} n'apparaissent plus dans l'énoncé de la Proposition \ref{prop:SO*(2n)_CNS_polyèdre}, car elles découlent d'autres équations données ci-dessus. En effet, il ne faut pas oublier que les coordonnées de $\mu$ et $\Lambda$ sont rangées par ordre décroissant, ce qui ajoute des conditions au polyèdre final. Par ailleurs, on connaît la forme du poids $\nu$ qui doit apparaître dans le membre de droite des équations de l'Exemple \ref{exple:équationsgénérales_Horn_pourU(3)}, et ici il vérifie $\nu_1=\nu_2=1$ et $\nu_3=0$ (il s'agit de $\beta_{1,2}$). Pour donner un exemple, l'équation $\mu_1+\Lambda^*_1\geqslant R$ découle de l'équation $\mu_1+\Lambda^*_2\geqslant R$, puisque $\Lambda^*_1\geqslant\Lambda^*_2$.

Le Théorème suivant conclut l'étude du polyèdre $\Delta_K(\Orb_{\Lambda})$ pour le groupe $SO^*(6)$.

\begin{theo}
\label{theo:SO*(2n)_équationsdupolyèdremoment}
Lorsque $G=SO^*(6)$ et $\Lambda\in\wedge_{\Q,+}^*\cap\Chol$, le polyèdre de Kirwan $\Delta_K(\Orb_{\Lambda})$ est défini par les équations affines suivantes
\begin{align}
-\xi_1+\xi_2+\xi_3&\geqslant-\Lambda_1+\Lambda_2+\Lambda_3, \label{eq:SO*(2n)_équationsdupolyèdre_eq1}\tag{B1}\\
\xi_1-\xi_2+\xi_3&\geqslant\Lambda_1-\Lambda_2+\Lambda_3, \label{eq:SO*(2n)_équationsdupolyèdre_eq2}\tag{B2}\\
\xi_1+\xi_2-\xi_3&\geqslant\Lambda_1+\Lambda_2-\Lambda_3, \label{eq:SO*(2n)_équationsdupolyèdre_eq3}\tag{B3}\\
\xi_1-\xi_2-\xi_3&\geqslant-\Lambda_1+\Lambda_2-\Lambda_3, \label{eq:SO*(2n)_équationsdupolyèdre_eq4}\tag{B4}\\
-\xi_1+\xi_2-\xi_3&\geqslant-\Lambda_1-\Lambda_2+\Lambda_3, \label{eq:SO*(2n)_équationsdupolyèdre_eq5}\tag{B5}\\
\xi_1\geqslant\xi_2\geqslant\xi_3. \notag
\end{align}
\end{theo}

\begin{proof}
Puisque les polyèdres intervenant ici sont rationnels, il suffit de montrer que leurs ensembles de points rationnels coïncident. Notons $\mathcal{P}_{\Lambda}$ le polyèdre de $\wedge^*_{\Q,+}$ défini par les équations \eqref{eq:SO*(2n)_équationsdupolyèdre_eq1} à \eqref{eq:SO*(2n)_équationsdupolyèdre_eq5}.

Commençons donc par supposer que $\mu$ appartient à $\mathcal{P}_{\Lambda}$. En sommant les deux inégalités \eqref{eq:SO*(2n)_équationsdupolyèdre_eq1} et \eqref{eq:SO*(2n)_équationsdupolyèdre_eq2}, on obtient $2\mu_3\geqslant 2\Lambda_3$ et donc $\mu_3\geqslant\Lambda_3$, ce qui est exactement l'inégalité \eqref{eq:SO*(2n)_équationsparHorn_eq7}. De même, sommer les inégalités \eqref{eq:SO*(2n)_équationsdupolyèdre_eq1} et \eqref{eq:SO*(2n)_équationsdupolyèdre_eq3} (resp. \eqref{eq:SO*(2n)_équationsdupolyèdre_eq2} et \eqref{eq:SO*(2n)_équationsdupolyèdre_eq3}) donne l'inégalité \eqref{eq:SO*(2n)_équationsparHorn_eq8} (resp. \eqref{eq:SO*(2n)_équationsparHorn_eq9}). On en déduit que $\mu$ vérifie
\[
\mu_1+\mu_2+\mu_3 \geqslant \Lambda_1+\Lambda_2+\Lambda_3,
\]
donc il existe un rationnel positif $R$ tel que $\mu_1+\mu_2+\mu_3=\Lambda_1+\Lambda_2+\Lambda_3 + 2R$. Il s'agit de l'égalité \eqref{eq:SO*(2n)_équationsparHorn_égalité}. On va sommer cette égalité avec les équations (B).

En sommant les équations \eqref{eq:SO*(2n)_équationsparHorn_égalité} et \eqref{eq:SO*(2n)_équationsdupolyèdre_eq1}, on obtient
\[
2\mu_2+2\mu_3\geqslant 2\Lambda_2+2\Lambda_3+2R,
\]
ou encore, en divisant par $2$, l'équation $\mu_2+\mu_3\geqslant \Lambda_2+\Lambda_3+R$, qui est égale à \eqref{eq:SO*(2n)_équationsparHorn_eq1}. De même, en sommant \eqref{eq:SO*(2n)_équationsparHorn_égalité} et \eqref{eq:SO*(2n)_équationsdupolyèdre_eq2} (resp. \eqref{eq:SO*(2n)_équationsparHorn_égalité} et \eqref{eq:SO*(2n)_équationsdupolyèdre_eq3}), on obtient l'inégalité \eqref{eq:SO*(2n)_équationsparHorn_eq2} (resp. \eqref{eq:SO*(2n)_équationsparHorn_eq3}).

Enfin, on peut voir que l'on obtient l'équation \eqref{eq:SO*(2n)_équationsparHorn_eq4} (resp. \eqref{eq:SO*(2n)_équationsparHorn_eq5}, resp. \eqref{eq:SO*(2n)_équationsparHorn_eq6}) en faisant $2$\eqref{eq:SO*(2n)_équationsparHorn_égalité}$+$\eqref{eq:SO*(2n)_équationsdupolyèdre_eq4}$+$\eqref{eq:SO*(2n)_équationsdupolyèdre_eq5} (resp. $\frac{1}{2}$\eqref{eq:SO*(2n)_équationsparHorn_égalité}$+\frac{1}{2}$\eqref{eq:SO*(2n)_équationsdupolyèdre_eq4}, resp $\frac{1}{2}$\eqref{eq:SO*(2n)_équationsparHorn_égalité}$+\frac{1}{2}$\eqref{eq:SO*(2n)_équationsdupolyèdre_eq5}). Donc on a bien $\mu\in\Delta_K(\Orb_{\Lambda})\cap\wedge_{\Q,+}^*$.

\bigskip

Maintenant, supposons que $\mu$ appartient à $\Delta_K(\Orb_{\Lambda})\cap\wedge_{\Q}^*$. Puisque $\Delta_K(\Orb_{\Lambda})$ est, par définition, contenu dans la chambre de Weyl $\got{t}_+^*$, on a $\mu_1\geqslant\mu_2\geqslant\mu_3$. D'après la Proposition \ref{prop:SO*(2n)_CNS_polyèdre}, il existe un rationnel positif $R$ tel que $\mu$ vérifie l'égalité \eqref{eq:SO*(2n)_équationsparHorn_égalité} et les neuf inégalités \eqref{eq:SO*(2n)_équationsparHorn_eq1} à \eqref{eq:SO*(2n)_équationsparHorn_eq9}. En procédant de manière opposée à la première partie de la preuve, c'est-à-dire, en effectuant le calcul des équations $2$\eqref{eq:SO*(2n)_équationsparHorn_eq1}$-$\eqref{eq:SO*(2n)_équationsparHorn_égalité} (resp. $2$\eqref{eq:SO*(2n)_équationsparHorn_eq2}$-$\eqref{eq:SO*(2n)_équationsparHorn_égalité}, resp. $2$\eqref{eq:SO*(2n)_équationsparHorn_eq3}$-$\eqref{eq:SO*(2n)_équationsparHorn_égalité}, resp. $2$\eqref{eq:SO*(2n)_équationsparHorn_eq5}$-$\eqref{eq:SO*(2n)_équationsparHorn_égalité}, resp. $2$\eqref{eq:SO*(2n)_équationsparHorn_eq6}$-$\eqref{eq:SO*(2n)_équationsparHorn_égalité}), on obtient l'inégalité \eqref{eq:SO*(2n)_équationsdupolyèdre_eq1} (resp. \eqref{eq:SO*(2n)_équationsdupolyèdre_eq2}, resp \eqref{eq:SO*(2n)_équationsdupolyèdre_eq3}, resp. \eqref{eq:SO*(2n)_équationsdupolyèdre_eq4}, resp. \eqref{eq:SO*(2n)_équationsdupolyèdre_eq5}). Donc $\mu\in\mathcal{P}_{\Lambda}$. On en conclut l'égalité $\Delta_K(\Orb_{\Lambda})\cap\wedge_{\Q}^* = \mathcal{P}_{\Lambda}$. D'où le résultat en passant aux adhérences dans $\got{t}^*$.
\end{proof}


\chapter[Structure symplectique des orbites holomorphes]{Structure symplectique des orbites coadjointes holomorphes}
\label{chap:symplecto_mcduff}



L'étude du polyèdre moment $\Delta_K(\Orb_{\Lambda})$ de la projection d'une orbite coadjointe holomorphe nécessite une meilleure connaissance de la structure symplectique de $\Orb_{\Lambda}$. Dans ce chapitre, nous prouvons l'existence d'un symplectomorphisme $K$-équivariant entre l'orbite $(\Orb_{\Lambda},\Omega_{\Orb_{\Lambda}})$ et une variété symplectique \og{}plus simple\fg{}, produit direct de l'orbite coadjointe compacte $(K\cdot\Lambda,\Omega_{K\cdot\Lambda})$ et d'un espace vectoriel symplectique. L'avantage de cette nouvelle variété symplectique est que sa structure hamiltonienne donne un polyèdre moment qui peut être décrit en termes de représentations du groupe $K$, description que l'on pourrait considérer comme analogue non-compact de la version algébrique du polytope moment donnée au paragraphe \ref{section:polytopemoment_algébrique}. 

\section{Symplectomorphisme de McDuff}
\label{section:symplecto_mcduff}

Soit $G$ un groupe de Lie réel semi-simple connexe non compact à centre fini et $\got{g}$ son algèbre de Lie. On fixe $\got{g}=\got{k}\oplus\got{p}$ une décomposition de Cartan de son algèbre de Lie et on note $K$ le sous-groupe de Lie connexe de $G$ d'algèbre de Lie $\got{k}$. Nous supposerons que l'espace symétrique $G/K$ est hermitien. Il s'agit en particulier d'une variété kählérienne à courbure négative difféomorphe à l'espace vectoriel $\got{p}$. Dans \cite{mcduff}, McDuff a prouvé qu'une telle variété kählérienne est symplectomorphe à $\got{p}$ muni de la forme symplectique linéaire $\Omega_{\got{p}}$ définie par
\[
\Omega_{\got{p}}(A,B) = B_{\got{g}}(A,\ad(z_0)B) \quad \text{pour tous $A,B\in\got{p}$},
\]
où $B_{\got{g}}$ désigne la forme de Killing sur $\got{g}$ et $z_0$ est un élément du centre de $\got{k}$ tel que $\ad(z_0)$ définisse une structure complexe $K$-invariante sur $\got{p}$. Un tel élément existe d'après la Proposition \ref{prop:equivalence_G/Khermitien}.

On peut également interpréter l'espace symétrique hermitien $G/K$ comme l'orbite coadjointe dans $\got{g}^*$ d'un certain élément de $\got{k}^*$, central pour $K$. Plus généralement, si on prend $\Lambda\in\got{k}^*$, la décomposition de Cartan nous donne également un difféomorphisme $K$-équivariant entre l'orbite coadjointe elliptique $\Glambda$ et la variété symplectique $\Klambda\times\got{p}$. Ceci nous amène à la question suivante : existe-t-il une généralisation du symplectomorphisme de McDuff au cas d'une orbite elliptique non centrale?

\medskip

Fixons un tore maximal $T$ de $K$, notons $\got{t}$ son algèbre de Lie et $\got{t}_+^*$ une chambre de Weyl de $\got{t}^*$. Au choix de l'élément $z_0\in\got{z}(\got{k})$ correspond un système de racines non compactes positives $\got{R}_n^+$, ce sont les racines apparaissant dans $\got{p}^+=\ker(\ad(z_0)|_{\got{p}_{\C}}-i)$, et donc une chambre holomorphe $\Chol$. Pour plus de détails, cf paragraphe \ref{section:la_chambre_holomorphe}.
%

Dans le cas où on fait l'hypothèse supplémentaire que $G$ est simple, Paradan a prouvé, dans \cite{Paradan}, que pour tout élément $\Lambda$ de $\Chol$, on a l'égalité $\Delta_K(\Glambda) = \Delta_K(\Klambda\times\got{p})$ des polyèdres moments respectifs des variétés hamiltoniennes $(\Glambda,\Omega_{G\cdot\Lambda})$ et $(\Klambda\times\got{p},\Omega_{K\cdot\Lambda\times\got{p}})$, pour leurs applications moments standards. Ici, $\Omega_{K\cdot\Lambda\times\got{p}}:=\Omega_{K\cdot\Lambda}\oplus\Omega_{\got{p}}$ désigne la somme directe des deux formes symplectiques $\Omega_{K\cdot\Lambda}$ et $\Omega_{\got{p}}$ sur le produit de variétés $\Klambda\times\got{p}$. 
Il a alors conjecturé que ces variétés symplectiques sont $K$-symplectomorphes, l'égalité des polyèdres moments corroborant cette idée.

Cette conjecture est vraie, même sans l'hypothèse de simplicité du groupe $G$, et nous proposons dans ce chapitre une preuve de ce résultat. Il s'agit de prouver le théorème suivant.

\begin{theo}
\label{theo:symplecto_principal}
On considère les hypothèses énoncées au début du paragraphe. Soit $\Lambda\in\Chol$. Alors il existe un symplectomorphisme $K$-équivariant entre les variétés symplectiques $(\Glambda,\Omega_{G\cdot\Lambda})$ et $(\Klambda\times\got{p},\Omega_{K\cdot\Lambda\times\got{p}})$ qui envoie $(k\Lambda,0)\in\Klambda\times\got{p}$ sur $k\Lambda\in G\cdot\Lambda$, pour tout $k\in K$.
\end{theo}

Comme conséquence directe du Théorème \ref{theo:symplecto_principal}, nous obtenons une nouvelle preuve de l'égalité des images des applications moments respectives.

\begin{coro}
\label{coro:égalité_polyèdresmoments}
Les applications moments $\Phi_{\Glambda}$ et $\Phi_{\Klambda\times\got{p}}$ ont la même image dans $\got{k}^*$. En particulier, leurs polyèdres moments respectifs $\Delta_K(\Glambda)$ et $\Delta_K(\Klambda\times\got{p})$ sont égaux.
\end{coro}

Travailler sur la variété $K$-hamiltonienne $(\Klambda\times\got{p},\Omega_{K\cdot\Lambda\times\got{p}}, \Phi_{K\cdot\Lambda\times\got{p}})$ en lieu et place de la variété $K$-hamiltonienne $(\Glambda,\Omega_{G\cdot\Lambda},\Phi_{G\cdot\Lambda})$, nous permettra d'obtenir les équations du polyèdre moment $\Delta_K(G\cdot\Lambda)$ de la projection d'une orbite holomorphe dans le Chapitre \ref{chap:calcul_exemples_par_pairesbiencouvrantes}.

\section{Une version non compacte du Théorème de Moser}
\label{section:moser}

Une méthode efficace pour trouver des symplectomorphismes entre variétés symplectiques est d'utiliser un argument de type Moser \cite{moser}. Dans le cadre non compact, il y a cependant un problème supplémentaire qui apparaît : trouver un champ de vecteurs dépendant du temps qui s'intègre complètement afin d'obtenir une isotopie. Par exemple, dans \cite{mcduff}, McDuff utilise la complétude au sens riemannien de sa variété, couplé au Théorème de comparaison de Rauch, pour pouvoir appliquer l'argument de Moser. Une autre possibilité est de travailler sur des variétés hamiltoniennes avec applications moments propres, comme par exemple dans \cite{karshon_tolman}. Ici, nous allons utiliser une variante de cette deuxième méthode.

Soit $K$ un groupe de Lie compact connexe, $E$ un espace vectoriel réel de dimension finie sur lequel agit linéairement $K$ et $M$ une variété compacte connexe munie d'une action de $K$. Nous travaillerons maintenant sur le fibré vectoriel trivial $M\times E$, muni de l'action diagonale de $K$. Soit $(\Omega_t)_{t\in[0,1]}$ une famille de formes symplectiques $K$-invariantes de $M\times E$ dépendant de manière lisse en $t$, telle que, pour chaque $t\in[0,1]$, la variété symplectique $(M\times E,\Omega_t)$ soit hamiltonienne pour l'action de $K$, d'application moment $\phi_t:M\times E\rightarrow\got{k}^*$.

Dans l'énoncé ci-dessous, $\left(T_{(m,0)}(\{m\}\times E)\right)^{\Omega_t}$ dénote l'orthogonal symplectique de l'espace tangent $T_{(m,0)}(\{m\}\times E)$ au-dessus de $(m,0)$ de la sous-variété $\{m\}\times E$, dans l'espace tangent $T_{(m,0)}(M\times E)$, par rapport à la forme symplectique $\Omega_t$.

\begin{theo}
\label{theo:existence_isotopie_noncompact_casgeneral}
%
Si les quatre assertions suivantes sont vérifiées,
\begin{enumerate}
\item Il existe une famille lisse $(\mu_t)_{t\in[0,1]}$ de $1$-formes différentielles $K$-invariantes telles que $d\mu_t = \frac{d}{dt}\Omega_t$ pour tout $t\in[0,1]$,
\item L'ensemble $\{\phi_t(m,0); m\in M, t\in[0,1]\}$ est borné dans $\got{k}^*$,
\item Pour tout $m\in M$ et tout $t\in[0,1]$, $\left(T_{(m,0)}(\{m\}\times E)\right)^{\Omega_t} = T_{(m,0)}(M\times\{0\})$,
\item Il existe deux réels $d,\gamma>0$ tels que
\[
\|\phi_t(m,v)\| \geqslant d\|v\|^{\gamma}, \quad \forall (m,v)\in M\times E,\ \forall t\in[0,1],
\]
\end{enumerate}
alors, il existe une isotopie $K$-équivariante $\rho_t:M\times E\rightarrow M\times E$ telle que $\rho_t^*\Omega_t = \Omega_0$ pour tout $t\in[0,1]$.

Si, de plus, pour un $m_0\in M$, on a $\mu_t|_{(m_0,0)}(u,0)=0$ pour tout $(t,u)\in[0,1]\times T_{m_0}M$, alors $\rho_t(m_0,0) = (m_0,0)$ pout tout $t\in[0,1]$.
\end{theo}

\begin{rema}
La condition $(4)$ du Théorème \ref{theo:existence_isotopie_noncompact_casgeneral} peut être considérée comme une condition d'uniforme propreté des applications moments $\phi_t$, pour $t$ parcourant $[0,1]$. Il s'agit d'une condition très forte, qui est cependant réalisée pour les variétés hamiltoniennes que l'on rencontrera dans la suite de ce chapitre.
\end{rema}

\begin{rema}
\label{rema:theo_moser_conditionsupplémentaire_toujours_vérifiée_pour_E=V}
Lorsque $M\times E\cong E$, c'est-à-dire $M=\{m_0\}$ est réduite à un point, la dernière condition \og{}$\mu_t|_{0}(0)=0$ pour tout $t\in[0,1]$\fg{} est évidemment toujours vérifiée. En particulier, si les quatre autres hypothèses sont vérifiées, il existera toujours un symplectomorphisme $K$-équivariant qui stabilise le vecteur nul.

L'hypothèse ($3$) est, quant à elle, clairement vérifiée dès que $M\times E\cong E$ est vectoriel, car elle revient à écrire
\[
\left(T_0 E\right)^{\Omega_t} = \{0\},
\]
condition impliquée par la non-dégénérescence de la $2$-forme alternée $\Omega_t|_0$ sur $T_0 E$.
\end{rema}

\begin{lemm}
\label{lemm:exactitude_famillede1formes}
Soit $(\mu_t)_{t\in[0,1]}$ une famille lisse de $1$-formes différentielles sur $M\times E$ $K$-invariantes. Il existe une famille lisse $(f_t)_{t\in[0,1]}$ de fonctions $C^{\infty}$ et $K$-invariantes sur $M\times E$, telles que 
\begin{enumerate}
\item[(\emph{i})] $df_t|_{T(M\times\{0\})} \equiv 0$,
\item[(\emph{ii})] $\imath(v)(\mu_t-df_t) = 0$ sur $M\times\{0\}$ pour tout $v\in E$.
\end{enumerate}
\end{lemm}

\begin{proof}
Pour tout $t\in[0,1]$, tout $v\in E$ et tout $m\in M$, on pose
\[
f_t(m,v) = 2\int_0^1 \mu_t|_{(m,sv)}(0,sv)ds.
\]
Les fonctions $f_t$ sont toutes $C^{\infty}$ et dépendent de manière lisse par rapport à $t$, par le théorème de dérivabilité sous le signe intégral. La fonction $f_t$ est $K$-invariante par $K$-invariance de $\mu_t$. La dernière assertion se vérifie également directement sur la définition de $f_t$, en passant en coordonnées locales si nécessaire.
\end{proof}

\begin{proof}[Preuve du Théorème \ref{theo:existence_isotopie_noncompact_casgeneral}]
D'après le Lemme \ref{lemm:exactitude_famillede1formes}, quitte à remplacer $\mu_t$ par $\mu_t-df_t$ pour une certaine famille lisse $(f_t)_{t\in[0,1]}$ de fonctions $C^{\infty}$ et $K$-invariantes sur $M\times E$, on peut supposer que $\mu_t$ est $K$-invariante et vérifie
\[
\mu_t|_{(m,0)}(0,v) = 0, \quad \forall m\in M, \ \forall v\in E,
\]
pour tout $t\in[0,1]$.

Définissons maintenant le champ de vecteurs $\xi_t$ sur $M\times E$ dépendant du temps par
\[
\imath(\xi_t)\Omega_t = -\mu_t, \qquad \text{pour tout $t\in[0,1]$}.
\]
Ce champ de vecteurs est lisse par rapport à $t$ et est $K$-invariant, car $\Omega_t$ et $\mu_t$ le sont. Le but est d'intégrer ce champ de vecteurs $\xi_t$ en une isotopie de $M\times E$.

D'après l'hypothèse sur $\mu_t$, on remarque que, pour tout $m\in M$ et tout $v\in E$, on a
\[
\Omega_t|_{(m,0)}\bigl(\xi_t(m,0),(0,v)\bigr) = -\mu_t|_{(m,0)}(0,v) = 0.
\]
Donc le vecteur tangent $\xi_t(m,0)$ appartient à $\left(T_{(m,0)}(\{m\}\times E)\right)^{\Omega_t}$, orthogonal symplectique  de $T_{(m,0)}\{m\}\times E$ dans $T_{(m,0)}(M\times E)$ pour la forme symplectique $\Omega_t|_{(m,0)}$. Or, d'après l'hypothèse $(3)$, on a $\left(T_{(m,0)}(\{m\}\times E)\right)^{\Omega_t} = T_{(m,0)}(M\times\{0\})$. D'où $\xi_t(m,0)$ appartient à $T_{(m,0)}(M\times\{0\})$ pour tout $m\in M$, pour tout $t\in[0,1]$.

\medskip

Nous considérons l'équation différentielle dépendant du temps $t\in[0,1]$ associée au champ de vecteurs $\xi_t$
\begin{equation}
\label{eq:equadiff_integration_xi_t}
\left\{\begin{array}{l}
\rho_0(x) = x, \\
\displaystyle\frac{d}{dt}\rho_t(x) = \xi_t(\rho_t(x)),
\end{array}\right.
\end{equation}
pour toute condition initiale $x\in M\times E$. Une application du Théorème de Cauchy-Lipschitz montre que le domaine de définition $\mathcal{D}\subset[0,1]\times M\times E$ des courbes intégrales est ouvert dans $[0,1]\times M\times E$.

Soit $r>0$ un réel positif quelconque. Définissons le voisinage ouvert et connexe $U_r:=M\times B(0,r)$ de $M\times\{0\}$ dans $M\times E$, la boule étant prise pour une norme euclidienne $K$-invariante $\|\cdot\|$ sur $E$ (une telle norme existe toujours car $K$ est compact). Son adhérence $\overline{U_r}$ est le compact $M\times \overline{B}(0,r)$ de $M\times E$. On définit l'intervalle
\[
I_r := \{\varepsilon\in[0,1]; \forall x\in\overline{U_r}, \text{ la courbe intégrale $\rho_t(x)$ est définie sur tout $[0,\varepsilon]$}\}.
\]
Nous voulons montrer l'égalité $I_r = [0,1]$.

\bigskip

\noindent{\it \'Etape 1: montrons que l'intervalle $I_r$ est ouvert dans $[0,1]$.}

Soit $\varepsilon\in[0,1[$ tel que la courbe intégrale $\rho_t(x)$ soit définie sur $[0,\varepsilon]$ pour tout $x\in\overline{U}_r$ (ceci est possible au moins pour $\varepsilon=0$). Puisque le domaine $\mathcal{D}$ est ouvert, pour chaque $x\in\overline{U}_r$, il existe un voisinage $V_x$ de $x$ dans $M\times E$ et un réel $\varepsilon_x>\varepsilon$ tel que $\rho_t(y)$ soit défini pour tout $y\in V_x$ et tout $t\in[0,\varepsilon_x[$. Or $\overline{U_r}$ est compact, il est donc recouvert par un nombre fini de voisinages $V_{x_1},\ldots,V_{x_s}$, où $x_1,\ldots,x_s$ sont des éléments de $\overline{U_r}$. Ceci implique que l'intervalle $I_r$ contient le voisinage $[0,\min_{i=1}^s \varepsilon_{x_i}[$ de $\varepsilon$ dans $[0,1]$. Il est donc clair que $I_r$ est ouvert dans $[0,1]$.

\bigskip

\noindent{\it \'Etape 2: montrons que l'intervalle $I_r$ est fermé dans $[0,1]$.}

Pour prouver que $I_r$ est fermé, nous aurons besoin de montrer plusieurs faits.

\smallskip

\begin{center}
\begin{minipage}{13cm}
{\bf Fait 1:} {\it Pour tout $m\in M$, la courbe intégrale $\rho_t(m,0)$ est définie sur tout $[0,1]$.}
\end{minipage}
\end{center}

En effet, on a pu prendre $\mu_t$ de sorte que
\[
\xi_t(m,0)\in T_{(m,0)}(M\times\{0\}), \quad \forall m\in M, \ \forall t\in[0,1].
\]
Cela implique que toute courbe intégrale $t\in[0,\varepsilon[\mapsto\rho_t(m,0)$ est contenue dans $M\times\{0\}$ par unicité locale, pour tout $m\in M$. Or $M$ est compact, donc la courbe intégrale maximale partant de $(m,0)$ est définie sur $[0,1]$ tout entier, pour tout $m\in M$.

\bigskip

Fixons maintenant un élément $m_0\in M$. Cet élément nous permet de définir la famille de vecteurs de $\got{k}^*$
\[
c_t := \phi_t\circ\rho_t(m_0,0) - \phi_0(m_0,0) \qquad \forall t\in[0,1],
\]
et notons
\[
C := \sup_{t\in[0,1]}\|c_t\| < +\infty.
\]
La finitude de cette borne supérieure est donnée par l'hypothèse $(2)$, puisque, pour tout $t\in[0,1]$, le vecteur $\phi_t\circ\rho_t(m_0,0)$ est contenu dans l'ensemble borné $\{\phi_t(m,0); m\in M, t\in[0,1]\}$ de $\got{k}^*$.

Nous définissons également le réel
\[
D_r := \sup_{x\in\overline{U}_r}\|\phi_0(x)\|,
\]
pour tout réel $r>0$. Il est clair que si l'on a deux réels positifs $r<s$, alors $D_r\leqslant D_s$.

\smallskip

\begin{center}
\begin{minipage}{13cm}
{\bf Fait 2:} {\it Soient $r>0$ et $\varepsilon\in\,]0,1]$ tels que, pour tout $x\in U_r$, la courbe intégrale $\rho_t(x)$ soit définie sur $[0,\varepsilon[$. Alors, pour tout $t\in[0,\varepsilon[$, on a 
\[
\rho_t(U_r)\subset\overline{U}_{\left(\frac{D_{r}+C}{d}\right)^{1/\gamma}}.
\]}
\end{minipage}
\end{center}

Pour $r>0$ et $\varepsilon\in\,]0,1]$ pris comme dans l'énoncé du Fait 2, nous obtenons pour tout $t\in[0,\varepsilon[$ une application lisse $\rho_t:U_r\rightarrow M\times E$. Puisque $\Omega_t$ est fermée, on a $L_{\xi_t}\Omega_t = d(i(\xi_t)\Omega_t) = -d\mu_t = -\frac{d}{dt}\Omega_t$. Donc $\rho_t^*(L_{\xi_t}\Omega_t + \frac{d}{dt}\Omega_t) = \frac{d}{dt}(\rho_t^*\Omega_t) = 0$ sur $U_r$, pour tout $t\in[0,\varepsilon[$. Par conséquent, on a $\rho_t^*\Omega_t = \Omega_0|_{U_r}$ pour tout $t\in[0,\varepsilon[$.

L'ouvert $U_r$ de $M\times E$ est clairement un voisinage $K$-stable de $M\times\{0\}$ dans $M\times E$. La restriction $\phi_0|_{U_r}:U_r\rightarrow\got{k}^*$ définit donc une application moment pour la $K$-variété symplectique $(U_r,\Omega_0|_{U_r})$. Puisque $\rho_t^*\Omega_t = \Omega_0|_{U_r}$, l'application $\rho_t^*\phi_t=\phi_t\circ\rho_t:U_r\rightarrow\got{k}^*$ est une autre application moment pour $(U_r,\Omega_0|_{U_r})$. Comme $U_r$ est connexe, on en déduit que pour tout $t\in[0,\varepsilon[$, l'application $\phi_t\circ\rho_t-\phi_0$ est constante. De la définition des constantes $c_t$, on a $\phi_t\circ\rho_t(x)=\phi_0(x)+c_t$, pour tout $x\in U_r$ et tout $t\in[0,\varepsilon[$. En particulier, on voit que l'on a
\begin{equation}
\label{eq:majoration_phi_t_rho_t}
\|\phi_t(\rho_t(x))\| \leqslant \|\phi_0(x)\| + \|c_t\| \leqslant D_r + C
\end{equation}
pour tout $x\in U_r$ et tout $t\in[0,\varepsilon[$.

Notons $\pi_V:M\times E\rightarrow E$ la seconde projection canonique. En appliquant l'hypothèse $(4)$, on déduit de \eqref{eq:majoration_phi_t_rho_t} que, pour tout $x\in U_r$ et tout $t\in[0,\varepsilon[$,
\[
D_r+C \geqslant \|\phi_t(\rho_t(x))\| \geqslant d\|\pi_V(\rho_t(x))\|^{\gamma}.
\]
Il en découle directement l'inclusion
\[
\rho_t(U_r) \subset \overline{U}_{\left(\frac{D_{r}+C}{d}\right)^{1/\gamma}}
\]
pour tout $t\in[0,\varepsilon[$.

\smallskip

\begin{center}
\begin{minipage}{13cm}
{\bf Fait 3:} {\it Soient $r>0$ et $\varepsilon\in]0,1]$ tels que, pour tout $x\in\overline{U_r}$, la courbe intégrale $\rho_t(x)$ soit définie sur $[0,\varepsilon[$. Alors, pour tout $t\in[0,\varepsilon[$, on a 
\[
\rho_t(\overline{U}_r)\subset\overline{U}_{\left(\frac{D_{r+1}+C}{d}\right)^{1/\gamma}}.
\]}
\end{minipage}
\end{center}

Soit $\tau\in[0,\varepsilon[$. Par hypothèse, $\rho_t(x)$ est définie sur $[0,\tau]$ pour tout $x\in\overline{U}_r$. En réalisant un raisonnement similaire à l'\'Etape 1, on peut montrer qu'il existe un voisinage ouvert $V$ de $\overline{U}_r$ tel que la courbe intégrale $\rho_t(y)$ soit définie sur $[0,\tau]$ pour tout $y\in V$. Par un argument classique de topologie, il existe un réel $r'>r$ tel que $\overline{U}_{r'}\subset V$. En effet, fixons $s>r$ et supposons que $\overline{U}_s\cap V^{c}\neq \emptyset$ (sinon $\overline{U}_s\subset V$ et $r'=s$ convient). On sait que l'intersection $\overline{U}_s\cap V^{c}$ est un compact de $M\times E$, donc la borne inférieure de l'ensemble $\{\|v\|; (x,v)\in\overline{U}_s\cap V^{c}\}$ est atteinte en un point $(x_0,v_0)$ de $\overline{U}_s\cap V^{c}$. Mais alors on a forcément $\|v_0\|> r$, puisque l'hypothèse $\|v_0\|\leqslant r$ impliquerait $(x_0,v_0)\in\overline{U}_r\subset V$, ce qui contredirait le fait que $(x_0,v_0)$ appartient à $V^{c}$. Ainsi, en prenant un réel $r'$ vérifiant l'encadrement
\[
r < r' < \inf \{\|v\|; (x,v)\in\overline{U}_s\cap V^{c}\},
\]
on a $\overline{U}_{r'}\subset V$.

Par conséquent, il existe un réel $r'>r$ tel que la courbe intégrale $\rho_t(y)$ soit définie sur $[0,\tau]$ pour tout $y\in U_{r'}$. On peut supposer que $r<r'\leqslant r+1$. On a donc $D_r \leqslant D_{r'}\leqslant D_{r+1}$.

Appliquons le Fait 2 à $U_{r'}$ et $[0,\tau[$. Ceci nous donne donc, pour tout $t\in[0,\tau[$, les inclusions
\[
\rho_t(\overline{U}_r) \subset \rho_t(U_{r'}) \subset \overline{U}_{\left(\frac{D_{r'}+C}{d}\right)^{1/\gamma}} \subset \overline{U}_{\left(\frac{D_{r+1}+C}{d}\right)^{1/\gamma}}.
\]
Ceci étant vrai pour tout $\tau\in[0,\varepsilon[$, on en déduit que pour tout $t\in[0,\varepsilon[$, on a $\rho_t(\overline{U}_r) \subset \overline{U}_{\left(\frac{D_{r+1}+C}{d}\right)^{1/\gamma}}$.

\bigskip

Notons $\varepsilon_r := \sup I_r$. Par définition de $I_r$, cela signifie que, pour tout $x\in \overline{U}_r$, la courbe intégrale $\rho_t(x)$ est définie au moins sur $[0,\varepsilon_r[$. Or, d'après le Fait 3, $\rho_t(\overline{U}_r) \subset \overline{U}_{\left(\frac{D_{r+1}+C}{d}\right)^{1/\gamma}}$ est vérifié pour tout $t\in[0,\varepsilon[$. Par conséquent, on peut prolonger ponctuellement en $t=\varepsilon_r$ chacune des courbes intégrales ayant pour condition initiale un élément de $\overline{U}_r$. D'où $\varepsilon$ est lui-même un élément de $I_r$. Finalement, cette propriété implique que $I_r$ est un intervalle fermé de $[0,1]$.

De plus, il est clair que $I_r$ est non vide, puisqu'il contient au moins $0$. L'intervalle $I_r$ est alors ouvert, fermé et non vide dans le connexe $[0,1]$. On en conclut que $I_r=[0,1]$. Autrement dit, toutes les courbes intégrales ayant pour condition initiale un élément de $\overline{U}_r$ sont définies sur $[0,1]$ tout entier. Mais puisque tout élément de $M\times E$ est contenu dans un certain $\overline{U}_r$, le champ de vecteurs $\xi_t$ est donc complet.

Remarquons également que pour tout $t\in[0,1]$, l'application lisse $\rho_t:M\times E\rightarrow M\times E$ est en réalité un difféomorphisme. Ceci est un résultat classique sur les champs de vecteurs dépendant du temps, que l'on pourra trouver par exemple dans \cite[Théorème 52]{lafontaine}. Ainsi, $\rho_t$ définit une isotopie de $M\times E$ et vérifie $\rho_t^*\Omega_t=\Omega_0$ pour tout $t\in[0,1]$. Et comme $\xi_t$ est $K$-invariant, $\rho_t$ est $K$-équivariante.

Il reste à prouver la dernière assertion. On a vu au début de la preuve que l'on pouvait se ramener à une famille lisse $(\mu_t)_{t\in[0,1]}$ de $1$-formes $K$-invariantes qui vérifient $\mu_t|_{(m,0)}(0,v) = 0$ pour tout $(m,v)\in M\times E$. Or, l'hypothèse sur $m_0$ devient maintenant :
\[
\mu_t|_{(m_0,0)} \equiv 0 \quad \forall t\in[0,1].
\]
Par conséquent, $\xi_t(m_0,0) = 0$ pour tout $t\in[0,1]$, et comme $\rho_t(m_0)$ est l'unique courbe intégrale de l'équation \eqref{eq:equadiff_integration_xi_t} pour $x=(m_0,0)$ vérifiant $\rho_0(m_0,0) = (m_0,0)$, on en déduit que $\rho_t(m_0,0)=(m_0,0)$ pour tout $t\in[0,1]$.
\end{proof}

L'avantage du Théorème \ref{theo:existence_isotopie_noncompact_casgeneral} est que la famille lisse $(\Omega_t)_{t\in[0,1]}$, représentant le chemin lisse reliant les deux formes symplectiques considérées sur $M\times E$, n'est pas limitée au segment, comme c'est le cas dans l'argument de Moser. Il nous donne donc plus de liberté pour trouver un chemin symplectique, sachant que l'ensemble des formes symplectiques sur une variété n'est pas convexe.

Cependant, il n'y a pas que des avantages. En effet, il devient alors difficile de prouver que chaque $\Omega_t$ est fermée et possède une application moment $\phi_t$, alors que ceci est direct dans le cas du segment. C'est pourquoi on se ramène à étudier le segment reliant deux formes symplectiques dès que cela est possible.

\begin{coro}
\label{coro:existence_isotopie_noncompact_parsegment}
Soient $\Omega_0$ et $\Omega_1$ deux formes symplectiques $K$-invariantes sur $M\times E$, d'applications moments $\phi_0$ et $\phi_1$. Posons $\Omega_t=t\Omega_1+(1-t)\Omega_0$ et $\phi_t = t\phi_1+(1-t)\phi_0$ pour tout $t\in[0,1]$. Si les quatre assertions suivantes sont vérifiées,
\begin{enumerate}
\item pour tout $t\in[0,1]$, $\Omega_t$ est symplectique sur $M\times E$,
\item il existe une $1$-forme différentielle $K$-invariante $\mu$ telle que $d\mu = \Omega_1-\Omega_0$,
\item pour tout $m\in M$ et tout $t\in\{0,1\}$, $\left(T_{(m,0)}(\{m\}\times E)\right)^{\Omega_t} = T_{(m,0)}(M\times\{0\})$,
\item il existe deux réels $d,\gamma>0$ tels que
\[
\|\phi_t(m,v)\| \geqslant d\|v\|^{\gamma}, \quad \forall (m,v)\in M\times E,\ \forall t\in[0,1],
\]
\end{enumerate}
alors $(M\times E,\Omega_0)$ et $(M\times E,\Omega_1)$ sont $K$-symplectomorphes.
\end{coro}

\begin{proof}
Ce résultat découle directement du Théorème \ref{theo:existence_isotopie_noncompact_casgeneral}, où $\frac{d}{dt}\Omega_t=\Omega_1-\Omega_0$ et $\mu_t=\mu$ pour tout $t\in[0,1]$ est $K$-équivariante et dépend de manière lisse en $t$. Ici, la famille d'applications moments $(\phi_t)_{t\in[0,1]}$ dépend continûment de $t$. Donc l'ensemble $\{\phi_t(m,0);t\in[0,1],m\in M\}$ est compact comme image continue du compact $[0,1]\times M$.

De plus, la condition (3) est en réalité valable pour tout $t\in[0,1]$, grâce au fait que $\Omega_t$ s'écrive sous la forme $\Omega_t=t\Omega_1+(1-t)\Omega_0$ et qu'elle soit symplectique par la condition (1). En effet, on voit clairement de l'écriture de $\Omega_t$ que
\[
\left(T_{(m,0)}(\{m\}\times E)\right)^{\Omega_t} \supset T_{(m,0)}(M\times\{0\}).
\]
Comme $\Omega_t$ est symplectique, on a égalité des deux sous-espaces vectoriels par un argument de dimension.

On a donc une isotopie $\rho_t:M\times E\rightarrow M\times E$ telle que $\rho_t^*\Omega_t=\Omega_0$ pour tout $t\in[0,1]$. En particulier, $\rho_1^*\Omega_1=\Omega_0$, donc $\rho_1$ est un symplectomorphisme $K$-équivariant entre $(M\times E,\Omega_0)$ et $(M\times E,\Omega_1)$.
\end{proof}

Nous travaillons ici avec un type très particulier de variété, puisqu'il s'agit d'un produit direct d'une variété compacte connexe et d'un espace vectoriel de dimension finie. La condition ($2$) du Corollaire \ref{coro:existence_isotopie_noncompact_parsegment} correspond ainsi à une version $K$-invariante du Lemme de Poincaré, où les deux applications homotopes sur $M\times E$ considérées sont l'identité de $M\times E$ et l'application composée $i\circ\pi$, avec $i:M\hookrightarrow M\times E$ l'injection de $M$ comme section nulle de $M\times E$ et $\pi:M\times E\rightarrow M$ la projection canonique.

Dans l'énoncé ci-dessous, nous dirons qu'une homotopie $F:N_1\times[0,1]\rightarrow N_2$ entre deux variétés $N_1$ et $N_2$ munies d'actions du groupe $K$, est $K$-équivariante si
\[
F(k\cdot x, t) = k\cdot F(x,t) \quad \forall x\in N_1, \forall k\in K.
\]
Cela revient à dire que $F$ est $K$-équivariante lorsqu'on munit $N_1\times[0,1]$ de l'action diagonale de $K$, avec action triviale de $K$ sur $[0,1]$.

\begin{lemm}
\label{lemm:lemmePoincaré_Kinvariant}
Soient $N_1$ et $N_2$ deux variétés munies d'une action du groupe $K$. Soient $f,g:N_1\rightarrow N_2$ deux application lisses telles qu'il existe une homotopie $K$-équivariante $F:N_1\times[0,1]\rightarrow N_2$ vérifiant $F(x,0)=f(x)$ et $F(x,1)=g(x)$ pour tout $x\in N_1$. Pour toute $p$-forme différentielle $\omega$ sur $N_2$, on pose la $(p-1)$-forme différentielle $h_F(\omega)$ définie sur $N_1$ par
\begin{equation}
\label{eq:defi_opérateur_h^p}
h_F(\omega)|_x := \int_{t=0}^{t=1} (F^*\omega)|_{(x,t)}, \quad \forall x\in N_1.
\end{equation}
L'application $h_F:\Omega^{p}(N_2)\rightarrow\Omega^{p-1}(N_1)$ est un opérateur linéaire qui vérifie les propriétés suivantes :
\begin{enumerate}
\item[(\emph{i})] Pour tout $k\in K$, on a $k^*(h_F(\omega)) = h_F(k^*\omega)$. En particulier, si $\omega$ est $K$-invariante, alors $h_F(\omega)$ est $K$-invariante.
\item[(\emph{ii})] $d\circ h_F + h_F\circ d = g^*-f^*$.
\end{enumerate}
\end{lemm}

La preuve de ce lemme n'est qu'une adaptation de la preuve standard du Lemme de Poincaré au cas $K$-équivariant. On pourra trouver une preuve dans le cas non équivariant dans \cite{warner}. La démonstration de la $K$-équivariance se voit directement dans la définition de l'opérateur $h_F$.

Nous allons appliquer ce lemme, dans notre contexte, aux applications $f=i\circ\pi$ et $g=\id_{M\times E}$. Une homotopie entre ces deux applications de $M\times E$ dans $M\times E$ est
\[
\begin{array}{cccc}
F: & M\times E\times[0,1]& \longrightarrow & M\times E \\
& (m,v,t) & \longmapsto & (m,tv).
\end{array}
\]
On peut facilement vérifier sur la formule \eqref{eq:defi_opérateur_h^p} que, pour tout $m\in M$ et pour toute $2$-forme différentielle $\Omega$ sur $M\times E$, la $1$-forme $h_F(\Omega)$ est nulle sur la sous-variété $M\times\{0\}$.

\begin{coro}
\label{coro:existence_isotopie_noncompact_parsegment_avec_condition_lemmedepoincaré}
Soient $\Omega_0$ et $\Omega_1$ deux formes symplectiques $K$-invariantes sur $M\times E$, d'applications moments $\phi_0$ et $\phi_1$. Posons $\Omega_t=t\Omega_1+(1-t)\Omega_0$ et $\phi_t = t\phi_1+(1-t)\phi_0$ pour tout $t\in[0,1]$. Si les quatre assertions suivantes sont vérifiées,
\begin{enumerate}
\item pour tout $t\in[0,1]$, $\Omega_t$ est symplectique sur $M\times E$,
\item la $2$ forme différentielle $\Omega_1-\Omega_0$ est dans le noyau de $i^*$,
\item pour tout $m\in M$ et tout $t\in\{0,1\}$, $\left(T_{(m,0)}(\{m\}\times E)\right)^{\Omega_t} = T_{(m,0)}(M\times\{0\})$,
\item il existe deux réels $d,\gamma>0$ tels que
\[
\|\phi_t(m,v)\| \geqslant d\|v\|^{\gamma}, \quad \forall (m,v)\in M\times E,\ \forall t\in[0,1],
\]
\end{enumerate}
alors il existe un symplectomorphisme $K$-équivariant entre $(M\times E,\Omega_0)$ et $(M\times E,\Omega_1)$ qui laisse fixe chaque point de la sous-variété $i(M) = M\times\{0\}$ de $M\times E$.
\end{coro}

\begin{proof}
Il suffit d'appliquer le Corollaire \ref{coro:existence_isotopie_noncompact_parsegment}, où toutes les hypothèses sont déjà vérifiées, sauf la condition ($2$). Pour ce faire, nous allons appliquer le Lemme \ref{lemm:lemmePoincaré_Kinvariant} à $M\times E$, $f=i\circ\pi$ et $g=\id_{M\times E}$. L'homotopie $F$ entre $f$ et $g$ choisie dans le paragraphe précédant l'énoncé du Corollaire \ref{coro:existence_isotopie_noncompact_parsegment_avec_condition_lemmedepoincaré} est bien $K$-équivariante, car l'action de $K$ sur $E$ est linéaire. Or, $\Omega_1-\Omega_0$ est dans le noyau de $i^*$. Donc d'après le Lemme \ref{lemm:lemmePoincaré_Kinvariant} ({\it ii}), la $1$-forme différentielle $K$-invariante $\mu=h_F(\Omega_1-\Omega_0)$ est telle que $d\mu = \Omega_1-\Omega_0$, puisque $\Omega_1-\Omega_0$ est fermée. L'hypothèse ($2$) du Corollaire \ref{coro:existence_isotopie_noncompact_parsegment} est donc vérifiée.

De plus, on a vu dans le paragraphe précédant l'énoncé du Corollaire \ref{coro:existence_isotopie_noncompact_parsegment_avec_condition_lemmedepoincaré} que $\mu = h^2(\Omega_1-\Omega_0)$ est nulle sur la sous-variété $M\times\{0\}$. Appliquant ceci à la dernière assertion du Théorème \ref{theo:existence_isotopie_noncompact_casgeneral}, on en conclut qu'il existe un symplectomorphisme $K$-équivariant entre $(M\times E,\Omega_0)$ et $(M\times E,\Omega_1)$ qui stabilise chaque point $M\times\{0\}$ de $M\times E$.
\end{proof}

\section{Propriétés de la décomposition de Cartan de $G$}
\label{section:difféo_Gamma}

Comme indiqué dans la section \ref{section:moser}, nous allons prouver le Théorème \ref{theo:symplecto_principal} en utilisant des arguments de type Moser. Pour appliquer le Théorème \ref{theo:existence_isotopie_noncompact_casgeneral} et le Corollaire \ref{coro:existence_isotopie_noncompact_parsegment}, il faut d'abord nous ramener à deux formes symplectiques sur la même variété $M=\Klambda\times\got{p}$. Or, la décomposition de Cartan de $G$ induit un difféomorphisme $\Gamma$ de $\Klambda\times\got{p}$ sur $\Glambda$. Par conséquent, il nous suffira de considérer les deux formes symplectiques $\Omega_{K\cdot\Lambda\times\got{p}}$ et $\Gamma^*\Omega_{G\cdot\Lambda}$ sur $\Klambda\times\got{p}$.

Dans ce paragraphe, nous considérons un groupe de Lie $G$ réel semi-simple connexe non compact à centre fini.

\subsection{Notations}
\label{subsection:notations_décomp_cartan}

Notons $B_{\got{g}}$ la forme de Killing sur $\got{g}$. Comme $\got{g}$ est semi-simple, elle possède un involution de Cartan $\theta$. Rappelons qu'une involution de Cartan $\theta$ est un involution de l'algèbre de Lie $\got{g}$, c'est-à-dire une involution linéaire qui est aussi un automorphisme de l'algèbre de Lie, telle que la forme bilinéaire symétrique
\[
B_{\theta}(X,Y) = -B_{\got{g}}(X,\theta(Y)) \qquad \text{pour tous $X,Y\in\got{g}$}
\]
est définie positive sur $\got{g}$.

On note alors les deux sous-espaces propres
\[
\got{k}:= \ker(\theta-\id_{\got{g}}) \qquad \text{et} \qquad \got{p}:= \ker(\theta+\id_{\got{g}})
\]
de $\theta$ dans $\got{g}$. Le sous-espace vectoriel $\got{k}$ est en fait une sous-algèbre de Lie de $\got{g}$. En effet, on a les inclusions $[\got{k},\got{k}]\subset\got{k}$, $[\got{p},\got{p}]\subset\got{k}$ et $[\got{k},\got{p}]\subset\got{p}$. Notons $K$ le sous-groupe connexe de $G$ d'algèbre de Lie $\got{k}$.

Il est bien connu qu'il existe un automorphisme de groupe de Lie $\Theta$ de $G$, de différentielle $\theta$, tel que $\Theta^2 = \id_G$, que le sous-groupe de $G$ des éléments fixés par $\Theta$ est exactement $K$ et qu'il existe un difféomorphisme
\begin{equation}
\label{eq:difféo_décomposition_de_Cartan}
(k,Z)\in K\times\got{p}\mapsto\exp(Z)k\in G.
\end{equation}
Le lecteur pourra trouver une preuve de ce fait dans \cite[Théorème 6.31]{knapp}. Ce difféomorphisme est la décomposition de Cartan\index{Décomposition de Cartan} au niveau du groupe $G$. On peut facilement vérifier qu'il est $K$-équivariant pour $K$ qui agit sur $K$ et $G$ par multiplication à gauche, sur $\got{p}$ par l'action adjointe et sur $K\times\got{p}$ par action diagonale.

De plus, comme $G$ est à centre fini, $K$ est un sous-groupe compact maximal de $G$, toujours d'après \cite[Théorème 6.31]{knapp}.

Soit $\Lambda$ un élément de $\got{t}^*$ et $\Glambda$ son orbite coadjointe dans $\got{g}^*$. Le difféomorphisme \eqref{eq:difféo_décomposition_de_Cartan} induit un difféomorphisme
\[
\begin{array}{cccc}
\Gamma : & \Klambda\times\got{p} & \longrightarrow & \Glambda \\
& (k\Lambda,Z) & \longmapsto & \exp(Z)k\Lambda
\end{array}
\]\index{$\Gamma$}
qui est lui aussi $K$-équivariant pour les actions évidentes.

\subsection{La différentielle de $\Gamma$}


Pour connaître l'expression de la forme symplectique $\Gamma^*\Omega_{G\cdot\Lambda}$, il est nécessaire de calculer la différentielle de $\Gamma$.

Soit $(k\Lambda,Z)\in M$. Alors, pour tout $(X,A)\in\got{k}\oplus\got{p}$, on a
\begin{align}
d\Gamma(k\Lambda,Z)([k,\overline{X}],A) & = \frac{d}{dt}(\Gamma(k\exp(tX)\Lambda,Z))|_{t=0} + \frac{d}{dt}(\Gamma(k\Lambda,Z+tA))|_{t=0} \notag\\
& = \frac{d}{dt}(\exp(Z)k\exp(tX)\Lambda)|_{t=0} + \frac{d}{dt}(\exp(Z+tA)k\Lambda)|_{t=0} \notag\\
& = [\exp(Z)k,\overline{X}] + \frac{d}{dt}\left(\exp(Z+tA)k\Lambda\right)|_{t=0}. \label{eq:relation_différentielledeGamma_1}
\end{align}
Pour tout $Z\in\got{g}$, on définit l'opérateur linéaire $\Psi_Z:\got{g}\rightarrow\got{g}$\index{$\Psi_Z$} par
\[
\Psi_Z := \int_0^1 e^{-s\ad(Z)}ds.
\]
Remarquons que, pour tout $g\in G$, on a
\[
\Psi_{\Ad(g)Z}(\Ad(g)X) = \Ad(g)\Psi_Z(X),
\]
pour tous $Z,X\in\got{g}$. On peut également écrire
\begin{equation}
\label{eq:defi_equiv_de_PsiZ}
\Psi_Z = \int_0^1 e^{-s\ad(Z)}ds = \sum_{n=0}^{\infty} \frac{(-1)^n\ad(Z)^n}{(n+1)!}.
\end{equation}
Cet opérateur vérifie les propriétés bien connues suivantes.

\begin{prop}
\label{prop:differentielle_de_exp}
Pour tout $Z\in\got{g}$, on a
\begin{enumerate}
\item $e^{\ad(Z)}\circ\Psi_Z = \Psi_{-Z}$,
\item $d\exp(Z) = d l_{\exp(Z)}(e) \circ \Psi_Z = d r_{\exp(Z)}(e)\circ\Psi_{-Z}$,
\end{enumerate}
où $l_{\exp(Z)}$ (resp. $r_{\exp(Z)}$) désigne la multiplication à gauche (resp. à droite) par $\exp(Z)$ dans $G$.
\end{prop}

On pourra trouver une preuve de cette proposition dans \cite{duistermaat}. Du point $(2)$ de la Proposition \ref{prop:differentielle_de_exp}, on voit que
\begin{equation}
\label{eq:relation_Psi_Z_et_exponentielle}
\frac{d}{dt}\left(\exp(Z+tX)\right)|_{t=0} = \frac{d}{dt}\left(\exp(Z)\exp(t\Psi_{Z}(X))\right)|_{t=0},
\end{equation}
pour tout $Z,X\in\got{g}$.

\begin{prop}
\label{prop:différentielle_de_Gamma}
Pour tout $(k\Lambda,Z)\in \Klambda\times\got{p}$ et tout $(X,A)\in\got{k}\oplus\got{p}$, on a
\begin{equation}
\label{eq:relation_différentielledeGamma_2}
d\Gamma(k\Lambda,Z)([k,\overline{X}],A) = [\exp(Z)k,\overline{X + \Psi_{\Ad(k^{-1})Z}(\Ad(k^{-1})A)}]\in G\times_{K_{\Lambda}}\got{g}/\got{g}_{\Lambda}.
\end{equation}
\end{prop}

\begin{proof}
En combinant les équations \eqref{eq:relation_différentielledeGamma_1} et \eqref{eq:relation_Psi_Z_et_exponentielle}, on obtient
\begin{align*}
d\Gamma(k\Lambda,Z)([k,\overline{X}],A) & = [\exp(Z)k,\overline{X}] + \frac{d}{dt}\left(k\exp\bigl(\Ad(k^{-1})(Z+tA)\bigr)\Lambda\right)|_{t=0} \\
& = [\exp(Z)k,\overline{X}] + [\exp(Z)k,\overline{\Psi_{\Ad(k^{-1})Z}(\Ad(k^{-1})A)}].
\end{align*}
Ceci nous donne bien la relation \eqref{eq:relation_différentielledeGamma_2}.
\end{proof}

\subsection{Propriétés de l'application $\Psi_Z$}

Commençons par regarder l'application linéaire $\Psi_Z:\got{g}\rightarrow\got{g}$, pour $Z\in\got{p}$ fixé.

\begin{prop}
Soit $Z\in\got{p}$. L'application linéaire $\Psi_Z:\got{g}\rightarrow\got{g}$ est symétrique pour $B_{\theta}$ et ses valeurs propres sont toutes strictement positives. En particulier, $\Psi_Z$ est un automorphisme linéaire de $\got{g}$.
\end{prop}

\begin{proof}
Puisque $Z$ appartient à $\got{p}$, il est bien connu que $\ad(Z):\got{g}\rightarrow\got{g}$ est symétrique pour $B_{\theta}$, cf \cite[Lemme 6.27]{knapp}. De l'écriture en série convergente \ref{eq:defi_equiv_de_PsiZ}, il est clair que $\Psi_Z$ est aussi symétrique pour $B_{\theta}$. Etant donné que cette application est définie à partir de l'intégrale
\[
\Psi_Z =  \int_0^1 e^{-s\ad(Z)}ds,
\]
ses valeurs propres seront les réels $\frac{1-\exp(-\nu_i)}{\nu_i}$, pour $i=1,\ldots,r$, où $\nu_1,\ldots,\nu_r$ sont toutes les valeurs propres de $\ad(Z)$. On a pris pour convention $\frac{1-\exp(0)}{0} := 1$. On en déduit que les valeurs propres de $\Psi_Z$ sont toutes strictement positives et que $\Psi_Z$ est inversible.
\end{proof}

Il sera intéressant par la suite de décomposer l'application linéaire $\Psi_Z:\got{k}\oplus\got{p}\rightarrow\got{k}\oplus\got{p}$, en suivant les propriétés du crochet
\[
[\got{k},\got{k}]\subset\got{k}, \quad [\got{p},\got{p}]\subset\got{k} \quad \text{et} \quad [\got{k},\got{p}]\subset\got{p}.
\]

\begin{defi}
Pour tout $Z\in\got{p}$, on a une décomposition de $\Psi_Z$ en une partie paire $\Psi_Z^+$ et une partie impaire $\Psi_Z^-$, qui sont définies par
\[
\Psi_Z^+ = \sum_{n\geqslant 0}\frac{\ad(Z)^{2n}}{(2n+1)!}, \qquad \text{et} \qquad \Psi_Z^- = -\sum_{n\geqslant 0}\frac{\ad(Z)^{2n+1}}{(2n+2)!}.
\]\index{$\Psi_Z^+$, $\Psi_Z^-$}
\end{defi}

\begin{rema}
Nous justifions la terminologie \og{}paire\fg{} et \og{}impaire\fg{} du fait que, étant donné que $Z$ appartient à $\got{p}$, l'application linéaire $\Psi_Z^+$ envoie $\got{k}$ (resp. $\got{p}$) sur $\got{k}$ (resp. $\got{p}$), alors que $\Psi_Z^-$ envoie $\got{k}$ dans $\got{p}$, et réciproquement.
\end{rema}

Alors que l'application $\Psi_Z^+$ définit un isomorphisme de $\got{g}$ dans lui-même, l'étude de $\Psi_Z^-$ est plus compliquée, puisqu'il ne laisse pas stables les sous-espaces $\got{k}$ et $\got{p}$, mais les intervertit. De plus, on voit facilement que $\Psi_Z^-(Z) = 0$, car $\ad(Z)Z = [Z,Z] = 0$. Donc $\Psi_Z^-$ n'est pas inversible.

Dans la section \ref{section:cas_général}, nous ferons apparaître l'application linéaire
\[
\chi_Z := \Psi_Z^-\circ(\Psi_Z^+)^{-1} : \got{g}\rightarrow\got{g},
\]\index{$\chi_Z$}
pour $Z\in\got{p}$. Cette application a la propriété cruciale suivante.

\begin{prop}
\label{prop:chi_Z_symetrique+vp}
Pour tout $Z\in\got{p}$, l'application linéaire $\chi_Z:\got{g}\rightarrow\got{g}$ est symétrique pour le produit scalaire $B_{\theta}$ sur $\got{g}$ et ses valeurs propres sont toutes dans l'intervalle $]-1,1[$.
\end{prop}

\begin{proof}
Rappelons que $\ad(Z):\got{g}\rightarrow\got{g}$ est symétrique pour $B_{\theta}$, car $Z$ appartient à $\got{p}$. Puisque $\Psi_Z^+ = \sum_{k\geqslant 0} \frac{\ad(Z)^{2k}}{(2k+1)!}$, cet endomorphisme de $\got{g}$ est aussi symétrique. Il est également défini positif et, si les valeurs propres de $\ad(Z)$ sont les réels $\nu_1,\ldots,\nu_r$, alors les valeurs propres de $\Psi_Z^+$ sont
\[
\sum_{k\geqslant 0}\frac{\nu_i^{2k}}{(2k+1)!} = \frac{e^{\nu_i}-e^{-\nu_i}}{2\nu_i} > 0 \quad \text{pour $i=1,\ldots,r$}.
\]
Pour le cas de $\Psi_Z^-$, nous avons $\Psi_Z^- = \sum_{k\geqslant 0} \frac{\ad(Z)^{2k+1}}{(2k+2)!}$, qui est également symétrique pour $B_{\theta}$ et ses valeurs propres sont
\[
\sum_{k\geqslant 0}\frac{\nu_i^{2k+1}}{(2k+2)!} = \frac{e^{\nu_i}+e^{-\nu_i}-2}{2\nu_i} \quad \text{pour $i=1,\ldots,r$}.
\]
Il est clair que $\ad(Z)$ commute avec $\Psi_Z^+$, donc il commute avec $(\Psi_Z^+)^{-1}$ également. De la linéarité de $(\Psi_Z^+)^{-1}$ et de l'écriture de $\Psi_Z^-$ comme somme de puissances de $\ad(Z)$, on en déduit que $\Psi_Z^-$ et $(\Psi_Z^+)^{-1}$ commutent, donc sont simultanément diagonalisables. Par conséquent, $\chi_Z = \Psi_Z^-\circ(\Psi_Z^+)^{-1}$ est symétrique pour $B_{\theta}$ et ses valeurs propres sont les réels
\[
\frac{e^{\nu_i}+e^{-\nu_i}-2}{2\nu_i}\frac{2\nu_i}{e^{\nu_i}-e^{-\nu_i}} = \frac{e^{2\nu_i}-2e^{\nu_i}+1}{e^{2\nu_i}-1} = \frac{e^{\nu_i}-1}{e^{\nu_i}+1}  \quad \text{pour $i=1,\ldots,r$}.
\]
Et ces nombres sont évidemment dans $]-1,1[$, d'où le résultat.
\end{proof}

\begin{rema}
D'après la Proposition \ref{prop:chi_Z_symetrique+vp}, les valeurs propres de $\chi_Z$ sont toutes bornées indépendamment de $Z\in\got{p}$. Ce sera un des arguments-clés de la preuve du Théorème \ref{theo:symplecto_principal}.
\end{rema}

\subsection{Les formes symplectiques $\Omega_{G\cdot\Lambda}$ et $\Gamma^*\Omega_{G\cdot\Lambda}$}

D'après la Définition \ref{defi:formesymplectique_kirillovkostantsouriau}, et en utilisant les notations du paragraphe \ref{subsection:orbites_coadjointes}, la forme symplectique $\Omega_{G\cdot\Lambda}$ est définie par la formule
\[
(\Omega_{G\cdot\Lambda})|_{g\Lambda}\left([g,(\overline{X},A)],[g,(\overline{Y},B)]\right) = \langle\Lambda,[X,Y]\rangle + \langle\Lambda,[A,B]\rangle,
\]
pour tout $g\in G$ et tous $(X,A),(Y,B)\in\got{k}\oplus\got{p}$. En effet, dans l'algèbre de Lie $\got{g}$, la décomposition de Cartan $\got{k}\oplus\got{p}$ induit les inclusions $[\got{k},\got{k}]\subset\got{k}$, $[\got{p},\got{p}]\subset\got{k}$ et $[\got{k},\got{p}]\subset\got{p}$. Et comme $\Lambda$ appartient à $\got{k}^*$, on a nécessairement $\langle\Lambda,[X,B]\rangle = \langle\Lambda,[A,Y]\rangle = 0$ pour tous $(X,A),(Y,B)\in\got{k}\oplus\got{p}$.

Passons maintenant à la forme symplectique $\Gamma^*\Omega_{G\cdot\Lambda}$, qui est, elle, définie sur $\Klambda\times\got{p}$. Le calcul de la différentielle de $\Gamma$ ayant été effectué dans la Proposition \ref{prop:différentielle_de_Gamma}, nous obtenons la formule
\begin{multline*}
(\Gamma^*\Omega_{G\cdot\Lambda})|_{(k\Lambda,Z)}\left(([k,\overline{X}],A),([k,\overline{Y}],B)\right) \\
= \langle\Lambda,[X+\Ad(k^{-1})\Psi_Z(A), Y+\Ad(k^{-1})\Psi_Z(B)]\rangle
\end{multline*}
pour tout $(k\Lambda,Z)\in M$, tous $(X,A),(Y,B)\in\got{k}\oplus\got{p}$. Pour séparer partie dans $\got{k}$ et partie dans $\got{p}$, on peut également l'écrire sous la forme
\begin{multline}
\label{eq:formule_GammaOmegaGLambda}
(\Gamma^*\Omega_{G\cdot\Lambda})|_{(k\Lambda,Z)}\left(([k,\overline{X}],A),([k,\overline{Y}],B)\right) \\
= \left\langle\Lambda,[X+\Ad(k^{-1})\Psi_Z^-(A), Y+\Ad(k^{-1})\Psi_Z^-(B)]\right\rangle + \left\langle k\Lambda,[\Psi_Z^+(A),\Psi_Z^+(B)]\right\rangle 
\end{multline}
pour tout $(k\Lambda,Z)\in M$ et tous $(X,A),(Y,B)\in\got{k}\oplus\got{p}$.

\section{Cas de l'espace symétrique hermitien $G/K$}
\label{section:cas_G/K}

On sait qu'il existe un élément $z_0\in\got{z}(K)$ tel que $\ad(z_0)|_{\got{p}}$ définisse une structure complexe $K$-invariante sur $\got{p}$. Le produit scalaire $B_{\theta}$ sur $\got{g}$ permet d'identifier $z_0$ avec un élément $\Lambda_0\in\got{t}^*$ de la chambre holomorphe $\Chol$. En effet, rappelons que pour toute racine non compacte positive $\beta$, on a $\beta(z_0)=1$, d'après le paragraphe \ref{section:la_chambre_holomorphe}.

L'espace homogène $G/K$ s'identifie alors à l'orbite coadjointe $\Glambda_0$. Le difféomorphisme $\Gamma$ s'écrit ici
\[
\begin{array}{cccl}
\Gamma_0 : & \got{p} & \longrightarrow & \Glambda_0 \\
& Z & \longmapsto & \exp(Z)\Lambda_0.
\end{array}
\]

\begin{rema}
On donne ici une notation spécifique pour $\Gamma_0$, car, dans la section suivante, les difféomorphismes $\Gamma$ et $\Gamma_0$ interviendront tous les deux.
\end{rema}

La forme symplectique $\Gamma_0^*\Omega_{G\cdot\Lambda_0}$ devient alors
\[
(\Gamma^*\Omega_{G\cdot\Lambda_0})|_Z(A,B) = \langle \Lambda_0, [\Psi_Z^+(A),\Psi_Z^+(B)]\rangle,
\]
puisque, $z_0$ étant dans le centre de $\got{k}$, le stabilisateur $K_{\Lambda_0}$ est égal à $K$ tout entier. En particulier, le terme $\langle \Lambda_0, [\Psi_Z^-(A),\Psi_Z^-(B)]\rangle$ est forcément nul.

\begin{theo}[McDuff]
\label{theo:symplectomorphisme_casG/K}
La variété symplectique $(\Glambda_0,\Omega_{G\cdot\Lambda_0})$ est symplectomorphe à l'espace vectoriel symplectique $(\got{p},\Omega_{\got{p}})$, par un difféomorphisme $K$-équivariant qui envoie $0\in\got{p}$ sur $\Lambda_0\in \Glambda_0$.
\end{theo}

Nous proposons ici une démonstration très différente de celle donnée par McDuff dans \cite{mcduff}. La preuve ci-dessous n'utilise aucune propriété géométrique de la variété $\Glambda_0$, contrairement à \cite{mcduff} qui montre cette propriété plus généralement pour les variétés kählériennes à courbure négative possédant un pôle.

Un des premiers points à s'assurer est que des applications moments propres existent pour $\Gamma_0^*\Omega_{G\cdot\Lambda_0}$ et $\Omega_{\got{p}}$.

On sait qu'une application moment associée à $\Gamma_0^*\Omega_{G\cdot\Lambda_0}$ est l'application
\[
\begin{array}{cccl}
\Phi_{\Gamma_0^*\Omega_{G\cdot\Lambda_0}}:&\got{p}&\longrightarrow &\got{k}^* \\
& Z & \longmapsto & (X\in\got{k}\mapsto\langle \exp(Z)\Lambda_0,X\rangle\in\R).
\end{array}
\]
Il s'agit tout simplement de l'application composée $\Phi_{G\cdot\Lambda_0}\circ\Gamma_0$.

\begin{lemm}
\label{lemm:applicationmoment_Phi_Gamma0OmegaGlambda0_propre}
Pour tout $Z\in\got{p}$, on a
\[
\langle\Phi_{\Gamma_0^*\Omega_{G\cdot\Lambda_0}}(Z)-\Lambda_0,z_0\rangle \geqslant \frac{1}{2}\|Z\|^2.
\]
En particulier, l'application moment $\Phi_{\Gamma_0^*\Omega_{G\cdot\Lambda_0}} :\got{p}\rightarrow\got{k}^*$ est propre.
\end{lemm}

\begin{proof}
Remarquons que, pour tout $Z\in\got{p}$ et tout $X\in\got{k}$, on a
\begin{align*}
\langle \exp(Z)\Lambda_0,X\rangle & = \langle \Lambda_0,e^{-\ad(Z)}X\rangle \\
& = \Bigl\langle\Lambda_0,\sum_{k=0}^{\infty}\frac{\ad(-Z)^{2k}}{(2k)!}X\Bigr\rangle \\
& = B_{\theta}\left(z_0,\sum_{k=0}^{\infty}\frac{\ad(Z)^{2k}}{(2k)!}X\right).
\end{align*}
Or, $\ad(Z)$ est symétrique pour la forme bilinéaire symétrique $B_{\theta}$. On obtient donc
\begin{align*}
\langle\exp(Z)\Lambda_0-\Lambda_0,z_0\rangle & = \sum_{k=1}^{\infty}\frac{1}{(2k)!}B_{\theta}(\ad(Z)^k z_0,\ad(Z)^k z_0) \\
& \geqslant \frac{1}{2}B_{\theta}([z_0,Z],[z_0,Z]) = \frac{1}{2}\|Z\|^2.
\end{align*}
Par conséquent, l'application $\Phi_{\Gamma_0^*\Omega_{G\cdot\Lambda_0}}-\Lambda_0:\got{p}\rightarrow\got{k}^*$ est propre. En effet, en tant qu'élément du dual de $\got{k}$, la norme de $\Phi_{\Gamma_0^*\Omega_{G\cdot\Lambda_0}}(Z)-\Lambda_0$ est la norme d'opérateur linéaire, donc
\begin{align*}
\|\Phi_{\Gamma_0^*\Omega_{G\cdot\Lambda_0}}(Z)-\Lambda_0\| & = \sup_{\|X\|=1}\langle\Phi_{\Gamma_0^*\Omega_{G\cdot\Lambda_0}}(Z)-\Lambda_0,X\rangle \\
& \geqslant \frac{\langle\Phi_{\Gamma_0^*\Omega_{G\cdot\Lambda_0}}(Z)-\Lambda_0,z_0\rangle}{\|z_0\|} \\
& \geqslant \frac{1}{2\|z_0\|}\|Z\|^2,
\end{align*}
quel que soit $Z\in\got{p}$. Cette application est donc bien propre. On en déduit que l'application $\Phi_{\Gamma_0^*\Omega_{G\cdot\Lambda_0}}$ est propre elle aussi.
\end{proof}

La seconde forme symplectique, $\Omega_{\got{p}}$, est, quant à elle, une forme symplectique linéaire sur $\got{p}$. On peut donc facilement vérifier qu'une application moment pour $\Omega_{\got{p}}$ est
\[
\begin{array}{cccl}
\Phi_{\Omega_{\got{p}}}:&\got{p}&\longrightarrow&\got{k}^* \\
& Z & \longmapsto & (X\mapsto\langle\Lambda_0,[[X,Z],Z]\rangle).
\end{array}
\]

\begin{lemm}
\label{lemm:applicationmoment_Phi_Omega_gotp_propre}
On a $\langle\Phi_{\Omega_{\got{p}}}(Z),z_0\rangle = \|Z\|^2$, pour tout $Z\in\got{p}$. En particulier, l'application moment $\Phi_{\Omega_{\got{p}}} :\got{p}\rightarrow\got{k}^*$ est propre.
\end{lemm}

\begin{proof}
Si la première assertion est vraie, alors, par un raisonnement similaire à celui de la preuve du Lemme \ref{lemm:applicationmoment_Phi_Gamma0OmegaGlambda0_propre}, on en déduit que $\Phi_{\Omega_{\got{p}}}$ est propre.

Il suffit donc de prouver que $\langle\Phi_{\Omega_{\got{p}}}(Z),z_0\rangle = \|Z\|^2$, pour tout $Z\in\got{p}$. On utilise la définition de $\Phi_{\Omega_{\got{p}}}$, qui nous donne les égalités suivantes,
\[
\langle\Phi_{\Omega_{\got{p}}}(Z),z_0\rangle = B_{\theta}(z_0,[[z_0,Z],Z]) = B_{\theta}(-\ad(z_0)^2Z,Z).
\]
Or, $\ad(z_0)|_{\got{p}}^2 = -\id_{\got{p}}$. D'où l'égalité $\langle\Phi_{\Omega_{\got{p}}}(Z),z_0\rangle = \|Z\|^2$.
\end{proof}

On aura également besoin du lemme suivant, qui est un analogue du Lemme de Poincaré aux familles lisses de formes différentielles.

\begin{lemm}
\label{lemm:lemmePoincare_versionfamillelisse}
Soit $(\omega_t)_{t\in[0,1]}$ une famille lisse de $2$-formes fermées sur $\got{p}$. Alors, il existe une famille lisse $(\mu_t)_{t\in[0,1]}$ de $1$-formes différentielles, telle que $\omega_t = d\mu_t$ pour tout $t\in[0,1]$. Et si, pour $t\in[0,1]$, la $2$-forme $\omega_t$ est $K$-invariante, alors $\mu_t$ peut être choisie $K$-invariante.
\end{lemm}

\begin{proof}
Il suffit de recopier la preuve du Lemme de Poincaré pour les $2$-formes différentielles sur $\R^n$ (voir par exemple \cite[4.18]{warner}), en effectuant de très légères modifications de notations. Le paramètre $t$ ne change aucun calcul et on obtient bien une famille lisse à l'arrivée.

La $K$-invariance de $\mu_t$ provient de la linéarité de l'action de $K$ sur $\got{p}$. Elle se vérifie alors directement sur la définition du $\mu_t$ choisi dans la preuve du Lemme de Poincaré.
\end{proof}

\begin{proof}[Preuve du Théorème \ref{theo:symplectomorphisme_casG/K}]
Notons, pour tout $t\in[0,1]$, la $2$-forme différentielle $\Omega_t$ sur $\got{p}$ définie par
\[
\Omega_t|_Z := (\Gamma_0^*\Omega_{G\cdot\Lambda_0})|_{tZ} \qquad \text{pour tout $Z\in\got{p}$}.
\]
En particulier, $\Omega_1 = \Gamma_0^*\Omega_{G\cdot\Lambda_0}$. On remarque également que, pour $t=0$, on obtient la forme symplectique linéaire
\[
\Omega_0|_Z(A,B) = \langle\Lambda_0,[A,B]\rangle = B_{\theta}(z_0,[A,B]) = \Omega_{\got{p}}|_Z(A,B),
\]
pour tous $Z,A,B\in\got{p}$, puisque $\Psi_0^+ = \id_{\got{p}}$. Lorsque $t\neq 0$, on peut voir que
\[
\Omega_t = \frac{1}{t^2}\eta_t^*(\Gamma_0^*\Omega_{G\cdot\Lambda_0}) = \frac{1}{t^2}\eta_t^*\Omega_1,
\]
où $\eta_t:\got{p}\rightarrow\got{p}$ est l'homothétie $Z\mapsto tZ$, pour tout $t\in[0,1]$. L'action de $K$ sur $\got{p}$ étant linéaire, $\eta_t$ commute avec l'action. On en déduit que $\Omega_t$ est $K$-invariante. De plus, comme $\Omega_1$ est fermée, on a $d\Omega_t = \frac{1}{t^2}d(\eta_t^*\Omega_1)=\frac{1}{t^2}\eta_t^*d\Omega_1 = 0$. Et $\Omega_1$ étant symplectique, il est clair que la forme bilinéaire alternée $(\Omega_t)|_Z=(\Omega_1)|_{tZ}$ est non dégénérée. On en conclut que $\Omega_t$ est une forme symplectique pour tout $t\in[0,1]$.

La famille $(\Omega_t)_{t\in[0,1]}$ est une famille lisse de formes symplectiques. Ainsi, la famille $(\frac{d}{dt}\Omega_t)_{t\in[0,1]}$ est elle aussi lisse, avec les $2$-formes différentielles $\frac{d}{dt}\Omega_t$ qui sont fermées puisque la différentielle extérieure $d$ et l'opérateur différentiel $\frac{d}{dt}$ commutent. Il ne reste plus qu'à appliquer le Lemme \ref{lemm:lemmePoincare_versionfamillelisse} pour obtenir l'existence d'une famille lisse $(\mu_t)_{t\in[0,1]}$ de $1$-formes $K$-invariantes sur $\got{p}$ telle que, pour tout $t\in[0,1]$, on ait $\frac{d}{dt}\Omega_t = d\mu_t$. Ceci prouve la condition $(1)$ du Théorème \ref{theo:existence_isotopie_noncompact_casgeneral}.

Pour tout $t\in]0,1]$, on pose
\[
\Phi_t := \frac{1}{t^2}\eta_t^*\Phi_{\Gamma_0^*\Omega_{G\cdot\Lambda_0}} - \frac{1}{t^2}\Lambda_0 :\got{p}\rightarrow\got{k}^*.
\]
Puisque $\eta_t$ et $\Phi_{\Gamma_0^*\Omega_{G\cdot\Lambda_0}}$ sont $K$-équivariantes, il en est de même de $\Phi_t$. On peut alors facilement vérifier que $\Phi_t$ est une application moment pour la forme symplectique $\Omega_t$ sur $\got{p}$, car $\Lambda_0$ est laissé fixe par l'action de $K$. Pour $t=0$, on pose $\Phi_0 := \Phi_{\Omega_{\got{p}}}$, qui est une application moment pour $\Omega_0=\Omega_{\got{p}}$.

Remarquons que $\Phi_{\Gamma_0^*\Omega_{G\cdot\Lambda_0}}(0) = \exp(0)\Lambda_0 = \Lambda_0$ et que $\eta_t(0) = 0$. Par conséquent, on a $\Phi_t(0) = 0$ pour tout $t\in]0,1]$. De plus, on a également $\Phi_0(0) = \Phi_{\Omega_{\got{p}}}(0) = 0$. Donc l'ensemble $\{\Phi_t(0); t\in[0,1]\}$ est réduit à un point. La condition $(2)$ du Théorème \ref{theo:existence_isotopie_noncompact_casgeneral} est donc vérifiée, puisque ici, $M = K\cdot\Lambda_0 = \{\Lambda_0\}$ est un singleton, c'est-à-dire que $K\cdot\Lambda_0\times\got{p} \simeq \got{p}$ est vectoriel. Par ailleurs, ce fait prouve trivialement la condition $(3)$, car ici, on a $(T_0 \got{p})^{\Omega_t|_0} = \{0\} = T_0\{0\}$, cf Remarque \ref{rema:theo_moser_conditionsupplémentaire_toujours_vérifiée_pour_E=V}.

Il ne reste plus qu'à prouver la condition $(4)$ du Théorème \ref{theo:existence_isotopie_noncompact_casgeneral}. Par un premier calcul, nous obtenons
\[
\langle\Phi_t(Z),z_0\rangle = \frac{1}{t^2}\langle\Phi_{\Gamma_0^*\Omega_{G\cdot\Lambda_0}}(tZ)-\Lambda_0,z_0\rangle \geqslant \frac{1}{2t^2}\|tZ\|^2 = \frac{1}{2}\|Z\|^2,
\]
pour tout $Z\in\got{p}$ et tout $t\in]0,1]$, grâce au Lemme \ref{lemm:applicationmoment_Phi_Gamma0OmegaGlambda0_propre}. D'après le Lemme \ref{lemm:applicationmoment_Phi_Omega_gotp_propre}, nous avons également $\langle\Phi_0(Z),z_0\rangle\geqslant \frac{1}{2}\|Z\|^2$ pour tout $Z\in\got{p}$. Par un raisonnement analogue à la preuve du Lemme \ref{lemm:applicationmoment_Phi_Gamma0OmegaGlambda0_propre}, cela prouve que $\|\Phi_t(Z)\|\geqslant\frac{1}{2\|z_0\|}\|Z\|^2$ pour tout $Z\in\got{p}$ et tout $t\in[0,1]$. La condition $(4)$ est donc prouvée et on conclut en appliquant le Théorème \ref{theo:existence_isotopie_noncompact_casgeneral} et la Remarque \ref{rema:theo_moser_conditionsupplémentaire_toujours_vérifiée_pour_E=V}, donnant l'existence d'un symplectomorphisme entre  $(\got{p},\Omega_0)$ et $(\got{p},\Omega_1)$ qui laisse fixe le point $0$.
\end{proof}

\section{Cas général}
\label{section:cas_général}

Dans cette dernière section, nous allons prouver l'énoncé du Théorème \ref{theo:symplecto_principal}. On se place à nouveau dans le cas où $\Lambda$ est un élément arbitraire de la chambre holomorphe $\Chol$. Interviendront maintenant les deux difféomorphismes $K$-équivariants $\Gamma:\Klambda\times\got{p}\rightarrow \Glambda$ et $\Gamma_0:\got{p}\rightarrow \Glambda_0$. En effet, nous travaillerons ici exclusivement sur la variété $K\cdot\Lambda\times\got{p}$. Nous considérerons sur $K\cdot\Lambda\times\got{p}$ les formes symplectiques suivantes,
\begin{enumerate}
\item[(\emph{i})] $\Omega_{K\cdot\Lambda\times\got{p}} = \Omega_{K\cdot\Lambda}\oplus\Omega_{\got{p}}$;
\item[(\emph{ii})] $\Omega^1:= \Omega_{K\cdot\Lambda}\oplus\Gamma_0^*\Omega_{G\cdot\Lambda_0}$;
\item[(\emph{iii})] $\Omega^{\delta}:= \Omega_{K\cdot\Lambda}\oplus (\delta\Gamma_0^*\Omega_{G\cdot\Lambda_0})$, pour $\delta>0$;
\item[(\emph{iv})] $\Gamma^*\Omega_{G\cdot\Lambda}$.
\end{enumerate}
Rappelons que la \og{}somme directe\fg{} de deux formes symplectiques est la forme symplectique canonique définie sur le produit direct des deux variétés symplectiques sous-jacentes.

Nous allons montrer dans cette section que les formes symplectiques $\Omega_{K\cdot\Lambda\times\got{p}}$ et $\Gamma^*\Omega_{G\cdot\Lambda}$ sont symplectomorphes. Pour ce faire, nous utiliserons à plusieurs reprises l'argument de Moser donné en section \ref{section:moser} pour prouver les symplectomorphismes indiqués dans le diagramme suivant,
\begin{diagram}
\Omega_{K\cdot\Lambda\times\got{p}} & \rTo^{\text{Théorème \ref{theo:symplectomorphisme_casG/K} }} & \Omega^1 & \rTo^{\text{paragraphe \ref{subsection:symplecto_entre_Omega^1_et_Omega^delta} }} & \Omega^{\delta} & \rTo^{\text{paragraphe \ref{subsection:dernier_symplecto} }} \Gamma^*\Omega_{G\cdot\Lambda}
\end{diagram}
Le premier symplectomorphisme découle directement du Théorème \ref{theo:symplectomorphisme_casG/K}. Les deux autres symplectomorphismes seront montrés dans la suite de cette section. De plus, on peut rajouter la condition que chacun des symplectomorphismes obtenus pour les trois flèches du diagramme ci-dessus laisse fixe les points de $K\cdot\Lambda\times\{0\}$. La composition de tels symplectomorphismes donne un symplectomorphisme de $(K\cdot\Lambda\times\got{p},\Omega_{K\cdot\Lambda\times\got{p}})$ sur $(K\cdot\Lambda\times\got{p},\Gamma^*\Omega_{G\cdot\Lambda})$ qui satisfait l'énoncé du Théorème \ref{theo:symplecto_principal}.

\subsection{Symplectomorphisme entre $\Omega^1$ et $\Omega^{\delta}$ sur $K\cdot\Lambda\times\got{p}$}
\label{subsection:symplecto_entre_Omega^1_et_Omega^delta}

Pour tout $\delta>0$, on définit la $2$-forme différentielle
\begin{equation}
\label{eq:formule_Omegadelta}
\Omega^{\delta} = \Omega_{K\cdot\Lambda}\oplus(\delta\Gamma_0^*\Omega_{G\cdot\Lambda_0}).
\end{equation}

\begin{lemm}
\label{lemm:Omegadelta_symplectique}
Pour tout $\delta>0$, la $2$-forme $\Omega^{\delta}$ est symplectique $K$-invariante sur $\Klambda\times\got{p}$ et l'application
\[
\Phi^{\delta} = \pi_{K\cdot\Lambda}^*\Phi_{K\cdot\Lambda} + \delta\pi_{\got{p}}^*(\Gamma_0^*\Phi_{G\cdot\Lambda_0})
\]
est une application moment pour $\Omega^{\delta}$, où $\pi_{K\cdot\Lambda}:\Klambda\times\got{p}\rightarrow\Klambda$ et $\pi_{\got{p}}:\Klambda\times\got{p}\rightarrow\got{p}$ sont les deux projections canoniques. De plus, on a
\begin{equation}
\label{eq:minoration_norme_de_Phidelta}
\|\Phi^{\delta}(k\Lambda,Z)\| \geqslant \frac{\delta}{2\|z_0\|}\|Z\|^2 \quad \text{pour tout $(k\Lambda,Z)\in\Klambda\times\got{p}$}.
\end{equation}
En particulier, l'application moment $\Phi^{\delta}$ est propre.
\end{lemm}

\begin{proof}
De manière plus générale, on peut regarder deux variétés symplectiques $(N_1,\omega_1)$ et $(N_2,\omega_2)$, et considérer la $2$-forme différentielle
\[
\omega^{a_1,a_2} = (a_1\omega_1)\oplus(a_2\omega_2) := \pi_1^*(a_1\omega_1) + \pi_2^*(a_2\omega_2),
\]
sur la variété produit $N_1\times N_2$, où $a_1,a_2$ sont deux réels non nuls et $\pi_i:N_1\times N_2\rightarrow N_i$ est la projection canonique, pour $i=1,2$. La $2$-forme $\omega^{a_1,a_2}$ est fermée, puisqu'on a
\[
d\omega^{a_1,a_2} = a_1d(\pi_1^*\omega_1) + a_2d(\pi_2^*\omega_2) = a_1\pi_1^*(d\omega_1) + a_2\pi_2^*(d\omega_2) = 0,
\]
car $\omega_1$ et $\omega_2$ sont fermées. Et comme $\omega_1$ et $\omega_2$ sont non dégénérées, il est facile de voir que la définition de $\omega^{a_1,a_2}$ et le fait que $a_1$ et $a_2$ sont non nuls impliquent que $\omega^{a_1,a_2}$ est aussi symplectique.

Si, de plus, $N_1$ et $N_2$ sont munies d'une action du groupe de Lie $K$, telle que $\omega_1$ et $\omega_2$ soient $K$-invariantes, alors $\omega^{a_1,a_2}$ est aussi $K$-invariante, car les projections $\pi_i$ sont équivariantes. Si $\phi_1$ (resp. $\phi_2$) est une application moment pour $\omega_1$ (resp. $\omega_2$), alors un calcul direct montre que $\phi^{a_1,a_2}:=a_1\pi_1^*\phi_1+a_2\pi_2^*\phi_2$ est une application moment pour $\omega^{a_1,a_2}$.

On en déduit directement le résultat pour $\Omega^{\delta}$. Il reste tout de même à montrer que $\Phi^{\delta}$ est propre. Soit $(k\Lambda,Z)\in\Klambda\times\got{p}$. En appliquant le Lemme \ref{lemm:applicationmoment_Phi_Gamma0OmegaGlambda0_propre}, on obtient l'inégalité
\[
\langle\Phi^{\delta}(k\Lambda,Z),z_0\rangle = \langle k\Lambda,z_0\rangle + \delta\langle\exp(Z)\Lambda_0,z_0\rangle \geqslant \langle \Lambda,z_0\rangle + \frac{\delta}{2}\|Z\|^2 \geqslant \frac{\delta}{2}\|Z\|^2.
\]
La positivité de $\langle \Lambda,z_0\rangle$ provient du fait que $\Lambda$ et $\Lambda_0$ sont deux éléments de $\Chol$. Par un argument analogue à la preuve du Lemme \ref{lemm:applicationmoment_Phi_Gamma0OmegaGlambda0_propre}, ceci prouve l'équation \eqref{eq:minoration_norme_de_Phidelta} et l'application moment $\Phi^{\delta}$ est propre.
\end{proof}

\begin{prop}
\label{prop:symplecto_Omegadelta_et_Omega1}
Pour tout réel $\delta>0$, les deux formes symplectiques $\Omega^{\delta}$ et $\Omega^1 = \Omega_{K\cdot\Lambda}\oplus\Gamma_0^*\Omega_{G\cdot\Lambda_0}$ sont symplectomorphes sur $\Klambda\times\got{p}$, pour un symplectomorphisme $K$-équivariant qui envoie $(k\Lambda,0)$ sur lui-même pour tout $k\in K$.
\end{prop}

\begin{proof}
Fixons $\delta$ strictement positif. 
Posons $\omega^{\delta}_t = t\Omega^1 + (1-t)\Omega^{\delta}$, pour tout $t\in[0,1]$. Il est clair que
\[
\omega_t = \Omega_{K\cdot\Lambda}\oplus\bigl((t+(1-t)\delta)\Gamma_0^*\Omega_{G\cdot\Lambda_0}\bigr) = \Omega^{t+(1-t)\delta} \quad \text{pour tout $t\in[0,1]$}.
\]
Puisque $\delta$ est strictement positif, on a $t+(1-t)\delta>0$, et par conséquent, $\omega_t$ est une forme symplectique pour tout $t\in[0,1]$, d'après le Lemme \ref{lemm:Omegadelta_symplectique}. De plus, l'écriture de $\omega_t$ sous forme d'une somme directe de formes symplectiques
\[
\omega_t = \Omega_{K\cdot\Lambda}\oplus\bigl((t+(1-t)\delta)\Gamma_0^*\Omega_{G\cdot\Lambda_0}\bigr)
\]
montre que $\omega_t$ vérifie la condition $(3)$ du Corollaire \ref{coro:existence_isotopie_noncompact_parsegment_avec_condition_lemmedepoincaré}.

Pour tout $t\in[0,1]$, la $2$-forme différentielle $\frac{d}{dt}\omega_t$ est égale à
\[
\frac{d}{dt}\omega_t = \Omega^1-\Omega^{\delta} = (1-\delta)\pi_{\got{p}}^*(\Gamma_0^*\Omega_{G\cdot\Lambda_0}).
\]
Elle est fermée, $K$-invariante et est clairement dans le noyau de $i^*$, où $i:\Klambda\hookrightarrow\Klambda\times\got{p}$ est l'inclusion canonique. Il s'agit de la condition ($2$) du Corollaire \ref{coro:existence_isotopie_noncompact_parsegment_avec_condition_lemmedepoincaré}.

D'après le Lemme \ref{lemm:Omegadelta_symplectique}, $\omega_t$ a pour application moment l'application $\Phi^{t+(1-t)\delta} = \pi_{K\cdot\Lambda}^*\Phi_{K\cdot\Lambda} + (t+(1-t)\delta)\pi_{\got{p}}^*(\Gamma_0^*\Phi_{G\cdot\Lambda_0})$ et on a
\[
\|\Phi^{t+(1-t)\delta}(k\Lambda,Z)\| \geqslant \frac{(t+(1-t)\delta)}{2\|z_0\|}\|Z\|^2 \geqslant \frac{\min\{1,\delta\}}{2\|z_0\|}\|Z\|^2.
\]
On peut donc appliquer le Corollaire \ref{coro:existence_isotopie_noncompact_parsegment_avec_condition_lemmedepoincaré}.
D'où le résultat annoncé.
\end{proof}

\subsection{Symplectomorphisme entre $\Omega^{\delta}$ et $\Gamma^*\Omega_{G\cdot\Lambda}$ sur $\Klambda\times\got{p}$}
\label{subsection:dernier_symplecto}

Ce dernier paragraphe est consacré à la preuve du théorème suivant, qui est un des arguments-clés de la preuve du Théorème \ref{theo:symplecto_principal}.

\begin{theo}
\label{theo:symplecto_Gamma*OmegaGlamnbda_et_Omegadelta}
Pour tout $\delta> b_{\Lambda}:=\sup_{\|u\|=1,\|v\|=1}\langle\Lambda,[u,v]\rangle$, il existe un symplectomorphisme $K$-équivariant entre les deux variétés symplectiques $(\Klambda\times\got{p},\Gamma^*\Omega_{G\cdot\Lambda})$ et $(\Klambda\times\got{p},\Omega^{\delta})$ qui laisse fixe chaque $(k\Lambda,0)$, pour $k$ parcourant $K$.
\end{theo}

L'idée est d'utiliser une fois encore le Corollaire \ref{coro:existence_isotopie_noncompact_parsegment}. Le point le plus délicat est de prouver que toutes les $2$-formes du segment reliant les deux formes symplectiques $\Omega^{\delta}$ et $\Gamma^*\Omega_{G\cdot\Lambda}$ sont elles aussi symplectiques. L'intervention du $\delta$ est ici primordiale.

\begin{theo}
\label{theo:segment_bien_symplectique}
Si $\delta>b_{\Lambda}:=\sup_{\|u\|=1,\|v\|=1}\langle\Lambda,[u,v]\rangle$, alors, pour tout $t\in[0,1]$, la $2$-forme $\Omega_t^{\delta} := t\Omega^{\delta}+(1-t)\Gamma^*\Omega_{G\cdot\Lambda}$ est symplectique.
\end{theo}

Ce résultat sera prouvé dans le paragraphe \ref{subsection:preuve_du_theo_segment_bien_symplectique}.

Nous aurons besoin du lemme suivant, qui sera utile aux démonstrations des Théorèmes \ref{theo:symplecto_Gamma*OmegaGlamnbda_et_Omegadelta} et \ref{theo:segment_bien_symplectique}.

Commençons par donner quelques notations. Pour toute racine $\alpha\in\got{R}$, le sous-espace de racine $\got{g}_{\alpha}\subset\got{g}_{\C}$ est défini ici par
\[
\{X\in\got{g}_{\C}; [H,X] = i\alpha(H) X, \forall H\in\got{t}\}.
\]
Pour tout élément $\lambda\in\got{t}^*$, on notera $H_{\lambda}$ l'unique élément de $\got{t}$ qui vérifie
\[
B_{\theta}(H_{\lambda},X) = \lambda(X) \quad \forall X\in\got{g}.
\]
Pour toute racine non compacte positive $\beta$, on peut fixer deux vecteurs non nuls $E_{\beta}\in\got{g}_{\beta}$ et $E_{-\beta}\in\got{g}_{-\beta}$ tels que $E_{-\beta} = \overline{E_{\beta}}$. Ainsi, $E_{\beta}+E_{-\beta}$ et $i(E_{\beta}-E_{-\beta})$ sont deux vecteurs de $\got{p}$. De plus, la famille
$\bigl(E_{\beta}+E_{-\beta}, i(E_{\beta}-E_{-\beta})\bigr)_{\beta\in\got{R}_n^+}$ définit une base réelle de $\got{p}$, et il est bien connu que cette base de $\got{p}$ est orthogonale pour $B_{\theta}$, cf \cite{knapp, bordemann}.


\begin{lemm}
\label{lemm:positivité_puissances_ad(Z)_pourHlambda}
Soient $\Lambda,\Lambda'\in\Chol$
. Alors, pour tout $Z\in\got{p}$, on a
\begin{equation}
\label{eq:minoration_Btheta(HadZ2H')}
B_{\theta}(H_{\Lambda},\ad(Z)^2 H_{\Lambda'}) \geqslant \left(\min_{\alpha\in\got{R}_n^+}\alpha(H_{\Lambda})\alpha(H_{\Lambda'})\right)\|Z\|^2.
\end{equation}
En particulier, si on pose $m_{\Lambda}=\min_{\alpha\in\got{R}_n^+}\alpha(H_{\Lambda})$, alors
\begin{equation}
\label{eq:minoration_Btheta(z0adZ2H)}
B_{\theta}(z_0,\ad(Z)^2 H_{\Lambda}) \geqslant m_{\Lambda}\|Z\|^2,
\end{equation}
et
\begin{equation}
\label{eq:minoration_Btheta(HadZ2H)}
B_{\theta}(H_{\Lambda},\ad(Z)^2 H_{\Lambda}) \geqslant m_{\Lambda}^2\|Z\|^2,
\end{equation}
pour tout $Z\in\got{p}$.
\end{lemm}

\begin{proof}
Soit $Z\in\got{p}$. Il s'écrit
\[
Z = \sum_{\beta\in\got{R}_n^+}\left(x_{\beta}^+(E_{\beta}+E_{-\beta})+x_{\beta}^-i(E_{\beta}-E_{-\beta})\right),
\]
avec $x_{\beta}^{\pm}\in\R$ pour tout $\beta\in\got{R}_n^+$.

Soit $H\in\got{t}$. Puisque $E_{\pm\beta}$ appartient à $\got{g}_{\pm\beta}$, on a $[H,E_{\pm\beta}] = \pm i\beta(H)E_{\pm\beta}$. On en déduit les deux égalités
\[
[H,E_{\beta} + E_{-\beta}] = \beta(H)\bigl(i(E_{\beta}-E_{-\beta})\bigr),
\]
et
\[
[H,i(E_{\beta}-E_{-\beta})] = -\beta(H)(E_{\beta} + E_{-\beta}).
\]
Par conséquent
\[
[H,Z] = \sum_{\beta\in\got{R}_n^+}\beta(H)\left(-x_{\beta}^-(E_{\beta}+E_{-\beta})+x_{\beta}^+i(E_{\beta}-E_{-\beta})\right),
\]
pour tout $H\in\got{t}$.

Prenons maintenant $\Lambda,\Lambda'\in\Chol$ et notons $H_{\Lambda}, H_{\Lambda'}\in\got{t}$ les éléments duaux respectifs obtenus grâce au produit scalaire $B_{\theta}$. Ces deux éléments de $\got{t}$ vérifient $\beta(H_{\Lambda})>0$ et $\beta(H_{\Lambda'})>0$ pour tout $\beta\in\got{R}_n^+$.
\begin{align*}
B_{\theta}(H_{\Lambda},\ad(Z)^2 H_{\Lambda'}) & = B_{\theta}([H_{\Lambda},Z],[H_{\Lambda'},Z]) \\
& = \sum_{\beta\in\got{R}_n^+}\beta(H_{\Lambda})\beta(H_{\Lambda'})\left((x_{\beta}^-)^2B_{\theta}(E_{\beta}+E_{-\beta},E_{\beta}+E_{-\beta})\right. \\
& \qquad \qquad \qquad \left.+ (x_{\beta}^+)^2B_{\theta}\bigl(i(E_{\beta}-E_{-\beta}),i(E_{\beta}-E_{-\beta})\bigr)\right).
\end{align*}
Puisque $\beta(H_{\Lambda})\beta(H_{\Lambda'})$ est positif pour tout $\beta\in\got{R}_n^+$, on obtient la minoration
\begin{align*}
B_{\theta}(H_{\Lambda},\ad(Z)^2 H_{\Lambda'}) & \geqslant \left(\min_{\beta\in\got{R}_n^+}\beta(H_{\Lambda})\beta(H_{\Lambda'})\right)\sum_{\beta\in\got{R}_n^+}\left((x_{\beta}^-)^2B_{\theta}(E_{\beta}+E_{-\beta},E_{\beta}+E_{-\beta})\right. \\
& \qquad \qquad \qquad \left.+ (x_{\beta}^+)^2B_{\theta}\bigl(i(E_{\beta}-E_{-\beta}),i(E_{\beta}-E_{-\beta})\bigr)\right) \\
& \geqslant \left(\min_{\beta\in\got{R}_n^+}\beta(H_{\Lambda})\beta(H_{\Lambda'})\right)B_{\theta}(Z,Z),
\end{align*}
puisque $B_{\theta}(E_{\beta}+E_{-\beta},E_{\beta}+E_{-\beta}) =  B_{\theta}\bigl(i(E_{\beta}-E_{-\beta}),i(E_{\beta}-E_{-\beta})\bigr)$ pour tout $\beta\in\got{R}_n^+$. Ceci prouve l'inégalité \eqref{eq:minoration_Btheta(HadZ2H')}.

Pour prouver l'inégalité \eqref{eq:minoration_Btheta(z0adZ2H)}, il suffit de prendre $\Lambda' = \Lambda_0$, pour avoir $H_{\Lambda'} = H_{\Lambda_0} = z_0$. Or, $\beta(z_0) = 1$ pour toute racine $\beta\in\got{R}_n^+$, donc $\min_{\beta\in\got{R}_n^+}\beta(H_{\Lambda})\beta(z_0) = m_{\Lambda}$. En remplaçant ce minimum dans \eqref{eq:minoration_Btheta(HadZ2H')}, on obtient bien l'inégalité \eqref{eq:minoration_Btheta(z0adZ2H)}.

Quant à l'inégalité \eqref{eq:minoration_Btheta(HadZ2H)}, il suffit de remarquer que $\min_{\beta\in\got{R}_n^+}\left(\beta(H_{\Lambda})^2\right) = m_{\Lambda}^2$, puisque tous les $\beta(H_{\Lambda})$ sont positifs et la fonction carrée est croissante sur l'intervalle $[0,+\infty[$.
\end{proof}

\begin{proof}[Preuve du Théorème \ref{theo:symplecto_Gamma*OmegaGlamnbda_et_Omegadelta}]
Pour pouvoir appliquer le Corollaire \ref{coro:existence_isotopie_noncompact_parsegment_avec_condition_lemmedepoincaré}, nous devons vérifier ses conditions $(1)$ à $(4)$. Tout d'abord, d'après le Théorème \ref{theo:segment_bien_symplectique}, l'hypothèse $\delta> b_{\Lambda}$ implique que la famille $(\Omega_t^{\delta})_{t\in[0,1]}$ est formée de formes symplectiques sur $\Klambda\times\got{p}$. Donc l'assertion $(1)$ est vérifiée.

Ensuite, on peut facilement voir, en regardant les formules \eqref{eq:formule_GammaOmegaGLambda} et \eqref{eq:formule_Omegadelta}, que la $2$-forme différentielle $\Gamma^*\Omega_{G\cdot\Lambda} - \Omega^{\delta}$ est dans le noyau de $i^*$, où $i:\Klambda\rightarrow \Klambda\times\got{p}$ est l'inclusion canonique. D'où l'hypothèse ($2$) est vérifiée.

Considérons maintenant $k\in K$ et regardons l'expression de $\Omega_t^{\delta}|_{(k\Lambda,0)}$. \'Etant donné que $\Psi_0(A) = \Psi_0^+(A) = A$ pour tout $A\in\got{p}$, on a
\[
\Omega_t^{\delta}|_{(k\Lambda,0)}\bigl(([k,\overline{X}],A),([k,\overline{Y}],B)\bigr) = \langle\Lambda,[X,Y]\rangle + \langle t\delta\Lambda_0+(1-t)\Lambda,[A,B]\rangle,
\]
pour tous $X,Y\in\got{k}$ et tous $A,B\in\got{p}$. On en déduit facilement l'assertion $(3)$ pour $\Omega_t^{\delta}$.

Le dernier point à montrer est que les applications moments $\phi_t:\Klambda\times\got{p}\rightarrow\got{k}^*$, pour tout $t\in[0,1]$, sont uniformément propres.

Une application moment pour la variété hamiltonienne $(\Klambda\times\got{p},\Gamma^*\Omega_{G\cdot\Lambda})$ est l'application définie pour tout $(k\Lambda,Z)\in \Klambda\times\got{p}$ par
\[
\Phi_{\Gamma^*\Omega_{G\cdot\Lambda}}(k\Lambda,Z) := (e^Zk\Lambda)|_{\got{k}} = k\Lambda\circ\left(\sum_{n\geqslant 0}\frac{\ad(Z)^{2n}}{(2n)!}\right).
\]
D'après le Lemme \ref{lemm:Omegadelta_symplectique}, l'application moment $\Phi^{\delta}$ est propre et sa formule générale est donnée pour tout $(k\Lambda,Z)\in\Klambda\times\got{p}$ par
\[
\Phi^{\delta}(k\Lambda,Z) := k\Lambda + \delta \Lambda_0\circ\left(\sum_{n\geqslant 0}\frac{\ad(Z)^{2n}}{(2n)!}\right).
\]
On obtient alors une application moment pour $\Omega_t^{\delta}$ en prenant l'application $\phi_t^{\delta} = t\Phi_{\Gamma^*\Omega_{G\cdot\Lambda}} + (1-t)\Phi^{\delta}$.

Montrons que la famille d'applications moments $(\phi_t^{\delta})_{t\in[0,1]}$ est uniformément propre. Pour ce faire, posons
\[
\Lambda_t = t\Lambda + (1-t)\delta \Lambda_0
\]
pour tout $t\in[0,1]$. L'élément associé dans $\got{t}$ est le vecteur $H_{\Lambda_t} = tH_{\Lambda} + (1-t)\delta z_0$. La forme linéaire $\Lambda_t$ est elle aussi dans $\Chol$ puisque cette chambre est un cône convexe.

Nous appliquons le vecteur $\Ad(k)H_{\Lambda_t}$ à la forme linéaire $\phi_t^{\delta}(k\Lambda,Z)\in\got{k}^*$ et nous obtenons
\begin{align*}
\langle\phi_t^{\delta}(k\Lambda,Z),\Ad(k)H_{\Lambda_t}\rangle & = \left\langle t\Phi_{\Gamma^*\Omega_{G\cdot\Lambda}}(k\Lambda,Z)+(1-t)\Phi^{\delta}(k\Lambda,Z), \Ad(k)H_{\Lambda_t}\right\rangle\ , \\
& = t B_{\theta}\bigl(\Ad(k)H_{\Lambda},\sum_{n\geqslant 0}\frac{\ad(Z)^{2n}}{(2n)!}\Ad(k)H_{\Lambda_t}\bigr) \\
& \quad + (1-t)\delta B_{\theta}\bigl(z_0,\sum_{n\geqslant 0}\frac{\ad(Z)^{2n}}{(2n)!}\Ad(k)H_{\Lambda_t}\bigr) \\
& \quad + (1-t)\langle k\Lambda, \Ad(k)H_{\Lambda_t}\rangle\ , \\
& = B_{\theta}\bigl(\Ad(k)H_{\Lambda_t},\sum_{n\geqslant 0}\frac{\ad(Z)^{2n}}{(2n)!}\Ad(k)H_{\Lambda_t}\bigr) + (1-t)\langle\Lambda, H_{\Lambda_t}\rangle\ .
\end{align*}
Puisque $\Lambda$ et $\Lambda_t$ sont deux éléments de $\Chol$, le crochet de dualité $\langle\Lambda,H_{\Lambda_t}\rangle$ est positif, car $\langle\Lambda,H_{\Lambda_t}\rangle = 2\sum_{\alpha\in\got{R}^+}\alpha(H_{\Lambda})\alpha(H_{\Lambda_t}) \geqslant 0$. Donc, grâce au fait que $\ad(Z)$ est symétrique pour $B_{\theta}$, on obtient les inégalités suivantes,
\begin{align*}
\langle\phi_t^{\delta}(k\Lambda,Z),\Ad(k)H_{\Lambda_t}\rangle & \geqslant \sum_{n\geqslant 0}\frac{1}{(2n)!}B_{\theta}\bigl(\ad(Z)^{n}\Ad(k)H_{\Lambda_t},\ad(Z)^{n}\Ad(k)H_{\Lambda_t}\bigr) \\
& \geqslant \sum_{n\geqslant 0}\frac{1}{(2n)!}B_{\theta}\bigl(\ad(\Ad(k^{-1})Z)^{n}H_{\Lambda_t},\ad(\Ad(k^{-1})Z)^{n}H_{\Lambda_t}\bigr) \\
& \geqslant \frac{1}{2}B_{\theta}\bigl(\ad(\Ad(k^{-1})Z)H_{\Lambda_t},\ad(\Ad(k^{-1})Z)H_{\Lambda_t}\bigr).
\end{align*}
On applique l'inégalité \eqref{eq:minoration_Btheta(HadZ2H)} du Lemme \ref{lemm:positivité_puissances_ad(Z)_pourHlambda}, ce qui nous donne
\[
\langle\phi_t^{\delta}(k\Lambda,Z),\Ad(k)H_{\Lambda_t}\rangle \geqslant \frac{m_{\Lambda_t}^2}{2}\|\Ad(k^{-1})Z\|^2 = \frac{m_{\Lambda_t}^2}{2}\|Z\|^2,
\]
puisque la norme $\|\cdot\|$ est $K$-invariante. Tout ceci nous permet maintenant de majorer la norme de $\phi_t^{\delta}(k\Lambda,Z)$,
\[
\|\phi_t^{\delta}(k\Lambda,Z)\| \geqslant \frac{\langle\phi_t^{\delta}(k\Lambda,Z),\Ad(k)H_{\Lambda_t}\rangle}{\|\Ad(k)H_{\Lambda_t}\|} \geqslant \frac{m_{\Lambda_t}^2}{2\|H_{\Lambda_t}\|}\|Z\|^2 \geqslant \left(\inf_{t\in[0,1]}\frac{m_{\Lambda_t}^2}{2\|H_{\Lambda_t}\|}\right)\|Z\|^2.
\]
La constante $\inf_{t\in[0,1]}\frac{m_{\Lambda_t}^2}{2\|H_{\Lambda_t}\|}$ est strictement positive. En effet, par continuité sur le compact $[0,1]$, il est clair que $\|H_{\Lambda_t}\|$ est majoré et ne s'annule jamais. De plus, pour toute racine non compacte positive $\beta$, la valeur de $\beta(H_{\Lambda_t})^2$ est minorée par une constante strictement positive. Sinon,  cela impliquerait qu'elle s'annulerait pour un certain $t$, toujour par propriété de continuité sur le compact $[0,1]$, ce qui est impossible car $\Lambda_t\in\Chol$ pour tout $t\in[0,1]$. Par conséquent, $m_{\Lambda_t}^2\geqslant \min\left\{\inf_{t\in[0,1]}\beta(H_{\Lambda_t})^2; \beta\in\got{R}_n^+\right\}$ est aussi minoré par un réel strictement positif.

Ceci prouve la condition $(4)$ du Corollaire \ref{coro:existence_isotopie_noncompact_parsegment_avec_condition_lemmedepoincaré}. On peut alors appliquer le Corollaire \ref{coro:existence_isotopie_noncompact_parsegment_avec_condition_lemmedepoincaré}, prouvant l'existence d'un symplectomorphisme $K$-équivariant entre les deux variétés symplectiques $(\Klambda\times\got{p},\Gamma^*\Omega_{G\cdot\Lambda})$ et $(\Klambda\times\got{p},\Omega^{\delta})$ qui fixe point par point la sous-variété $K\cdot\Lambda\times\{0\}$.
\end{proof}

\subsection{Preuve du Théorème \ref{theo:segment_bien_symplectique}}
\label{subsection:preuve_du_theo_segment_bien_symplectique}


On définit les $2$-formes différentielles sur $\Klambda\times\got{p}$ suivantes,
\[
\omega_0^{\delta}|_{(k\Lambda,Z)}\bigl(([k,X],A),([k,Y],B)\bigr) = \langle\Lambda,[X,Y]\rangle + \delta B_{\theta}(z_0,[A,B]),
\]
pour tout $\delta >0$, et
\begin{align*}
\omega_1|_{(k\Lambda,Z)}\bigl(([k,X],A),([k,Y],B)\bigr) & = \langle\Lambda,[X+\Ad(k^{-1})\chi_Z(A),Y+\Ad(k^{-1})\chi_Z(B)]\rangle \\
& \quad + \langle\Lambda,[\Ad(k^{-1})A,\Ad(k^{-1})B]\rangle,
\end{align*}
pour tout $(k\Lambda,Z)\in \Klambda\times\got{p}$ et tous vecteurs $X\oplus A, Y\oplus B \in \got{k}/\got{k}_{\Lambda}\oplus\got{p}$.

\begin{lemm}
\label{lemm:condnécessaire_si_formeslinéaires_opposées}
Soit $\delta>0$. Supposons qu'il existe $(k\Lambda,Z)\in\got{p}$, $X\oplus A\in\got{k}/\got{k}_{\Lambda}\oplus\got{p}$ non nul et $c>0$ tels que
\begin{equation}
\label{eq:formeslinéaires_opposées}
\omega_1|_{(k\Lambda,Z)}\bigl(([k,X],A),([k,Y],B)\bigr) = -c\,\omega_0^{\delta}|_{(k\Lambda,Z)}\bigl(([k,X],A),([k,Y],B)\bigr),
\end{equation}
pour tout $Y\oplus B\in\got{k}/\got{k}_{\Lambda}\oplus\got{p}$. Alors $\delta \leqslant b_{\Lambda}$.
\end{lemm}

\begin{proof}
Puisque $\omega_0^{\delta}$ et $\omega_1$ sont $K$-équivariantes, quitte à prendre de nouvelles notations sur $Z$, $X$ et $A$, on peut supposer que $k=1$. Si l'hypothèse de l'énoncé est vérifiée pour $(\Lambda,Z)$, $X\oplus A$ et $c>0$, alors on a
\[
\langle\Lambda,[X+\chi_Z(A),Y+\chi_Z(B)]\rangle+\langle\Lambda,[A,B]\rangle = -c(\langle\Lambda,[X,Y]\rangle+\delta B_{\theta}(z_0,[A,B]))
\]
pour tout $Y\oplus B\in\got{k}/\got{k}_{\Lambda}\oplus\got{p}$. En prenant en particulier $B=0$, cela nous donne
\[
\langle\Lambda,[X+\chi_Z(A),Y]\rangle = -c\langle\Lambda,[X,Y]\rangle, \quad \forall Y\in\got{k}/\got{k}_{\Lambda}.
\]
On en déduit que $X+\chi_Z(A) = -c X$, et donc $\chi_Z(A) = -(c+1)X$, modulo $\got{k}_{\Lambda}$. On voit que l'on a nécessairement $A\neq0$, car sinon $X$ serait également nul, mais on a supposé que le vecteur $X\oplus A$ de $\got{k}/\got{k}_{\Lambda}\oplus\got{p}$ était non nul.

On remarque également qu'il ne nous reste plus que l'égalité
\[
\langle\Lambda,[X+\chi_Z(A),\chi_Z(B)]\rangle+\langle\Lambda,[A,B]\rangle = -c\delta B_{\theta}(z_0,[A,B]),
\]
pour tout $B\in\got{p}$, ce qui est équivalent à
\begin{equation}
\label{eq:égalité_condnécessaire_1}
\delta B_{\theta}(z_0,[A,B]) + \frac{1}{c}\langle\Lambda,[A,B]\rangle = \langle\Lambda,[X,\chi_Z(B)]\rangle,
\end{equation}
pour tout $B\in\got{p}$, puisque $X+\chi_Z(A) = -c X$ modulo $\got{k}_{\Lambda}$.

En remplaçant $B$ par le vecteur $-[z_0,A]\in\got{p}$ dans \eqref{eq:égalité_condnécessaire_1}, on obtient
\[
0 < \delta B_{\theta}(z_0,[A,-[z_0,A]]) + \frac{1}{c}\langle\Lambda,[A,-[z_0,A]]\rangle = -\frac{1}{c+1}\langle\Lambda,[\chi_Z(A),\chi_Z(-[z_0,A])]\rangle
\]
Or, $\|A\|^2 = B_{\theta}(A,A) = B_{\theta}(z_0,[A,-[z_0,A]]))$. De plus, $\frac{1}{c}\langle\Lambda,[A,-[z_0,A]]\rangle\geqslant 0$ puisque
\[
\frac{1}{c}\langle\Lambda,[A,-[z_0,A]]\rangle = \frac{1}{c}B_{\theta}(H_{\Lambda},[A,-[z_0,A]]) = \frac{1}{c}B_{\theta}(z_0,\ad(A)^2 H_{\Lambda}) \geqslant 0,
\]
la minoration par $0$ étant donnée par l'inégalité \eqref{eq:minoration_Btheta(z0adZ2H)} du Lemme \ref{lemm:positivité_puissances_ad(Z)_pourHlambda} appliqué à $A\in\got{p}$. En découlent les inégalités suivantes,
\begin{align*}
\delta\|A\|^2 & \leqslant \delta B_{\theta}(z_0,[A,-[z_0,A]]) + \frac{1}{c}\langle\Lambda,[A,-[z_0,A]]\rangle \\
& \leqslant \frac{1}{c+1}\langle\Lambda,[\chi_Z(A),\chi_Z([z_0,A])]\rangle \\
& \leqslant \frac{1}{c+1} b_{\Lambda}\|\chi_Z(A)\|.\|\chi_Z([z_0,A])\|.
\end{align*}
Rappelons que $c$ est positif, donc $0<\frac{1}{c+1}\leqslant 1$. De plus, d'après la Proposition \ref{prop:chi_Z_symetrique+vp}, l'endomorphisme $\chi_Z$, symétrique pour $B_{\theta}$, a toutes ses valeurs propres dans $]-1,1[$. Donc
\[
\|\chi_Z(W)\| \leqslant \|W\| \quad \text{pour tout $W\in\got{g}$}.
\]
On obtient par conséquent $\|\chi_Z([z_0,A])\| \leqslant \|[z_0,A]\| = \|A\|$. On en déduit que $\delta\|A\|^2 \leqslant b_{\Lambda}\|A\|^2$. Or, on a vu que $A$ est non nul. On en conclut l'inégalité $\delta\leqslant b_{\Lambda}$.
\end{proof}

\begin{proof}[Preuve du Théorème \ref{theo:segment_bien_symplectique}]
Le seul point non trivial est de montrer que $\Omega_t$ est non dégénérée en tout point de $\Klambda\times\got{p}$. Par $K$-invariance de $\Omega_t$, il suffit de le montrer aux points de $\{\Lambda\}\times\got{p}$.

Fixons $\delta > b_{\Lambda}$. Par contraposée de l'énoncé du Lemme \ref{lemm:condnécessaire_si_formeslinéaires_opposées}, pour tout $Z\in\got{p}$, tout $X\oplus A\in\got{k}/\got{k}_{\Lambda}\oplus\got{p}$ et toute constante $c>0$, on ne peut pas avoir l'égalité de formes linéaires suivante :
\[
\imath\bigl(([1,X],A)\bigr)\omega_1|_{(\Lambda,Z)} = -c \left(\imath\bigl(([1,X],A)\bigr)\omega_0^{\delta}|_{(\Lambda,Z)}\right).
\]
Cela signifie que pour $Z\in\got{p}$ et $X\oplus A\in\got{k}/\got{k}_{\Lambda}\oplus\got{p}$ fixés, avec $X\oplus A$ non nul, il existe forcément un vecteur $Y\oplus B\in\got{k}/\got{k}_{\Lambda}\oplus\got{p}$ tel que
\[
\langle\Lambda,[X,Y]\rangle+\delta B_{\theta}(z_0,[A,B]) >0
\]
et
\[
\langle\Lambda,[X+\chi_Z(A),Y+\chi_Z(B)]\rangle+\langle\Lambda,[A,B]\rangle >0.
\]
En effet, si deux formes linéaires non nulles $f_1$ et $f_2$ sur un $\R$-espace vectoriel $F$ de dimension finie sont telles que
\[
\forall x\in F, \quad f_1(x) >0 \quad \Longrightarrow \quad f_2(x)\leqslant 0,
\]
alors ces formes linéaires ont même noyau, et par conséquent il existe une constante $c\in \R^*$ tel que $f_2 = cf_1$. Mais nécessairement $c<0$ car sinon $f_1(x)>0$ implique $f_2(x)=cf_1(x)>0$.

Or, de la définition de $\chi_Z$ et du fait que l'application linéaire $\Psi_Z^+:\got{p}\rightarrow\got{p}$ est un isomorphisme pour tout $Z\in\got{p}$, cela implique les deux inégalités
\[
\Omega^{\delta}|_{(\Lambda,Z)}\left(\left([1,\overline{X}],(\Psi_Z^+)^{-1}(A)\right),\left([1,\overline{Y}],(\Psi_Z^+)^{-1}(B)\right)\right) >0
\]
et
\[
(\Gamma^*\Omega_{G\cdot\Lambda})|_{(\Lambda,Z)}\left(\left([1,\overline{X}],(\Psi_Z^+)^{-1}(A)\right),\left([1,\overline{Y}],(\Psi_Z^+)^{-1}(B)\right)\right) > 0.
\]
Il est alors clair que $\Omega_t^{\delta}|_{(\Lambda,Z)}\left(\left([1,\overline{X}],(\Psi_Z^+)^{-1}(A)\right),\left([1,\overline{Y}],(\Psi_Z^+)^{-1}(B)\right)\right) > 0$, d'où la non-dégénérescence de $\Omega_t^{\delta}|_{(\Lambda,Z)}$ pour tout $t\in[0,1]$.
\end{proof}


\chapter[Incursion en Théorie Géométrique des Invariants]{Incursion en Théorie Géométrique des Invariants. \'Etude des équations du polyèdre moment de la variété hamiltonienne $K\cdot\Lambda\times E$}
\label{chap:projectiondorbitecoadjointe+GIT}


Après la digression réalisée dans le chapitre précédent pour prouver l'existence d'un symplectomorphisme $K$-équivariant entre une orbite coadjointe holomorphe et la variété symplectique $(K\cdot\Lambda\times\got{p},\Omega_{K\cdot\Lambda\times\got{p}})$, nous rentrons maintenant de plain-pied dans l'étude des équations de la projection d'orbite $\Delta_K(\Orb_{\Lambda})$. Ce chapitre présente la première étape de cette étude, proposant un calcul des équations pour une famille plus large de variétés. Il s'agit des variétés de la forme $K\cdot\Lambda\times E$, où $E$ est un espace vectoriel hermitien sur lequel $K$ agit en préservant la forme hermitienne. Cette étude nécessitera l'utilisation d'un résultat qui sera démontré dans le chapitre suivant.

\section{Une structure hamiltonienne sur $K\cdot\Lambda\times E$}
\label{section:polydremoment_KLambdaE}

Soit $(E,h)$ un espace vectoriel hermitien et $U(E)$ le groupe unitaire associé, d'algèbre de Lie $\got{u}(E)$. On note $\Omega_E$ la partie imaginaire de $-h$. C'est une forme symplectique sur $E$, stable par l'action de $U(E)$. L'espace vectoriel (réel) symplectique $(E,\Omega_E)$ est $U(E)$-hamiltonien, d'application moment $\Phi_{U(E)}:E\rightarrow \got{u}(E)^*$ définie par $\langle\Phi_{U(E)}(v),X\rangle = \frac{1}{2}\Omega_E(v,Xv)$.

Soit $\zeta:K\rightarrow U(E)$ un morphisme de groupes de Lie et $\Phi_E : E\rightarrow \got{k}^*$ l'application moment obtenue par composition de $\Phi_{U(E)}$ et de la transposée $^t(d\zeta):\got{u}(E)^*\rightarrow\got{k}^*$.

\begin{prop}[\cite{paradan_fgq}, Lemme 5.2]
Les conditions suivantes sont équivalentes :
\begin{enumerate}
\item L'application $\Phi_E$ est propre,
\item $\Phi_E^{-1}(0) = \{0\}$.
\end{enumerate}
\end{prop}

Dans \cite{paradan_fgq}, $K$ est un sous-groupe de $U(E)$, mais cela peut être aisément étendu au cas où on a juste un morphisme de groupes de Lie $K\rightarrow U(E)$.

\begin{exple}
\label{exple:Phi_gotp_propre}
Lorsque $E = \got{p}$ provient d'un espace symétrique hermitien $G/K$, on peut lui associer la forme symplectique linéaire $K$-invariante
\[
\Omega_{\got{p}}(X,Y) = B_{\theta}(X, \ad(z_0)Y) \quad\mbox{pour tous $X,Y\in\got{p}$},
\]
avec laquelle on peut définir un produit scalaire hermitien $K$-invariant $h_{\got{p}}$ sur $\got{p}$ correspondant. On aura alors $\Phi_{\got{p}}(v) = -[v,[z_0,v]]$, pour $v\in\got{p}$,  si on identifie $\got{k}^*$ à $\got{k}$ par la relation $\langle\xi,X\rangle = -B_{\theta}(\tilde{\xi},X)$, pour tout $X\in\got{k}$. Ceci donne $\langle\Phi_{\got{p}}(v),z_0\rangle = \|[z_0,v]\|^2 >0$, si $v\neq 0$, pour la norme associée à la forme bilinéaire symétrique définie positive $B_{\theta}$.
\end{exple}

Soit $\Lambda\in\got{t}_+^*$. Nous pouvons maintenant considérer la variété $K\cdot\Lambda \times E$, munie de l'action diagonale de $K$, pour l'action standard de $K$ sur l'orbite coadjointe $K\cdot\Lambda$, et l'action sur $E$ induite par le morphisme de groupes $\zeta:K\rightarrow U(E)$.

On définit sur $K\cdot\Lambda\times E$ une structure symplectique comme produit des deux variétés symplectiques $(K\cdot\Lambda,\Omega_{K\cdot\Lambda})$ et $(E,\Omega_E)$. La forme symplectique obtenue sur $K\cdot\Lambda\times E$ est définie par
\[
\Omega_{K\cdot\Lambda\times E} = \pi_{K\cdot\Lambda}^*\Omega_{K\cdot\Lambda} + \pi_E^*\Omega_E,
\]
où $\pi_{K\cdot\Lambda}: K\cdot\Lambda\times E \rightarrow K\cdot\Lambda$ et $\pi_E:K\cdot\Lambda\times E\rightarrow E$ désignent les deux projections canoniques. L'application moment canonique qui lui est associée est l'application
\[
\begin{array}{cccl}
\Phi_{K\cdot\Lambda\times E} : & K\cdot\Lambda\times E & \longrightarrow & \got{k}^* \\
 & (k\Lambda, v) & \longmapsto & k\Lambda + \Phi_E(v).
\end{array}
\]
On définit alors le polyèdre moment
\[
\Delta_K(K\cdot\Lambda\times E) := \Phi_{K\cdot\Lambda\times E}(K\cdot\Lambda\times E)\cap\got{t}_+^*\index{$\Delta_K(K\cdot\Lambda\times E)$}
\]
associé à la variété hamiltonienne $(K\cdot\Lambda\times E, \Omega_{K\cdot\Lambda\times E},\Phi_{K\cdot\Lambda\times E})$.

Il faut noter que la variété $K\cdot\Lambda\times E$ n'est pas compacte. Le Théorème de convexité (Théorème \ref{theo:convexitéhamiltonienne}) sera tout de même applicable sur cette variété hamiltonienne. 
En effet, d'après \cite[Théorème 1.5]{paradan_fgq}, si $(E,\Omega_E)$ est un espace vectoriel symplectique muni d'une action $K$-hamiltonienne d'application moment $\Phi_E:E\rightarrow \got{k}^*$ propre, alors l'application moment $\Phi_{K\cdot\Lambda\times E}:K\cdot\Lambda\times E\rightarrow \got{k}^*$ sera également propre. L'ensemble $\Delta_K(K\cdot\Lambda\times E)$ sera alors un ensemble convexe rationnel localement polyèdral, d'après \cite{lerman98}.

On demandera donc, dans la suite, que l'application moment $\Phi_E$ soit propre.

Notons $\wedge^*_+ := \wedge^*\cap\got{t}^*_+$ (resp $\wedge^*_{\Q,+} := (\wedge^*\otimes_{\Z}\Q)\cap\got{t}^*_+$) l'ensemble des poids dominants (resp. poids rationnels dominants) de $K$. Pour tout $\nu\in\wedge^*_+$ poids dominant de $K$, $V_{\nu}$ désignera la représentation irréductible de $K$ de plus haut poids $\nu$. Enfin, notons $\C[E]$ l'algèbre des polynômes sur $E$.

\begin{theo}
\label{theo:polyèdremoment_et_représentationsirréductibles}
Soient $\Lambda\in\wedge^*_{\Q,+}$ et $\mu\in\wedge^*_{\Q,+}$. Alors $\mu\in\Delta_K(K\cdot\Lambda\times E)$ si et seulement s'il existe un entier $n\geqslant 1$ tel que $(n\Lambda,n\mu)\in(\wedge^*)^2$ et
\[
\left(V^*_{n\mu}\otimes V_{n\Lambda}\otimes \C[E]\right)^K \neq 0.
\]
\end{theo}
%
%
%

\begin{rema}
Nous allons voir que ce résultat est un corollaire de \cite[Théorème 4.9]{sjamaar}. Il peut également se montrer en utilisant la quantification géométrique formelle, de façon similaire à \cite[Lemme 2.5]{Paradan}.
\end{rema}

La preuve qui suit a été communiquée à l'auteur par R. Sjamaar.

\begin{proof}[Preuve du Théorème \ref{theo:polyèdremoment_et_représentationsirréductibles}]
Il est connu que l'orbite coadjointe $K\cdot\Lambda$ est un quotient symplectique du fibré cotangent $T^*K$ par rapport à l'action de multiplication à droite de $K$. Le lecteur pourra trouver les détails de la construction du quotient symplectique de $T^*K$ dans \cite[Section 29]{guillemin_sympltechniques} et dans \cite[Section 7.3]{sjamaar}. Par conséquent, $K\cdot\Lambda\times E$ est un quotient symplectique de $T^*K\times E$.

Précisons ce quotient symplectique. Tout d'abord, rappelons que le fibré cotangent $T^*K$ s'identifie à $K\times\got{k}^*$ par translations à gauche. Le groupe $K\times K$ agit sur $K\times\got{k}^*\times E$ par
\[
(k,k')\cdot(l,\mu,v) := (klk'^{-1},k'\mu,k'v), \quad \forall(k,k')\in K\times K, \forall (l,\mu,v)\in K\times\got{k}^*\times E.
\]
Cette action est hamiltonienne pour l'action de $K\times K$, et une application moment associée est
\[
\Phi_{K\times K}(l,\mu,v) = (l\mu,-\mu+\Phi_E(v))\in\got{k}^*\times\got{k}^*,
\]
pour tout $(l,\mu,v)\in K\times\got{k}^*\times E$. Le sous-groupe distingué $K\times\{1\}$ de $K\times K$ agit par restriction sur la variété $K\times \got{k}^*\times E$ (par multiplication à gauche); cette action est hamiltonienne pour l'application moment
\[
\Phi_{K\times\{1\}}(l,\mu,v) = l\mu\in\got{k}^*, \quad \forall (l,\mu,v)\in K\times\got{k}^*\times E.
\]
On remarque que, pour tout $\mu\in\got{t}^*$, la variété $K\cdot\mu\times E$ s'identifie de manière $K$-équivariante au quotient symplectique
\[
\Phi_{K\times\{1\}}^{-1}(K\cdot\mu)/(K\times\{1\}),
\]
où $K$ agit sur ce quotient par l'intermédiaire du sous-groupe $\{1\}\times K$ de $K\times K$. De plus, on voit clairement que, pour tout $(k,v)\in K\times E$, on a 
\[
\Phi_{K\cdot\Lambda\times E}(k\Lambda,v) = k\Lambda+\Phi_E(v) = -k(-\Lambda)+\Phi_E(v) = \Phi_{\{1\}\times K}(l,k(-\Lambda),v),
\]
pour tout $l$ parcourant $K$. En particulier, remplaçant $l$ par $w_0k^{-1}$, on obtient aussi
\[
\Phi_{K\times K}(w_0k^{-1},k(-\Lambda),v) = (-w_0\Lambda,\Phi_{K\cdot\Lambda\times E}(k\Lambda,v)).
\]
On en déduit que le polyèdre moment $\Delta_K(K\cdot\Lambda\times E)$ s'obtient en intersectant le polyèdre moment $\Delta_{K\times K}(T^*K\times E)$ par l'espace affine $\{-w_0\Lambda\}\times \got{t}^*$, c'est-à-dire
\begin{equation}
\label{eq:lien_DeltaK(KLambdatimesE)_et_DeltaKtimesK(T*KtimesE)}
\Delta_K(K\cdot\Lambda\times E) = \mathrm{pr}_2\left(\Delta_{K\times K}(T^*K\times E)\cap (W\cdot\{-w_0\Lambda\}\times \got{t}^*)\right), 
\end{equation}
où $\mathrm{pr}_2:\got{t}^*\times\got{t}^*\rightarrow\got{t}^*$ est la projection sur la deuxième variable.

D'un autre côté, le fibré cotangent $T^*K$ est isomorphe en tant que $K\times K$-variété hamiltonienne à $\KC$, groupe complexifié du groupe de Lie compact connexe $K$. Or $\KC$ est un groupe algébrique affine, donc $\KC\times E$ est également affine. En appliquant \cite[Théorème 4.9]{sjamaar}, on obtient que $\Delta_{K\times K}(T^*K\times E)$ est le cône engendré par le monoïde
\[
\left\{(\mu,\nu)\in(\wedge^*_+)^2; \ V_{(\mu,\nu)}^{K\times K} \subset \C[\KC\times E]\right\}.
\]
Or, $\C[\KC\times E] = \C[\KC]\otimes\C[E]$, et, d'après le théorème de Frobenius, $\C[\KC] = \bigoplus_{\delta\in\wedge^*_+}V_{\delta}\otimes V_{\delta}^*$. D'où un couple $(\mu,\nu)\in(\wedge^*_+)^2$ vérifie $V_{(\mu,\nu)}^{K\times K} \subset \C[\KC\times E]$ si et seulement si $V_{\mu}$ apparaît dans la décomposition en représentations irréductibles de $V_{\nu}^*\otimes\C[E]$.

On conlut alors grâce à \eqref{eq:lien_DeltaK(KLambdatimesE)_et_DeltaKtimesK(T*KtimesE)} que $\mu\in\wedge_{\Q,+}^*$ appartient à $\Delta_K(K\cdot\Lambda\times E)$ si et seulement s'il existe un entier $n\geqslant 1$ tel que $(n\mu,n\Lambda)$ appartienne à $(\wedge^*)^2$ et
\[
\left(V_{n\mu}^*\otimes V_{-w_0(n\Lambda)}^*\otimes\C[E]\right)^K \neq 0,
\]
c'est-à-dire $\left(V_{n\mu}^*\otimes V_{n\Lambda}\otimes\C[E]\right)^K \neq 0$.
\end{proof}

Nous définissons maintenant l'objet que nous allons étudier dans ce chapitre. Il s'agit de la version algébrique du polyèdre moment de $K\cdot\Lambda\times E$.

\begin{defi}
\label{defi:DGIT}
Pour tout $\Lambda\in\wedge^*_{\Q,+}$, nous définissons l'ensemble 
\[
\Delta_K^{\mathrm{GIT}}(K\cdot\Lambda\times E) = \left\{\mu\in\wedge^*_{\Q,+} \left|
\begin{array}{l}
\exists n\in\N^* \mbox{ tel que } (n\mu,n\Lambda)\in(\wedge^*)^2, \\
\mbox{et } \left(V^*_{n\mu}\otimes V_{n\Lambda}\otimes \C[E]\right)^K \neq 0
\end{array}\right.\right\}.
\]\index{$\DGIT$}
\end{defi}

Le corollaire suivant permet de lier la version hamiltonienne et la version algébrique du polyèdre moment de la variété $K\cdot\Lambda\times E$.

\begin{coro}
\label{coro:DeltaClassique_DeltaGIT}
Supposons que $\Phi_E: E\rightarrow \got{k}^*$ est propre. Pour tout $\Lambda\in\wedge^*_{\Q,+}$, nous avons $\Delta_K(K\cdot\Lambda\times E)\cap\wedge^*_{\Q} = \DGIT$ et
$\Delta_K(K\cdot\Lambda\times E)$ est l'adhérence de $\DGIT$ dans $\got{t}^*$.
\end{coro}

\begin{proof}
Ce résultat découle directement du Théorème \ref{theo:polyèdremoment_et_représentationsirréductibles} et de la définition de $\DGIT$.

Pour la dernière assertion, puisque l'application moment $\Phi_E: E\rightarrow \got{k}^*$ est propre, comme il a été dit dans un paragraphe précédant l'énoncé du Théorème \ref{theo:polyèdremoment_et_représentationsirréductibles}, l'application moment $\Phi_{K\cdot\Lambda\times E}:K\cdot\Lambda\times E\rightarrow \got{k}^*$ est également propre et l'ensemble convexe localement polyédral $\Delta_K(K\cdot\Lambda\times E)$ est rationnel. Par conséquent, il est donc égal à l'adhérence de l'ensemble de ses points rationnels. D'où le résultat.
\end{proof}

L'ensemble $\DGIT$ manque de symétrie entre les variables $\mu$ et $\Lambda$. Nous définissons donc l'ensemble
\begin{equation}
\label{eq:gammaQ_version_K}
\Pi_{\Q}(E) = \left\{ (\mu,\nu)\in(\wedge^*_{\Q,+})^2 \left|
\begin{array}{l}
\exists n\geqslant 1 \text{ tel que } (n\mu,n\nu)\in(\wedge^*)^2 \\
\mbox{et } \left(V_{n\mu}^*\otimes V_{n\nu}^*\otimes \C[E]\right)^K \neq 0
\end{array}\right.\right\}.
\end{equation}\index{$\Pi_{\Q}(E)$}

\begin{lemm}
\label{lemm:lien_DeltaKLambda_GammaQ}
Pour tout $\Lambda\in\wedge^*_{\Q,+}$, on a
\[
\DGIT = \left\{\xi\in\wedge^*_{\Q}; (\xi,-w_0\Lambda)\in\Pi_{\Q}(E)\right\},
\]
où $w_0$ est le plus long élément du groupe de Weyl $W=N_K(T)/T$. Si, de plus, on suppose que $\Phi_E: E\rightarrow \got{k}^*$ est propre, alors
\[
\Delta_K(K\cdot\Lambda\times E) = \left\{\xi\in\got{t}^*; (\xi,-w_0\Lambda)\in\overline{\Pi_{\Q}(E)}\right\},
\]
où $\overline{\Pi_{\Q}(E)}$ désigne l'adhérence de $\Pi_{\Q}(E)$ dans $\got{t}^*$.
\end{lemm}

\begin{proof}
Ce résultat vient tout simplement du fait que $V_{k\Lambda}\cong V^*_{k(-w_0\Lambda)}$ et des définitions de $\DGIT$ et $\Pi_{\Q}(E)$. La deuxième assertion découle du Corollaire \ref{coro:DeltaClassique_DeltaGIT}.
\end{proof}

\begin{rema}
L'ensemble $\Pi_{\Q}(E)$ a des propriétés géométriques très intéressantes que nous allons étudier dans la section suivante. Il est en particulier un cône convexe polyédral. En appliquant le Lemme \ref{lemm:lien_DeltaKLambda_GammaQ}, on voit qu'il suffit de trouver les équations du polyèdre $\Pi_{\Q}(E)$ pour obtenir celles de $\DGIT$.
\end{rema}

Avant de passer à la section suivante, nous allons démontrer le Théorème \ref{theo:DeltaK(OrbLambda)_par_représentationsirréductibles} (dont l'énoncé est donné page \pageref{theo:DeltaK(OrbLambda)_par_représentationsirréductibles}). C'est un corollaire direct du Théorème \ref{theo:polyèdremoment_et_représentationsirréductibles} et du résultat principal du Chapitre 4.

\begin{proof}[Preuve du Théorème \ref{theo:DeltaK(OrbLambda)_par_représentationsirréductibles}]
Le Théorème \ref{theo:symplecto_principal} nous donne l'égalité $\Delta_K(\Orb_{\Lambda}) = \Delta_K(K\cdot\Lambda\times\got{p})$ dès que $\Lambda\in\Chol$. Enfin, d'après le Théorème \ref{theo:polyèdremoment_et_représentationsirréductibles} appliqué à $E=\got{p}\simeq\got{p}^-$ et à l'application moment propre $\Phi_{\got{p}}$, les points rationnels de $\Delta_K(K\cdot\Lambda\times\got{p})$ sont exactement les $\mu\in\wedge_{\Q,+}^*$ qui vérifient $(N\mu,N\Lambda)\in(\wedge^*)^2$ et $\left(V^*_{N\mu}\otimes V_{N\Lambda}\otimes \C[\got{p}^-]\right)^K \neq 0$ pour un certain entier strictement positif $N$.
\end{proof}

\section{Quelques éléments de GIT}
\label{section:GIT}

Le GIT, ou \emph{Théorie Géométrique des Invariants}, a pour but d'étudier les quotients $X/\!/G$, où $X$ est une variété projective irréductible munie d'une action d'un groupe algébrique $G$, le plus souvent réductif. Le quotient standard $X/G$ n'étant pas en général une variété algébrique (les orbites ne sont pas toutes fermées), on est amené à définir un \og{}bon\fg{} quotient, le quotient GIT, que l'on note $X/\!/G$. Nous n'entrerons pas ici dans les détails de la Théorie Géométrique des Invariants. Nous n'aborderons que les éléments utiles à la compréhension de la suite de ce chapitre. Pour plus d'informations sur le GIT, le lecteur pourra se référer à \cite{dolgachev03, ressayre08, heinzner_migliorini}. Nous suivrons les notations et conventions de \cite{ressayre08}.

\subsection{Notations}
\label{subsection:GIT_notations}

Dans la suite de cette section, $K$ désignera un groupe de Lie réel compact connexe et $\KC$ son complexifié. Le groupe $\KC$ est un groupe algébrique complexe réductif connexe et $K$ est un sous-groupe compact maximal de $\KC$.

Soit $X$ une variété algébrique projective sur $\C$, munie d'une action algébrique de $\KC$. Nous noterons $\Pic(X)$ l'ensemble des classes d'isomorphismes des fibrés en droites sur la variété $X$. Cet ensemble peut être muni d'une structure canonique de groupe abélien, induite par le produit tensoriel des fibrés en droites sur $X$.

\begin{defi}
Une $\KC$-linéarisation d'un fibré en droites $\mathcal{L}$ sur $X$ est un relèvement de l'action de $\KC$ sur $X$ en une action sur $\mathcal{L}$ qui est linéaire sur les fibres. Un fibré en droites $\KC$-linéarisé est la donnée d'un fibré en droites $\mathcal{L}$ sur $X$ et d'une $\KC$-linéarisation de $\mathcal{L}$. On notera $\PicG(X)$ le groupe des classes d'isomorphismes des fibrés en droites $\KC$-linéarisés de $X$. La structure de groupe de $\PicG(X)$ est à nouveau induite par le produit tensoriel.
\end{defi}


Nous noterons $\PicQ(X):=\Pic(X)\otimes_{\Z}\Q$ (resp. $\PicGQ(X):=\PicG(X)\otimes_{\Z}\Q$) le $\Q$-espace vectoriel engendré par les éléments de $\Pic(X)$ (resp. $\PicG(X)$) et, pour tout $\mathcal{L}\in\PicG(X)$, $\mathrm{H}^0(X,\mathcal{L})$ le $\KC$-module des sections régulières de $\mathcal{L}$.

Soit $T$ un tore maximal de $K$ et $\got{t}$ son algèbre de Lie. Nous fixons $\TC$ le tore maximal de $\KC$ tel que $T=K\cap\TC$ et $B$ un sous-groupe de Borel de $\KC$ contenant $\TC$. On notera $W = W(\KC;\TC) = N_{\KC}(\TC)/\TC$ le groupe de Weyl de $\KC$ relatif à son tore maximal $\TC$.

Dans la suite, nous identifierons le groupe des caractères de $\TC$ avec le réseau des poids $\wedge^*\subset\got{t}^*$. Pour un poids dominant $\mu\in\wedge^*$, nous noterons $V_{\mu}$ la représentation irréductible de $\KC$ (ou, de manière équivalente, de $K$), de plus haut poids $\mu$.

\begin{nota}
\label{nota:variétéX_M}
Pour $M$ une représentation complexe de $\KC$ (ou, de manière équivalente, de $K$), on définit la variété projective lisse
\[
X_M:=\KC/B\times \KC/B\times\mathbb{P}(M),
\]\index{$X_M$}
munie de l'action diagonale de $\KC$.
\end{nota}

\begin{nota}
Pour tout $\TC$-module $M$, on notera $\WT(M)\subset\wedge^*$\index{$\WT(M)$} l'ensemble des poids de $\TC$ sur $M$.
\end{nota}

\begin{nota}
Si $F$ est un sous-ensemble d'un $\Z$-module (resp. $\Q$-espace vectoriel, resp. $\R$-espace vectoriel) $M$, on notera $\mathcal{C}_{\Z}(F)$\index{$\mathcal{C}_{\Z}(F)$, $\mathcal{C}_{\Q}(F)$, $\mathcal{C}_{\R}(F)$} (resp. $\mathcal{C}_{\Q}(F)$, resp. $\mathcal{C}_{\R}(F)$) le monoïde (resp. le $\Q$-cône convexe, resp. le $\R$-cône convexe) engendré par $F$ dans $M$.
\end{nota}

\subsection{Fibrés en droites amples et semi-amples}

Parmi tous les fibrés en droites sur la variété projective $X$, certains sont de première importance. Toute application algébrique de la variété projective $X$ dans un espace projectif $\mathbb{P}^n$ est définie par le choix d'un fibré en droites $\mathcal{L}$ et d'un ensemble $s_0,\ldots,s_n$ de sections globales de $\mathcal{L}$. L'application est définie par
\begin{equation}
\label{eq:fibréendroites_morphisme_sanspointbase}
x\in X\longmapsto [s_0(x):\ldots:s_n(x)]\in\mathbb{P}^n.
\end{equation}
Celle-ci sera bien définie si, pour tout $x\in X$, une des sections $s_i$ ne s'annule pas en $x$. Nous dirons que $\mathcal{L}$ est \emph{engendré par ses sections globales} s'il existe un nombre fini de sections globales de $\mathcal{L}$ telles que, pour tout $x\in X$, au moins une de ces sections ne s'annule pas en $x$. Un fibré engendré par ses sections globales donnera bien le morphisme algébrique \eqref{eq:fibréendroites_morphisme_sanspointbase} pour les sections globales $s_0,\ldots,s_n$ correspondantes. Inversement, un morphisme $f : X\rightarrow \mathbb{P}^n$ définira un fibré en droite $\mathcal{L}_f = f^*(\mathcal{O}(1))$, tiré-en-arrière du fibré canonique sur $\mathbb{P}^n$, qui sera engendré par ses sections globales sur $X$.

Un fibré en droites $\mathcal{L}$ engendré par ses sections globales, sera \emph{très ample} si l'application algébrique \eqref{eq:fibréendroites_morphisme_sanspointbase} associée est une immersion fermée et $\mathcal{L}$ est isomorphe au tiré-en-arrière de $\mathcal{O}(1)$ sur $X$.

\begin{defi}
Un fibré en droites $\mathcal{L}$ est \emph{ample} (resp. \emph{semi-ample}) s'il existe un entier strictement positif $n$ tel que $\mathcal{L}^{\otimes n}$ est très ample (resp. engendré par ses sections globales).
\end{defi}

Nous noterons $\PicG(X)^+$ (resp. $\PicG(X)^{++}$) l'ensemble des fibrés en droites $\KC$-linéarisés et semi-amples (resp. amples) sur $X$. Nous noterons aussi $\PicG(X)_{\Q}^+:=\mathcal{C}_{\Q}(\PicG(X)^+)$ et $\PicG(X)_{\Q}^{++}:=\mathcal{C}_{\Q}(\PicG(X)^{++})$ les cônes convexes engendrés respectivement par les éléments semi-amples et amples de $\PicG(X)_{\Q}$. On peut remarquer que nous avons les inclusions $\PicG(X)^{++}\subset\PicG(X)^+$ et $\PicG(X)_{\Q}^{++}\subset\PicG(X)_{\Q}^+$.

\subsection{Semi-stabilité}

Pour un fibré en droites $\KC$-linéarisé $\mathcal{L}$ sur la variété projective $X$, un des objets fondamentaux du GIT associés à $\mathcal{L}$ est l'ensemble de ses points semi-stables sur $X$,
\[
\XssL = \left\{x\in X;\ \exists k\geqslant 1, \exists s\in \mathrm{H}^0(X,\mathcal{L}^{\otimes k})^{\KC}, \mbox{ tel que } s(x)\neq 0\right\}.
\]
Ceci n'est pas la définition standard de $\Xss(\mathcal{L})$, mais ça l'est si $\mathcal{L}$ est ample. Dans ce cas, l'ouvert sur lequel $\sigma$ ne s'annule pas est affine. Cette propriété est utile pour définir un bon quotient de $\XssL$.

Quand $\mathcal{L}$ est ample, nous avons un quotient catégorique $\pi : \XssL \rightarrow \XssL\quotient\KC$, tel que $\XssL\quotient\KC$ soit une variété projective et $\pi$ soit affine.

On remarque aisément que pour tout fibré en droites $\KC$-linéarisé $\mathcal{L}$ et pour tout entier strictement positif $n$, on a $\Xss(\mathcal{L}) = \Xss(\mathcal{L}^{\otimes n})$. Nous pouvons donc définir $\Xss(\mathcal{L})$ pour tout élément $\mathcal{L}\in\PicG(X)_{\Q}^+$.

\subsection{Cônes ample et semi-ample}

Nous allons maintenant introduire la notion de cône ample $\Ca$ et de cône semi-ample $\Csa$ associés à une variété projective.

Soit $X$ une variété projective irréductible, munie d'une action algébrique du groupe $\KC$. Nous avons défini, pour un fibré en droites $\mathcal{L}$ sur $X$, l'ensemble de ses points semi-stables et avons indiqué que, dans le cas des fibrés amples, on pouvait définir un quotient catégorique de $\XssL$. La question qui vient maintenant est de savoir pour quels $\mathcal{L}\in\PicG(X)$ l'ensemble $\XssL$ est non vide. Nous définissons donc le \emph{cône semi-ample}\index{Cône semi-ample}
\[
\Csa = \{\mathcal{L}\in\PicG(X)_{\Q}^+; \ \Xss(\mathcal{L}) \neq\emptyset\}.
\]\index{$\Csa$, $\Ca$}
et le \emph{cône ample}\index{Cône ample}
\[
\Ca = \Csa_{\Q} \cap \PicG(X)_{\Q}^{++}.
\]
Les théorèmes suivant vont justifier la terminologie de ces ensembles.

\begin{theo}[\cite{dolgachev98, ressayre00}]
L'ensemble $\Ca$ est un cône convexe polyédral dans $\PicGQ(X)$, fermé dans $\PicGQ(X)^{++}$.
\end{theo}

Nous nous intéresserons plus tard à un type très particulier de variétés projectives, les variétés de type $X_M$, définies en Notation \ref{nota:variétéX_M}, pour $M$ une représentation complexe du groupe $\KC$. Ces variétés vérifient les hypothèses du théorème suivant. Soit $\hKC$ un groupe réductif connexe tel que $\KC\subset\hKC$ et soit $Q$ (resp. $\hat{Q}$) un sous-groupe parabolique de $\KC$ (resp. $\hKC$).

\begin{theo}[\cite{ressayre08}, Proposition 10]
\label{theo:CsaX_conepolyconvfermé_et_adhérencedeCaX}
Si $X=\KC/Q\times\hKC/\hat{Q}$, alors le cône semi-ample $\Csa$ est un cône convexe polyédral fermé dans $\PicGQ(X)$. De plus, si $\Ca$ n'est pas vide, son adhérence dans $\PicGQ(X)$ est $\Csa$.
\end{theo}

\subsection{Groupe de Picard $\KC$-linéarisé de la variété $X_M$}

Dans ce paragraphe, nous allons décrire le groupe de Picard $\PicG(X_M)$ de la variété $X_M=\KC/B\times\KC/B\times\mathbb{P}(M)$, où $M$ est une représentation complexe de $\KC$.

Commençons tout d'abord par la variété $\KC/B$. A tout élément $\mu\in\wedge^*$ est associé de manière univoque un caractère de $\TC$. Cette association provient de l'identification, que nous avons fixée au début de cette section, entre le réseau des poids $\wedge^*$ et le groupe des caractères de $\TC$. Un tel caractère se prolonge de manière unique en un caractère de $B$. Nous noterons alors $\C_{\mu}$ la représentation de $B$ sur $\C$ associée à ce caractère de $B$. Ceci permet de définir l'application $\mu\in\wedge^*\mapsto\KC\times_B\C_{-\mu}\in\Pic(\KC/B)$.

Nous pouvons remarquer que, pour tout $\mu\in\wedge^*$, le fibré en droites $\KC\times_B\C_{-\mu}$ est muni d'une action canonique de $\KC$, définie par
\[
g\cdot[k,z] = [gk,z] \text{ pour tout $g\in\KC$ et tout $[k,z]\in\KC\times_B\C_{-\mu}$}.
\]
Cette action relève l'action de $\KC$ depuis la base $\KC/B$ sur le fibré tout entier et l'action est évidemment linéaire sur les fibres. Il s'agit donc d'une $\KC$-linéarisation du fibré en droites $\KC\times_B\C_{-\mu}$, que nous noterons $\mathcal{L}_{\mu}$. Nous obtenons l'application $\mu\in\wedge^*\mapsto\mathcal{L}_{\mu}\in\PicG(\KC/B)$, qui, composée avec le foncteur oubli, redonne l'application $\wedge^*\rightarrow\Pic(\KC/B)$ définie ci-dessus. Il est bien connu que cette application est un isomorphisme entre $\wedge^*$ et $\PicG(\KC/B)$, cf \cite[Section 3.1]{knopkraftvurst}. De plus, le Théorème de Borel-Weil implique que les fibrés $\KC$-linéarisés sur $\KC/B$ engendrés par leurs sections globales sont les fibrés $\mathcal{L}_{\mu}$ provenant d'un poids $\mu$ dominant. Par conséquent, il est clair que l'on a
\[
\PicG(\KC/B)_{\Q}^+ \simeq \wedge^*_{\Q,+}.
\]
Il est aussi bien connu que les fibrés très amples sur la variété des drapeaux $\KC/B$ correspondent aux éléments dominants réguliers de $\wedge^*$. L'ensemble des éléments dominants réguliers sera noté $\wedge^*_{\scriptscriptstyle ++}$. Ainsi,
\[
\PicG(\KC/B)_{\Q}^{++} \simeq \wedge^*_{\Q,\scriptscriptstyle++}.
\]

\bigskip

Intéressons-nous maintenant au dernier facteur apparaissant dans $X_M$, c'est-à-dire l'espace projectif $\mathbb{P}(M)$, avec action induite de $\KC$.

Pour tout $k\in\Z$, nous définissons $\C_k$ comme étant l'espace vectoriel complexe $\C$ muni de l'action de $\C^*\cong\mathbb{G}_m$ définie par
\[
\forall w\in\C,\ \forall z\in\C^*,\ z\cdot w := z^k w.
\]
Pour tout $k\in\Z$, nous posons le fibré en droites $\mathcal{L}_k:=M\setminus\{0\}\times_{\C^*}\C_k$ sur $\mathbb{P}(M)$. Il est bien connu que l'application $k\in\Z\mapsto\mathcal{L}_k\in\Pic(\mathbb{P}(M))$ est un isomorphisme de groupes abéliens. Remarquons que $\mathcal{L}_{1}$ correspond au fibré $\mathcal{O}(1)$ de $\mathbb{P}(M)$ et, par conséquent, puisque l'on a $\mathcal{L}_k = \mathcal{L}_{1}^{\otimes k}$, le fibré en droites $\mathcal{L}_k$ est isomorphe au fibré $\mathcal{O}(k)$. On en déduit que les fibrés en droites semi-amples (resp. amples) de $\Pic(\mathbb{P}(M))$ sont les $\mathcal{L}_k$, pour $k$ parcourant $\N$ (resp. $k$ parcourant $\N\setminus\{0\}$).

Maintenant, notons $\mathcal{X}(\KC)$ le groupe des caractères de $\KC$. Par définition, il s'agit du groupe des homomorphismes rationnels de groupes algébriques de $\KC$ vers $\mathbb{G}_m$. Pour tout $(\chi,k)\in\car(\KC)\times\Z$, l'action de $\KC$ sur $\mathcal{L}_k$ définie par
\[
g\cdot[x,z] := [g\cdot x,\chi(g^{-1})z] \text{ pour tout $g\in\KC$ et tout $[x,z]\in\mathcal{L}_k$},
\]
est une $\KC$-linéarisation du fibré en droites $\mathcal{L}_k$. Nous noterons $\mathcal{L}_{\chi,k}$ le fibré en droites $\KC$-linéarisé de $\mathbb{P}(M)$ obtenu.

\begin{theo}
Le groupe de Picard $\KC$-linéarisé $\PicG(\mathbb{P}(M))$ est isomorphe à $\Z\times\car(\KC)$.
\end{theo}

\begin{proof}
Nous venons de voir que l'on peut définir, pour tout couple $(\chi,k)\in\car(\KC)\times\Z$, une $\KC$-linéarisation sur $\mathcal{L}_k$, le fibré $\KC$-linéarisé associé étant $\mathcal{L}_{\chi,k}$. On en déduit que le morphisme oubli $\alpha:\PicG(\mathbb{P}(M))\rightarrow\Pic(\mathbb{P}(M))$ est surjectif. D'après \cite[Lemme 2.2]{knopkraftvurst} (ou encore \cite[Théorème 7.1]{dolgachev03}), $\ker\alpha$ est isomorphe au groupe abélien $\car(\KC)$. Par conséquent, tout fibré en droite $\KC$-linéarisé de $\mathbb{P}(M)$ est isomorphe à $\mathcal{L}_{\chi,k}$ pour un certain couple $(\chi,k)\in\car(\KC)\times\Z$.
\end{proof}

Nous identifions $\car(\KC)$ avec le sous-$\Z$-module $(\wedge^*)^{\KC}:=(\got{k}_{\C}^*)^{\KC}\cap\wedge^*$ de $\wedge^*$. On peut remarquer que $(\got{k}_{\C}^*)^{\KC}$ est égal au dual $\got{z}(\got{k}_{\C})^* := [\got{k}_{\C},\got{k}_{\C}]^{\bot}$ du centre de $\got{k}_{\C}$. Il est clair que, pour tout $(\chi,k)\in\car(\KC)\times\Z$, le fibré $\KC$-linéarisé $\mathcal{L}_{\chi,k}$ est isomorphe au produit tensoriel de fibrés en droites $\KC$-linéarisés $\mathcal{L}_k\otimes(\mathbb{P}(M)\times\C_{-\chi})$, où $\KC$ agit uniquement sur le premier facteur de $\mathbb{P}(M)\times\C_{-\chi}$.

%
%

La proposition suivante est un analogue de \cite[Lemme 13]{ressayre08}.

\begin{prop}
\label{prop:surjection_grpPicardKC}
L'application $\Q$-linéaire
\begin{equation}
\label{eq:application_Qlin_surjective_sur_gpPicardKC}
\begin{array}{ccc}
\wedge^*_{\Q}\times\wedge^*_{\Q}\times(\wedge^*_{\Q})^{\KC}\times\Q & \longrightarrow & \PicG(X_M)_{\Q} \\
(\mu,\nu,\chi,r) & \longmapsto & \mathcal{L}_{\mu,\nu,\chi,r} := \mathcal{L}_{\mu}\boxtimes\mathcal{L}_{\nu}\boxtimes\mathcal{L}_{\chi,r}
\end{array}
\end{equation}
est surjective. Son noyau est le sous-espace $\{(-\mu-\chi,\mu,\chi,0); \mu,\chi\in(\wedge^*_{\Q})^{\KC}\}$. L'ensemble des fibrés en droites $\KC$-linéarisés semi-amples (resp. amples) sur $X_M$ est l'image par l'application ci-dessus du cône convexe $\wedge^*_{\Q,+} \times \wedge^*_{\Q,+} \times (\wedge^*_{\Q})^{\KC} \times \Q_{\geqslant 0}$ (resp. $\wedge^*_{\Q,\scriptscriptstyle++} \times \wedge^*_{\Q,\scriptscriptstyle++} \times (\wedge^*)^{\KC} \times \Q_{>0}$).
\end{prop}

\begin{proof}
Soient trois éléments $\mu$, $\nu$ et $\chi$ de $(\wedge^*)^{\KC}$. Ils correspondent à trois caractères du groupe $\KC$. Le fibré en droites $\KC$-linéarisé $\mathcal{L}_{\mu,\nu,\chi,0}$ provient alors d'une $\KC$-linéarisation du fibré trivial. Plus exactement, on peut voir que $\mathcal{L}_{\mu,\nu,\chi,0} = X_M\times \C_{-(\mu+\nu+\chi)}$ comme élément de $\PicG(X_M)$. On en déduit que toute $\KC$-linéarisation du fibré trivial sur $X_M$ provient d'un $\mathcal{L}_{\mu,\nu,\chi,0}$.

Soit $\mathcal{L}\in\PicG(X_M)$. Notons $\mathcal{L}'$ le fibré en droites obtenu en oubliant l'action de $\KC$ sur $\mathcal{L}$. D'après \cite[Proposition 1.5]{mumford94} ou bien \cite[Corollaire 7.2]{dolgachev03}, il existe un entier $n\geqslant 1$ et une $\KC\times\KC\times GL_{\C}(M)$-linéarisation $\mathcal{M}$ de $\mathcal{L}'^{\otimes n}$. 
Or on a l'isomorphisme
\[
\begin{array}{ccc}
\PicG(\KC/B)\times\PicG(\KC/B)\times\Pic^{GL_{\C}(M)}(\mathbb{P}(M)) &\stackrel{\sim}{\longrightarrow} & \Pic^{\KC\times\KC\times GL_{\C}(M)}(X_M) \\
(\mathcal{L}_{\mu},\mathcal{L}_{\nu},\mathcal{L}_{\hat{\nu}}) & \longmapsto & \mathcal{L}_{\mu}\boxtimes\mathcal{L}_{\nu}\boxtimes\mathcal{L}_{\hat{\nu}}
\end{array}
\]
où $\mu$ et $\nu$ désignent deux caractères de $\KC$, et $\hat{\nu}$ un caractère de $GL_{\C}(M)$. Par réduction de l'action de $GL_{\C}(M)$ sur $\mathcal{L}_{\hat{\nu}}$ à celle de $\KC$, on obtient un certain couple $(\chi,k)\in(\wedge^*)^{\KC} \times \Z$ tel que le fibré $\mathcal{L}_{\chi,k}$ définisse une $\KC$-linéarisation de $\mathcal{L}_{\hat{\nu}}$. Par conséquent, $\mathcal{L}_{\mu,\nu,\chi,k}^*\otimes \mathcal{L}^{\otimes n}$ est le fibré trivial sur $X_M$. D'après le premier paragraphe de la preuve, on en déduit que $\mathcal{L}_{\mu,\nu,\chi,k}^*\otimes \mathcal{L}^{\otimes n}$ est dans l'image de l'application \eqref{eq:application_Qlin_surjective_sur_gpPicardKC} donnée dans l'énoncé. On en conlut que $\mathcal{L}$ est égal à un certain $\mathcal{L}_{\mu',\nu',\chi',r}$, avec $(\mu',\nu',\chi',r)\in\wedge^*_{\Q}\times\wedge^*_{\Q}\times(\wedge^*_{\Q})^{\KC}\times\Q$. D'où la surjectivité de l'application \eqref{eq:application_Qlin_surjective_sur_gpPicardKC}.

Soit maintenant $(\mu,\nu,\chi,k)\in\wedge^*\times\wedge^*\times(\wedge^*)^{\KC}\times\Z$ tel que $\mathcal{L}_{\mu,\nu,\chi,k}$ est le fibré $\KC$-linéarisé trivial. Si on oublie l'action de $\KC$, cela signifie que les fibrés $\mathcal{L}_{\mu}$, $\mathcal{L}_{\nu}$ et $\mathcal{L}_{\chi,k}$ sont les fibrés triviaux. Par conséquent, $\mu$ et $\nu$ sont forcément des caractères de $\KC$, c'est-à-dire $\mu$, $\nu\in(\wedge^*)^{\KC}$, et $k=0$. Dans ce cas, on obtient l'égalité $\mathcal{L}_{\mu,\nu,\chi,k} = X_M\times \C_{-(\mu+\nu+\chi)}$ comme fibré $\KC$-linéarisé. Ce dernier est trivial si et seulement si $\mu+\nu+\chi = 0$. On en déduit alors directement que le noyau de l'application linéaire \eqref{eq:application_Qlin_surjective_sur_gpPicardKC} est $\{(-\mu-\chi,\mu,\chi,0); \mu,\chi\in(\wedge^*_{\Q})^{\KC}\}$.
\end{proof}

\begin{rema}
En regardant le noyau de l'application linéaire \eqref{eq:application_Qlin_surjective_sur_gpPicardKC} donnée dans l'énoncé de la Proposition \ref{prop:surjection_grpPicardKC}, on se rend compte que la donnée de $\chi$ est superflue, puisqu'elle est déjà incluse dans les deux premières variables $\mu$ et $\nu$. Il suffira dans la suite de considérer l'application linéaire
\begin{equation}
\label{eq:application_Qlin_surjective_sur_gpPicardKC_v2}
\begin{array}{cccc}
\pi^{\KC}:&\wedge^*_{\Q}\times\wedge^*_{\Q}\times\Q & \longrightarrow & \PicG(X_M)_{\Q} \\
&(\mu,\nu,r) & \longmapsto & \mathcal{L}_{\mu,\nu,r} := \mathcal{L}_{\mu}\boxtimes\mathcal{L}_{\nu}\boxtimes\mathcal{L}_{r}.
\end{array}
\end{equation}
L'application $\pi^{\KC}$ est elle aussi surjective, et son noyau est $\ker(\pi^{\KC}) = \{(\mu,-\mu,0); \mu\in(\wedge^*_{\Q})^{\KC}\}$.
\end{rema}

\subsection{Cône semi-ample de la variété $X_{E\oplus\C}$}

Nous allons tout particulièrement nous intéresser au cas d'une représentation $M = E\oplus \C$, où $\zeta:\KC\rightarrow GL_{\C}(E)$ est une représentation complexe de $\KC$ et où $\KC$ agit trivialement sur le facteur $\C$. Dans l'énoncé ci-dessous, $\wedge^*_{\Q,\scriptscriptstyle ++}$ désignera l'ensemble des éléments dominants et $\KC$-réguliers de $\wedge^*_{\Q}$, et $\C_{\leqslant d}[E]$ le $\KC$-module des polynômes de degré inférieur ou égal à $d$ sur $E$. Rappelons que $\Pi_{\Q}(E)$ a été défini en \eqref{eq:gammaQ_version_K} par
\[
\Pi_{\Q}(E) = \left\{ (\mu,\nu)\in(\wedge^*_{\Q,+})^2 \left|
\begin{array}{l}
\exists n\geqslant 1\text{ tel que } (n\mu,n\nu)\in(\wedge^*)^2 \\
\mbox{et }\left(V_{n\mu}^*\otimes V_{n\nu}^*\otimes \C[E]\right)^{\KC} \neq 0
\end{array}\right.\right\}.
\]

\begin{prop}
\label{prop:Csa_deféquivalente_avec_irrép}
Soit $(\mu,\nu,r)\in\wedge^*_{\Q,+}\times\wedge^*_{\Q,+}\times\Q_{\geqslant 0}$, et $\mathrm{pr}:\wedge^*_{\Q}\times\wedge^*_{\Q}\times\Q\rightarrow\wedge^*_{\Q}\times\wedge^*_{\Q}$ la projection linéaire canonique.
\begin{enumerate}
\item On a $\mathcal{L}_{\mu,\nu,r}\in C_{\Q}(\XEC)^+$ si et seulement s'il existe $n\geqslant 1$ tel que l'on a $(n\mu,n\nu,n\chi)\in(\wedge^*)^3$, $nr\in\Z$ et
\[
\left(V_{n\mu}^*\otimes V_{n\nu}^*\otimes \C_{\leqslant nr}[E]\right)^{\KC}\neq 0.
\]
\item L'ensemble $C_{\Q}(\XEC)^+$ est l'adhérence de $C_{\Q}(\XEC)^{++}$ dans $\PicGQ(\XEC)$.
\item On a $\Pi_{\Q}(E) = \mathrm{pr}\left((\pi^{\KC})^{-1}(C_{\Q}(\XEC)^+)\right)$.
\end{enumerate}
\end{prop}

\begin{proof}
Par définition, pour $(\mu,\nu,r)\in\wedge^*_{\Q,+}\times\wedge^*_{\Q,+}\times\Q_{\geqslant 0}$, le fibré en droites $\mathcal{L}_{\mu,\nu,r}$ est dans le cône semi-ample $C_{\Q}(\XEC)^+$ si et seulement s'il existe un entier $n\geqslant 1$ tel que $(n\mu,n\nu,nr)\in\wedge^*\times\wedge^*\times\Z$ et $\Xss(\mathcal{L}_{n\mu,n\nu,nr}) \neq \emptyset$. Cette dernière assertion est équivalente à ce que l'espace vectoriel $\mathrm{H}^0(X,\mathcal{L}_{n\mu,n\nu,nr}^{\otimes m})^{\KC}$ contienne un élément non nul, pour un certain $m\geqslant 1$. Ceci provient de la définition de l'ensemble $\Xss(\mathcal{L}_{n\mu,n\nu,nr})$. Quitte à multiplier $n$ par $m$, on peut supposer que l'espace vectoriel $\mathrm{H}^0(X,\mathcal{L}_{n\mu,n\nu,nr})^{\KC}$ est non nul. Le $\KC$-module $\mathrm{H}^0\left(\mathbb{P}(E\oplus\C),\mathcal{L}_{nr}\right)$ est isomorphe au $\KC$-module $\C_{\leqslant nr}[E]$ des polynômes sur $E$ de degré inférieur ou égal à $nr$ (qui est bien un entier positif, car $\mathcal{L}_{nr}$ est semi-ample dans $\Pic(\mathbb{P}(E\oplus\C))$).

D'après le théorème de Borel-Weil, nous savons que pour tout $\nu\in\wedge^*_+$, $\mathrm{H}^0(X,\mathcal{L}_{n\nu})$ est isomorphe, en tant que $\KC$-module, à $V_{n\nu}^*$. Nous obtenons donc
\[
\mathrm{H}^0(X,\mathcal{L}_{n\mu,n\nu,nr}) \cong V_{n\mu}^*\otimes V_{n\nu}^*\otimes \C_{\leqslant nr}[E].
\]
Ceci prouve la première assertion.

Par définition, nous avons $C_{\Q}(\XEC)^{++} = C_{\Q}(\XEC)^+ \cap \pi^{\KC}(\wedge^*_{\Q,\scriptscriptstyle ++}\times\wedge^*_{\Q,\scriptscriptstyle ++}\times\Q_{>0})$. On remarque que, pour tout $\mu\in\wedge^*_{\Q,\scriptscriptstyle ++}$, tout $r\in\Q_{>0}$ et tout $n\in\N^*$ tels que $(n\mu,nr)\in\wedge^*\times\Z$, on a $(V_{n\mu}^*\otimes V_{n(-w_0\mu)}^*\otimes\C_{\leqslant nr}[E])^{\KC}\neq 0$, car il contient le sous-espace $(V_{n\mu}^*\otimes V_{n(-w_0\mu)}^*\otimes\C)^{\KC}\neq 0$, où $w_0$ est le plus élément du groupe de Weyl de $K$ relativement à $T$. Par conséquent, $\mathcal{L}_{\mu,-w_0\mu,r}$ est dans $C_{\Q}(\XEC)^{++}$. Et la seconde assertion découle du Théorème \ref{theo:CsaX_conepolyconvfermé_et_adhérencedeCaX}.

Enfin, d'après la définition de $\Pi_{\Q}(E)$ et grâce à l'assertion (1), il est clair que l'image de $(\pi^{\KC})^{-1}(C_{\Q}(\XEC)^+)$ par la projection $\mathrm{pr}$ est exactement l'ensemble $\Pi_{\Q}(E)$.
\end{proof}

\begin{rema}
\label{rema:GammaQ(E)_est_coneconvpolyfermé}
D'après le Théorème \ref{theo:CsaX_conepolyconvfermé_et_adhérencedeCaX}, nous savons que $C_{\Q}(\XEC)^+$ est un cône convexe polyédral fermé de $\PicG(\XEC)$. Puisque $\pi^{\KC}$ est une projection linéaire entre espaces de dimensions finies, $(\pi^{\KC})^{-1}(C_{\Q}(\XEC)^+)$ est également un cône convexe polyédral fermé de $\wedge^*_{\Q}\times\wedge^*_{\Q}\times\Q$. On en déduit, par la Proposition \ref{prop:Csa_deféquivalente_avec_irrép} ($3$), que $\Pi_{\Q}(E)$ est un cône convexe polyédral fermé de $\wedge^*_{\Q}\times\wedge^*_{\Q}$, comme projection linéaire du cône convexe polyédral fermé $(\pi^{\KC})^{-1}(C_{\Q}(\XEC)^+)$.
\end{rema}

\begin{rema}
\label{rema:cns_C_Q(XEC)intérieurnonvide}
Dans la preuve de l'assertion ($2$) de la Proposition \ref{prop:Csa_deféquivalente_avec_irrép}, nous avons en fait prouvé que $(\pi^{\KC})^{-1}(C_{\Q}(\XEC)^+)$ est non vide, donc $\Pi_{\Q}(E)$ est lui aussi non vide. Cette dernière propriété se voit facilement puisque $(0,0)$ appartient à $\Pi_{\Q}(E)$. En outre, puisque $\pi^{\KC}$ est une application linéaire entre deux espaces vectoriels de dimensions finies, il est clair que $C_{\Q}(\XEC)^+$ est d'intérieur non vide dans $\PicGQ(\XEC)$ si et seulement si le cône convexe $(\pi^{\KC})^{-1}(C_{\Q}(\XEC)^+)$ est d'intérieur non vide dans $\wedge^*_{\Q}\times\wedge^*_{\Q}\times\Q$.
\end{rema}

Nous allons décrire plusieurs propriétés géométriques de l'ensemble $\Pi_{\Q}(E)$. La plus importante pour notre étude donne une condition équivalente à ce que $\Pi_{\Q}(E)$ soit d'intérieur non vide. Cette propriété sera en effet très utile pour déterminer un ensemble fini d'équations définissant le polyèdre convexe $C_{\Q}(\XEC)^+$ (cf section \ref{section:redondance_équations}).

%

Soit $U$ le radical unipotent de $B$ et $U^-$ le radical unipotent du sous-groupe de Borel $B^-$ opposé à $B$ et contenant $\TC$. On considère l'action de $\KC\times\KC$ sur $\KC\times E$ définie, pour tout $(k_1,k_2)\in \KC\times \KC$ et tout $(k,v)\in G\times E$, par
\[
(k_1,k_2)\cdot (k,v) := (k_1kk_2^{-1},\zeta(k_1)v),
\]
ainsi que l'action induite sur $\C[\KC\times E]$. Comme $\TC\times\TC$ normalise le sous-groupe $U\times U^-$ de $\KC\times\KC$, la dernière action induit une action de $\TC\times\TC$ sur $\C[\KC\times E]^{U\times U^-}$.

Nous noterons $\Pi_{\Z}(E)$ l'ensemble des points entiers de $\Pi_{\Q}(E)$.

\begin{theo}
\label{theo:GammaQ_intérieurnonvide}
Soient $E$ une représentation complexe de $\KC$, associée au morphisme de groupes $\zeta:\KC\rightarrow GL_{\C}(E)$, et $(\mu,\nu)\in(\wedge_+^*)^2$.
\begin{enumerate}
\item $\Pi_{\Q}(E) = \Q_{>0}\cdot\Pi_{\Z}(E)$.
\item L'intérieur de $\Pi_{\Q}(E)$ est non vide dans $\wedge_{\Q}^*\times\wedge_{\Q}^*$ si et seulement si $\ker(\zeta)$ est fini.
\end{enumerate}
\end{theo}

La démonstration du point (2) nécessitera l'utilisation du lemme suivant.

\begin{lemm}
\label{lemm:codim(PiQ(X))=dim(ker_action_TC)}
Soit $X$ une variété algébrique affine sur $\C$ sur laquelle agit algébriquement un tore $\TC$. Notons $\Pi_{\Z}(X)$ le monoïde des poids de $\TC$ dans $\C[X]$. Alors la codimension de $\Pi_{\Q}(X):=\Q_{>0}\cdot\Pi_{\Z}(X)$ dans $\wedge_{\Q}^*$ est égale à la dimension du noyau de l'action de $\TC$ sur la variété $X$.
\end{lemm}

\begin{proof}[Preuve du Lemme \ref{lemm:codim(PiQ(X))=dim(ker_action_TC)}]
Nous identifions toujours le réseau des poids $\wedge^*$ avec le groupe des caractères de $\TC$.

Soit $\lambda$ un sous-groupe à un paramètre de $\TC$ dont l'action est triviale sur $X$ (c'est-à-dire son image est dans le noyau de l'action de $\TC$ sur $X$). Alors clairement $\lambda$ agit trivialement sur l'algèbre $\C[X]$ des fonctions régulières sur $X$. Par conséquent, pour tout poids $\mu$ de $\Pi_{\Z}(X)$, on a $\langle\lambda,\mu\rangle = 0$.

Réciproquement, soit $\lambda$ un sous-groupe à un paramètre de $\TC$ tel que $\langle\lambda,\mu\rangle=0$ pour tout poids $\mu$ de $\Pi_{\Z}(X)$. Cela signifie donc que $\lambda$ laisse fixe toute fonction régulière de $\C[X]$. Or, comme $X$ est affine, $\C[X]$ sépare les points. On en déduit que l'action de $\lambda$ sur $X$ fixe tous les points de $X$. D'où l'image de $\lambda$ est contenue dans le noyau de l'action de $\TC$.

On en conclut que la dimension du noyau de l'action de $\TC$ sur $X$ est égale à la dimension de l'orthogonal de $\Pi_{\Q}(X)$, qui est elle-même égale à la codimension de $\Pi_{\Q}(X)$ dans $\wedge_{\Q}^*$.
\end{proof}

Nous pouvons maintenant commencer la preuve du théorème.

\begin{proof}[Preuve du Théorème \ref{theo:GammaQ_intérieurnonvide}]
La première assertion découle directement du fait que $\Pi_{\Q}(E)$ est un cône dans $(\wedge^*_{\Q})^2$, d'après la Remarque \ref{rema:GammaQ(E)_est_coneconvpolyfermé}.

Le reste de la preuve est inspirée de la preuve de \cite[Proposition 3]{lrcone}. Par le théorème de Frobenius, le $\KC\times\KC$-module $\C[\KC]$ se décompose en une somme directe de sous-espaces stables
\[
\C[\KC] = \bigoplus_{\nu\in\wedge_+^*}V_{\nu}\otimes V^*_{\nu},
\]
où, pour tout $\nu\in\wedge_+^*$, l'action de $\KC\times\KC$ sur $V_{\nu}\otimes V^*_{\nu}$ est définie par 
\[
(k_1,k_2)\cdot v\otimes\varphi = (k_1\cdot v)\otimes(k_2\cdot\varphi),
\]
pout tout $(k_1,k_2)\in\KC\times\KC$ et tout $(v,\varphi)\in V_{\nu}\times V^*_{\nu}$, et étendue par linéarité à tout $V_{\nu}\otimes V^*_{\nu}$.

Par conséquent, nous obtenons
\begin{align*}
\C[\KC\times E]^{U\times U^-} & = (\C[\KC]\otimes \C[E])^{U\times U^-} = \Big(\bigoplus_{\nu\in\wedge^*_+} V_{\nu}\otimes V_{\nu}^*\otimes \C[E]\Big)^{U\times U^-} \\
& = \bigoplus_{\nu\in\wedge^*_+}(V_{\nu}\otimes\C[E])^U\otimes V_{\nu}^{*U^-},
\end{align*}
où $\TC$ agit sur $V_{\nu}^{*U^-}$ par le caratère $-\nu$. En ce qui concerne l'action de $\TC\times\TC$, chaque sous-espace $(V_{\nu}\otimes\C[E])^U\otimes V_{\nu}^{*U^-}$ est stable, avec la première composante de $\TC\times\TC$ agissant sur $(V_{\nu}\otimes\C[E])^U$, la seconde sur $V_{\nu}^{*U^-}$.

On en déduit que l'ensemble des poids de $\TC\times \TC$ sur $\C[\KC\times E]^{U\times U^-}$ est égal à l'ensemble des couples $(\mu,-\nu)\in(\wedge^*)^2$ où $(\mu,\nu)\in(\wedge^*_+)^2$ sont tels que $\mu$ est un poids de $\TC$ sur $(V_{\nu}\otimes\C[E])^U$, autrement dit, $V_{\mu}\subset V_{\nu}\otimes\C[E]$. Et cette assertion est équivalente à $(V_{\mu}^*\otimes V_{\nu}\otimes\C[E])^{\KC}\neq 0$, c'est-à-dire, $(\mu,-w_0\nu)\in\Pi_{\Z}(E)$.

L'application $l:(\mu,\nu)\in(\wedge^*_{\Q})^2\mapsto(\mu,w_0\nu)\in(\wedge^*_{\Q})^2$ est clairement un isomorphisme $\Q$-linéaire, qui induit une application $\Z$-linéaire bijective de $(\wedge^*)^2$ dans lui-même. Ainsi, l'ensemble des poids de $\TC\times \TC$ dans $\C[\KC\times E]^{U\times U^-}$ est égal à l'ensemble $l(\Pi_{\Z}(E))$ des points entiers de $l(\Pi_{\Q}(E))$.

Maintenant, puisque $\Pi_{\Q}(E) = \Q_{>0}\cdot\Pi_{\Z}(E)$ et $l$ est un isomorphisme, la codimension de $\Pi_{\Q}(E)$ dans $(\wedge^*_{\Q})^2$ est égal à la dimension du noyau de l'action de  $\TC\times\TC$ sur la variété $(\KC\times E)\quotient(U\times U^-)$, d'après le Lemme \ref{lemm:codim(PiQ(X))=dim(ker_action_TC)}. Par conséquent, $\Pi_{\Q}(E)$ est d'intérieur non vide dans $(\wedge^*)^2$ si et seulement si le noyau de l'action de $\TC\times\TC$ sur $(\KC\times E)\quotient(U\times U^-)$ est fini.

Les actions de $U$ et $U^-$, induites par l'action de $\KC\times\KC$ sur $\KC\times E$, commutent. On a donc
\[
\C[\KC\times E]^{U\times U^-} = \left((\C[\KC]\otimes \C[E])^{U^-}\right)^U = \left(\C[\KC]^{U^-}\otimes \C[E]\right)^U,
\]
puisque $U^-$ agit trivialement sur $E$ et donc aussi sur $\C[E]$. Cependant, $BU^-$ est ouvert dans $\KC$, donc $\C[\KC]^{U^-}$ s'injecte canoniquement dans $\C[B]$ en tant que $B$-modules, pour l'action à gauche de $B$. En général, ces deux algèbres ne sont pas isomorphes, sauf si $\KC$ est un tore. Cependant, leurs corps de fractions sont isomorphes, puisque $B$ s'identifie à un ouvert affine de $\KC/U$. On en déduit que $\C(\KC\times E)^{U\times U^-}$ est isomorphe à $\C(B\times E)^U$ comme $\TC\times\TC$-module.

De plus, $B$ est le produit semi-direct des groupes $U$ et $\TC$, où $\TC$ normalise $U$. Donc il existe un isomorphisme $\TC\times\TC$-équivariant canonique entre $\C(B\times E)^U$ et $\C(\TC\times E)$, où $\TC\times\TC$ agit sur $\TC\times E$ par une action analogue à celle de $\KC\times \KC$ sur $\KC\times E$. D'où, $(\KC\times E)\quotient(U\times U^{-})$ est birationnellement $\TC\times\TC$-équivariant à la variété $\TC\times E$.

Maintenant, on voit clairement que le noyau de l'action de $\TC\times\TC$ sur $\TC\times E$ est fini si et seulement si le noyau de l'action de $\TC$ sur $E$ est fini, c'est-à-dire, $\ker(\zeta|_{\TC})$ est fini. En effet, comme $\TC$ est abélien, on peut aisément prouver que ces deux noyaux sont canoniquement isomorphes.

Il reste donc à prouver que le groupe $\ker\zeta$ est fini si et seulement si $\ker(\zeta|_{\TC})$ est fini. La première implication est évidente.

Supposons donc que $\ker(\zeta|_{\TC})$ est fini, autrement dit, $\ker(\zeta)\cap \TC$ est fini. Soit $S$ un tore maximal de $\ker\zeta$ et $\tilde{S}$ un tore maximal de $\KC$ contenant $S$. Puisque $\KC$ est réductif, $\tilde{S}$ est conjugué à $\TC$, donc $S$ est conjugué à un sous-tore de $\TC$. Or $\ker\zeta$ est un sous-groupe distingué de $\KC$, par conséquent $S$ est conjugué à un sous-groupe de $\ker\zeta\cap \TC$, qui est fini. Donc $S$ est trivial et $\ker\zeta$ est fini. Ceci conclut la preuve de l'assertion ($3$).
%
\end{proof}

Comme conséquence du théorème précédent, nous obtenons une condition nécessaire et suffisante pour que $C_{\Q}(\XEC)^+$ soit d'intérieur non vide. C'est l'énoncé du corollaire suivant. Ce corollaire est lui-même conséquence de la propriété suivante de $C_{\Q}(\XEC)^+$.

\begin{lemm}
\label{lemm:propriétédeCQ(XEC)_danslacomposante_r}
Pour tout triplet $(\mu,\nu,r)\in (\pi^{\KC})^{-1}(C_{\Q}(\XEC)^+)$ et tout $r'\in\Q$, si $r'\geqslant r$ alors $(\mu,\nu,r')$ est aussi dans $(\pi^{\KC})^{-1}(C_{\Q}(\XEC)^+)$.
\end{lemm}

\begin{proof}
%
Soit $(\mu,\nu,r)$ appartiennant à $(\pi^{\KC})^{-1}(C_{\Q}(\XEC)^+)$, alors il existe un entier $n\geqslant 1$ tel que $(n\mu,n\nu)\in(\wedge^*)^2$, $nr\in\Z$ et $(V_{n\mu}^*\otimes V_{n\nu}^*\otimes\C_{\leqslant nr}[E])^{\KC}\neq 0$. Il est clair que, pour tout $m\geqslant 1$, le $\KC$-module $V_{n\mu}^*\otimes V_{n\nu}^*\otimes\C_{\leqslant nr}[E]$ est contenu dans $V_{n\mu}^*\otimes V_{n\nu}^*\otimes\C_{\leqslant nm(r+1)}[E]$. Par conséquent, $(V_{n\mu}^*\otimes V_{n\nu}^*\otimes\C_{\leqslant nm(r+1)}[E])^{\KC}\neq 0$. Donc, pour tout $d\geqslant 1$, $(\mu,\nu,m(r+1))$ est aussi dans $(\pi^{\KC})^{-1}(C_{\Q}(\XEC)^+)$. De la convexité de $(\pi^{\KC})^{-1}(C_{\Q}(\XEC)^+)$ dans $\wedge^*_{\Q}\times\wedge^*_{\Q}\times\Q$, comme $r+1>0$, on en déduit que $(\mu,\nu,r')$ appartient à $(\pi^{\KC})^{-1}(C_{\Q}(\XEC)^+)$ pour tout rationel $r'\geqslant r$.
\end{proof}

\begin{coro}
\label{coro:côneample_intérieurnonvide}
Les cônes convexes polyédraux $C_{\Q}(\XEC)^+$ et $C_{\Q}(\XEC)^{++}$ sont d'intérieur non vide si et seulement si $\ker\zeta$ est fini.
\end{coro}

\begin{proof}
Tout d'abord, il faut remarquer que $C_{\Q}(\XEC)^{++}$ est d'intérieur non vide si et seulement si $C_{\Q}(\XEC)^+$ est d'intérieur non vide, puisque $C_{\Q}(\XEC)^+$ est l'adhérence de $C_{\Q}(\XEC)^{++}$ dans  $\PicGQ(\XEC)$, d'après la Proposition \ref{prop:Csa_deféquivalente_avec_irrép} ($2$).

De plus, d'après le Théorème \ref{theo:GammaQ_intérieurnonvide} et la Remarque \ref{rema:GammaQ(E)_est_coneconvpolyfermé}, il est suffisant de prouver que le cône convexe polyédral $(\pi^{\KC})^{-1}(C_{\Q}(\XEC)^+)$ est d'intérieur non vide dans $\wedge^*_{\Q}\times\wedge^*_{\Q}\times\Q$ si et seulement si $\Pi_{\Q}(E)$ est lui-même d'intéreur non vide dans $\wedge^*_{\Q}\times\wedge^*_{\Q}$.

Tout d'abord, on sait que $\Pi_{\Q}(E)$ est la projection linéaire de l'ensemble convexe $(\pi^{\KC})^{-1}(C_{\Q}(\XEC)^+)$, d'après l'assertion (1) de la Proposition \ref{prop:Csa_deféquivalente_avec_irrép}. Par conséquent, si $(\pi^{\KC})^{-1}(C_{\Q}(\XEC)^+)$ est d'intérieur non vide dans $\wedge^*_{\Q}\times\wedge^*_{\Q}\times\Q$, alors l'intérieur de $\Pi_{\Q}(E)$ est aussi non vide.

Réciproquement, supposons que l'intérieur du cône convexe $(\pi^{\KC})^{-1}(C_{\Q}(\XEC)^+)$ est vide. 
Comme il s'agit d'un ensemble convexe, il est donc contenu dans un certain hyperplan $\mathcal{H}$, défini par l'équation
\[
f_1(\xi_1)+f_2(\xi_2)+ax = 0, \qquad (\xi_1,\xi_2,x)\in\wedge^*_{\Q}\times\wedge^*_{\Q}\times\Q,
\]
avec $f_1,f_2$ deux formes linéaires rationnelle sur $\wedge^*_{\Q}$ et $a\in\Q$.

Mais, d'après le Lemme \ref{lemm:propriétédeCQ(XEC)_danslacomposante_r}, si $(\mu,\nu,r)$ appartient à $(\pi^{\KC})^{-1}(C_{\Q}(\XEC)^+)$, alors pour tout rationnel $r'\geqslant r$, on doit aussi avoir $(\mu,\nu,r')\in (\pi^{\KC})^{-1}(C_{\Q}(\XEC)^+)$. Cela donne donc
\[
f_1(\mu)+f_2(\nu)+ar = 0 = f_1(\mu)+f_2(\nu)+ar',
\]
pour tout rationnel $r'\geqslant r$. On en déduit que $a=0$, puisque $(\pi^{\KC})^{-1}(C_{\Q}(\XEC)^+)$ est non vide d'après la Remarque \ref{rema:cns_C_Q(XEC)intérieurnonvide}. Alors, au moins une des deux formes linéaires $f_1$ ou $f_2$ est non nulle. De plus, tout $(\mu,\nu)\in\Pi_{\Q}(E)$ vérifie l'équation
\[
f_1(\mu)+f_2(\nu) = 0,
\]
puisqu'il doit exister $r\in\Q_{\geqslant 0}$ tel que $(\mu,\nu,r)$ est dans $(\pi^{\KC})^{-1}(C_{\Q}(\XEC)^+)$, d'après la Proposition \ref{prop:Csa_deféquivalente_avec_irrép} ($3$). On en conclut que $\Pi_{\Q}(E)$ est inclus dans un hyperplan de $\wedge^*_{\Q}\times\wedge^*_{\Q}$ et, donc, il est d'intérieur vide.
\end{proof}

\subsection{Critère numérique de Hilbert-Mumford}
\label{subsection:critère_HilbertMumford}

Fixons un fibré $\mathcal{L}$ dans $\PicG(X)$. Pour un point $x\in X$ et un sous-groupe à un paramètre $\lambda$ de $\KC$, la limite $\lim_{t\rightarrow 0} \lambda(t)\cdot x$ peut être définie comme suit: puisque $X$ est projective, l'application rationnelle $f:t\in\C^*\mapsto \lambda(t)\cdot x\in X$ peut être prolongé en une application algébrique $\tilde{f}:\mathbb{P}^1\rightarrow X$. On posera alors $\lim_{t\rightarrow 0}\lambda(t)\cdot x = \tilde{f}(0)$.

Si on note $z = \lim_{t\rightarrow 0}\lambda(t)\cdot x$, ce point $z$ est alors fixé par l'image de $\lambda$ dans $\KC$, c'est-à-dire, l'action induite de $\C^*$ sur $X$ par l'intermédiaire de $\lambda$ laisse fixe $z$. Comme $\mathcal{L}$ est $G$-linéarisé, le groupe $\C^*$ va alors agir sur la fibre $\mathcal{L}_z$. Cette action définit un élément $\mu^{\mathcal{L}}(x,\lambda)$ de $\Z$, par $\lambda(t)\cdot v = t^{-\mu^{\mathcal{L}}(x,\lambda)}v$, pour tout $t\in\C^*$ et tout $v\in\mathcal{L}_z$.

On peut facilement vérifier que les nombres $\mu^{\mathcal{L}}(x,\lambda)$ vérifient les propriétés suivantes:
\begin{enumerate}
\item $\mu^{\mathcal{L}}(g.x,g\cdot\lambda\cdot g^{-1}) = \mu^{\mathcal{L}}(x,\lambda)$, pour tout $g\in \KC$,
\item l'application $\mathcal{L}\mapsto\mu^{\mathcal{L}}(x,\lambda)$ est un morphisme de groupes de $\PicG(X)$ dans $\Z$.
\item pour tout $n\in\N$, on a $\mu^{\mathcal{L}}(x,n\lambda) = n\mu^{\mathcal{L}}(x,\lambda)$,
\end{enumerate}
où $n\lambda$ est le sous-groupe à un paramètre de $\KC$ défini par $(n\lambda)(t):=\lambda(t^n)$, pour tout $t\in\C^*$.

\begin{defi}
 Un sous-groupe à un paramètre $\lambda$ de $\KC$ est dit \emph{indivisible}\index{Sous-groupe à un paramètre indivisible} si, pour tout sous-groupe à un paramètre $\lambda'$ de $\KC$ et tout entier $n>1$, $n\lambda'$ n'est pas égal à $\lambda$.
\end{defi}

Ces nombres $\mu^{\mathcal{L}}(x,\lambda)$ donnent une caractérisation de l'ensemble des points semi-stables de $X$ pour $\mathcal{L}$. En effet, d'après \cite{mumford94}, dans le cas ample, et \cite{ressayre08} (Lemme $2$) dans le cas semi-ample,
\[
x\in\Xss(\mathcal{L}) \ \Longleftrightarrow \ \mu^{\mathcal{L}}(x,\lambda) \leqslant 0, \ \text{pour tout $\lambda$ sous-groupe à un paramètre de $\KC$}.
\]\index{Critère de Hilbert-Mumford}
On remarque que l'on aurait pu seulement considérer les sous-groupes à un paramètre indivisibles de $\KC$ dans l'équivalence écrite ci-dessus, d'après l'assertion ($3$).

\`A un sous-groupe à un paramètre $\lambda$ de $\KC$, on peut associer le sous-groupe parabolique
\[
P(\lambda) = \left\{g\in \KC; \ \lim_{t\rightarrow 0} \lambda(t)\cdot g\cdot \lambda(t)^{-1} \mbox{ existe dans } \KC\right\}.
\]\index{$P(\lambda)$}
Notons $\KC^{\lambda}$ le centralisateur de l'image de $\lambda$ dans $\KC$. Ce groupe $\KC^{\lambda}$ est un sous-groupe de Levi de $P(\lambda)$. De plus, si $g\in P(\lambda)$, alors l'élément $\overline{g} = \lim_{t\rightarrow 0}\lambda(t)\cdot g \cdot\lambda(t)^{-1}\in \KC$ est un élément de $\KC^{\lambda}$. Enfin, pour $g\in P(\lambda)$, on a $\mu^{\mathcal{L}}(x,\lambda) = \mu^{\mathcal{L}}(x,g\cdot\lambda\cdot g^{-1})$.

\bigskip

Nous aurons, dans la suite, besoin de calculer la valeur de $\mu^{\mathcal{L}}(x,\lambda)$ pour certains points $x$ de la variété produit $X_M = \KC/B\times \KC/B \times \mathbb{P}(M)$, où $M$ est une représentation complexe de $\KC$. Les trois énoncés suivants répondent à cette question. La première proposition est un résultat bien connu.

\begin{prop}
\label{prop:produit_deuxfibrés_critèrenumérique}
Si $(X_1,\mathcal{L}_1)$ et $(X_2,\mathcal{L}_2)$ sont deux variétés munies de fibrés en droites $\KC$-linéarisés, alors
\[
\mu^{\mathcal{L}_1\boxtimes\mathcal{L}_2}((x_1,x_2),\lambda) = \mu^{\mathcal{L}_1}(x_1,\lambda) + \mu^{\mathcal{L}_2}(x_2,\lambda),
\]
pour tout $(x_1,x_2)\in X_1\times X_2$.
\end{prop}

\begin{rema}
\label{rema:defi_hilbertmumford_pour_fibrés_rationnels}
En particulier, pour tout entier $n\geqslant 1$, on a $\mu^{\mathcal{L}^{\otimes n}}(x,\lambda) = n\mu^{\mathcal{L}}(x,\lambda)$ pour tout fibré en droites $\mathcal{L}\in\PicG(X)$, tout sous-groupe à un paramètre $\lambda$ de $\KC$ et tout point $x\in X$. Ceci nous permet de définir la valeur de $\mu^{\bullet}(x,\lambda)$ pour les fibrés rationnels, en posant, pour tout $\mathcal{L}\in\PicG(X)_{\Q}$,
\[
\mu^{\mathcal{L}}(x,\lambda) := \frac{\mu^{\mathcal{L}^{\otimes n}}(x,\lambda)}{n},
\]
où $n$ est un entier strictement positif tel que $\mathcal{L}^{\otimes n}\in\PicG(X)$.
\end{rema}

Pour les deux prochains lemmes, nous supposerons que $\lambda$ est maintenant un sous-groupe à un paramètre de $\TC$.

\begin{lemm}
\label{prop:critèrenumérique_KCsurB}
Soit $w$ un élément du groupe de Weyl $W$ et $\nu$ un caractère de $B$. Alors $\mu^{\mathcal{L}_{\nu}}(w^{-1}B/B, \lambda) = \langle w\lambda,\nu\rangle$.
\end{lemm}

\begin{proof}
L'élément $w^{-1}B/B$ de $\KC/B$ est un point fixe de l'action de $\TC$ sur $\KC/B$. Par conséquent, la valeur de $\mu^{\mathcal{L}_{\nu}}(w^{-1}B/B, \lambda)$ provient de l'action de $\lambda$ sur la fibre au-dessus de $w^{-1}B/B$.

Soit $z\in\C_{-\nu}$ et $t\in\C^*$. Alors, par définition de l'action de $\KC$ sur $\mathcal{L}_{\nu} = G\times_B \C_{-\nu}$, nous avons $\lambda(t)\cdot[w^{-1},z] = [\lambda(t)w^{-1},z] = [w^{-1}(w\lambda(t)w^{-1}),z]$. Or, $w$ est un élément du groupe de Weyl $W = W(\KC;\TC)$ et $\lambda(t)$ est un élément de $\TC$, donc $w\lambda(t)w^{-1}$ est dans $\TC\subset B$. On en déduit
\[
\lambda(t)\cdot[w^{-1},z] = [w^{-1},(w\lambda(t)w^{-1})\cdot z] = [w^{-1},t^{\langle w\lambda,-\nu\rangle}z] = [w^{-1},t^{-\langle w\lambda,\nu\rangle}z],
\]
ce qui signifie que $\mu^{\mathcal{L}_{\nu}}(w^{-1}B/B, \lambda) = \langle w\lambda,\nu\rangle$.
\end{proof}

\begin{lemm}
\label{prop:critèrenumérique_ProjectifdeE}
Soit $M$ une représentation complexe non nulle de $\KC$. Soient $v\in M$ non nul et $n\in\Z$ tels que, pour tout $t\in\C^*$, on ait $\lambda(t)\cdot v = t^n v$. Alors, pour tout $k>0$, on a $\mu^{\mathcal{L}_k}([v],\lambda) = nk$, où $[v]$ est la classe de $v$ dans $\mathbb{P}(M)$.
\end{lemm}

\begin{proof}
Puisque $\lambda(t)\cdot v = t^n v$ pour tout $t\in\C^*$, le point $[v]$ de $\mathbb{P}(M)$ est fixé par $\lambda(\C^*)$. Il suffit donc de déterminer l'action de $\lambda(\C^*)$ sur la fibre au-dessus de $[v]$ dans $\mathcal{L}_k$. Soient $z\in\C$ et $t\in\C^*$. Alors
\[
\lambda(t)\cdot[v,z] = [\lambda(t)\cdot v, z] = [t^n v,z].
\]
Or, par définition de l'action de $\C^*$ sur $M\setminus\{0\}\times \C_k$ utilisée pour définir $\mathcal{L}_k = M\setminus\{0\}\times_{\C^*}\C_k$, nous avons $[a v, a^k z] = [v,z]$ pour tout $a\in\C^*$. Par conséquent, nous avons aussi $[a v, z] = [v, a^{-k}z]$ pour tout $a\in\C^*$. Et en particulier pour $a = t^n \neq 0$, cela donne
\[
\lambda(t)\cdot[v,z] = [t^n v,z] = [v, (t^n)^{-k}z] = [v, t^{-nk}z].
\]
On en conclut que $\mu^{\mathcal{L}_k}([v],\lambda) = nk$.
\end{proof}

\subsection{Paires bien couvrantes}
\label{subsubsection:gpalg_git_pairesbiencouvrantes}

Cette notion a été introduite par Ressayre dans \cite{ressayre08}. Elle permet de donner une condition nécessaire et suffisante pour qu'un élément $\mathcal{L}\in\PicG(X)^{++}$ (resp. $\mathcal{L}\in\PicG(X)^+$) soit dans le cône ample $\Ca$ (resp. cône semi-ample $\Csa$), en termes d'inégalités linéaires.

Nous supposerons dans ce paragraphe que $X$ est une variété projective lisse. Cette propriété sera vérifiée pour toutes les variétés que nous considérerons dans la suite, puisque nous ne rencontrerons que des produits de variétés des drapeaux.

Dans la suite, si $\lambda$ est un sous-groupe à un paramètre de $\KC$, on notera $X^{\lambda}$ l'ensemble des points de $X$ fixés par le sous-groupe $\lambda(\C^*)$ de $\KC$.

\begin{defi}
Soit $\lambda$ un sous-groupe à un paramètre de $\KC$ et $C$ une composante irréductible de $X^{\lambda}$. Soit $C^+ := \{x\in X\ ; \ \lim_{t\rightarrow 0} \lambda(t)x \in C\}$. Considérons l'application $\KC$-équivariante suivante :
\[
\begin{array}{cccc}
\eta : & \KC\times_{P(\lambda)} C^+ & \longrightarrow & X \\
& [g,x] & \longmapsto & g\cdot x.
\end{array}
\]
La paire $(C,\lambda)$\index{$(C,\lambda)$} est dite \emph{couvrante}\index{Paire couvrante} (resp. \emph{dominante}\index{Paire dominante}) si $\eta$ est une application birationnelle (resp. dominante). La paire $(C,\lambda)$ sera \emph{bien couvrante}\index{Paire bien couvrante} si elle est couvrante et s'il existe un ouvert $P(\lambda)$-stable $\Omega$ de $C^+$ intersectant $C$ tel que $\eta$ induise un isomorphisme de $\KC\times_{P(\lambda)}\Omega$ sur un ouvert de $X$.
\end{defi}

Pour un fibré $\mathcal{L}\in\PicG(X)$ et un sous-groupe à un paramètre $\lambda$ de $\KC$ fixés, l'application $x\mapsto\mu^{\mathcal{L}}(x,\lambda)$ est à valeur dans $\Z$. Pour une composante irréductible $C$ de $X^{\lambda}$, la valeur de $\mu^{\mathcal{L}}(x,\lambda)$ provient de l'action rationnelle de $\C^*$ sur la fibre au-dessus de $x\in C$. La valeur de $\mu^{\mathcal{L}}(x,\lambda)$ ne dépend pas de l'élément $x$ choisi dans $C$, ce qui permet de définir le nombre $\mu^{\mathcal{L}}(C,\lambda)$\index{$\mu^{\mathcal{L}}(C,\lambda)$}. Remarquons, de plus, que $\mu^{\mathcal{L}}(x,\lambda)$ ne dépendra pas non plus de $x\in C^+$ et vaudra aussi $\mu^{\mathcal{L}}(C,\lambda)$.

\begin{lemm}[\cite{ressayre08}, Lemme 3]
\label{lemm:ressayre_lemme3}
Soit $(C,\lambda)$ une paire dominante et $\mathcal{L}\in\Csa$. Alors $\mu^{\mathcal{L}}(C,\lambda)\leqslant 0$.
\end{lemm}

Ce premier résultat ne donne qu'une condition nécessaire pour qu'un fibré en droites soit dans le cône semi-ample : si un fibré $\mathcal{L}\in\PicG(X)_{\Q}^{++}$ appartient à $\Csa$, alors $\mu^{\mathcal{L}}(C,\lambda)\leqslant 0$ pour toute paire dominante $(C,\lambda)$. La proposition suivante assure que la réciproque est vraie, grâce à l'usage primordial des paires bien couvrantes.

\begin{theo}[\cite{ressayre08}, Proposition 4]
\label{theo:ressayre_equations_coneample}
Supposons que $X$ est une variété projective lisse. Alors le cône ample $\Ca$ (resp. semi-ample $\Csa$) est l'ensemble des $\mathcal{L}\in\PicG(X)_{\Q}^{++}$ (resp. $\mathcal{L}\in\PicG(X)_{\Q}^+$) tels que pour toute paire bien couvrante $(C,\lambda)$, on ait $\mu^{\mathcal{L}}(C,\lambda) \leqslant 0$.
\end{theo}

%

Pour un certain type de variétés $X$, comprenant les variétés de la forme $X_M$, avec $M$ une représentation complexe de $\KC$, nous pouvons obtenir un ensemble plus petit d'équations déterminant le cône ample de $X$. Dans l'énoncé suivant, $\hKC$ est un groupe réductif complexe contenant $\KC$ et $\hat{Q}$ est un sous-groupe parabolique de $\hKC$.

\begin{theo}[\cite{ressayre08}, Théorème 3]
 \label{theo:ressayre_equations_coneample_avecconditionssup}
Soit la variété $X = \KC/B\times\hKC/\hat{Q}$. Supposons que $\Ca$ soit d'intérieur non vide dans $\PicGQ(X)$. Soit un fibré $\mathcal{L}\in\PicGQ(X)^{++}$. Alors $\mathcal{L}$ est dans $\Ca$ si et seulement si pour toute paire bien couvrante $(C,\lambda)$ de $X$ telle qu'il existe $x\in C$ vérifiant $(\KC)_x^{\circ} = \lambda(\C^*)$, on a $\mu^{\mathcal{L}}(C,\lambda)\leqslant 0$.
\end{theo}

\begin{rema}
\label{rema:selimiterauxsgprà1paramindivdominants}
 On peut noter que, pour tout sous-groupe à un paramètre $\lambda$ de $\KC$ et tout $g\in\KC$, l'ensemble des points fixes de $g\lambda g^{-1}$ dans $X$ est $X^{g\lambda g^{-1}} = g\cdot X^{\lambda}$, et $C$ est une composante irréductible de $X^{\lambda}$ si et seulement si $g\cdot C$ est une composante irréductible de $X^{g\lambda g^{-1}}$. Ainsi, nous pouvons appliquer les assertions ($1$) et ($3$) du paragraphe \ref{subsection:critère_HilbertMumford} pour montrer qu'il est suffisant de considérer, dans les énoncés des Théorèmes \ref{theo:ressayre_equations_coneample} et \ref{theo:ressayre_equations_coneample_avecconditionssup}, uniquement les paires bien couvrantes $(C,\lambda)$ avec $\lambda$ sous-groupe à un paramètre dominant indivisible de $\TC$.
\end{rema}

\section{Réduction de la redondance des équations}
\label{section:redondance_équations}

Le but principal de ce chapitre est de déterminer un ensemble d'équations déterminant complètement le polyèdre moment $\Delta_K(K\cdot\Lambda\times E)$, pour certaines représentations complexes $E$ de $K$. Le Lemme \ref{lemm:lien_DeltaKLambda_GammaQ} et la Proposition \ref{prop:Csa_deféquivalente_avec_irrép} nous amènent à considérer le cône semiample $C_{\Q}(\XEC)^+$ et ses équations. Rappelons que $\XEC = \KC/B\times\KC/B\times\mathbb{P}(E\oplus\C)$\index{$\XEC$} a été définie en section \ref{subsection:GIT_notations}.

Par le Théorème \ref{theo:ressayre_equations_coneample}, nous obtenons un ensemble d'équations qui déterminent le cône semiample $C_{\Q}(\XEC)^+$. Ses équations sont indexées par les paires bien couvrantes de la variété $\XEC$. Cependant, cet ensemble d'équations est infini, et ceci pose problème pour appliquer la projection linéaire au cône semi-ample afin d'obtenir des équations pour le cône convexe polyédral $\Pi_{\Q}(E)$. Nous allons consacrer cette section à montrer que le cône convexe polyédral $C_{\Q}(\XEC)^+$ peut finalement être décrit par seulement une partie finie des paires bien couvrantes de la variété $\XEC$. Ici, nous porterons tout particulièrement notre attention aux sous-groupes à un paramètre.

Dans cette section, $K$ désignera un groupe de Lie réel compact connexe, $T$ un tore maximal de $K$, $\KC$ et $\TC$ leurs complexifiés respectifs avec $\TC\subset\KC$, $B$ un sous-groupe de Borel de $\KC$ contenant $\TC$ et $\zeta:\KC\rightarrow GL_{\C}(E)$ une représentation algébrique complexe de $\KC$ \textbf{de noyau fini}. Nous allons étudier les paires bien couvrantes de la variété $\XEC$.

\subsection{Description des paires de $\XEC$}

Soit $\beta\in\WT(E)$ un poids de $\TC$ dans $E$. On définit le sous-espace de poids associé
\[
E_{\beta} := \{v\in E\,; \ d\zeta(H)v = \beta(H)v, \text{ pour tout $H\in \got{t}_{\C}$}\},
\]
et, pour tout $k\in\Z$ et tout $\lambda$ sous-groupe à un paramètre de $\TC$,
\[
E_{\lambda,k} := \{ v\in E\,; \lambda(t)\cdot v= t^kv, \forall t\in\C^*\}.
\]
On peut remarquer que $E_{\lambda,0}=E^{\lambda}$ est le sous-espace de $E$ des vecteurs fixés par $\lambda$.

Fixons maintenant un sous-groupe à un paramètre $\lambda$ dominant de $\TC$. Nous noterons à nouveau $W$ le groupe de Weyl $W(\KC;\TC)$, $P=P(\lambda)$ le groupe parabolique associé à $\lambda$ dans $\KC$ et $W_{\lambda}$ le groupe de Weyl du Levi $\KC^{\lambda}$ de $P$.

Il est clair que $X^{\lambda} = (\KC/B)^{\lambda}\times (\KC/B)^{\lambda} \times \mathbb{P}(E)^{\lambda}$. Les deux premiers termes sont de la forme
\[
(\KC/B)^{\lambda} = \bigcup_{w\in W_{\lambda}\backslash W}\KC^{\lambda}wB/B.
\]
Nous aurons également, de manière évidente, $\mathbb{P}(E)^{\lambda} = \bigcup_{m\in\Z} C_m$, où $C_m = \mathbb{P}(E_{\lambda,m})$, pour tout $m\in\Z$, et $C_0 = \mathbb{P}(\C\oplus E^{\lambda})$. Pour $(w,w',m)\in W/W_{\lambda}\times W/W_{\lambda}\times\Z$, nous définissons
\[
C(w,w',m) = \KC^{\lambda}w^{-1}B/B \times \KC^{\lambda}w'^{-1}B/B \times C_m.
\]\index{$C(w,w',m)$}
Nous mettons des $^{-1}$ pour garder la notation utilisée dans \cite{ressayre08}. Mis ensemble, cela nous donne
\[
X^{\lambda} = \bigcup_{\stackrel{w,w'\in W/W_{\lambda}}{m\in\Z}} C(w,w',m),
\]
où chaque $C(w,w',m)$ est une composante irréductible de $X^{\lambda}$.

\subsection{Sous-groupes à un paramètre admissibles}

Dans ce paragraphe, nous adaptons à nos besoins la définition de sous-groupe à un paramètre admissible donnée dans \cite[section 7.3.2]{ressayre08}.

\begin{defi}
 Soit $M$ un $\TC$-module. Un sous-tore de $\TC$ est dit $M$-\emph{admissible} s'il existe $v\in M$ tel que $S$ soit la composante neutre du stabilisateur $(\TC)_v$ du point $v$ de $M$.
\end{defi}

La proposition suivante donne une définition équivalente de sous-tore $M$-admissible.

\begin{prop}
Un sous-tore $S$ de $\TC$ est $M$-admissible si et seulement s'il est la composante neutre de l'intersection des noyaux d'une famille finie de caractères de $\TC$ de $M$.
\end{prop}

\begin{proof}
En effet, si on fixe $(u_1,\ldots,u_n)$ une base de $M$ formée de vecteurs propres communs à l'action de $\TC$ sur $M$ et si on note $\chi_i$ le caractère de $\TC$ associée à l'action de $\TC$ sur $\C u_i$, alors, il est clair que, pour tout $x=\sum_{i=1}^nx_iu_i$, on a $(\TC)_x = \bigcap_{i \text{ t.q. } x_i\neq 0}\ker\chi_i$.
\end{proof}

\begin{rema}
\label{rema:autre_déf_tore_Madmissible}
Lorsque l'on considère un sous-groupe à un paramètre $\lambda$ de $\TC$, son image $\lambda(\C^*)$ est un sous-tore de $\TC$. En conséquence, on dira qu'un sous-groupe à un paramètre $\lambda$ de $\TC$ est $M$-admissible\index{Sous-groupe à un paramètre $M$-admissible} si le sous-tore $\lambda(\C^*)$ de $\TC$ est $M$-admissible.

On peut remarquer qu'un sous-groupe à un paramètre $\lambda$ de $\TC$ s'identifie canoniquement à son générateur dans l'algèbre de Lie $\got{t}_{\C}$, que l'on notera également $\lambda$. Ainsi, $\lambda$ sera $M$-admissible si et seulement si $\C\lambda\subset\got{t}_{\C}$ est égal à l'intersection dans $\got{t}_{\C}$ des noyaux d'une famille finie de poids de $\WT(M)$.
\end{rema}

\begin{rema}
Si on utilise les notations de \cite[7.3.2]{ressayre08}, $S$ est admissible au sens de Ressayre si et seulement si $S$ est $\hat{\got{g}}/\got{g}$-admissible, où $\hat{\got{g}}/\got{g}$ est bien un $\TC$-module.
\end{rema}

\begin{rema}
\label{rema:sgrpà1paramindivMadmissible_fini}
Lorsque $M$ est un $\TC$-module de dimension finie (comme $\C$-espace vectoriel), l'ensemble $\WT(M)$ des poids de $\TC$ sur $M$ est nécessairement fini. Par conséquent, l'ensemble des sous-groupes à un paramètre indivisibles $M$-admissibles de $\TC$ sera lui aussi fini.
\end{rema}

Le morphisme de groupes algébriques $\zeta:\KC\rightarrow GL_{\C}(E)$ induit un morphisme de groupes algébriques $\zeta\oplus\id_{\C}:\KC\rightarrow GL_{\C}(E\oplus\C)$ canonique. L'algèbre de Lie $\got{gl}_{\C}(E\oplus\C)$ de $GL_{\C}(E\oplus\C)$ est donc un $\KC$-module, pour l'action induite du morphisme $\zeta\oplus\id_{\C}:\KC\rightarrow GL_{\C}(E\oplus\C)$ et de l'action adjointe de $GL_{\C}(E\oplus\C)$ sur $\got{gl}_{\C}(E\oplus\C)$. C'est donc aussi un $\TC$-module.

\begin{prop}
\label{prop:admissibilité_par_stabilisateur}
Soit $(C,\lambda)$ une paire de $\XEC$. S'il existe $x\in C$ tel que $(\KC)_x^{\circ} = \lambda(\C^*)$, alors $\lambda$ est $\got{gl}_{\C}(E\oplus\C)$-admissible.
\end{prop}

\begin{proof}
Par commodité, on pose $\hKC = GL_{\C}(E\oplus\C)$. Puisque $\mathbb{P}(E\oplus\C)$ est un espace $\hKC$-homogène, nous pouvons identifier $\XEC$ à la variété $\tilde{X} = \KC/B\times\KC/B\times\hKC/\hat{Q}$, où $\hat{Q}$ est le stabilisateur dans $\hKC$ d'une droite de $E\oplus\C$. Cette identification est $\KC$-équivariante, par l'intermédiaire du morphisme de groupes $\zeta\oplus\id_{\C}:\KC\rightarrow \hKC$. \`A la paire $(C,\lambda)$ correspond la paire $(\tilde{C},\lambda)$ de $\tilde{X}$ et à $x$ on associe le point $\tilde{x}$, qui aura forcément même stabilisateur que $x$. L'ensemble $\tilde{C}$ est une composante irréductible de $\tilde{X}^{\lambda}$, il s'agit donc d'un $\tilde{C}(w,w',\hat{w})$, pour un triplet $(w,w',\hat{w})\in W/W_{\lambda} \times W/W_{\lambda} \times \hat{W}_{\hat{Q}}\backslash\hat{W}/\hat{W}_{\hat{P}}$. Le point $\tilde{x}$ s'écrit donc
\[
\tilde{x} = (gwB/B,g'w'B/B,\hat{g}\hat{w}\hat{Q}/\hat{Q}).
\]
Nous pourrions conclure directement, en appliquant \cite[Lemme 17]{ressayre08}, si $\tilde{C}(w,w',\hat{w})$ était une variété des drapeaux complète, autrement dit, si $\hat{Q}$ était un sous-groupe de Borel de $\hKC$. Ici, ce n'est pas le cas en général (il s'agit d'un sous-groupe parabolique maximal d'un $GL$).

Soit $\hat{B}$ un sous-groupe de Borel de $\hKC$ contenant $\zeta\oplus\id_{\C}(B)$. Quitte à changer le sous-groupe à un paramètre $\hat{Q}$, on peut supposer $\hat{B}\subset\hat{Q}$. Cela est toujours possible car le sous-groupe de Borel $\hat{B}$ de $\hKC$ doit au moins avoir un point fixe dans l'espace projectif $\mathbb{P}(E\oplus\C)$. Il suffit alors de prendre pour $\hat{Q}$ le stabilisateur du point de $\mathbb{P}(E\oplus\C)$ fixé par $\hat{B}$.

Prenons maintenant $\hat{w}_1$ un élément quelconque de la double classe $\hat{w}\in\hat{W}_{\hat{Q}}\backslash\hat{W}/\hat{W}_{\hat{P}}$, et considérons la composante irréductible $\tilde{C}_{\hat{B}}(w,w',\hat{w}_1)$ de $\KC/B\times\KC/B\times\hKC/\hat{B}$. Nous allons montrer qu'il existe un élément $\tilde{x}_{\hat{B}}$ de $\tilde{C}_{\hat{B}}(w,w',\hat{w}_1)$ qui vérifie $(\KC)_{\tilde{x}_{\hat{B}}}^{\circ} = \lambda(\C^*)$.

Pour ce faire, posons $\tilde{x}_{\hat{B}} = (gw^{-1}B/B,g'w'^{-1}B/B,\hat{g}\hat{w}_1^{-1}\hat{B}/\hat{B})\in\tilde{C}_{\hat{B}}(w,w',\hat{w}_1)$. Prenons $h$ un élément de $(\KC)_{\tilde{x}_{\hat{B}}}$. Alors on a
\[
h\cdot(gw^{-1}B/B) = gw^{-1}B/B, \quad h\cdot(g'w'^{-1}B/B) = g'w'^{-1}B/B
\]
et
\[
h\cdot(\hat{g}\hat{w}_1^{-1}\hat{B}/\hat{B}) = \hat{g}\hat{w}_1^{-1}\hat{B}/\hat{B}.
\]
Ce dernier point est équivalent à $h\hat{g}\hat{w}_1^{-1} \in \hat{g}\hat{w}_1^{-1}\hat{B}$. Par conséquent, on a $h\hat{g}\hat{w}_1^{-1}\hat{Q} \in \hat{g}\hat{w}_1^{-1}\hat{Q}$, c'est-à-dire $h\cdot(\hat{g}\hat{w}_1^{-1}\hat{Q}/\hat{Q}) = \hat{g}\hat{w}_1^{-1}\hat{Q}/\hat{Q}$, car $\hat{B}\subset\hat{Q}$. On en déduit que $(\KC)_{\tilde{x}_{\hat{B}}}^{\circ} \subset (\KC)_{\tilde{x}}^{\circ}$. On obtient donc $\lambda(\C^*)\subset (\KC)_{\tilde{x}_{\hat{B}}}^{\circ} \subset (\KC)_{\tilde{x}}^{\circ} = (\KC)_x^{\circ} = \lambda(\C^*)$. Cela implique que $(\KC)_{\tilde{x}_{\hat{B}}}^{\circ} = \lambda(\C^*)$ et, forcément, $(\KC^{\lambda})_{\tilde{x}_{\hat{B}}}^{\circ} = \lambda(\C^*)$. Or, la composante irréductible $\tilde{C}_{\hat{B}}(w,w',\hat{w}_1)$ est une variété des drapeaux complète de $\KC^{\lambda}\times \KC^{\lambda}\times\hKC^{\lambda}$. On conclut la preuve grâce au Lemme 17 de \cite{ressayre08}.
\end{proof}

Le théorème suivant donne maintenant un ensemble d'équations du cône convexe polyédral $C_{\Q}(\XEC)^+$.

\begin{theo}
\label{theo:équations_pairesbiencouvrantes_admissible}
Soit $(\mu,\nu,r)\in\wedge_{\Q,+}^*\times\wedge_{\Q,+}^*\times\Q_{\geqslant 0}$. Alors $(\mu,\nu,r)$ appartient à $(\pi^{\KC})^{-1}(C_{\Q}(\XEC)^{+})$ si et seulement si
\begin{equation}
\label{eq:équation_pairebiencouvrantedonnée}
\langle w\lambda,\mu\rangle + \langle w'\lambda,\nu\rangle + mr \leqslant 0,
\end{equation}
pour tout sous-groupe à un paramètre indivisible dominant $\got{gl}_{\C}(E\oplus\C)$-admissible $\lambda$ de $\TC$ et pour tout $(w,w',m)\in W/W_{\lambda}\times W/W_{\lambda}\times\Z$ tels que $(C(w,w',m),\lambda)$ soit une paire bien couvrante de $\XEC$.
\end{theo}

\begin{proof}
Ceci est la transposition de \cite[Théorème 9]{ressayre08} au cas de la variété $\XEC=\KC/B\times\KC/B\times\mathbb{P}(E\oplus\C)$. Puisque $C_{\Q}(\XEC)^{+}$ est l'adhérence de $C_{\Q}(\XEC)^{++}$ d'après la Proposition \ref{prop:Csa_deféquivalente_avec_irrép}, il suffit de le prouver pour $C_{\Q}(\XEC)^{++}$.

L'élément $(\mu,\nu,r)\in\wedge_{\Q,+}^*\times\wedge_{\Q,+}^*\times\Q_{\geqslant 0}$ sera dans $(\pi^{\KC})^{-1}(C_{\Q}(\XEC)^{+})$ si et seulement si le fibré en droites ample $\mathcal{L}_{\mu,\nu,r} = \mathcal{L}_{\mu}\boxtimes\mathcal{L}_{\nu}\boxtimes\mathcal{L}_{r}$ est dans $C_{\Q}(\XEC)^{+}$. Nous avons, pour tout $(w,w',m)\in W/W_{\lambda}\times W/W_{\lambda}\times \Z$,
\begin{align*}
\mu^{\mathcal{L}_{\mu,\nu,r}}(C(w,w',m),\lambda) & = \mu^{\mathcal{L}_{\mu}}(G^{\lambda}w^{-1} B/B,\lambda) + \mu^{\mathcal{L}_{\nu}}(G^{\lambda}w'^{-1} B/B,\lambda) + \mu^{\mathcal{L}_{r}}(C_m,\lambda) \\
& = \langle w\lambda,\mu\rangle + \langle w'\lambda,\nu\rangle + mr.
\end{align*}
Ceci provient de la Proposition \ref{prop:produit_deuxfibrés_critèrenumérique}, des Lemmes \ref{prop:critèrenumérique_KCsurB} et \ref{prop:critèrenumérique_ProjectifdeE} et de la Remarque \ref{rema:defi_hilbertmumford_pour_fibrés_rationnels}. Ceci nous donne donc l'équation $\mu^{\bullet}(C(w,w',m),\lambda)$.

Notons $\mathcal{C}$ le cône convexe polyédral de $\wedge_{\Q,\scriptscriptstyle++}^*\times\wedge_{\Q,\scriptscriptstyle++}^*\times\Q_{> 0}$ défini par les équations \eqref{eq:équation_pairebiencouvrantedonnée} pour tout sous-groupe à un paramètre $\got{gl}_{\C}(E\oplus\C)$-admissible indivisible et dominant $\lambda$ de $\TC$ et pour tout $(w,w',m)\in W/W_{\lambda}\times W/W_{\lambda}\times\Z$ tels que $(C(w,w',m),\lambda)$ soit une paire bien couvrante de $\XEC$. Montrons que $\mathcal{C}$ est égal à $(\pi^{\KC})^{-1}(C_{\Q}(\XEC)^{++})$.

Soit $(\mu,\nu,r)\in (\pi^{\KC})^{-1}(C_{\Q}(\XEC)^{++})$, c'est-à-dire $\mu^{\mathcal{L}_{\mu,\nu,r}}(C(w,w',m),\lambda)\leqslant 0$ pour toute paire bien couvrante $(C(w,w',m),\lambda)$ de $\XEC$ d'après le Théorème \ref{theo:ressayre_equations_coneample}. Or, d'après ce qui a été vu dans le paragraphe précédent, cela signifie que $(\mu,\nu,r)$ vérifie l'équation \eqref{eq:équation_pairebiencouvrantedonnée} pour toute les paires bien couvrantes de $\XEC$, donc il vérifie en particulier cette équation pour le sous-ensemble des paires bien couvrantes $(C,\lambda)$ avec $\lambda$ indivisible dominant $\got{gl}_{\C}(E\oplus\C)$-admissible. Donc $(\mu,\nu,r)\in\mathcal{C}$.

Inversement, prenons un élément $(\mu,\nu,r)\in\mathcal{C}$. Soit $(C(w,w',m),\lambda)$ une paire bien couvrante de $\XEC$ telle qu'il existe $x\in C(w,w',m)$ avec $(\KC)_x^{\circ} = \lambda(\C^*)$. D'après la Proposition \ref{prop:admissibilité_par_stabilisateur}, $\lambda$ est $\got{gl}_{\C}(E\oplus\C)$-admissible. Puisque $(\mu,\nu,\chi,r)\in\mathcal{C}$, cela signifie que $(\mu,\nu,r)$ vérifie l'équation \eqref{eq:équation_pairebiencouvrantedonnée} pour la paire bien couvrante $(C(w,w',m),\lambda)$, c'est-à-dire, le fibré $\mathcal{L}_{\mu,\nu,r}$ vérifie l'équation $\mu^{\mathcal{L}_{\mu,\nu,r}}(C(w,w',m),\lambda)\leqslant 0$. On utilise cette fois le Théorème \ref{theo:ressayre_equations_coneample_avecconditionssup} pour conclure que $\mathcal{L}_{\mu,\nu,r}\in C_{\Q}(\XEC)^{++}$. D'où l'égalité des cônes convexes polyédraux $\mathcal{C} = (\pi^{\KC})^{-1}(C_{\Q}(\XEC)^{++})$.
\end{proof}

\subsection{Admissiblité spéciale pour les paires $(C(w,w',0),\lambda)$}

Nous allons voir, dans ce paragraphe, que nous pouvons être plus précis sur l'admissibilité d'un sous-groupe à un paramètre $\lambda$ apparaissant dans une paire bien couvrante $(C(w,w',0),\lambda)$.

Nous nous inspirons du Lemme 17 de \cite{ressayre08} pour démontrer le résultat suivant.

\begin{prop}
\label{prop:soustore_Eadmissible}
Soit $S$ un sous-tore de $\TC$. Nous considérons l'action de $\KC^S$ sur la variété $X' = \KC^S/B^S\times\KC^S/B^S\times\mathbb{P}(E^S\oplus\C)$. S'il existe $x\in X'$ tel que la composante neutre de $(\KC^S)_x$ soit égale à $S$, alors $S$ est $E$-admissible.
\end{prop}

\begin{lemm}
\label{lemm:ouvertdense_stabilisateurminimal}
Soient $Y$ une $\KC^S$-variété projective, et $\mathcal{U}$ un ouvert non vide de $\KC^S/B^S\times Y$. Si $g_0B^S\in\pi_1(\mathcal{U})\subset \KC^S/B^S$, l'ensemble $\{y\in Y; \ (g_0B^S/B^S,y)\in \mathcal{U}\}$ est ouvert dans $\KC^S/B^S\times Y$.
\end{lemm}

\begin{proof}
Si $X_1,X_2$ sont deux variétés, alors la projection $\pi_1:X_1\times X_2\rightarrow X_1$ est régulière, donc continue. Par conséquent, l'ensemble $\pi_1^{-1}(\{a\}) = \{(a,b);b\in X_2\}$ est fermé dans $X_1\times X_2$. Ici, nous prenons $X_1 = \KC^S/B^S$ et $X_2 = Y$.

On sait, de plus, que si $X_1$ et $X_2$ sont projectives, ce qui est le cas ici, alors les projections canoniques $\pi_1$ et $\pi_2$ sont fermées

Fixons $g_0B^s\in \KC^S/B^S$. L'ensemble $\pi_1^{-1}(\{g_0B^S/B^S\})$ est fermé dans $\KC^S/B^S\times Y$. Nécessairement, $\mathcal{U}^c\cap\pi_1^{-1}(\{g_0B^S/B^S\})$ est un fermé de $\KC^S/B^S\times Y$. Puisque $\pi_2$ est fermée, $\pi_2(\mathcal{U}^c\cap \pi_1^{-1}(\{g_0B^S/B^S\})$ est un fermé de $Y$. Remarquant que
\[
\pi_2(\mathcal{U}^c\cap \pi_1^{-1}(\{g_0B^S/B^S\}) =  \left\{y\in Y; \ (g_0B^S/B^S,y)\notin \mathcal{U}\right\},
\]
on en conclut que son complémentaire 
\[
\left\{y\in Y; (g_0B^S/B^S,y)\in \mathcal{U}\right\}
\]
est ouvert dans $Y$.
\end{proof}

\begin{proof}[Preuve de la Proposition \ref{prop:soustore_Eadmissible}]
Supposons qu'il existe $x\in X'$ tel que $(\KC^S)_x^{\circ} = S$. Le tore $S$ agit trivialement sur $X'$, donc pour tout $y\in X'$, $S$ est contenu dans le stabilisateur de $y$ dans $\KC^S$. Or la dimension du stabilisateur en un point est une application semi-continue inférieurement, donc la condition $(\KC^S)_x^{\circ} = S$ est une condition ouverte dans $X'$. On note $\mathcal{U}$ l'ensemble des $x$ de $X$ tel que $(\KC^S)_x^{\circ} = S$. Il est non vide par hypothèse. 

Vu que les éléments de $\KC^S$ centralisent $S$, $\mathcal{U}$ est un ouvert $\KC^S$-stable de $X'$. Si on note la $\KC^S$-variété projective $Y = \KC^S/B^S\times\mathbb{P}(E^S\oplus\C)$ et $\pi_1:\KC^S/B^S\times Y\rightarrow\KC^S/B^S$ la projection sur le premier facteur, alors clairement $B^S\in \pi_1(\mathcal{U})$. Ceci provient du fait que $\mathcal{U}$ est non vide et qu'il est $\KC^S$-stable. Par $\KC^S$-homogénéité de $\KC^S/B^S$, on a même $\pi_1(\mathcal{U}) = \KC^S/B^S$.

D'après le Lemme \ref{lemm:ouvertdense_stabilisateurminimal}, pour $g_0B^S = B^S$, l'ensemble
\[
\mathcal{V} := \left\{(gB^S/B^S,[v])\in\KC^S/B^S\times\mathbb{P}(E^S\oplus\C); (B^S/B^S,gB^S/B^S,[v])\in \mathcal{U}\right\}
\]
est un ouvert de $\KC^S/B^S\times\mathbb{P}(E^S\oplus\C)$. On remarque que, pour tout $(gB^S/B^S,[v])\in\mathcal{V}$, la composante neutre du stabilisateur du point $(B^S/B^S,gB^S/B^S,[v])$ de $X'$ est
\begin{align}
(\KC^S)_{(B^S/B^S,gB^S/B^S,[v])}^{\circ} & = \left((\KC^S)_{B^S/B^S} \cap (\KC^S)_{(gB^S/B^S,[v])}\right)^{\circ} \notag\\
\label{eq:stabilisateur_intersecté_avec_BS} & = \left(B^S \cap (\KC^S)_{(gB^S/B^S,[v])}\right)^{\circ}.
\end{align}
Or, par hypothèse, elle est aussi égale à $S$, puisque $(B^S/B^S,gB^S/B^S,[v])\in\mathcal{U}$. L'équation \eqref{eq:stabilisateur_intersecté_avec_BS} peut s'écrire maintenant
\[
 \left(B^S_{(gB^S/B^S,[v])}\right)^{\circ} = S,
\]
pour tous les éléments $(gB^S/B^S,[v])$ de l'ouvert dense $\mathcal{V}$ de $\KC^S/B^S\times\mathbb{P}(E^S\oplus\C)$. La composante neutre du stabilisateur dans $B^S/S$ d'un point de $\mathcal{V}$ est donc réduite à l'élément neutre. On en déduit que l'isotropie générique de $B^S/S$ agissant sur $\KC^S/B^S\times\mathbb{P}(E^S\oplus\C)$ est finie, puisque elle est finie pour tous les éléments de l'ouvert dense $\mathcal{V}$ de $\KC^S/B^S\times\mathbb{P}(E^S\oplus\C)$. D'après la décomposition de Bruhat, la cellule ouverte $B^Sw_0B^S/B^S$ de $\KC^S/B^S$ est isomorphe à $B^S/\TC$. Par conséquent, l'isotropie générique de $B^S/S$ agissant sur $B^S/\TC\times\mathbb{P}(E^S\oplus\C)$ est finie.

Le groupe $B^S$ est le produit semi-direct de $U^S$ par $\TC$, où $U$ est la partie unipotente de $B$. Le sous-groupe $U^S$ est distingué dans $B^S$ et il agit librement sur $B^S/\TC\times\mathbb{P}(E^S\oplus\C)$. En effet, la décomposition de $B^S$ en produit semi-direct de $U^S$ et $\TC$ donne un isomorphisme $U^S$-équivariant entre $B^S/\TC$ et $U^S$, où $U^S$ agit sur lui-même par multiplication à gauche. Remarquons que cet isomorphisme est également $\TC$-équivariant pour l'action par conjugaison de $\TC$ sur $U^S$. Ceci implique que le tore $\TC/S$ agit sur la variété quotient $\left(B^S/\TC\times\mathbb{P}(E^S\oplus\C)\right)/U^S\cong\mathbb{P}(E^S\oplus\C)$ (identification $\TC$-équivariante) avec isotropie générique finie. Donc, pour $x\in\mathbb{P}(E^S\oplus\C)$ générique, nous avons $(\TC)_x^{\circ} = S$.

Remarquons que $E^S$ est un ouvert $\KC^S$-stable de $\mathbb{P}(E^S\oplus\C)$. Par conséquent, pour $x\in E^S$ générique, nous avons $(\TC)_x^{\circ} = S$. En particulier, $S$ est la composante neutre du stabilisateur dans $\TC$ d'au moins un $x\in E$. On en conclut donc que $S$ est $E$-admissible.
\end{proof}

\begin{rema}
Nous pouvons aisément vérifier que l'ensemble des poids de $\TC$ dans $\got{gl}_{\C}(E\oplus\C)$ est
\[
 \WT(\got{gl}_{\C}(E\oplus\C)) = \{\beta-\beta'; \beta,\beta'\in\WT(E)\}\cup\WT(E).
\]
Par conséquent, l'ensemble $\WT(E)$ est contenu dans $\WT(\got{gl}_{\C}(E\oplus\C))$. Ainsi, un sous-groupe à un paramètre $E$-admissible de $\TC$ est nécessairement $\got{gl}_{\C}(E\oplus\C)$-admissible.
\end{rema}

\begin{coro}
\label{coro:mégalzéro_admissible_implique_Eadmissible}
Dans l'énoncé du Théorème \ref{theo:équations_pairesbiencouvrantes_admissible}, parmi les paires bien couvrantes $(C(w,w',0),\lambda)$ de $X$ telles que $\lambda$ est dominant $\got{gl}_{\C}(E\oplus\C)$-admissible indivisible, nous pouvons éliminer toutes celles où $\lambda$ n'est pas $E$-admissible.
\end{coro}

\begin{proof}
Soit $(C(w,w',0),\lambda)$ une paire bien couvrante de $\XEC$ telle qu'il existe $x\in C(w,w',0)$ avec $(\KC)_x^{\circ} = \lambda(\C^*)$. Il est clair que $C(w,w',0)$ est isomorphe de façon $\KC^{\lambda}$-équivariante à la variété $X' = \KC^{\lambda}/B^{\lambda}\times\KC^{\lambda}/B^{\lambda}\times\mathbb{P}(E^{\lambda}\oplus\C)$. Donc d'après la Proposition \ref{prop:soustore_Eadmissible}, $\lambda$ est $E$-admissible. En effet, nous avons $\lambda(\C^*)\subset (\KC^{\lambda})_x^{\circ} = (\KC)_x^{\circ}\cap \KC^{\lambda} \subset (\KC)_x^{\circ} = \lambda(\C^*)$. Donc $(\KC^{\lambda})_x^{\circ} = \lambda(\C^*)$. Le reste découle à nouveau du Corollaire \ref{coro:côneample_intérieurnonvide} et du Théorème \ref{theo:ressayre_equations_coneample_avecconditionssup}.
\end{proof}

\section{\'Equations du polyèdre moment $\DGIT$}
\label{section:équations_polyèdremoment_KLambdaE}

%

\subsection{Les équations du polyèdre}

Soit $K$ un groupe de Lie réel compact connexe, $T$ un tore maximal de $K$, $n$ la dimension de $T$, $(E,h)$ un espace hermitien de dimension $r$ et $\zeta: K\rightarrow U(E)$ un morphisme de groupes de Lie de noyau fini tel que l'application moment associée $\Phi_E:E\rightarrow\got{k}^*$ (définie en section \ref{section:polydremoment_KLambdaE}) soit propre. Cette action se complexifie en $\zeta : \KC \rightarrow GL_{\C}(E)$, qui est aussi de noyau fini. Remarquons que le noyau de $\zeta$ est fini si et seulement si le noyau du morphisme
\[
\begin{array}{cccl}
\zeta\oplus\id_{\C}:& \KC & \longrightarrow & GL_{\C}(E\oplus\C) \\
& g & \longmapsto & \zeta(g)\oplus\id_{\C}
\end{array}
\]
est fini. Nous allons donc pouvoir appliquer les résultats de la section \ref{section:redondance_équations}.

Soit $\TC$ un tore maximal de $\KC$ contenant $T$ et $B$ un sous-groupe de Borel de $\KC$ contenant $\TC$. Rappelons que nous avons défini la variété $\XEC = \KC/B\times\KC/B\times\mathbb{P}(E\oplus\C)$, où l'action de $\KC$ sur $\mathbb{P}(E\oplus\C)$ est induite de l'action linéaire $\zeta\oplus\id_{\C}$ de $\KC$ sur $E\oplus\C$. L'action de $\KC$ étant triviale sur le facteur $\C$, l'ensemble $\WT(E\oplus\C)$ contiendra le poids nul.


On rappelle que l'ensemble des poids $\WT(E)$ engendre $\got{t}^*$ si et seulement si $\ker\zeta$ est fini.

%

Le Théorème \ref{theo:équations_pairesbiencouvrantes_admissible}, le Corollaire \ref{coro:mégalzéro_admissible_implique_Eadmissible} et la Remarque \ref{rema:selimiterauxsgprà1paramindivdominants} nous amènent à définir les deux ensembles de paires bien couvrantes suivants.

\begin{defi}
\label{defi:P(E)_et_P0(E)}
On définit $\mathcal{P}(E)$\index{$\mathcal{P}(E)$} l'ensemble des paires $(C(w,w',m),\lambda)$ bien couvrantes de $\XEC$ telles que $\lambda$ soit dominant $\got{gl}_{\C}(E\oplus\C)$-admissible indivisible et, pour le cas où $m=0$, $E$-admissible.

On note $\mathcal{P}_0(E)$\index{$\mathcal{P}_0(E)$} le sous-ensemble de $\mathcal{P}(E)$ formé des paires qui vérifient en plus $m=0$ (les $\lambda$ qui apparaissent sont donc $E$-admissibles)
\end{defi}


Remarquons que $\lambda$ est $E$-admissible si et seulement s'il est orthogonal à un des hyperplans de $\got{t}^*$ engendrés par des poids de $\WT(E)$. De tels hyperplans existent car $\WT(E)$ engendre $\got{t}^*$. D'après la Remarque \ref{rema:sgrpà1paramindivMadmissible_fini}, comme $E$ est de dimension finie, il n'existe qu'un nombre fini de $\lambda$ dominants $\got{gl}_{\C}(E\oplus\C)$-admissibles indivisibles. Et pour chacun de ces $\lambda$, il n'y a qu'un nombre fini de paires $(C,\lambda)$ (bien couvrantes ou non), car le groupe de Weyl $W$ est fini. On en déduit que $\mathcal{P}(E)$ et $\mathcal{P}_0(E)$ sont des ensembles finis.

Rappelons la définition donnée en section \ref{section:polydremoment_KLambdaE} de l'ensemble $\Pi_{\Q}(E)$,
\[
\Pi_{\Q}(E) = \left\{ (\mu,\nu)\in(\wedge^*_{\Q,+})^2 \left|
\begin{array}{l}
\exists n\geqslant 1 \text{ tel que } (n\mu,n\nu)\in(\wedge^*)^2, \\
\mbox{et } \left(V_{n\mu}^*\otimes V_{n\nu}^*\otimes \C[E]\right)^K \neq 0
\end{array}\right.\right\}.
\]
L'énoncé suivant est l'élément-clé pour obtenir un ensemble d'équations de $\Pi_{\Q}(E)$ à partir de celles de $C_{\Q}(\XEC)^+$.

\begin{theo}
 \label{theo:cn_pairebiencouvrante_poidsnul}
Soit $M$ une représentation complexe de $\KC$ telle que le poids nul appartienne à $\WT(M)$, et considérons la variété $X_M$. Soit $(w,w',m)\in W/W_{\lambda}\times W/W_{\lambda}\times\Z$ tel que la paire $(C(w,w',m),\lambda)$ soit bien couvrante dans $X_M$. Alors $m\leqslant 0$.
\end{theo}

La preuve de ce théorème utilise des outils complètement différents de ceux qui ont été présentés jusque-là, nous la reportons donc au Chapitre \ref{chap:PairesBienCouvrantes}. Dans le cas présent, c'est-à-dire lorsque $M=E\oplus\C$ avec action triviale de $\KC$ sur le facteur $\C$, la condition $0\in\WT(E\oplus\C)$ est clairement remplie.

Nous pouvons maintenant énoncer et prouver le résultat nous donnant un ensemble d'équations du cône convexe polyédral $\Pi_{\Q}(E)$. Notons $\mathrm{pr}:\wedge^*_{\Q}\times\wedge^*_{\Q}\times\Q\rightarrow\wedge^*_{\Q}\times\wedge^*_{\Q}$ la projection linéaire canonique. La Proposition \ref{prop:Csa_deféquivalente_avec_irrép} montre qu'alors $\Pi_{\Q}(E) = \mathrm{pr}\left((\pi^{\KC})^{-1}(C_{\Q}(\XEC)^+)\right)$. Cela signifie qu'un élément $(\mu,\nu)\in(\wedge^*_{\Q,+})^2$ appartient à $\Pi_{\Q}(E)$ si et seulement s'il existe un rationnel $r$ tel que le fibré $\KC$-linéarisé $\mathcal{L}_{\mu,\nu,r}$ appartienne à $C_{\Q}(\XEC)^+$.

\begin{theo}
\label{theo:équations_GammaQ}
Soit $(\mu,\nu)\in(\wedge^*_{\Q,+})^2$. Le couple $(\mu,\nu)$ appartient à $\Pi_{\Q}(E)$ si et seulement si, pour toute paire $(C(w,w',0),\lambda)$ de $\mathcal{P}_0(E)$, on a
\begin{equation}
\label{eq:theo_equations_polyèdremoment}
\langle w\lambda,\mu\rangle + \langle w'\lambda,\nu\rangle \leqslant 0.
\end{equation}
\end{theo}

\begin{proof}
Soit $\mathcal{C}$ le cône convexe polyédral défini par les équations \eqref{eq:theo_equations_polyèdremoment}. Montrons que $\mathcal{C}$ est inclus dans $\Pi_{\Q}(E)$.

Soit $(\mu,\nu)\in\mathcal{C}$. Par définition, $(\mu,\nu)$ vérifie les équations \eqref{eq:équation_pairebiencouvrantedonnée}, avec $m = 0$, pour toute paire $(C(w,w',0),\lambda)$ de $\mathcal{P}_0(E)$. Soit $(C(w,w',m),\lambda)$ une paire de $\mathcal{P}(E)\backslash\mathcal{P}_0(E)$, c'est-à-dire, $m\neq 0$. D'après le Théorème \ref{theo:cn_pairebiencouvrante_poidsnul}, l'entier $m$ est négatif. Donc $-m>0$, puisque $m\neq 0$.

Nous définissons le rationnel
\[
r_0 = \max_{C(w,w',m)\in\mathcal{P}(E)\backslash\mathcal{P}_0(E)}\left\{\frac{\langle w\lambda,\mu\rangle + \langle w'\lambda,\nu\rangle}{-m}\right\}.
\]
Ce maximum existe, car l'ensemble $\mathcal{P}(E)$ est fini, et est bien rationnel car chaque $\langle w\lambda,\mu\rangle + \langle w'\lambda,\nu\rangle$ est rationnel. Quitte à le remplacer par $\max\{r_0,1\}$, on peut supposer que $r_0$ est strictement positif. Nous avons donc, pour tout $(C(w,w',m),\lambda)\in\mathcal{P}(E)\backslash\mathcal{P}_0(E)$, l'inégalité $\langle w\lambda,\mu\rangle + \langle w'\lambda,\nu\rangle \leqslant (-m)r_0$, car $-m$ est strictement positif. Donc $\langle w\lambda,\mu\rangle + \langle w'\lambda,\nu\rangle + mr_0 \leqslant 0$. Combiné à ce qui a été dit au premier paragraphe, on en déduit que
\[
\langle w\lambda,\mu\rangle + \langle w'\lambda,\nu\rangle + mr_0 \leqslant 0
\]
pour toute paire $(C(w,w',m),\lambda)$ de $\mathcal{P}(E)$, et le Théorème \ref{theo:équations_pairesbiencouvrantes_admissible} montre que $\mathcal{L}_{\mu,\nu,r_0}$ appartient à $C_{\Q}(\XEC)^+$, ce qui signifie que $(\mu,\nu)$ est un élément de $\Pi_{\Q}(E)$. D'où $\mathcal{C}\subset\Pi_{\Q}(E)$.

Réciproquement, soit $(\mu,\nu)\in\Pi_{\Q}(E)$. Il existe donc un rationnel positif $r$ tel que $\mathcal{L}_{\mu,\nu,r}\in C_{\Q}(\XEC)^+$. C'est-à-dire, $(\mu,\nu,r)$ vérifie l'inégalité $\langle w\lambda,\mu\rangle + \langle w'\lambda,\nu\rangle + mr \leqslant 0$ pour toute paire $(C(w,w',m),\lambda)$ de $\mathcal{P}(E)$. En particulier, le couple $(\mu,\nu)$ vérifie l'équation $\langle w\lambda,\mu\rangle + \langle w'\lambda,\nu\rangle \leqslant 0$ pour toute paire $(C(w,w',0),\lambda)$ de $\mathcal{P}_0(E)$. Ce sont les équations \eqref{eq:theo_equations_polyèdremoment}. On en déduit $\Pi_{\Q}(E) \subset \mathcal{C}$. Cela termine la preuve du théorème.
\end{proof}

Comme corollaire direct, nous obtenons un ensemble fini d'équations décrivant le polyèdre convexe $\DGIT$.

\begin{coro}
\label{coro:équationsgénérales_DGIT}
Soit $\Lambda\in\wedge^*_{\Q,+}$. Le polyèdre moment $\DGIT$
est égal au polyèdre convexe
\[
\left\{\xi\in\wedge_{\Q,+}^*;\ \langle w\lambda,\xi\rangle \leqslant \langle w_0w'\lambda,\Lambda\rangle \mbox{ pour toute paire } (C(w,w',0),\lambda)\in\mathcal{P}_0(E)\right\}.
\]
\end{coro}

\begin{proof}
C'est une conséquence immédiate du Théorème \ref{theo:équations_GammaQ} et du Lemme \ref{lemm:lien_DeltaKLambda_GammaQ}.
\end{proof}


L'énoncé qui suit est un autre résultat géométrique qui découle directement du corollaire précédent.

\begin{coro}
Nous avons
\[
\DGIT[E] = \left\{\xi\in\wedge_{\Q,+}^*;\ \langle w\lambda,\xi\rangle \leqslant 0 \mbox{ pour toute paire } (C(w,w',0),\lambda)\in\mathcal{P}_0(E)\right\}.
\]
Plus généralement, soit $\Lambda\in\wedge^*_{\Q,+}$ central. Le polyèdre moment $\DGIT$
est égal au polyèdre convexe
\[
\left\{\xi\in\wedge_{\Q,+}^*;\ \langle w\lambda,\xi-\Lambda\rangle \leqslant 0 \mbox{ pour toute paire } (C(w,w',0),\lambda)\in\mathcal{P}_0(E)\right\}.
\]
En particulier, on a $\DGIT = \Lambda + \DGIT[E]$.
\end{coro}

\subsection{Paires dominantes dans les équations de $\DGIT$}
\label{subsection:pairesdominantes_dans_DGIT}

D'un point de vue théorique, l'utilisation des paires bien couvrantes permet d'obtenir un ensemble d'équations pour le cône semi-ample $C_{\Q}^+(\XEC)$ avec peu de redondance d'information. Par ailleurs, lorsque la variété $X$ est un produit $\KC/B\times\hKC/\hat{B}$ de variétés des drapeaux complètes, avec $\KC\subseteq\hKC$, \cite[Théorème 10]{ressayre08} prouve que cet ensemble d'équations est minimal pour décrire le cône semi-ample $\Csa$.

Malheureusement, en passant du cône $C_{\Q}^+(\XEC)$ au cône convexe polyédral $\Pi_{\Q}(E)$ par projection linéaire, il est très probable que de la redondance supplémentaire apparaisse dans les équations de $\Pi_{\Q}(E)$ obtenues par les paires bien couvrantes de $\mathcal{P}_0(E)$.

Néanmoins, en pratique, le calcul des paires bien couvrantes de la variété $\XEC$ peut se révéler fastidieux, car il fait intervenir une condition cohomologique forte, cf Théorème \ref{theo:cns_pairebiencouvrante}. On peut toutefois réduire la difficulté dans la recherche d'équations de $\Pi_{\Q}(E)$ si, d'un autre côté, on diminue nos attentes dans l'optimalité de l'ensemble des équations que l'on souhaite obtenir.

\begin{lemm}
\label{lemm:équations_de_GammaQ_et_pairesdominantes}
Soit $(C(w,w',0),\lambda)$ une paire dominante de $\XEC$. Alors, pour tout $(\mu,\nu)\in\Pi_{\Q}(E)$, on a
\[
\langle w\lambda,\mu\rangle + \langle w'\lambda,\nu\rangle \leqslant 0.
\]
\end{lemm}

\begin{proof}
Soit $(\mu,\nu)\in\Pi_{\Q}(E)$. Il existe donc un rationnel positif $r$ tel que $(\mu,\nu,0,r)$ appartienne à $C_{\Q}(\XEC)^+$, c'est-à-dire, le fibré en droites $\mathcal{L} := \mathcal{L}_{\mu,\nu,0,r}$ appartient à $C_{\Q}(\XEC)^+$. Or, d'après le Lemme \ref{lemm:ressayre_lemme3}, comme la paire $(C(w,w',0),\lambda)$ est dominante, on doit avoir $\mu^{\mathcal{L}}(C(w,w',0),\lambda) \leqslant 0$. Cette dernière équation peut s'écrire également
\[
\langle w\lambda,\mu\rangle + \langle w'\lambda,\nu\rangle + \langle \lambda,0\rangle + 0r \leqslant 0,
\]
comme indiqué dans la preuve du Théorème \ref{theo:équations_pairesbiencouvrantes_admissible}. Ceci prouve l'assertion.
\end{proof}

Notons de manière analogue au paragraphe précédent, $\overline{\mathcal{P}}_0(E)$ (resp. $\overline{\mathcal{P}}_0^{adm}(E)$) l'ensemble des paires dominantes $(C(w,w',0),\lambda)$ de $\XEC$ telles que $\lambda$ soit dominant indivisible (resp. dominant indivisible $E$-admissible). On a évidemment $\mathcal{P}_0(E)\subset\overline{\mathcal{P}}_0^{adm}(E)\subset\overline{\mathcal{P}}_0(E)$.

\begin{theo}
Soit $\mathcal{P}_1$ un ensemble de paires $(C(w,w',0),\lambda)$ de $\XEC$ tel que $\mathcal{P}_0(E)\subseteq\mathcal{P}_1\subseteq\overline{\mathcal{P}}_0(E)$. Alors le cône convexe polyédral $\Pi_{\Q}(E)$ est l'ensemble des $(\mu,\nu)\in(\wedge_{\Q,+})^2$ tels que
\begin{equation}
\label{eq:inégalité_pour_P1}
\langle w\lambda,\mu\rangle + \langle w'\lambda,\nu\rangle \leqslant 0
\end{equation}
pour toute paire $(C(w,w',0),\lambda)$ de $\mathcal{P}_1$.
\end{theo}

\begin{proof}
Puisque $\mathcal{P}_0(E)\subseteq\mathcal{P}_1$, il est clair que si le couple $(\mu,\nu)\in(\wedge_{\Q,+})^2$ vérifie les équations \eqref{eq:inégalité_pour_P1} pour les paires de $\mathcal{P}_1$, alors il vérifiera ces mêmes équations pour les paires de $\mathcal{P}_0(E)$. Le Théorème \ref{theo:équations_pairesbiencouvrantes_admissible} permet d'en déduire qu'un tel couple $(\mu,\nu)$ est dans $\Pi_{\Q}(E)$.

Réciproquement, si $(\mu,\nu)\in\Pi_{\Q}(E)$, alors le Lemme \ref{lemm:équations_de_GammaQ_et_pairesdominantes} montre que $(\mu,\nu)$ vérifie l'équation \eqref{eq:inégalité_pour_P1} pour toute paire dominante, donc en particulier pour toute paire de $\mathcal{P}_1\subseteq\overline{\mathcal{P}}_0(E)$.
\end{proof}

On en déduit un nouvel ensemble d'équations du polyèdre convexe rationnel $\DGIT$, résultat analogue au Corollaire \ref{coro:équationsgénérales_DGIT}.

\begin{coro}
Soient $\Lambda\in\wedge_{\Q,+}^*$ et $\mathcal{P}_1$ un ensemble de paires $(C(w,w',0),\lambda)$ de $\XEC$ tel que $\mathcal{P}_0(E)\subseteq\mathcal{P}_1\subseteq\overline{\mathcal{P}}_0(E)$. Alors on a
\[
\DGIT = \left\{\xi\in\wedge_{\Q,+}^*;\ \langle w\lambda,\xi\rangle \leqslant \langle w_0w'\lambda,\Lambda\rangle, \forall (C(w,w',0),\lambda)\in\mathcal{P}_1\right\}.
\]
\end{coro}

En particulier, ce corollaire est valable pour $\mathcal{P}_1 = \overline{\mathcal{P}}_0^{adm}(E)$, c'est-à-dire, le polyèdre convexe polyédral $\DGIT$ est aussi égal à l'ensemble des vecteurs $\xi\in\wedge_{\Q,+}^*$ qui vérifient l'équation $\langle w\lambda,\xi\rangle \leqslant \langle w_0w'\lambda,\Lambda\rangle$ pour toute paire dominante $(C(w,w',0),\lambda)$ avec $\lambda$ dominant indivisible $E$-admissible.

\subsection{Propriétés géométriques de $\DGIT$}
\label{subsection:propriétésgéométriquesdeDGIT}

Nous terminons ce chapitre en donnant plusieurs propriétés géométriques du polyèdre $\DGIT$. Cela concerne les faces de codimension $1$ qui rencontrent le point $\Lambda$.

Parmi les paires de $\mathcal{P}_0(E)$, les paires de type $(C(w,w_0w,0),\lambda)$ sont très particulières. En effet, elles donnent les faces de codimension $1$ du cône convexe polyédral $\CQ(-\WT(E))$. Avant de prouver ce résultat, nous avons besoin d'énoncer le théorème suivant, qui utilise des résultats du Chapitre \ref{chap:PairesBienCouvrantes}.

\begin{theo}
 \label{theo:paires_w'=w0w_sonttoutesbiencouvrantes}
Soit $\lambda$ un sous-groupe à un paramètre dominant de $\TC$ et soit $(w,m)\in W/W_{\lambda}\times\Z$ tel que $C(w,w_0w,m)\neq\emptyset$. Alors $(C(w,w_0w,m),\lambda)$ est une paire bien couvrante de $\XEC$ et, pour tout $\beta\in\WT(E)$, on a $\langle\lambda,\beta\rangle\geqslant m$.
\end{theo}

\begin{proof}
 Ce résultat découle directement du Théorème \ref{theo:cns_pairebiencouvrante} et du Corollaire \ref{coro:info_longueurs_éléments_pour_pairebiencouvrante}.
\end{proof}

Lorsque $\Lambda\in\wedge_{\Q,+}^*$ est quelconque, c'est-à-dire que $\Lambda$ n'a aucune hypothèse de régularité (son stabilisateur dans $K$ peut contenir strictement $T$), il est difficile de décrire exactement les faces de $\DGIT$ au voisinage de $\Lambda$. Cependant, nous pouvons tout de même donner une information partielle.

\begin{prop}
\label{prop:faces_ConeWtT(E)_et_pairesbiencouvrantes}
Toute paire bien couvrante $(C(w,w_0w,0),\lambda)$ de $\mathcal{P}_0(E)$ définit une face de codimension $1$ de $\CQ(-\WT(E))$. Inversement, toute face de codimension $1$ du cône convexe polyédral $\CQ(-\WT(E))$ provient d'une telle paire bien couvrante. En particulier, pour tout $\Lambda\in\wedge^*_{\Q,+}$, on a
\[
\DGIT \subset \Lambda + \CQ(-\WT(E)).
\]
\end{prop}

\begin{proof}
Soit $(C(w,w_0w,0),\lambda)$ une paire de $\mathcal{P}_0(E)$. 
Le sous-groupe à un paramètre $\lambda$ est $E$-admissible, donc il existe $n-1$ poids $\beta_{i_1},\ldots,\beta_{i_{n-1}}$ formant une famille libre de $\got{t}_{\C}^*$ tels que $\C\lambda = \cap_{j=1}^{n-1}\ker\beta_{i_j}$. De plus, d'après le Théorème \ref{theo:paires_w'=w0w_sonttoutesbiencouvrantes}, nous avons $\langle\lambda,\beta\rangle \geqslant 0$ pour tout poids $\beta\in\WT(E)$. Le groupe de Weyl laissant stable l'ensemble des poids de $\TC$ dans $E$, tout poids $\beta\in\WT(E)$ vérifiera aussi $\langle w\lambda,-\beta\rangle \leqslant 0$, mais nous aurons aussi $\langle w\lambda,-\beta_{i_k}\rangle = 0$ pour $k=1,\ldots,n-1$. Par conséquent, l'équation $\langle w\lambda,\cdot\rangle \leqslant 0$ définit une équation de codimension $1$ de $\CQ(-\WT(E))$.

Réciproquement, soit $\mathcal{F}$ une face de codimension 1 de $\CQ(-\WT(E))$. Ce cône convexe polyédral est rationnel, donc il existe un sous-groupe à un paramètre $\lambda_{\mathcal{F}}$ indivisible tel que $\langle \lambda_{\mathcal{F}},-\beta\rangle \leqslant 0$ pour tout $\beta\in\WT(E)$, et $\langle \lambda_{\mathcal{F}},-\beta'_{i_k}\rangle = 0$ pour $n-1$ poids $\beta'_{i_1},\ldots,\beta'_{i_{n-1}}$ linéairement indépendants. Il existe $w\in W$ tel que $\lambda = w^{-1}\lambda_{\mathcal{F}}$ soit dominant. De plus, comme les $\beta'_{i_k}$, $k\in\{1,\ldots,n-1\}$ sont linéairement indépendants, nous avons nécessairement $\C(w^{-1}\lambda_{\mathcal{F}}) = \cap_{j=1}^{n-1}\ker(w^{-1}\beta'_{i_j})$ et, donc, $w^{-1}\lambda_{\mathcal{F}}$ est un sous-groupe à un paramètre dominant $E$-admissible indivisible de $\TC$. Et d'après le Théorème \ref{theo:paires_w'=w0w_sonttoutesbiencouvrantes}, la paire $(C(w,w_0w,0),w^{-1}\lambda_{\mathcal{F}})$ est bien couvrante, c'est donc une paire de $\mathcal{P}_0(E)$. Et, bien sûr, $\langle w(w^{-1}\lambda_{\mathcal{F}}),\cdot\rangle = \langle \lambda_{\mathcal{F}},\cdot\rangle = 0$ est l'équation de $\mathcal{F}$.

Si maintenant $x$ appartient à $\DGIT$, cet élément de $\wedge_{\Q,+}^*$ vérifiera les équations affines
\[
\langle w\lambda,x\rangle \leqslant \langle w_0w'\lambda,\Lambda\rangle,
\]
pour toute paire $(C(w,w',0),\lambda)\in\mathcal{P}_0(E)$. En particulier, si $\mathcal{F}$ est une face de codimension $1$ de $\CQ(-\WT(E))$, si $\lambda_{\mathcal{F}}$ est le sous-groupe à un paramètre $E$-admissible indivisible associé et si $w_{\mathcal{F}}\in W/W_{\lambda}$ est tel que $w_{\mathcal{F}}^{-1}\lambda_{\mathcal{F}}$ soit dominant $E$-admissible indivisible, alors la paire bien couvrante $(C(w_{\mathcal{F}},w_0w_{\mathcal{F}},0),w_{\mathcal{F}}^{-1}\lambda_{\mathcal{F}})$ est dans $\mathcal{P}_0(E)$, et
\[
\langle w_{\mathcal{F}}(w_{\mathcal{F}}^{-1}\lambda_{\mathcal{F}}),x-\Lambda\rangle = \langle \lambda_{\mathcal{F}},x-\Lambda\rangle \leqslant 0,
\]
car ici $w' = w_0w$. Ceci est vrai pour toute face de codimension $1$ de $\CQ(-\WT(E))$, on a donc $x -\Lambda \in\CQ(-\WT(E))$. Autrement dit, $\DGIT\subset \Lambda + \CQ(-\WT(E))$.
\end{proof}

Le résultat précédent affirme que le polyèdre convexe $\DGIT$ est contenu dans le polyèdre convexe $\Lambda + \CQ(-\WT(E))$, pour tout $\Lambda\in\wedge^*_{\Q,+}$. Nous pouvons montrer également le fait suivant.

\begin{prop}
\label{prop:Lambda_dans_DGIT}
Pour tout $\Lambda\in\wedge_{\Q,+}^*$, l'élément $\Lambda$ est dans $\DGIT$.
\end{prop}

\begin{proof}
Soit $n\in\N^*$ tel que $n\Lambda\in\wedge^*$. On sait que $(V_{n\Lambda}^*\otimes V_{n\Lambda})^{\KC}$ n'est pas réduit à l'élément neutre. Or, clairement, on a $(V_{n\Lambda}^*\otimes V_{n\Lambda})^{\KC}\subseteq (V_{n\Lambda}^*\otimes V_{n\Lambda}\otimes\C[E])^{\KC}$. D'où $\Lambda$ appartient à $\DGIT$.
\end{proof}

\chapter[Critère cohomologique des paires bien couvrantes]{Critère cohomologique des paires bien couvrantes de la variété $\KC/B\times\KC/B\times\mathbb{P}(M)$}
\label{chap:PairesBienCouvrantes}


Dans le chapitre \ref{chap:projectiondorbitecoadjointe+GIT}, nous avons établi une liste d'équations affines qui définissent le polyèdre convexe rationnel $\DGIT$ lorsque $\Lambda$ est un poids rationnel de la chambre holomorphe $\Chol$. Ces équations proviennent d'une partie des équations linéaires déterminant le cône ample $C_{\Q}(\XEC)^+$, par projection linéaire.

De manière générale, nous avons vu, dans le paragraphe \ref{subsubsection:gpalg_git_pairesbiencouvrantes}, que nous pouvons avoir une description des cônes ample et semi-ample d'une variété projective lisse $X$ par des inéquations linéaires indexées par les paires bien couvrantes de $X$.

Le plus souvent, il se révèle difficile de déterminer l'ensemble de toutes les paires bien couvrantes d'une variété projective fixée. Cependant, dans notre contexte, nous pouvons nous restreindre à l'étude d'une famille de variétés projectives possédant une forme très intéressante. La Proposition $11$ de \cite{ressayre08} donne une condition nécessaire et suffisante pour qu'une paire $(C,\lambda)$ soit bien couvrante, dans le cas d'une $\KC$-variété $Y$ de la forme $Y=\KC/Q\times\hKC/\hat{Q}$, où $\KC$ est un sous-groupe réductif connexe d'un groupe réductif complexe connexe $\hKC$ et $Q$ (resp. $\hat{Q}$) un sous-groupe parabolique de $\KC$ (resp. $\hKC$).

Notre variété $X_M := \KC/B\times\KC/B\times\mathbb{P}(M)$ est un exemple de telle $\KC$-variété, lorsque $M$ est un $\KC$-module de dimension finie.

Dans ce chapitre, nous améliorons le critère cohomologique \cite[Proposition 11]{ressayre08} permettant de calculer l'ensemble des paires bien couvrantes de $X_M$. D'un point de vue théorique, ce critère a une conséquence qui nous a permis d'obtenir, dans la section \ref{section:équations_polyèdremoment_KLambdaE}, les équations de $\DGIT$ à partir des équations de $C_{\Q}(\XEC)^+$, en prenant $M=E\oplus\C$, avec action triviale de $\KC$ sur le facteur $\C$. D'un autre côté, ce critère nous permettra dans le Chapitre \ref{chap:calcul_exemples_par_pairesbiencouvrantes} de calculer les polyèdres moments $\Delta_K(\Orb_{\Lambda})$ pour certains groupes de Lie réels simples $G$.

\section{Le critère principal}
\label{section:critèreprincipal}

Dans cette section, nous utiliserons les mêmes notations que dans le paragraphe \ref{subsection:GIT_notations}, pour les groupes $\TC\subset B\subset\KC$. Soit $\zeta:\KC\rightarrow GL_{\C}(M)$ une représentation complexe de $\KC$ et $r:=\dim_{\C}M$. On note $\got{t}_{\C}$, $\got{b}$ et $\got{k}_{\C}$ les algèbres de Lie respectives. Nous fixons pour le reste de ce chapitre un sous-groupe à un paramètre $\lambda$ dominant de $\TC$. On notera $P(\lambda)$ le sous-groupe parabolique de $\KC$ associé à $\lambda$. Il est standard, c'est-à-dire, il contient $B$, puisque $\lambda$ est dominant.


Soit $P$ un quelconque sous-groupe parabolique de $\KC$ contenant $B$. On note $W_P$ le groupe de Weyl du sous-groupe de Levi de $P$ contenant $T$. Pour tout $w\in W/W_P$, la sous-variété algébrique $X_w^P:=\overline{BwP/P}$ de $\KC/P$ est appelée la variété de Schubert associée à $w$. Bien qu'une variété de Schubert ne soit pas forcément lisse, on peut tout de même lui définir une classe fondamentale, qui coïncide avec la définition standard de classe fondamentale lorsque la variété de Schubert est lisse. Cette définition est expliquée dans l'Annexe \ref{chap:classefondamentale}. Les classes fondamentales des variétés de Schubert forment une base du $\Z$-module libre $\mathrm{H}_*(\KC/P,\Z)$, ce qui nous permet de définir sa base duale $\{\sigma_w^P; w\in W/W_P\}$ dans $\coh{\KC/P}$. On notera également $\pt := \sigma_{w_0}^P$ la classe du point, c'est-à-dire, le générateur de $\mathrm{H}^{2\dim_{\C}(\KC/P)}(\KC/P,\Z)$. Ici, $w_0$ désigne le plus long élément de $W$.

Désormais, on considère le sous-groupe parabolique $P=P(\lambda)$ et $W_{\lambda}$ le groupe de Weyl de son sous-groupe de Levi $\KC^{\lambda}$. Remarquons que $W_{\lambda}$ s'identifie au stabilisateur de $\lambda$ dans $W$ (ce qui justifie la notation $W_{\lambda}$). Posons $W^{\lambda}$ l'ensemble des représentants de longueur maximale des classes de $W/W_{\lambda}$. On définit l'application $\jmath: gB\in\KC/B\mapsto gP(\lambda)\in\KC/P(\lambda)$ et l'application induite en cohomologie
\[
\jmath^*:\coh{\KC/P(\lambda)}\longrightarrow\coh{\KC/B}.
\]
Puisque $\jmath$ est surjective, l'homomorphisme d'anneaux $\jmath^*$ est injectif. Ici, nous considérons $\coh{\KC/P(\lambda)}$ et $\coh{\KC/B}$ comme anneaux pour la somme induite par leur structure de $\Z$-module et le produit définit par le produit cup.

Il est bien connu que, pour tout $w$ dans $W^{\lambda}$, on a l'égalité $\jmath^*(\sigma_{w_0w}^{P(\lambda)}) = \sigma_{w_0w}^B$ dans $\coh{\KC/B}$, car $w_0w$ est le plus petit élément de $w_0wW_{\lambda}$. Ceci redonne l'injectivité de $\jmath^*$. En particulier, nous avons $\jmath^*(\pt)=\sigma_{w_0w_{\lambda}}^B$, où $w_{\lambda}$ est le plus long élément de $W_{\lambda}$, cf \cite{brion_flagvar} pour plus de détails.

Nous rappelons que pour tout $k\in\Z$, le sous-espace $M_{\lambda,k}$ de $M$ est défini par
\[
M_{\lambda,k}:=\{v\in M; \zeta(\lambda(t))v = t^kv,\forall t\in\C^*\}.
\]
Pour tout $m\in\Z$, on définit les sous-espaces
\[
M_{< m}:=\bigoplus_{k< m}M_{\lambda,k} \qquad \text{et} \qquad M_{\geqslant m}:=\bigoplus_{k\geqslant m}M_{\lambda,k}.
\]
Pour tout $m\in\Z$, les trois sous-espaces $M_{\lambda,m}$, $M_{< m}$ et $M_{\geqslant m}$ sont stables par $\KC^{\lambda}$. On remarque que l'ensemble $\WT(M_{< m})$ est égal à l'ensemble des poids $\beta$ de $M$ tels que $\langle\lambda,\beta\rangle < m$. De même, $\beta'\in\WT(M_{\geqslant m})$ si et seulement si $\beta'\in\WT(M)$ et $\langle\lambda,\beta'\rangle\geqslant m$. En particulier, on a l'union disjointe
\[
 \WT(M) = \WT(M_{< m}) \amalg \WT(M_{m\geqslant m}),
\]
quel que soit $m\in\Z$. Pour tout $\beta\in\WT(M)$, on notera $n_{\beta}$ la multiplicité du poids $\beta$ sur $M$, c'est-à-dire $n_{\beta}:=\dim_{\C}M_{\beta}$, où $M_{\beta}$ est l'espace de poids $\beta$ dans $M$.

Soit $\Theta:\wedge^*\rightarrow\mathrm{H}^2(\KC/B,\Z)$ le morphisme qui envoie un poids $\mu\in\wedge^*$ de $\TC$ sur la première classe de Chern $\Theta(\mu) = c_1(\mathcal{L}_{\mu})$ du fibré en droites $\mathcal{L}_{\mu}$ de poids $\mu$.

On notera $\rho$ la demi-somme des racines positives de $\got{k}_{\C}$.

\begin{theo}
\label{theo:cns_pairebiencouvrante}
\label{THEO:CNS_PAIREBIENCOUVRANTE}
Soit $(w,w',m)\in W^{\lambda}\times W^{\lambda}\times\Z$ tel que $C(w,w',m)$ soit non vide. La paire $(C(w,w',m),\lambda)$ de $X_{M}$ est bien couvrante si et seulement si soit $w' = w_0ww_{\lambda}$ et $M_{< m} = 0$, soit les deux assertions suivantes sont simultanément vérifiées:
\begin{enumerate}
\item[(\emph{i})] $\sigma_{w_0w}^B\,.\,\sigma_{w_0w'}^B\,.\,\prod_{\beta\in\WT(M_{< m})}\Theta(-\beta)^{n_{\beta}}= \jmath^*(\pt)$,
\item[(\emph{ii})] $\langle w\lambda + w'\lambda,\rho\rangle + \sum_{k<m}(m-k)\dim_{\C}(M_{\lambda,k}) = 0$.
\end{enumerate}
\end{theo}

Le Théorème \ref{theo:cns_pairebiencouvrante} sera démontré dans la section \ref{section:démo_theo_pairebiencouvrante}. Dans l'énoncé du théorème, le cas $w'=w_0ww_{\lambda}$ et $M_{< m}=0$ est un cas particulier de (\emph{i}) et (\emph{ii}), si on utilise la convention que le produit d'éléments parcourant un ensemble vide est égal à $\sigma_{\id}^B = 1$, cf Lemme \ref{lemm:équation_lambda_casextrême}. Par ailleurs, les paires bien couvrantes $(C(w,w_0w,0),\lambda)$ sont très particulières. En effet, elles donnent les équations des faces du cône convexe polyédral $\CQ(-\WT(E))$ contenant $\Lambda$, quand on prend $M=E\oplus\C$, d'après les résultats du paragraphe \ref{subsection:propriétésgéométriquesdeDGIT}.

Lorsque le tore maximal $\TC$ fixe un élément non trivial dans $M$ ou, de manière équivalente, si le poids nul est un poids de la représentation $\rho$, nous sommes capables de donner une information supplémentaire sur l'entier $m$ qui peut apparaître dans les paires bien couvrantes. C'est le résultat énoncé par le Théorème \ref{theo:cn_pairebiencouvrante_poidsnul}. Ce résultat est essentiel pour pouvoir passer d'un ensemble fini d'équations du cône convexe polyédral $C_{\Q}(\XEC)^+$ à un ensemble fini d'équations du cône convexe polyédral $\Gamma_{\Q}(E)$, pour enfin obtenir des équations de $\DGIT$.


\begin{proof}[Preuve du Théorème \ref{theo:cn_pairebiencouvrante_poidsnul}]
 Soit $(w,w',m)\in W^{\lambda}\times W^{\lambda} \times\Z$ tel que $C(w,w',m)$ soit non vide dans $(X_M)^{\lambda}$.  Supposons que $m$ soit strictement positif. Alors le poids nul est dans $\WT(M_{< m})$ et $\prod_{\beta\in\WT(M_{< m})}\Theta(-\beta)^{n_{\beta}} = 0$, car $\Theta$ est un morphisme de $\Z$-modules. D'après le Théorème \ref{theo:cns_pairebiencouvrante}, on en déduit que la paire $(C(w,w',m),\lambda)$ n'est pas bien couvrante, car $M_{< m} \neq 0$ et l'assertion ($i$) de l'énoncé du Théorème \ref{theo:cns_pairebiencouvrante} n'est pas vérifiée.
\end{proof}

Nous terminons cette section avec un corollaire, qui donne une condition nécessaire simple pour qu'une paire $(C(w,w',m),\lambda)$ soit bien couvrante. Cette condition nécessaire concerne les longueurs de $w$ et $w'$.

Soit $(C(w,w',m),\lambda)$ une paire bien couvrante. D'après le Théorème \ref{theo:cns_pairebiencouvrante}, nous avons nécessairement
\[
 \sigma_{w_0w}^B\,.\,\sigma_{w_0w'}^B\,.\,\prod_{\beta\in\WT(M_{< m})}\Theta(-\beta)^{n_{\beta}}= \jmath^*(\pt),
\]
donc nous avons aussi une égalité sur les degrés en cohomologie. En effet, l'égalité précédente implique que
\[
 l(w_0w)+l(w_0w')+\dim_{\C}(M_{< m}) = \dim_{\C}(\KC/P(\lambda)).
\]
Or, pour tout $u\in W$, on a $l(w_0u) = l(w_0)-l(u)$, car $w_0$ est le plus grand élément de $W$. Par conséquent, on obtient
\[
 l(w)+l(w') = \dim_{\C}(M_{< m}) + 2l(w_0)-\dim_{\C}(\KC/P(\lambda)),
\]
qui s'écrit également de la manière suivante,
\[
 l(w)+l(w') = l(w_0)+l(w_{\lambda})+\dim_{\C}(M_{< m}).
\]

\begin{coro}
\label{coro:info_longueurs_éléments_pour_pairebiencouvrante}
 Soit $(w,w',m)\in W^{\lambda}\times W^{\lambda}\times\Z$ tel que $(C(w,w',m),\lambda)$ soit une paire bien couvrante. Alors,
\begin{equation}
 \label{eq:égalité_longueurs_casw'=w_0w}
l(w)+l(w') = l(w_0)+l(w_{\lambda})+\dim_{\C}(M_{< m}).
\end{equation}
En particulier, on a $M_{< m} = 0$ si et seulement si $w'=w_0ww_ {\lambda}$.
\end{coro}

\begin{proof}
 Il reste seulement à prouver la dernière assertion. Par la Remarque \ref{rema:N=r_ssi_M<m=0}, qui sera donnée en section \ref{section:notations_et_paramétrage}, et par le Théorème \ref{theo:cns_pairebiencouvrante}, si $M_{< m} = 0$, alors $w'=w_0ww_{\lambda}$ et un calcul direct nous donne
\begin{equation}
 \label{eq:égalité_longueurs_casw'=w_0ww_lambda}
 l(w)+l(w') = l(w_0)+l(w_{\lambda}).
\end{equation}
Nous avons utilisé ici le fait que $l(w'w_{\lambda}) = l(w')-l(w_{\lambda})$, car $w'$ est le plus long élément de $w'W_{\lambda}$ et $w_{\lambda}$ est le plus long élément de $W_{\lambda}$.

Réciproquement, si $w'=w_0ww_{\lambda}$, alors les équations \eqref{eq:égalité_longueurs_casw'=w_0w} et \eqref{eq:égalité_longueurs_casw'=w_0ww_lambda} sont satisfaites, donc $\dim_{\C}(M_{< m}) = 0$, d'où finalement $M_{< m} = 0$.
\end{proof}

\section{Notations et paramétrage}
\label{section:notations_et_paramétrage}

Avant de pouvoir démontrer le Théorème \ref{theo:cns_pairebiencouvrante}, nous allons choisir une bonne identification de $\mathbb{P}(M)$ avec une variété des drapeaux $\hKC/\hat{Q}$ \emph{ad hoc}, définie par un homomorphisme de groupes $f:\KC\rightarrow\hKC$ et un tore maximal $\hTC$ de $\hKC$ tels que $f(\TC)\subset\hTC$ et $\hTC\subset\hat{Q}$. Le choix d'une telle identification nous amène donc à réaliser un paramétrage convenable de notre espace vectoriel complexe $M$.

Rappelons que nous avons fixé une fois pour toute un sous-groupe à un paramètre dominant $\lambda$ de $\TC$ au début de la section \ref{section:critèreprincipal}. 

La premier point à remarquer est que l'on a les inclusions $\zeta(B)\subset\zeta(P(\lambda))\subset \hat{P}(\lambda)$, où $\hat{P}(\lambda)$ désigne le sous-groupe parabolique associé au sous-groupe à un paramètre $\zeta\circ\lambda$ de $GL_{\C}(M)$. Le sous-groupe $\zeta(B)$ de $\hat{P}(\lambda)$ est résoluble et connexe, donc il existe un sous-groupe de Borel $\hat{B}$ de $\hat{P}(\lambda)$ qui contient $\zeta(B)$. Le groupe $\hat{B}$ est alors un sous-groupe de Borel de $GL_{\C}(M)$. Fixons également un tore maximal de $\hTC$ de $\hat{B}$ contenant $\zeta(\TC)$.

Nous remarquons que, de la définition même des sous-espaces $M_{\lambda,k}$ pour $k$ parcourant $\Z$, le fait que $\zeta(\lambda(\C^*))$ soit contenu dans $\hTC$ implique que les sous-espaces $M_{\lambda,k}$ sont tous $\hTC$-stables. De plus, si on numérote de manière décroissante $\{k_1>k_2>\cdots>k_s\} = \{k\in\Z; M_{\lambda,k}\neq 0 \}$ les poids de $\lambda$ sur $M$, et si on note les sous-espaces $V_i = \oplus_{j=1}^i M_{\lambda,k_j}$, alors on voit également que le drapeau $0\subsetneq V_1 \subsetneq \cdots \subsetneq V_s = M$ est stable par $\hat{P}(\lambda)$, donc stable par $\hat{B}$. Or $\hat{B}$ est un groupe résoluble connexe, donc l'action de $\hat{B}$ sur chaque espace vectoriel $V_{i+1}/V_i$ est trigonalisable, ce qui nous donne au final un drapeau complet $0\subsetneq V'_1 \subsetneq \dots \subsetneq V'_r = M$ stable par $\hat{B}$ qui est imbriqué dans le drapeau $0\subsetneq V_1 \subsetneq \cdots \subsetneq V_s = M$, c'est-à-dire,
\[
0\subsetneq V'_1\subsetneq \cdots \subsetneq V'_{\dim V_1+\ldots+\dim V_i} = \bigoplus_{j=1}^i V_j \quad \text{pour tout $i=1,\ldots,s$}.
\]

Le tore maximal $\hTC$ étant formé d'éléments semi-simples de $GL_{\C}(M)$, on peut aisément vérifier qu'il existe une base $\mathcal{B}_{\lambda}=(u_1,\ldots,u_r)$ de $M$ formée de vecteurs propres communs de l'action de $\hTC$ sur $M$ tels que, pour tout $i=1,\ldots,r$, on ait
\[
\mathrm{Vect}(u_1,\ldots,u_i) = V'_i.
\]
Et évidemment, chaque $u_i$ est en particulier un vecteur propre de $\zeta(\TC)$. Pour tout $i\in\{1,\ldots,r\}$, on notera donc $\beta_i\in\WT(M)$ le poids de $\TC$ sur $M$ tel que $u_i\in M_{\beta_i}$. 
%
%
Nous aurons alors, pour tout $i\in\{1,\ldots,r\}$,
\[
\zeta(\lambda(t))u_i = t^{\langle\lambda,\beta_i\rangle}u_i \quad \text{pour tout $t\in\C^*$}.
\]
Il est clair que, du choix du drapeau $0\subsetneq V'_1 \subsetneq \dots \subsetneq V'_r = M$ donnant la base $\mathcal{B}_{\lambda}$, on obtient
\[
 \langle\lambda,\beta_i\rangle \geqslant \langle\lambda,\beta_{i+1}\rangle \quad \text{pour tout $i\in\{1,\ldots,n-1\}$}.
\]

Maintenant que nous avons fait ce choix de base de $M$, nous pouvons identifier l'espace projectif $\mathbb{P}(M)$ avec la variété des drapeaux $\hKC/\hat{Q}$, où $\hKC = GL_r(\C)$ et $\hat{Q}$ est le stabilisateur dans $\hKC$ de la droite vectorielle $\C u_1$, c'est-à-dire, le sous-groupe parabolique maximal
\begin{equation}
\label{eq:sousgrpparaboliquemax_hatQ}
\hat{Q} = \left(\begin{array}{c|ccc}
* & * & \ldots & * \\\hline
0 & * & \ldots & * \\
\vdots & \vdots & \ddots & \vdots \\
0 & * & \ldots & *
\end{array}\right)
\end{equation}
de $\hKC$. Cette identification est canoniquement définie par l'application
\[
 \hat{g}\hat{Q}\in\hKC/\hat{Q}\longmapsto[\hat{g}\cdot u_1]\in\mathbb{P}(M).
\]
Cette application est $\KC$-équivariante si on munit $\hKC/\hat{Q}$ de l'action induite par le morphisme de groupes
\[
 \begin{array}{cccl}
f_{\lambda} : & \KC & \longrightarrow & \hKC = GL_r(\C) \\
& g & \longmapsto & \mathcal{M}at_{\mathcal{B}_{\lambda}}(\zeta(g)).
 \end{array}
\]

Nous identifions de façon évidente $GL_{\C}(M)$ à $\hKC= GL_r(\C)$. Nous conservons les mêmes notations pour le sous-groupe de Borel $\hat{B}$ et le tore maximal $\hTC$ de $\hKC$ correspondants. Des paragraphes précédents, on voit que $\hat{B}$ est alors le sous-groupe des matrices triangulaires supérieures de $GL_r(\C)$ et $\hTC$ celui des matrices diagonales. Par la définition de $\hat{Q}$, nous avons $\hTC\subset\hat{B}\subset\hat{Q}$.


A partir de maintenant, nous identifions la $\KC$-variété $X_M$ au produit de variétés des drapeaux $\KC/B\times\KC/B\times\hKC/\hat{Q}$, munie de l'action diagonale de $\KC$. En suivant cette identification, nous avons une autre description simple et complète des composantes irréductibles de la variété projective $X_M^{\lambda}$. En effet, on a
\[
 X_M^{\lambda} = \bigcup_{\begin{array}{c}w,w'\in W/W_{\lambda}\\ \hat{w}\in\hat{W}_{\hat{Q}}\backslash\hat{W}/\hat{W}_{\lambda}\end{array}} \KC^{\lambda}w^{-1}B/B\times\KC^{\lambda}w'^{-1}B/B\times\hKC^{\lambda}\hat{w}^{-1}\hat{Q}/\hat{Q},
\]
où $\hat{W}=S_r$ (resp. $\hat{W}_{\lambda}$, resp. $\hat{W}_{\hat{Q}}=S_{r-1}$) est le groupe de Weyl de $\hKC$ (resp. du sous-groupe de Levi $\hKC^{\lambda}$ de $\hat{P}(\lambda)$, resp. du sous-groupe de Levi de $\hat{Q}$).

Pour $i=1,\ldots,r-1$, notons $s_i$ l'endomorphisme de permutation simple associée à la base canonique de $\C^r$. C'est-à-dire $s_i(u_i) = u_{i+1}$, $s_i(u_{i+1})=u_i$ et $s_i(u_k)=u_k$ si $k\notin\{i,i+1\}$. Posons $\hat{w}_k=s_1\circ\dots\circ s_{k-1}$, pour $k=2,\ldots,r$, et $\hat{w}_1 = \id$, éléments de $\hat{W}$. \label{spécial:définition_hatw_k}

Nous noterons également $\hat{w}_{\lambda}$ (resp. $\hat{w}_{\hat{Q}}$) le plus long élément du groupe de Weyl $\hat{W}_{\lambda}$ (resp. $\hat{W}_{\hat{Q}}$).

\begin{lemm}
\label{lemm:C_m_et_wQwN}
 Soit $N=\dim_{\C}(M_{\geqslant m})$. Alors $\mathbb{P}(M_{\lambda,m}) = \hKC^{\lambda}\hat{w}_N^{-1}\hat{Q}/\hat{Q}$ et $\hat{w}_{\hat{Q}}\hat{w}_N$ est le plus long élément de la classe $\hat{W}_{\hat{Q}}\hat{w}_N\hat{W}_{\lambda}$ de $\hat{W}_{\hat{Q}}\backslash\hat{W}/\hat{W}_{\lambda}$.
\end{lemm}

\begin{proof}
Nous savons déjà que $\mathbb{P}(M_{\lambda,m})$ est une composante irréductible de $\mathbb{P}(M)^{\lambda}$. Donc $\mathbb{P}(M_{\lambda,m}) = \hKC^{\lambda}\hat{w}^{-1}\hat{Q}/\hat{Q}$, pour une certaine classe $\hat{w}\in\hat{W}_{\hat{Q}}\backslash\hat{W}/\hat{W}_{\lambda}$. De plus, par l'identification $\hat{g}\hat{Q}\in\hKC/\hat{Q}\mapsto[\hat{g}\cdot u_1]\in\mathbb{P}(M)$, on a
\[
 \mathbb{P}(M_{\lambda,m}) = \hKC^{\lambda}\hat{w}^{-1}\hat{Q}/\hat{Q} \quad \text{si et seulement si} \quad \hat{w}^{-1}\cdot u_1 \in M_{\lambda,m}.
\]
Cela ne dépend pas du représentant de la classe $\hat{W}_{\hat{Q}}\hat{w}\hat{W}_{\lambda}$, car la droite $\C u_1$ est stabilisée par $\hat{Q}$, donc par les éléments de $\hat{W}_{\hat{Q}}$. Or, par une récurrence simple, on voit que $\hat{w}_N^{-1}\cdot u_1 = s_{N-1}\circ\cdots\circ s_1(u_1) = u_N$, car $s_i(u_i) = u_{i+1}$ pour tout $i\in\{1,\ldots,N-1\}$. Donc $\hat{w}_N^{-1}\cdot u_1 = u_N$ appartient à $M_{\lambda,m}$, d'après la définition de $N$, de $\mathcal{B}_{\lambda}$ et de l'ordre choisi pour les éléments de cette base de $M$.

La classe $\hat{W}_{\hat{Q}}\hat{w}_N\hat{W}_{\lambda}$ est stable par multiplication à gauche par un élément de $\hat{W}_{\hat{Q}}$. C'est donc l'union disjointe de classes de $\hat{W}_{\hat{Q}}\backslash\hat{W}$. D'après le Lemme \ref{lemm:éléments_les_plus_courts_de_W_mod_WQ}, l'ensemble $\{\hat{w}_k;k=1,\ldots,r\}$ est un système de plus courts représentants des classes de $\hat{W}_{\hat{Q}}\backslash\hat{W}$. Donc $\hat{W}_{\hat{Q}}\hat{w}_N\hat{W}_{\lambda}$ se décompose en une union disjointe de classes modulo $\hat{W}_{\hat{Q}}$ à gauche,
\[
 \hat{W}_{\hat{Q}}\hat{w}_N\hat{W}_{\lambda} = \bigcup_{k\in\{1,\ldots,r\},\langle\lambda,\beta_k\rangle=m}\hat{W}_{\hat{Q}}\hat{w}_k.
\]
Toujours par sa définition, $N$ est le plus grand entier de $\left\{k\in\{1,\ldots,r\}; \langle\lambda,\beta_k\rangle = m\right\}$. Or pour deux entiers $k<k'$, on a $l(v\hat{w}_k) = l(v)+k<l(v)+k' = l(v\hat{w}_{k'})$, pour tout élément $v\in\hat{W}_{\hat{Q}}$, d'après le Lemme \ref{lemm:éléments_les_plus_courts_de_W_mod_WQ}. On en conclut que $\hat{w}_{\hat{Q}}\hat{w}_N$ est le plus long élément de $\hat{W}_{\hat{Q}}\backslash\hat{W}/\hat{W}_{\lambda}$.
\end{proof}

\begin{rema}
 \label{rema:N=r_ssi_M<m=0}
On remarque que $N=\dim_{\C}(M_{\geqslant m}) = r - \dim_{\C}(M_{< m})$. On en déduit que $N=r$ si et seulement si $M_{< m} =0$.
\end{rema}

On déduit du Lemme \ref{lemm:C_m_et_wQwN} que la composante irréductible $C(w,w',m)$ de $X_M^{\lambda}$ est identifiée à la composante irréductible $C(w,w',\hat{w}_{\hat{Q}}\hat{w}_N)$ de $(\KC/B\times\KC/B\times\hKC/\hat{Q})^{\lambda}$, où $N=\dim_{\C}(M_{\geqslant m})$.

Soit $i_{\lambda}:g\in\KC\mapsto(g,f_{\lambda}(g))\in\KC\times\hKC$ l'injection que nous allons étudier dans la section suivante pour démontrer le Théorème \ref{theo:cns_pairebiencouvrante}. Cette injection induit l'immersion fermée $i_{\lambda}:\KC/P(\lambda)\rightarrow\KC/P(\lambda)\times\hKC/\hat{P}(\lambda)$. Et cette dernière induit une application en cohomologie
\[
i_{\lambda}^*:\coh{\KC/P(\lambda)\times\hKC/\hat{P}(\lambda)}\longrightarrow\coh{\KC/P(\lambda)}.
\]
Nous définissons également l'application
\[
 f_{\lambda}^{P(\lambda)}: gP(\lambda)\in\KC/P(\lambda)\longmapsto f_{\lambda}(g)\hat{P}(\lambda)\in\hKC/\hat{P}(\lambda),
\]
et l'application induite en cohomologie
\[
 (f_{\lambda}^{P(\lambda)})^*:\coh{\hKC/\hat{P}(\lambda)}\longrightarrow\coh{\KC/P(\lambda)}.
\]
Nous définissons de manière similaire $f_{\lambda}^B$ et $(f_{\lambda}^B)^*$ en remplaçant $P(\lambda)$ et $\hat{P}(\lambda)$ par respectivement $B$ et $\hat{B}$.

\section{Critère cohomologique}
\label{section:critère_cohomologique}

Comme première étape de la démonstration du Théorème \ref{theo:cns_pairebiencouvrante}, nous allons donner un critère cohomologique pour qu'une paire soit dominante (resp. couvrante). En effet, toute paire bien couvrante est en particulier couvrante (et donc dominante).

Soit $\tKC$ un groupe réductif connexe et $\KC$ un sous-groupe réductif connexe. Notons $i:\KC\hookrightarrow\tKC$ l'injection associée. Soit $\TC$ (resp. $\tTC$) un tore maximal et $B$ (resp. $\tilde{B}$) un sous-groupe de Borel de $\KC$ (resp. $\tKC$) tels que $\TC\subset B\subset\tilde{B}\supset\tTC\supset \TC$. Soit $Q$ (resp. $\tilde{Q}$) un sous-groupe parabolique de $\KC$ (resp. $\tKC$) contenant $\TC$ (resp. $\tTC$). On ne demande pas que $Q$ soit contenu dans $\tilde{Q}$. Rappelons que $\rho$ (resp. $\hat{\rho}$, resp. $\tilde{\rho}$) dénote la demi-somme des racines positives de $\got{k}_{\C}$ (resp. $\hat{\got{k}}_{\C}$, resp. $\tilde{\got{k}}_{\C}$).

\begin{lemm}[\cite{ressayre08}, Lemme 14]
\label{lemm:ressayre_lemme14}
Soit $\lambda$ un sous-groupe à un paramètre dominant de $\TC$ et $(w,\tilde{w})\in W\times\tilde{W}$ tel que $w$ (resp. $\tilde{w}$) est l'élément le plus long de la classe $W_{Q}wW_{\lambda}$ (resp. $\tilde{W}_{\tilde{Q}}\tilde{w}\tilde{W}_{\lambda}$). Alors :
\begin{enumerate}
\item la paire $(C(w,\tilde{w}),\lambda)$ est dominante si et seulement si $\sigma_{w_0w}^{P(\lambda)}\,.\,i^*\bigl(\sigma_{\tilde{w}_0\tilde{w}}^{\tilde{P}(\lambda)}\bigr) \neq 0$,
\item la paire $(C(w,\tilde{w}),\lambda)$ est couvrante si et seulement si $\sigma_{w_0w}^{P(\lambda)}\,.\,i^*\bigl(\sigma_{\tilde{w}_0\tilde{w}}^{\tilde{P}(\lambda)}\bigr) = \pt$.
\end{enumerate}
\end{lemm}


Fixons une fois pour toute un sous-groupe à un paramètre dominant $\lambda$ de $\TC$. Nous travaillons maintenant sur la variété $X_M$, c'est-à-dire qu'on applique le Lemme \ref{lemm:ressayre_lemme14} à $\tKC=\KC\times \hKC$ et à l'immersion fermée $i=i_{\lambda}:g\in\KC\hookrightarrow (g,i(g))\in\KC\times\hKC$ définie en fin de section \ref{section:notations_et_paramétrage}.

Soit $(w,w',m)$ un triplet de $W^{\lambda} \times W^{\lambda} \times \Z$ tel que $C(w,w',m)$ est non vide. On pose $\tilde{w} = (w',\hat{w}_{\hat{Q}}\hat{w}_N)\in\tilde{W}=W\times\hat{W}$. La paire $(w,\tilde{w})$ vérifie les hypothèses du Lemme \ref{lemm:ressayre_lemme14}. Et, par le Lemme \ref{lemm:C_m_et_wQwN}, la paire $(C(w,w',m),\lambda)$ de $X_M$ est dominante (resp. couvrante, resp. bien couvrante) si et seulement si la paire $(C(w,\tilde{w}),\lambda)$ de $\KC/B\times\KC/B\times\hKC/\hat{Q}$ est dominante (resp. couvrante, resp. bien couvrante).

\begin{lemm}
\label{lemm:alternatives_cohomologiques}
Soit $(w,w',m)\in W^{\lambda}\times W^{\lambda}\times\Z$. 
 Alors :
\begin{enumerate}
\item soit $M_{<0} = 0$, et alors $\jmath^*\left(\sigma_{w_0w}^{P(\lambda)}\,.\,i_{\lambda}^*\bigl(\sigma_{(w_0w',\hat{w}_0\hat{w}_{\hat{Q}}\hat{w}_N)}^{P(\lambda)\times\hat{P}(\lambda)}\bigr)\right) = \sigma_{w_0w}^{P(\lambda)}\,.\,\sigma_{w_0w'}^{P(\lambda)}$,
\item sinon $\jmath^*\left(\sigma_{w_0w}^{P(\lambda)}\,.\,i_{\lambda}^*\bigl(\sigma_{(w_0w',\hat{w}_0\hat{w}_{\hat{Q}}\hat{w}_N)}^{P(\lambda)\times\hat{P}(\lambda)}\bigr)\right) = \sigma_{w_0w}^{P(\lambda)}\,.\,\sigma_{w_0w'}^{P(\lambda)}\,.\,\prod_{\beta\in\WT(M_{< m})}\Theta(-\beta)^{n_{\beta}}$.
\end{enumerate}
\end{lemm}

\begin{proof}
Tout d'abord, on remarque que
\[
i^*\left(\sigma_{\tilde{w}_0\tilde{w}}^{\tilde{P}(\lambda)}\right) = i_{\lambda}^*\left(\sigma_{w_0w'}^{P(\lambda)}\otimes\sigma_{\hat{w}_0\hat{w}_{\hat{Q}}\hat{w}_N}^{\hat{P}(\lambda)}\right) = \sigma_{w_0w'}^{P(\lambda)}\,.\,\left(f_{\lambda}^{P(\lambda)}\right)^*\left(\sigma_{\hat{w}_0\hat{w}_{\hat{Q}}\hat{w}_N}^{\hat{P}(\lambda)}\right),
\]
puisque $i_{\lambda}$ est la composition des applications $(\id\times f_{\lambda}^{P(\lambda)})\circ\Delta$, où $\Delta$ est l'application diagonale $\KC/P\rightarrow \KC/P\times \KC/P$. Alors $i_{\lambda}^* = \Delta^*\circ(\id\times (f_{\lambda}^{P(\lambda)})^*)$ et $\Delta^*$ est le produit cup.

Puisque $\jmath^*$ est un morphisme d'anneaux pour le produit cup et $w_0w$ (resp. $w_0w'$) est le plus court élément de $w_0wW_{\lambda}$ (resp. $w_0wW_{\lambda}$), on a
\[
\jmath^*\left(\sigma_{w_0w}^{P(\lambda)}\,.\,\sigma_{w_0w'}^{P(\lambda)}\,.\,(f_{\lambda}^{P(\lambda)})^*\bigl(\sigma_{\hat{w}_0\hat{w}_{\hat{Q}}\hat{w}_N}^{\hat{P}(\lambda)}\bigr)\right) = \sigma_{w_0w}^{B}\,.\,\sigma_{w_0w'}^{B}\,.\,\jmath^*\left((f_{\lambda}^{P(\lambda)})^*\bigl(\sigma_{\hat{w}_0\hat{w}_{\hat{Q}}\hat{w}_N}^{\hat{P}(\lambda)}\bigr)\right).
\]
De plus, nous avons le diagramme commutatif suivant
\[
\begin{CD}
   \KC/B @>\jmath>> \KC/P(\lambda) \\
   @Vf_{\lambda}^BVV @VVf_{\lambda}^{P(\lambda)}V\\
   \hKC/\hat{B} @>\hat{\jmath}>> \hKC/\hat{P}(\lambda)
\end{CD}
\]
qui donne
\[
\jmath^*\left((f_{\lambda}^{P(\lambda)})^*\bigl(\sigma_{\hat{w}_0\hat{w}_{\hat{Q}}\hat{w}_N}^{\hat{P}(\lambda)}\bigr)\right) = (f_{\lambda}^B)^*\left(\hat{\jmath}^*\bigl(\sigma_{\hat{w}_0\hat{w}_{\hat{Q}}\hat{w}_N}^{\hat{P}}\bigr)\right) = (f_{\lambda}^B)^*\bigl(\sigma_{\hat{w}_0\hat{w}_{\hat{Q}}\hat{w}_N}^{\hat{B}}\bigr),
\]
car $\hat{w}_0\hat{w}_{\hat{Q}}\hat{w}_N$ est le plus court élément de sa classe de $\hat{W}/\hat{W}_{\lambda}$. Nous en déduisons donc l'égalité
\[
\jmath^*\left(\sigma_{w_0w}^{P(\lambda)}\,.\,\sigma_{w_0w'}^{P(\lambda)}\,.\,(f_{\lambda}^{P(\lambda)})^*\bigl(\sigma_{\hat{w}_0\hat{w}_{\hat{Q}}\hat{w}_N}^{\hat{P}(\lambda)}\bigr)\right) = \sigma_{w_0w}^B\,.\,\sigma_{w_0w'}^B\,.\,(f_{\lambda}^B)^*\bigl(\sigma_{\hat{w}_0\hat{w}_{\hat{Q}}\hat{w}_N}^{\hat{B}}\bigr),
\]
pour tout $(w,w',m)\in W^{\lambda} \times W^{\lambda} \times \Z$.

Nous utilisons la formule de Chevalley pour calculer $(f_{\lambda}^B)^*\bigl(\sigma_{\hat{w}_0\hat{w}_{\hat{Q}}\hat{w}_N}^{\hat{B}}\bigr)$. Tous les résultats nécessaires sont réunis dans le paragraphe \ref{subsection:formule_Chevalley}.

Du Lemme \ref{lemm:relation_wtilde_wchapeau}, on a $\hat{w}_0\hat{w}_{\hat{Q}}\hat{w}_{N} = s_{r-1}\ldots s_{N+1}s_N$ lorsque $N\in\{1,\ldots,r-1\}$, et $\hat{w}_0\hat{w}_{\hat{Q}}\hat{w}_{r} = \id$. Donc l'équation \eqref{eq:image_par_flambdatilde_de_sigmawtilde_finale} donne
\[
(f_{\lambda}^B)^*(\sigma_{\hat{w}_0\hat{w}_{\hat{Q}}\hat{w}_N}^{\hat{B}}) = (f_{\lambda}^B)^*(\sigma_{s_{r-1}\ldots s_N}^{\hat{B}}) = \Theta\bigl((-\beta_{N+1}) \ldots (-\beta_{r})\bigr),
\]
quand $N<r$, et
\[
(f_{\lambda}^B)^*(\sigma_{\hat{w}_0\hat{w}_{\hat{Q}}\hat{w}_r}^{\hat{B}}) = (f_{\lambda}^B)^*(\sigma_{\id}^{\hat{B}}) = \sigma_{\id}^B.
\]
En outre, on remarque que $\Theta(\beta_{N+1}\ldots\beta_{r}) = \prod_{\beta\in\WT(M_{<m})}\Theta(\beta)^{n_{\beta}}$, car $\beta_{N+1},\ldots,\beta_r$ sont les poids de $M_{<m}$ comptés avec multiplicité.

On obtient maintenant l'alternative suivante:
\begin{itemize}
\item soit $N=r$, et alors $\jmath^*\left(\sigma_{w_0w}^{P(\lambda)}\,.\,i_{\lambda}^*\bigl(\sigma_{(w_0w',\hat{w}_0\hat{w}_{\hat{Q}}\hat{w}_N)}^{P(\lambda)\times\hat{P}(\lambda)}\bigr)\right) = \sigma_{w_0w}^{P(\lambda)}\,.\,\sigma_{w_0w'}^{P(\lambda)}$,
\item soit $1\leqslant N<r$, et alors
\[
\jmath^*\left(\sigma_{w_0w}^{P(\lambda)}\,.\,i_{\lambda}^*\bigl(\sigma_{(w_0w',\hat{w}_0\hat{w}_{\hat{Q}}\hat{w}_N)}^{P(\lambda)\times\hat{P}(\lambda)}\bigr)\right) = \sigma_{w_0w}^{P(\lambda)}\,.\,\sigma_{w_0w'}^{P(\lambda)}\,.\,\prod_{\beta\in\WT(M_{< m})}\Theta(-\beta)^{n_{\beta}}.
\]
\end{itemize}
Or $N=r-\dim_{\C}M_{<0}$ d'après la Remarque \ref{rema:N=r_ssi_M<m=0}. On en déduit alors le résultat annoncé.
\end{proof}

La Proposition suivante donne un critère cohomologique pour qu'une paire de $X_M$ soit dominante (resp. couvrante).

\begin{prop}
\label{prop:critère_cohomologique_paire_dominante/couvrante}
Soit $(w,w',m)\in W^{\lambda}\times W^{\lambda}\times\Z$ tel que la paire $(C(w,w',m),\lambda)$ soit non vide. Alors la paire $(C(w,w',m),\lambda)$ est dominante (resp. couvrante) si et seulement
\begin{enumerate}
\item soit $M_{<0} = 0$ et $\sigma_{w_0w}^{P(\lambda)}\,.\,\sigma_{w_0w'}^{P(\lambda)}\neq 0$ (resp. $M_{<0} = 0$ et $w' = w_0ww_{\lambda}$),
\item soit  $\sigma_{w_0w}^{P(\lambda)}\,.\,\sigma_{w_0w'}^{P(\lambda)}\,.\,\prod_{\beta\in\WT(M_{< m})}\Theta(-\beta)^{n_{\beta}} \neq 0$ (resp. $\dots=\jmath^*(\pt)$).
\end{enumerate}
\end{prop}

Ce résultat découle des Lemmes \ref{lemm:ressayre_lemme14} et \ref{lemm:alternatives_cohomologiques}, de l'injectivité du morphisme $\jmath^*:\coh{\KC/P(\lambda)}\rightarrow\coh{\KC/B}$ , ainsi que de la Proposition \ref{prop:produit_deuxclassesschubert_dansG/P_vaut_classept} donnée ci-dessous. La Proposition \ref{prop:produit_deuxclassesschubert_dansG/P_vaut_classept} est un fait bien connu, généralisant le résultat pour le cas des sous-groupes de Borel dont on pourra trouver la preuve dans \cite[Lemme 1 et Proposition 1]{demazure} ou \cite{chevalley} par exemple.

\begin{prop}
\label{prop:produit_deuxclassesschubert_dansG/P_vaut_classept}
Soit $(w,w')\in(W^{\lambda})^2$ tel que $l(w)+l(w')\geqslant l(w_0)+l(w_{\lambda})$. Alors
\[
\sigma_{w}^{P(\lambda)}\,.\,\sigma_{w'}^{P(\lambda)} = \left\{\begin{array}{ll}
\pt & \text{si $w'=w_0ww_{\lambda}$,} \\
0 & \text{sinon}.
\end{array}\right.
\]
\end{prop}

Nous commençons par prouver deux lemmes. Pour $w\in W^{\lambda}$, nous noterons $X_w^{P(\lambda)} = \overline{BwP(\lambda)/P(\lambda)}$ la variété de Schubert généralisée associée à $w$.

\begin{lemm}
\label{lemm:intersection_classes_de_schubert_généralisées}
Soit $(w,w')\in (W^{\lambda})^2$.
\begin{enumerate}
\item Si $l(w)\leqslant l(w')$ et $w\neq w'$, alors $X_w^{P(\lambda)}\cap w_0X_{w_0w'w_{\lambda}}^{P(\lambda)} = \emptyset$.
\item L'intersection $X_w^{P(\lambda)}\cap w_0X_{w_0ww_{\lambda}}^{P(\lambda)}$ est transverse et réduite à $wP(\lambda)/P(\lambda)$.
\end{enumerate}
\end{lemm}

\begin{proof}
Comme $BwP(\lambda)/P(\lambda)$ est une orbite pour l'action de $B$ sur $\KC/B$, $X_w^{P(\lambda)}$ sera stable par $B$, donc union des $X_{w'}^{P(\lambda)}$, avec $\dim X_{w'}^{P(\lambda)} \leqslant \dim X_{w}^{P(\lambda)}$, et égalité si et seulement si $w=w'$ modulo $W_{\lambda}$. Le cas de l'égalité provient de la décomposition de Bruhat de $\KC/P(\lambda)$ en cellules disjointes $\cup_{w'\in W^{\lambda}}Bw'P(\lambda)/P(\lambda)$, cf \cite{BGG}. De plus, la dimension complexe d'une variété algébrique singulière $X_w^{P(\lambda)}$ est égale $l(w)-l(w_{\lambda})$, car $w$ est le plus long élément de sa classe $wW_{\lambda}$. Donc pour $w,w'\in W^{\lambda}$,
\begin{equation}
\label{eq:implications_inclusion_variétésdeschubert}
(w'P(\lambda)/P(\lambda)\in X_w^{P(\lambda)}) \Longleftrightarrow (X_{w'}^{P(\lambda)}\subset X_w^{P(\lambda)}) \Longrightarrow (l(w') < l(w) \mbox{ ou } w'=w).
\end{equation}

Fixons $w,w'\in W^{\lambda}$ tels que $l(w)\leqslant l(w')$, et supposons que $X_w^{P(\lambda)}\cap w_0 X_{w_0w'}^{P(\lambda)} \neq \emptyset$. Si $Y$ est une composante irréductible de $X_w^{P(\lambda)}\cap w_0 X_{w_0w'}^{P(\lambda)}$, c'est une variété projective non vide, stable par $\TC$, car chacune des deux variétés $X_w^{P(\lambda)}$ et $w_0 X_{w_0w'}^{P(\lambda)}$ est stable par $\TC$. Le groupe $\TC$ possède donc un point fixe dans $Y$, par le théorème de Borel (Théorème 21.2 de \cite{humphreys}). C'est en particulier un point fixe dans $\KC/P(\lambda)$, or le stabilisateur dans $\TC$ d'un point $gP(\lambda)/P(\lambda)$ est égal à $g\TC g^{-1}\cap\TC$. Donc $(\TC)_g = g\TC g^{-1}\cap\TC = \TC$ si et seulement si $g$ normalise $\TC$, c'est-à-dire $g=w''\in W$. Et quitte à multiplier à droite $g$ par un élément convenable de $W_{\lambda}\subset P(\lambda)$, on peut supposer que $g=w''\in W^{\lambda}$ (c'est-à-dire de longueur maximale dans sa classe modulo $W_{\lambda}$). Donc $w''P(\lambda)/P(\lambda)\in X_w^{P(\lambda)}$ et $w_0w''P(\lambda)/P(\lambda)\in X_{w_0w'}^{P(\lambda)}$. Cela implique donc, d'après \eqref{eq:implications_inclusion_variétésdeschubert}, que
\[
(l(w'') < l(w) \mbox{ ou } w'' = w)\ \mbox{ et } \ (l(w_0w''w_{\lambda}) < l(w_0w'w_{\lambda}) \mbox{ ou } w'' = w'),
\]
la dernière assertion venant du fait que le plus long élément de la classe $w_0w'W_{\lambda}$ est $w_0w'w_{\lambda}$. Ceci découle du fait que, pour tout $u\in W$, $l(w_0u) = l(w_0) - l(u)$. En particulier, $l(w_0w''w_{\lambda}) < l(w_0w'w_{\lambda})$ si et seulement si $l(w''w_{\lambda}) > l(w'w_{\lambda})$. Et comme $w'$ et $w''$ sont les éléments les plus longs de leurs classes, $l(w''w_{\lambda}) = l(w'') - l(w_{\lambda})$ et $l(w'w_{\lambda}) = l(w') - l(w_{\lambda})$. D'où $l(w_0w''w_{\lambda}) < l(w_0w'w_{\lambda})$ est équivalent à $l(w'') > l(w')$. De l'hypothèse $l(w)\leqslant l(w')$, on a nécessairement $w = w'' = w'$. On en déduit que si $w \neq w'$, alors $X_w^{P(\lambda)}\cap w_0 X_{w_0w'}^{P(\lambda)} = \emptyset$.

Maintenant, montrons le point (2). D'après le Théorème de transversalité de Kleiman \cite{kleiman}, \cite[Proposition 3]{belkale_kumar}, il existe un ouvert $\mathcal{U}$ de $\KC$ tel que $X_{w}^{P(\lambda)}$ intersecte $gw_0X_{w_0ww_{\lambda}}^{P(\lambda)}$ proprement et transversalement pour tout $g\in\mathcal{U}$. Mais $\KC$ est irréductible et $BB^-$ est un ouvert de $\KC$, donc $\mathcal{U}$ intersecte $BB^- = BU^- \cong B\times U^-$. Or $X_w^{P(\lambda)}$ est $B$-invariante et $w_0X_{w_0ww_{\lambda}}^{P(\lambda)}$ est $B^-$-invariante, donc pour $g=bu\in\mathcal{U}\cap BB^-$, on a
\[
X_{w}^{P(\lambda)}\cap (gw_0X_{w_0ww_{\lambda}}^{P(\lambda)}) = (bX_{w}^{P(\lambda)})\cap (bw_0X_{w_0ww_{\lambda}}^{P(\lambda)}) = b\cdot(X_{w}^{P(\lambda)}\cap w_0X_{w_0ww_{\lambda}}^{P(\lambda)}).
\]
Puisque l'intersection $X_{w}^{P(\lambda)}\cap (gw_0X_{w_0ww_{\lambda}}^{P(\lambda)})$ est transverse, c'est aussi le cas de l'intersection $X_{w}^{P(\lambda)}\cap w_0X_{w_0ww_{\lambda}}^{P(\lambda)}$. On en déduit que $X_{w}^{P(\lambda)}\cap w_0X_{w_0ww_{\lambda}}^{P(\lambda)}$ est de dimension $0$, c'est-à-dire est un ensemble fini de points de $\KC/P(\lambda)$. Or chacun de ces points est une composante irréductible de $X_{w}^{P(\lambda)}\cap w_0X_{w_0ww_{\lambda}}^{P(\lambda)}$, donc les points de l'intersection $X_{w}^{P(\lambda)}\cap w_0X_{w_0ww_{\lambda}}^{P(\lambda)}$ sont fixés par $\TC$. On a vu qu'un tel point est unique, c'est $wP(\lambda)/P(\lambda)$. D'où le résultat.
\end{proof}

Dans l'énoncé suivant, si $X$ est une sous-variété complexe singulière de $\KC/P(\lambda)$, $[X]$ désigne la classe de cohomologie obtenue par dualité de Poincaré à partir de la classe fondamentale de la sous-variété complexe $X$, cf Annexe \ref{chap:classefondamentale}. En particulier, on a $[X_w^{P(\lambda)}] = \sigma_{w_0ww_{\lambda}}^{P(\lambda)}$ pour tout $w\in W^{\lambda}$.

\begin{lemm}
\label{lemm:produit_2classesschubert_et_classedelintersection}
Pour tout $(w,w')\in(W^{\lambda})^2$, on a
\[
\sigma_w^{P(\lambda)}\,.\,\sigma_{w'}^{P(\lambda)} = [X_{w_0ww_{\lambda}}^{P(\lambda)}\cap w_0X_{w_0w'w_{\lambda}}^{P(\lambda)}].
\]
\end{lemm}

\begin{proof}
On applique à nouveau le Théorème de transversalité de Kleiman, de manière identique à la deuxième partie de la preuve du Lemme \ref{lemm:intersection_classes_de_schubert_généralisées}. Il existe un ouvert non vide $\mathcal{U}$ de $\KC$ tel que $X_{w}^{P(\lambda)}$ intersecte $gw_0X_{w_0w'w_{\lambda}}^{P(\lambda)}$ proprement et transversalement pour tout $g\in\mathcal{U}$. L'ouvert $\mathcal{U}$ de $\KC$ intersecte l'ouvert $BU^-$. De plus, $X_w^{P(\lambda)}$ est $B$-invariante et $w_0X_{w_0w'w_{\lambda}}^{P(\lambda)}$ est $B^-$-invariante, donc pour tout $g=bu\in\mathcal{U}\cap BB^-\subset\KC$, on a
\[
X_{w}^{P(\lambda)}\cap (gw_0X_{w_0w'w_{\lambda}}^{P(\lambda)}) = (bX_{w}^{P(\lambda)})\cap (bw_0X_{w_0w'w_{\lambda}}^{P(\lambda)}) = b\cdot(X_{w}^{P(\lambda)}\cap w_0X_{w_0w'w_{\lambda}}^{P(\lambda)}).
\]
Par conséquent, on a $[X_{w}^{P(\lambda)}]\,.\,[X_{w_0w'w_{\lambda}}^{P(\lambda)}] = [X_{w}^{P(\lambda)}]\,.\,[w_0X_{w_0w'w_{\lambda}}^{P(\lambda)}] = [X_{w}^{P(\lambda)}\cap w_0X_{w_0w'w_{\lambda}}^{P(\lambda)}]$. On en conclut que $\sigma_w^{P(\lambda)}\,.\,\sigma_{w'}^{P(\lambda)} = [X_{w_0ww_{\lambda}}^{P(\lambda)}]\,.\,[X_{w_0ww_{\lambda}}^{P(\lambda)}] = [X_{w_0ww_{\lambda}}^{P(\lambda)}\cap w_0 X_{w_0w'w_{\lambda}}^{P(\lambda)}]$.
\end{proof}

\begin{proof}[Preuve de la Proposition \ref{prop:produit_deuxclassesschubert_dansG/P_vaut_classept}]
Puisque $l(w)+l(w')\geqslant l(w_0)+l(w_{\lambda})$, on a clairement $l(w_0ww_{\lambda}) \leqslant l(w')$. Donc, d'après le Lemme \ref{lemm:intersection_classes_de_schubert_généralisées}, on a 
\[
X_{w_0ww_{\lambda}}^{P(\lambda)}\cap w_0 X_{w_0w'w_{\lambda}}^{P(\lambda)} = \left\{\begin{array}{l}
w_0ww_{\lambda}P(\lambda) \quad \text{si $w' = w_0ww_{\lambda}$,}\\
\emptyset \quad \text{sinon.}
\end{array}\right.
\]
Le résultat découle alors directement du Lemme \ref{lemm:produit_2classesschubert_et_classedelintersection}.
\end{proof}

\section{Démonstration du Théorème \ref{theo:cns_pairebiencouvrante}}
\label{section:démo_theo_pairebiencouvrante}

Nous allons maintenant prouver la condition nécessaire et suffisante pour qu'une paire de $X_M$ soit bien couvrante. Le théorème qui suit, de \cite{ressayre08}, donne un critère dans un contexte plus général que celui qui nous intéresse dans ce chapitre. Nous utiliserons les notations introduites dans la section \ref{section:critère_cohomologique}.

\renewcommand{\theenumi}{\textit{\alph{enumi}}}

\begin{theo}[\cite{ressayre08}, Proposition 11]
\label{theo:cns_paire_est_biencouvrante_casgénéral}
Soit $(w,\tilde{w})\in W\times\tilde{W}$ tel que $w$ (resp. $\tilde{w}$) est l'élément le plus long de la classe $W_{Q}wW_{\lambda}$ (resp. $\tilde{W}_{\tilde{Q}}\tilde{w}\tilde{W}_{\lambda}$). Alors la paire $(C(w,\tilde{w}),\lambda)$ est bien couvrante si et seulement si les deux assertions suivantes sont vérifiées :
\begin{enumerate}
\item $\sigma_{w_0w}^{P(\lambda)}\,.\,i^*\bigl(\sigma_{\tilde{w}_0\tilde{w}}^{\tilde{P}(\lambda)}\bigr) = \pt$,

\item $\langle \lambda,\rho+w^{-1}\rho\rangle + \langle i^*(\lambda),\tilde{\rho}+\tilde{w}^{-1}\tilde{\rho}\rangle = \langle\lambda,2\rho\rangle$.
\end{enumerate}
\end{theo}

Dans le cas de la variété $X_M$, nous sommes capables d'améliorer les formules données dans l'énoncé du théorème ci-dessus.

Fixons de nouveau un sous-groupe à un paramètre dominant $\lambda$ de $\TC$. On applique le Théorème \ref{theo:cns_paire_est_biencouvrante_casgénéral} à nouveau à $\tKC=\KC\times \hKC$ et à la projection $i=i_{\lambda}:\KC\hookrightarrow \KC\times\hKC$ définie en fin de section \ref{section:notations_et_paramétrage}.

Soit $(w,w',m)$ un triplet de $W^{\lambda} \times W^{\lambda} \times \Z$ tel que $C(w,w',m)$ est non vide. On pose $\tilde{w} = (w',\hat{w}_{\hat{Q}}\hat{w}_N)\in\tilde{W}=W\times\hat{W}$. La paire $(w,\tilde{w})$ vérifie les hypothèses du Théorème \ref{theo:cns_paire_est_biencouvrante_casgénéral}. Et, par le Lemme \ref{lemm:C_m_et_wQwN}, $(C(w,w',m),\lambda)$ est une paire bien couvrante de $X_M$ si et seulement si $(C(w,\tilde{w}),\lambda)$ est bien couvrante dans $\KC/B\times\KC/B\times\hKC/\hat{Q}$.

\bigskip

L'assertion (\emph{a}) du Théorème \ref{theo:cns_paire_est_biencouvrante_casgénéral} pour $(w,\tilde{w})$, est équivalente à l'alternative suivante:
\begin{itemize}
\item soit $N=r$ et $w' = w_0ww_{\lambda}$,
\item soit $1\leqslant N<r$ et $\sigma_{w_0w}^B\,.\,\sigma_{w_0w'}^B\,.\,\prod_{\beta\in\WT(M_{<m})}\Theta(-\beta)^{n_{\beta}} = \jmath^*(\pt)$.
\end{itemize}
d'après les Lemmes \ref{lemm:ressayre_lemme14} et \ref{lemm:alternatives_cohomologiques}, la Proposition \ref{prop:critère_cohomologique_paire_dominante/couvrante} et la Remarque \ref{rema:N=r_ssi_M<m=0}.

Il ne reste qu'à prouver que l'assertion ($b$) du Théorème \ref{theo:cns_paire_est_biencouvrante_casgénéral}, pour $(w,\tilde{w})$ avec $\tilde{w}=(w',\hat{w}_{\hat{Q}}\hat{w}_N)$, est équivalente à l'équation linéaire
\[
\langle w\lambda+w'\lambda,\rho\rangle+\sum_{k<m}(m-k)\dim_{\C}(M_{\lambda,k}) = 0.
\]
Nous allons essentiellement utiliser le lemme suivant.

\begin{lemm}
\label{lemm:gammahat}
On a $\hat{\rho}+(\hat{w}_{\hat{Q}}\hat{w}_N)^{-1}\hat{\rho} = \sum_{l=N+1}^{r}\halpha_{k,l}$.
\end{lemm}

\begin{proof}
D'après le Lemme \ref{lemm:relation_wtilde_wchapeau}, on sait que
\[
w_0\hat{w}_{\hat{Q}}\hat{w}_N = \check{w}_N := s_{r-1}\ldots s_{N+1}s_N.
\]
Donc $\hat{w}_{\hat{Q}}\hat{w}_N = w_0\check{w}_N$ et
\[
(\hat{w}_{\hat{Q}}\hat{w}_N)^{-1}\hat{\rho} + \hat{\rho} = \check{w}_N^{-1}\hat{w}_0\hat{\rho} + \hat{\rho}.
\]
Or il est clair que $(\check{w}_N^{-1}\hat{w}_0\hat{\rho} + \hat{\rho})$ est la somme des racines $\halpha\in\hat{\got{R}}^+$ telles que $\hat{w}_0\check{w}_N(\halpha)$ est positive. Comme $\hat{w}_0$ permute les deux ensembles de racines $\hat{\got{R}}^+$ et $\hat{\got{R}}^-$, $(\check{w}_N^{-1}\hat{w}_0\hat{\rho} + \hat{\rho})$ est la somme des racines positives $\halpha$ telles que $\check{w}_N(\halpha)$ est négative. Ces racines sont les racines $\halpha_{N,l}$, pour $N+1 \leqslant l \leqslant r$. Le lemme en résulte.
\end{proof}

Il est clair qu'en développant $\langle i_{\lambda}^*(\lambda),\tilde{\rho}+\tilde{w}^{-1}\tilde{\rho}\rangle$ selon la définition de $i_{\lambda}$, de $\tilde{\rho}$ et $\tKC$, on obtient
\[
\langle i_{\lambda}^*(\lambda),\tilde{\rho}+\tilde{w}^{-1}\tilde{\rho}\rangle = \langle\lambda,\rho+w'^{-1}\rho\rangle + \langle f_{\lambda}^*(\lambda),\hat{\rho}+(\hat{w}_{\hat{Q}}\hat{w}_N)^{-1}\hat{\rho}\rangle.
\]
Par le Lemme \ref{lemm:gammahat} et en utilisant le fait que, pour tout $k=N+1,\ldots,r$, $\langle\lambda,\hat{\alpha}_{N,k}\rangle = \langle\lambda,\beta_N\rangle-\langle\lambda,\beta_k\rangle$, on obtient
\[
\langle f_{\lambda}^*(\lambda),\hat{\rho}+(\hat{w}_{\hat{Q}}\hat{w}_N)^{-1}\hat{\rho}\rangle = \sum_{l=N+1}^r(\langle\lambda,\beta_N\rangle -\langle\lambda,\beta_l\rangle).
\]
De plus, on a évidemment $M_{<m} = \bigoplus_{l=N+1}^r M_{\beta_l}$. Cela implique que
\[
\langle f_{\lambda}^*(\lambda),\hat{\rho}+(\hat{w}_{\hat{Q}}\hat{w}_N)^{-1}\hat{\rho}\rangle = \sum_{k<m}(m-k)\dim_{\C}(M_{\lambda,k}) \ .
\]
Par conséquent, l'assertion ($b$) du Théorème \ref{theo:cns_paire_est_biencouvrante_casgénéral} peut être réécrite de la manière suivante,
\begin{align*}
\langle \lambda,\rho+w^{-1}\rho\rangle + \langle i^*(\lambda),\tilde{\rho}+\tilde{w}^{-1}\tilde{\rho}\rangle-\langle\lambda,2\rho\rangle & = \langle \lambda,w^{-1}\rho+w'^{-1}\rho\rangle \\
& + \sum_{k<m}(m-k)\dim_{\C}(M_{\lambda,k}).
\end{align*}
Donc nous avons prouvé que l'assertion ($b$) du Théorème \ref{theo:cns_paire_est_biencouvrante_casgénéral} pour $(w,\tilde{w})$, est équivalente à $\langle w\lambda+w'\lambda,\rho\rangle+\sum_{k<m}(m-k)\dim_{\C}(M_{\lambda,k}) = 0$.

\begin{lemm}
\label{lemm:équation_lambda_casextrême}
Pour tout $w\in W^{\lambda}$, on a
\[
\langle w\lambda+w_0ww_{\lambda}\lambda,\rho\rangle=0.
\]
\end{lemm}

\begin{proof}
Cela provient directement de $w_0\rho = -\rho$ et $w_{\lambda}\lambda = \lambda$, puisque $w_{\lambda}\in W_{\lambda}$.
\end{proof}

Dans le cas $N=r$, nous avons $M_{< m} = 0$, d'après la Remarque \ref{rema:N=r_ssi_M<m=0}. Et cela implique clairement que $\sum_{k<m}(m-k)\dim_{\C}(M_{\lambda,k}) = 0$. Par le Lemme \ref{lemm:équation_lambda_casextrême}, on voit que l'équation
\[
\langle w\lambda+w'\lambda,\rho\rangle+\sum_{k<m}(m-k)\dim_{\C}(M_{\lambda,k}) = 0
\]
est toujours satisfaite. C'est-à-dire, si $w'=w_0ww_{\lambda}$, alors l'assertion ($b$) est toujours vérifiée.

\bigskip

On en conclut que nous avons prouvé que la paire $(C(w,w',m),\lambda)$ est bien couvrante si et seulement si on a soit $w'=w_0ww_{\lambda}$ et $M_{<m}=0$ (c'est-à-dire $N=r$), soit les deux assertions suivantes sont vérifiées :
\begin{enumerate}
\item $\sigma_{w_0w}^B\,.\,\sigma_{w_0w'}^B\,.\,\prod_{\beta\in\WT(M_{<m})}\Theta(-\beta)^{n_{\beta}} = \sigma_{w_0w_{\lambda}}^B$,
\item $\langle w\lambda+w'\lambda,\rho\rangle+\sum_{k<m}(m-k)\dim_{\C}(M_{\lambda,k}) = 0$.
\end{enumerate}
C'est exactement l'énoncé du Théorème \ref{theo:cns_pairebiencouvrante}. Ceci termine la preuve.

\chapter[Projection d'orbites coadjointes holomorphes]{Projection d'orbites coadjointes holomorphes. Exemples}
\label{chap:calcul_exemples_par_pairesbiencouvrantes}


Nous terminons l'étude de la projection d'orbites coadjointes holomorphes en accomplissant le calcul explicite des équations de ces polyèdres dans plusieurs cas de la classification des espaces symétriques hermitiens irréductibles donnée dans le paragraphe \ref{section:description_espace_symetrique_hermitien}.

\section{\'Equations du polyèdre moment $\Delta_K(\Orb_{\Lambda})$}

Nous utilisons les résultats des Chapitres \ref{chap:symplecto_mcduff} et \ref{chap:projectiondorbitecoadjointe+GIT} pour déterminer de manière générale les équations de la projection $\Delta_K(\Orb_{\Lambda})$ d'une orbite coadjointe holomorphe .

\subsection{Deux descriptions du polyèdre moment de $\Orb_{\Lambda}$}
\label{subsection:projection_moment_de_GLambda}

Soit $G$ un groupe de Lie réel semi-simple connexe non compact à centre fini, $K$ un sous-groupe compact maximal. Nous continuons à supposer que l'espace symétrique $G/K$ est hermitien. Nous noterons $\got{g}$ et $\got{k}$ les algèbres de Lie respectives. L'algèbre de Lie compacte $\got{k}$ provient d'une décomposition de Cartan $\got{g} = \got{k}\oplus\got{p}$.

D'après le paragraphe \ref{section:description_espace_symetrique_hermitien}, nous pouvons fixer un élément $z_0\in\got{z}(K)$ tel que $\ad(z_o)|_{\got{p}}^2 = -\mathrm{id}_{\got{p}}$. Rappelons que $\got{p}^{\pm}$ est le sous-$K$-module $\ker(\ad(z_o)\mp i)$ de $\got{p}_{\C}$.


Soit $T$ un tore maximal de $K$ et $\got{t}$ son algèbre de Lie. Fixons $\got{t}^*_+$ une chambre de Weyl de $\got{t}^*$.

Prenons un élément $\Lambda$ de $\Chol$. Soit $\Phi_{\Orb_{\Lambda}}:\Orb_{\Lambda}\rightarrow\got{k}^*$ l'application moment donnée par la projection d'orbite définie au paragraphe \ref{section:proj_orbites_defi_et_exemples}.

Rappellons que le but principal de cette thèse est de déterminer les faces du polyèdre $\Delta_K(\Orb_{\Lambda}) := \Phi_{\Orb_{\Lambda}}(\Orb_{\Lambda})\cap\got{t}^*_+$, autrement dit, donner les équations affines qui caractérisent cet ensemble. Nous allons, pour cela, utiliser un résultat qui donne une autre description de $\Delta_K(\Orb_{\Lambda})$ comme polyèdre moment d'une variété plus simple à étudier.

Nous avons vu dans le Chapitre \ref{chap:symplecto_mcduff} que l'orbite coadjointe holomorphe $\Orb_{\Lambda}$ est $K$-symplectomorphe à la variété symplectique $(K\cdot\Lambda\times\got{p},\Omega_{K\cdot\Lambda\times\got{p}})$, munie de l'action diagonale de $K$. Ici $\Omega_{K\cdot\Lambda\times\got{p}}$ désigne la forme symplectique sur $K\cdot\Lambda\times\got{p}$ obtenue comme produit direct de la forme symplectique de Kirillov-Kostant-Souriau $\Omega_{K\cdot\Lambda}$ sur l'orbite coadjointe compacte $K\cdot\Lambda$ et de la forme symplectique linéaire $\Omega_{\got{p}}$ définie sur $\got{p}$ par
\begin{equation}
\label{eq:forme_symplectique_linéaire_sur_gotp}
\Omega_{\got{p}}(X,Y) = B_{\got{g}}(X, \ad(z_0)Y) \quad\mbox{pour tous $X,Y\in\got{p}$}.
\end{equation}
Le produit direct de ces deux formes symplectiques est donné par
\[
\Omega_{K\cdot\Lambda\times\got{p}} = \pi_{K\cdot\Lambda}^*\Omega_{K\cdot\Lambda} + \pi_{\got{p}}^*\Omega_{\got{p}},
\]
où $\pi_{K\cdot\Lambda}:K\cdot\Lambda\times\got{p}\rightarrow K\cdot\Lambda$ et $\pi_{\got{p}}:K\cdot\Lambda\times\got{p}\rightarrow\got{p}$ sont les deux projections canoniques. L'application moment sur $\got{p}$ associée est l'application $\Phi_{\got{p}}:\got{p}\rightarrow\got{k}^*$ définie pour tout $v\in\got{p}$ par
\[
\langle\Phi_{\got{p}}(v),X\rangle = \frac{1}{2}B_{\got{g}}(\ad(X)v,\ad(z_0)v) \text{ pour tout } X\in\got{k}.
\]
Une application moment pour l'action hamiltonienne de $K$ sur $K\cdot\Lambda\times\got{p}$ est
\[
\Phi_{K\cdot\Lambda\times\got{p}}:(\xi,v)\in K\cdot\Lambda\times \got{p} \longmapsto \xi+\Phi_{\got{p}}(v)\in\got{k}^*.
\]
On définit alors le polyèdre moment
\[
\Delta_K(K\cdot\Lambda\times\got{p}) := \Phi_{K\cdot\Lambda\times\got{p}}(K\cdot\Lambda\times\got{p})\cap\got{t}_+^*
\]
associé à la variété hamiltonienne $(K\cdot\Lambda\times\got{p}, \Omega_{K\cdot\Lambda\times\got{p}}, \Phi_{K\cdot\Lambda\times\got{p}})$.

Comme corollaire direct du Théorème \ref{theo:symplecto_principal}, nous avons le résultat suivant, qui est tout simplement l'énoncé du Corollaire \ref{coro:égalité_polyèdresmoments}.

\begin{theo}
\label{theo:polyèdre_GLambda_KLambdap}
Soit $G$ un groupe de Lie réel semi-simple connexe non compact à centre fini, tel que l'espace symétrique $G/K$ soit hermitien. Alors, pour tout $\Lambda\in\Chol$, on a
\[
\Delta_K(\Orb_{\Lambda}) = \Delta_K(K\cdot\Lambda\times\got{p}).
\]
\end{theo}


\begin{rema}
\label{rema:isomorphisme_entre_gotp_et_gotp-}
Pour éviter toute ambiguïté, dans la suite, nous utiliserons généralement $\got{p}^-$ à la place de $\got{p}$ pour accentuer le fait que la structure complexe $K$-invariante considérée sur $\got{p}$ est la structure complexe définie par $-\ad(z_0)$. On constate qu'en tant que sous-espaces vectoriels réels de $\got{p}_{\C}$, les espaces $\got{p}$ et $\got{p}^-$ sont isomorphes pour l'application linéaire $L:X\in\got{p}\mapsto X+i\ad(z_0)X\in\got{p}^-$, qui vérifie $L(-\ad(z_0)X) = iL(X)$ pour tout $X\in\got{p}$.
\end{rema}

On s'est donc ramené à l'étude des équations du polyèdre moment de la variété $K\cdot\Lambda\times\got{p}^-$, produit direct de l'orbite coadjointe compacte $K\cdot\Lambda$, où $\Lambda\in\Chol\subset\got{t}_+^*$, et de la représentation complexe de dimension finie $\got{p}^-$.

On peut définir une structure hermitienne $K$-invariante sur $\got{p}^-$ par
\[
h_{\got{p}^-}(X,Y) = B_{\got{g}}(X, Y) - iB_{\got{g}_0}(X, \ad(z_0)Y) \mbox{ pour tous } X,Y\in\got{p},
\]
pour l'isomorphisme $K$-équivariant $L:\got{p}\rightarrow\got{p}^-$ défini dans la Remarque \ref{rema:isomorphisme_entre_gotp_et_gotp-}. La forme symplectique $K$-invariante associée est
\[
\Omega_{\got{p}^-}(X,Y) = B_{\got{g}}(X, \ad(z_0)Y) \mbox{ pour tous } X,Y\in\got{p},
\]
c'est-à-dire la forme symplectique correspondant à \eqref{eq:forme_symplectique_linéaire_sur_gotp} par l'isomorphisme $L$. On note $\zeta: k\in K\rightarrow \Ad(k)|_{\got{p}^-}\in U(\got{p}^-)$ le morphisme de groupes ainsi obtenu.

Soit $\Lambda$ un élément de $\wedge_{\Q}^*\cap\Chol$. Le Théorème \ref{theo:polyèdre_GLambda_KLambdap} et le Corollaire \ref{coro:DeltaClassique_DeltaGIT} montrent que $\Delta_K(\Orb_{\Lambda}) = \Delta_K(K\cdot\Lambda\times\got{p}^-)$ est l'adhérence dans $\got{t}^*$ du polyèdre convexe $\DGIT[K\cdot\Lambda\times\got{p}^-]$ de $\wedge^*_{\Q}$. Nous appliquons donc les résultats de la section \ref{section:équations_polyèdremoment_KLambdaE}, pour $E = \got{p}^-$. Comme vu dans l'Exemple \ref{exple:Phi_gotp_propre}, l'application moment $\Phi_{\Orb_{\Lambda}}$ associée est propre.

L'hypothèse qui nous manque est la propriété de finitude du noyau du morphisme $\zeta:K\rightarrow U(\got{p}^-)$ nécessaire pour terminer l'étude du polyèdre $\Delta_K(K\cdot\Lambda\times\got{p}^-)$. Dans le cas des espaces symétriques hermitiens irréductibles, nous avons une réponse positive à la question de finitude du noyau du morphisme $\zeta$.

\begin{lemm}
\label{lemm:noyau_de_zeta_fini}
Si $G$ est simple, alors le noyau du morphisme de groupes $\zeta:K\rightarrow U(\got{p}^-)$ est égal au centre $Z(G)$ de $G$. En particulier, $\ker\zeta$ est fini.
\end{lemm}

\begin{proof}
Soit $k\in\ker\zeta$. Cela signifie que, pour tout $Y\in\got{p}^-$, nous avons $\Ad(k)Y = Y$. Puisque les deux $K$-modules $\got{p}$ et $\got{p}^-$ sont isomorphes, nous aurons aussi $\Ad(k)Y' = Y'$ pour tout $Y'\in\got{p}$. On a donc, pour tous $Y,Z\in\got{p}$,
\[
\Ad(k)[Y,Z] = [\Ad(k)Y,\Ad(k)Z] = [Y,Z].
\]
Par conséquent, $\Ad(k)$ est l'identité sur $[\got{p},\got{p}] = \got{k}$, ceci étant vrai car $\got{g}$ est simple (d'après \cite{knapp}, problème $24$ page $430$). Par linéarité, puisque $\got{g} = \got{k} \oplus \got{p}$, $\Ad(k)$ laisse fixe chaque élément de $\got{g}$. D'où $k\in \ker(\Ad) = Z(G)$. On en déduit l'inclusion $\ker\zeta\subset Z(G)$. Et il est clair que $Z(G)\subset\ker\zeta$. D'où le résultat.
\end{proof}

\subsection{Les équations de la projection d'orbite coadjointe holomorphe}

Maintenant nous considérons le groupe produit direct
\begin{equation}
\label{eq:decomposition_de_G_en_facteurs_simples}
G\cong K_0\times G_1\times\cdots\times G_s,
\end{equation}
où $K_0$ est un groupe de Lie semi-simple compact connexe à centre fini, $s\geqslant 1$ et $G_,\ldots,G_s$ des groupes de Lie simples connexes non compacts à centres finis.

Chaque algèbre de Lie $\got{g}_i$, associée au groupe $G_i$, a une décomposition de Cartan $\got{g}_i=\got{k}_i\oplus\got{p}_i$, pour $i=1,\ldots,s$, induisant une décomposition de Cartan $\got{g} = \got{k}\oplus\got{p}$ pour $\got{g}$, où
\[
\got{k} := \bigoplus_{i=0}^s\got{k}_i, \qquad\text{et}\qquad \got{p} := \bigoplus_{i=1}^s\got{p}_i.
\]
Si $K,K_1,\ldots,K_s$ désignent les sous-groupes de Lie connexes respectifs de $G,G_1,\ldots,G_s$ d'algèbres de Lie $\got{k},\got{k}_1,\ldots,\got{k}_s$, alors l'isomorphisme entre $G$ et $K_0\times G_1\times\cdots\times G_s$ induit un isomorphisme entre $K$ et $K_0\times K_1\times \cdots\times K_s$.

Fixons, pour chaque $i=0,\ldots,s$, un tore maximal $T_i$ de $K_i$ et notons $T$ le tore de $K$ s'identifiant à $T_0\times T_1\times \cdots\times T_s$ par l'isomorphisme précédent. Il s'agit également d'un tore maximal de $K$. On note leurs algèbres de Lie respectives $\got{t}$ et $\got{t}_0,\got{t}_1,\ldots,\got{t}_s$, et choisissons des chambres de Weyl $\got{t}_+^*$ et $(\got{t}_0^*)_+,(\got{t}_1^*)_+,\ldots,(\got{t}_s^*)_+$, de sorte que $\got{t}_+^*\cong (\got{t}_0^*)_+\times(\got{t}_1^*)_+\times\cdots\times(\got{t}_s^*)_+$.

Soit $\Lambda\in\got{t}_+^*$, qui s'identifie à l'élément $(\Lambda_0,\Lambda_1,\ldots,\Lambda_s)$ de $(\got{t}_0^*)_+\times(\got{t}_1^*)_+\times\cdots\times(\got{t}_s^*)_+$. Alors l'orbite coadjointe $\Orb_{\Lambda}$ s'identifie clairement au produit des orbites coadjointes $\Orb_{\Lambda_0}\times\Orb_{\Lambda_1}\times\cdots\times\Orb_{\Lambda_s}$.

L'action induite de $K$ sur $\Orb_{\Lambda}$ correspond à l'action de $K_0\times K_1\times \cdots\times K_s$ sur $\Orb_{\Lambda_0}\times\Orb_{\Lambda_1}\times\cdots\times\Orb_{\Lambda_s}$. Ces deux actions sont hamiltoniennes et on peut aisément vérifier que l'on a l'égalité
\[
\Delta_K(\Orb_{\Lambda}) = \Delta_{K_0}(\Orb_{\Lambda_0})\times\Delta_{K_1}(\Orb_{\Lambda_1})\times\cdots\times \Delta_{K_s}(\Orb_{\Lambda_s}).
\]

Le terme $\Delta_{K_0}(\Orb_{\Lambda_0})$ est réduit au point $\Lambda_0$, puisqu'il s'agit de la projection d'une orbite coadjointe de $K_0$ par rapport à l'action hamiltonienne de ce même groupe $K_0$. Par conséquent, on a
\[
\Delta_K(\Orb_{\Lambda}) = \{\Lambda_0\}\times\Delta_{K_1}(\Orb_{\Lambda_1})\times\cdots\times \Delta_{K_s}(\Orb_{\Lambda_s}).
\]
On peut ainsi supposer que le groupe $G$ est simple non-compact, ou bien, ce qui est équivalent, que $G/K$ est symétrique hermitien irréductible.

Pour le reste de ce paragraphe, nous supposerons donc que $G$ vérifie l'hypothèse suivante :
\smallskip
\begin{equation}
\label{eq:hypothese_sur_G_sans_facteur_compact}
\begin{minipage}{12cm}
$G$ est un groupe de Lie réel simple connexe non-compact à centre fini, avec $G/K$ hermitien.
\end{minipage}
\end{equation}

Le Lemme \ref{lemm:noyau_de_zeta_fini} nous indique que nous pouvons appliquer les résultats des Chapitres \ref{chap:projectiondorbitecoadjointe+GIT} et \ref{chap:PairesBienCouvrantes} afin de déterminer les équations de la projection d'orbite $\Delta_K(\Orb_{\Lambda})$.

Rappelons que l'ensemble fini $\mathcal{P}_0(\got{p}^-)$ a été introduit en Définition \ref{defi:P(E)_et_P0(E)}. Il s'agit de l'ensemble des paires bien couvrantes $(C(w,w',0),\lambda)$ de $X_{\got{p}^-\oplus\C} = \KC/B\times\KC/B\times\mathbb{P}(\got{p}^-\oplus\C)$ telles que $\lambda$ soit dominant indivisible $\got{p}^-$-admissible.

\begin{theo}[\'Equations de $\Delta_K(\Orb_{\Lambda})$]
\label{theo:équations_DeltaK_G0Lambda}
On suppose que $G$ vérifie l'hypothèse \eqref{eq:hypothese_sur_G_sans_facteur_compact}. Soient $\Lambda\in\wedge^*_{\Q}\cap\Chol$ et $\xi\in\got{t}_+^*$. Alors $\xi$ appartient à $\Delta_K(\Orb_{\Lambda})$ si et seulement s'il vérifie les équations
\[
\langle w\lambda,\xi\rangle \leqslant \langle w_0w'\lambda,\Lambda\rangle
\]
pour toute paire $(C(w,w',0),\lambda)\in\mathcal{P}_0(\got{p}^-)$.
\end{theo}

\begin{proof}
La preuve se déduit directement du fait que $\Delta_K(\Orb_{\Lambda}) = \Delta_K(K\cdot\Lambda\times\got{p}^-)$ est l'adhérence de $\DGIT[K\cdot\Lambda\times\got{p}^-]$ dans $\got{t}^*$, du Corollaire \ref{coro:équationsgénérales_DGIT} et du Lemme \ref{lemm:noyau_de_zeta_fini}.
\end{proof}


Remarquons qu'ici, les poids de $\TC$ sur $\got{p}^-$ sont les racines non compactes négatives, c'est-à-dire $\WT(\got{p}^-) = \got{R}_n^-$. De plus, un sous-groupe à un paramètre dominant $\lambda$ est $\got{p}^-$-admissible s'il existe $n-1$ racines non compactes positives $\beta_1,\ldots,\beta_{n-1}$ telles que $\C\lambda = \cap_{i=1}^{n-1}\ker\beta_i$.

Du paragraphe précédent, nous pouvons tirer les propriétés géométriques attendues sur le polyèdre convexe $\Delta_K(\Orb_{\Lambda})$.

\begin{prop}
\label{prop:propriétés_géométriques_polyèdremomentG0}
On suppose que $G$ vérifie l'hypothèse \eqref{eq:hypothese_sur_G_sans_facteur_compact}. Soit $\Lambda\in\wedge^*_{\Q}\cap\Chol$. Alors $\Lambda\in\Delta_K(\Orb_{\Lambda})$ et $\Delta_K(\Orb_{\Lambda})\subset \Lambda + \CR(\got{R}_n^+)$.
\end{prop}

\begin{proof}
Ces assertions découlent directement des Propositions \ref{prop:faces_ConeWtT(E)_et_pairesbiencouvrantes} et \ref{prop:Lambda_dans_DGIT}.
\end{proof}

\begin{prop}
\label{prop:polyèdremoment_contenudans_Chol}
On suppose que $G$ vérifie l'hypothèse \eqref{eq:hypothese_sur_G_sans_facteur_compact}. Soit $\Lambda\in\wedge^*_{\Q}\cap\Chol$. Alors $\DGIT[K\cdot\Lambda\times\got{p}^-]\subset\Chol$. En particulier, pour tout $\Lambda\in\wedge^*_{\Q}\cap\Chol$, on a $\Delta_K(\Orb_{\Lambda})\subset\Chol$.
\end{prop}

\begin{proof}
D'après la Proposition \ref{prop:faces_ConeWtT(E)_et_pairesbiencouvrantes}, on a $\DGIT[K\cdot\Lambda\times\got{p}^-] \subset \Lambda + \CQ(\got{R}_n^+)$. 
La chambre $\Chol$ est égale à
\[
\Chol = \{\xi\in\got{t}_+^*; (\beta,\xi) > 0, \forall\beta\in\got{R}_n^+\},
\]
où $(\cdot,\cdot)$ est le produit scalaire sur $\got{t}^*$ induit par le produit scalaire $B_{\theta}$ sur $\got{g}$. Pour toutes racines $\beta,\beta'\in\got{R}_n^+$, la somme $\beta+\beta'$ ne peut pas être une racine d'après la preuve de \cite[Lemme 7.128]{knapp}, ni $0$, donc d'après \cite[Proposition 2.48 (e)]{knapp}, nécessairement $(\beta,\beta')\geqslant 0$. Donc tout élément $\xi$ de $\CR(\got{R}_n^+)$ vérifie $(\beta,\xi)\geqslant 0$ pour toute racine $\beta\in\got{R}_n^+$. Et si $\Lambda$ appartient à $\wedge^*_{\Q}\cap\Chol$, on a $(\beta,\xi)> 0$ pour tout $\xi\in(\Lambda + \CR(\got{R}_n^+))$ et pour tout $\beta\in\got{R}_n^+$. Il est maintenant clair que $\DGIT[K\cdot\Lambda\times\got{p}^-] \subset (\Lambda + \CQ(\got{R}_n^+))\cap\got{t}_+^* \subset \Chol$, dès que $\Lambda$ appartient à $\wedge^*_{\Q}\cap\Chol$. Et la dernière assertion provient du fait que $\Delta_K(\Orb_{\Lambda})$ est l'enveloppe convexe de $\DGIT[K\cdot\Lambda\times\got{p}^-]$ dans $\got{t}^*$ lorsque $\Lambda\in\wedge^*_{\Q}\cap\Chol$.
\end{proof}

\section[Sous-groupes à un paramètre dominants $\got{p}^-$-admissibles]{Sous-groupes à un paramètre dominants indivisibles et $\got{p}^-$-admissibles pour le cas des groupes classiques}

Dans cette section, nous allons donner la liste complète de tous les sous-groupes à un paramètre dominants indivisibles et $\got{p}^-$-admissibles apparaissant pour chaque espace symétrique hermitien $G/K$ avec $G$ simple connexe non compact à centre fini classique, c'est-à-dire $G = Sp(2n,\R)$, $SU(p,q)$, $SO^*(2n)$ et $SO(p,2)$.

\subsection{Le groupe $SU(n,1)$, $n\geqslant 2$}
\label{subsection:ssgpes1param_SU(n,1)}

Le cas du groupe $SU(n,1)$ est le plus simple. En effet, ici, le sous-groupe compact maximal est isomorphe à $U(n)$ et $\got{p}^-$ est isomorphe à $\C^n$. Le nombre de racines non compactes négatives (qui sont les poids de l'action de $U(n)$ sur $\got{p}^-$) est égal à la dimension du tore maximal des matrices diagonales de $U(n)$, et on peut facilement vérifier qu'elles forment une famille libre de formes linéaires sur $\got{t}$. Notant $\beta_1,\ldots,\beta_n$ les racines non compactes positives de $SU(n,1)$, il est clair que les sous-groupes à un paramètres $\got{p}^-$-admissibles sont les $\lambda$ tels que
\[
\C\lambda = \bigcap_{i\neq k}\ker\beta_i
\]
pour un certain $k\in\{1,\ldots,n\}$.

Rappelons que $\beta_k = 2e_k^* + \sum_{j\neq k}e_j^*$ et notons $\lambda_k = ne_k - \sum_{j\neq k}e_j$. On vérifie aisément que $\C\lambda_k = \bigcap_{i\neq k}\ker\beta_i$ et $\langle\lambda_k,\beta_k\rangle = n+1$. Donc les sous-groupes à un paramètre indivisibles et $\got{p}^-$-admissibles sont les sous-groupes à un paramètre $\pm\lambda_k$, pour $k=1,\ldots,n$. On obtient l'énoncé suivant. Ici, $\TC$ désigne le groupe des matrices diagonales de $GL_n(\C)$.

\begin{theo}
\label{theo:ssgpes1param_SU(n,1)}
L'ensemble des sous-groupes à un paramètre dominants indivisibles $\got{p}^-$-admissibles de $\TC$ associés à $G=SU(n,1)$ est
\[
\mathcal{A}dm^+_{SU(n,1)} = \{\lambda_1,-\lambda_n\},
\]
où $\lambda_1 = (n,-1,\ldots,-1)$ et $-\lambda_n = (1,\ldots,1,-n)$.
\end{theo}

%

\subsection{Le groupe $Sp(2n,\R)$, $n\geqslant 2$}
\label{subsection:ssgpe1param_Sp(R,n)}

Lorsque $G=Sp(2n,\R)$, le sous-groupe compact maximal est à nouveau $U(n)$ et le $U(n)$-module $\got{p}^-$ est isomorphe à la représentation standard $S^2((\C^n)^*)$ de $U(n)$ (et le $U(n)$-module $\got{p}^+$ est isomorphe à $S^2(\C^n)$). Les racines non compactes positives sont les formes linéaires
$\beta_{i,j} = e_i^*+e_j^*$, avec $1\leqslant i, j\leqslant n$. Remarquons que $\beta_{i,j}=\beta_{j,i}$ pour tout $i,j$. Il suffit donc de paramétrer les racines non compactes positives par $\beta_{i,j}$ avec $1\leqslant i\leqslant j\leqslant n$. Nous noterons spécifiquement $\beta_i = \beta_{i,i} = 2e_i^*$, pour $i=1,\ldots,n$.


Soit $\lambda=(\lambda_1,\ldots,\lambda_n)$ un sous-groupe à un paramètre dominant de $\TC$. On note $L:=\{\lambda_k; k=1,\ldots,n\}$ et $\Lsym:=(L\cap(-L))\cap\N$. Nous noterons également, pour $x\in\Z$,
\[
I(x) :=\{1\leqslant i\leqslant n; \lambda_i=x\}.
\]
Remarquons qu'un entier $l$ appartient à $L$ si et seulement si $I(l)\neq \emptyset$. De plus, un entier positif $l$ appartient à $\Lsym$ si et seulement si $I(l)\neq\emptyset$ et $I(-l)\neq\emptyset$. Puisque $\lambda$ est dominant, pour $l\in \Lsym\setminus\{0\}$ et tout élément $i\in I(l)$ et $j\in I(-l)$, on aura $i< j$. Notons enfin $\WT(\got{p}^+_{\lambda,0})$ 
l'ensemble des racines non compactes positives $\beta$ telles que $\langle\lambda,\beta\rangle =0$, autrement dit, $\C\lambda\subset\ker\beta$. Le sous-groupe à un paramètre $\lambda$ sera $\got{p}^-$-admissible si et seulement si $\C\lambda=\bigcap_{\beta\in\WT(\got{p}^+_{\lambda,0})
}\ker\beta$, d'après la Remarque \ref{rema:autre_déf_tore_Madmissible}.

Une racine $\beta_{i,j}$, $1\leqslant i\leqslant j\leqslant n$, appartient à $\WT(\got{p}^+_{\lambda,0})$ 
si et seulement si $\lambda_j=-\lambda_i$. Distinguant deux cas, $\beta_{i,j}\in\WT(\got{p}^+_{\lambda,0})$ 
si et seulement si $j\in I(-\lambda_i)$. Nous obtenons donc une partition de l'ensemble $\WT(\got{p}^+_{\lambda,0})$ 
en
\[
\WT(\got{p}^+_{\lambda,0})
 = \bigcup_{l\in \Lsym} \{\beta_{i,j}; i\in I(l),j\in I(-l)\}.
\]
Par définition des éléments $\beta_{i,j}$, nous voyons clairement que pour tout $l\in \Lsym$, le sous-espace engendré par l'ensemble de vecteurs $\{\beta_{i,j}; i\in I(l),j\in I(-l)\}$ est contenu dans $\mathrm{Vect}\{e_i^*; i\in I(l)\cup I(-l)\}$. De plus, si $l\neq l'$ sont deux éléments de $\Lsym$, alors $(I(l)\cup I(-l))\cap (I(l')\cup I(-l')) = \emptyset$. On en déduit l'égalité suivante,
\[
\mathrm{Vect}(\WT(\got{p}^+_{\lambda,0})
) = \bigoplus_{l\in \Lsym} \mathrm{Vect}\{\beta_{i,j}; i\in I(l),j\in I(-l)\}.
\]
Donc, en termes de dimension, on a
\[
\dim\mathrm{Vect}(\WT(\got{p}^+_{\lambda,0})
) = \sum_{l\in \Lsym} \dim\mathrm{Vect}\{\beta_{i,j}; i\in I(l),j\in I(-l)\}.
\]
Lorsque $l=0$, on a $\mathrm{Vect}\{\beta_{i,j}; i\in I(0),j\in I(0)\} = \mathrm{Vect}\{e_i^*; i\in I(0)\}$, ce qui implique $\dim\mathrm{Vect}\{\beta_{i,j}; i\in I(0),j\in I(0)\} = |I(0)|$. Il reste donc à calculer les dimensions $\dim\mathrm{Vect}\{\beta_{i,j}; i\in I(l),j\in I(-l)\}$ lorsque $l>0$. Dans cette situation, nous avons $I(l)\cap I(-l) = \emptyset$, mais $I(l)$ et $I(-l)$ sont non vides.

\begin{lemm}
\label{lemm:dimension_sousespace_engendré_par_famille_uij}
Soient $E$ et $F$ deux espaces vectoriels non nuls, $(e_1,\ldots,e_p)$ une base de $E$, $(f_1,\ldots,f_q)$ une base de $F$. On définit les vecteurs $u_{i,j}:=(e_i,f_j)\in E\times F$, pour tout $(i,j)\in\{1,\ldots,p\}\times\{1,\ldots,q\}$. Alors
\[
\dim\mathrm{Vect}\{u_{i,j}; 1\leqslant i\leqslant p, 1\leqslant j\leqslant q\} = p + q - 1 = \dim E + \dim F - 1.
\]
Si, de plus, $I_1,I_2\subset\{1,\ldots,p\}$ (resp. $J_1,J_2\subset\{1,\ldots,p\}$) vérifient $I_1\cap I_2=\emptyset$ (resp. $J_1\cap J_2 = \emptyset$), alors
\[
\mathrm{Vect}\{u_{i_1,j_1}\}_{(i_1,j_1)\in I_1\times J_1}\cap\mathrm{Vect}\{u_{i_2,j_2}\}_{(i_2,j_2)\in I_2\times J_2} = 0
\]
\end{lemm}

\begin{proof}
Quitte à changer les notations, on peut supposer que $\dim E = p\leqslant q=\dim F$. Le cas $p=q=1$ est trivial. On suppose donc $q\geqslant 2$. On définit la famille
\[
\mathcal{L} = (u_{1,1},\ldots,u_{p,1},u_{1,2},\ldots,u_{1,q}).
\]
Puisque pour tout $i=1,\ldots,p$, $j=1,\ldots,q$, on a $u_{i,j} = u_{i,1} + u_{1,j} - u_{1,1}$, la famille $\mathcal{L}$ engendre l'espace vectoriel $\mathrm{Vect}\{u_{i,j}; 1\leqslant i\leqslant p, 1\leqslant j\leqslant q\}$. Montrons qu'elle est aussi libre. Considérons donc la combinaison linéaire
\[
\sum_{i=1}^p x_iu_{i,1} + \sum_{j=2}^q y_ju_{1,j} = 0
\]
En appliquant successivement à cette égalité les formes linéaires sur $E\times F$, $(e_i^*,0)$ et $(0,f_j^*)$ pour $1< i\leqslant p$ et $1<j\leqslant q$, on obtient $x_2=\ldots=x_p = 0$ et $y_2=\ldots=y_q = 0$. Il ne reste donc que $\lambda_1 u_{1,1} = 0$. Mais comme $u_{1,1}$ n'est pas le vecteur nul, nous avons nécessairement $\lambda_1=0$. On en conclut que la famille $\mathcal{L}$ est libre, donc une base de $\mathrm{Vect}\{u_{i,j}; 1\leqslant i\leqslant p, 1\leqslant j\leqslant q\}$. Le résultat découle du fait que $\mathcal{L}$ est composée de $p+q-1$ éléments.

Pour montrer la dernière assertion, il suffit de voir que pour tout $(i_k,j_k)\in I_k\times J_k$, avec $k=1,2$, on a $u_{i_k,j_k} \in\mathrm{Vect}\{e_i\}_{i\in I_k}\times\mathrm{Vect}\{f_j\}_{j\in J_k}$. Par conséquent, on obtient
\begin{multline*}
\mathrm{Vect}\{u_{i_1,j_1}\}_{(i_1,j_1)\in I_1\times J_1}\cap\mathrm{Vect}\{u_{i_2,j_2}\}_{(i_2,j_2)\in I_2\times J_2} \\  \subset (\mathrm{Vect}\{e_i\}_{i\in I_1}\cap\mathrm{Vect}\{e_i\}_{i\in I_2})\times(\mathrm{Vect}\{f_j\}_{j\in J_2}\cap\mathrm{Vect}\{f_j\}_{j\in J_2}) = 0,
\end{multline*}
puisque $(e_1,\ldots,e_p)$ (resp. $(f_1,\ldots,f_q)$) est une base de $E$ (resp. $F$) et puisque les ensembles $I_1$ et $I_2$ (resp. $J_1$ et $J_2$) sont disjoints.
\end{proof}

\begin{rema}
\label{rema:somme_directe_sousespace_engendré_par_famille_uij_récurrence}
Par une récurrence évidente, on peut montrer que si on a des parties disjointes $I_1,\ldots,I_k$ (resp. $J_1,\ldots,J_k$) de $\{1,\ldots,p\}$ (resp. $\{1,\ldots,q\}$), alors les sous-espaces $\mathrm{Vect}\{u_{i_1,j_1}\}_{(i_1,j_1)\in I_1\times J_1},\ldots,\mathrm{Vect}\{u_{i_k,j_k}\}_{(i_k,j_k)\in I_k\times J_k}$ sont en somme directe dans $E\times F$.

Il faut aussi noter qu'en général, on n'a pas mieux que l'inclusion stricte
\[
\mathrm{Vect}\{u_{i_1,j_1}\}_{(i_1,j_1)\in I_1\times J_1}\oplus\mathrm{Vect}\{u_{i_2,j_2}\}_{(i_2,j_2)\in I_2\times J_2} \subsetneq \mathrm{Vect}\{u_{i,j}\}_{(i,j)\in (I_1\cup I_2)\times(J_1\cup J_2)}.
\]
Par exemple, en prenant $I_k=\{i_k\}$ et $J_k=\{j_k\}$, avec $k=1,2$ et $i_1\neq i_2$, $j_1\neq j_2$, on aura $\dim\mathrm{Vect}\{u_{i_1,j_1}\}_{(i_1,j_1)\in I_1\times J_1}\oplus\mathrm{Vect}\{u_{i_2,j_2}\}_{(i_2,j_2)\in I_2\times J_2} = 2$, alors que le sous-espace vectoriel $\mathrm{Vect}\{u_{i,j}\}_{(i,j)\in (I_1\cup I_2)\times(J_1\cup J_2)}$ sera de dimension $|I_1\cup I_2|+|J_1\cup J_2|-1 = 3$.
\end{rema}

\begin{lemm}
\label{lemm:dimension_sousespace_engendré_par_famille_betaij_parrapportàl}
Soit $l\in \Lsym$ tel que $l\neq 0$. Alors on a
\[
\dim\mathrm{Vect}\{\beta_{i,j}; i\in I(l),j\in I(-l)\} = |I(l)|+|I(-l)| - 1.
\]
\end{lemm}

\begin{proof}
On pose $E=\mathrm{Vect}\{e_i^*; i\in I(l)\}$ et $F=\mathrm{Vect}\{e_j^*; j\in I(-l)\}$. Puisque $I(l)\cap I(-l) = \emptyset$ , alors $E\cap F = 0$ et on se retrouve dans la configuration de l'énoncé du Lemme \ref{lemm:dimension_sousespace_engendré_par_famille_uij}. La base utilisée dans $E$ est la base $(e_i^*)_{i\in I(l)}$ et celle de $F$ est $(e_j^*)_{j\in I(-l)}$. Les éléments $\beta_{i,j}$, avec $i\in I(l)$ et $j\in I(-l)$ correspondent bien aux vecteurs $u_{i,j}$. Enfin, comme les dimensions de ces espaces vectoriels sont $\dim E = |I(l)|$ et $\dim F = |I(-l)|$, on obtient bien le résultat attendu.
\end{proof}

Notons maintenant $\Lasym
 = L\setminus(L\cap(-L))$. Des Lemmes \ref{lemm:dimension_sousespace_engendré_par_famille_uij} et \ref{lemm:dimension_sousespace_engendré_par_famille_betaij_parrapportàl} et de la Remarque \ref{rema:somme_directe_sousespace_engendré_par_famille_uij_récurrence}, on déduit les égalités suivantes,
\begin{align*}
\dim\mathrm{Vect}(\WT(\got{p}^+_{\lambda,0})
) & = \sum_{l\in \Lsym} \dim\mathrm{Vect}\{\beta_{i,j}; i\in I(l),j\in I(-l)\} \\
& = |I(0)| + \sum_{l\in \Lsym, l\neq 0} (|I(l)|+|I(-l)| - 1) \\
& = \left(|I(0)| + \sum_{l\in \Lsym, l\neq 0} |I(l)|+|I(-l)|\right) - |\Lsym\setminus\{0\}| \\
& = n - \sum_{l\in\Lasym}|I(l)| - |\Lsym\setminus\{0\}|.
\end{align*}
Or, $\Lasym\subset L$, donc pour tout $l\in\Lasym$, $I(l)$ est non vide. On en conclut que la dimension de l'espace vectoriel $\mathrm{Vect}(\WT(\got{p}^+_{\lambda,0})
)$ est égale à $n-1$ (c'est-à-dire $\lambda$ est $\got{p}^-$-admissible) si et seulement si l'une des deux alternatives suivantes est vérifiée :
\renewcommand{\theenumi}{\textit{\roman{enumi}}}
\begin{enumerate}
\item $|\Lsym\setminus\{0\}| = 1$ et $\Lasym=\emptyset$;
\item $\Lsym\setminus\{0\} = \emptyset$, $\Lasym=\{l\}$ et $|I(l)| = 1$.
\end{enumerate}
Dans le premier cas, la valeur $0$ peut apparaître ou non dans $L$. Dans le second, la valeur $0$ doit apparaître exactement $n-1$ fois (avec $n\geqslant 2$). On note
\begin{equation}
\label{eq:lambda_k,l}
\lambda_{k,l} = (\underbrace{1,\ldots,1}_{k\ \mbox{termes}},\underbrace{0,\ldots,0}_{n-(k+l)\ \mbox{termes}},\underbrace{-1,\ldots,-1}_{l\ \mbox{termes}}),
\end{equation}
pour tout $k=1,\ldots, n-1$, $l=1,\ldots,n-k$. On peut maintenant en déduire le théorème suivant.

\begin{theo}
\label{theo:ssgpe1param_Sp(R,n)}
L'ensemble des sous-groupes à un paramètre dominants indivisibles $\got{p}^-$-admissibles de $\TC$ associés à $G=Sp(2n,\R)$ est
\[
\mathcal{A}dm^+_{Sp(2n,\R,} = \{(1,0,\ldots,0),(0,\ldots,0,-1)\}\cup\{\lambda_{k,l}; k=1,\ldots, n-1, l=1,\ldots,n-k\}.
\]
Le cardinal de $\mathcal{A}dm^+_{Sp(2n,\R)}$ est égal à $\frac{n(n-1)}{2}+2$.
\end{theo}

\subsection{Le groupe $SO^*(2n)$, $n\geqslant 3$}
\label{subsection:ssgpe1param_SO*(2n)}

Le groupe $G=SO^*(2n)$ est le dernier de la liste à avoir $U(n)$ pour sous-groupe compact maximal. De plus, $\got{p}^-$ est isomorphe au $U(n)$-module standard $\wedge^2((\C^n)^*)$. Les racines non compactes positives sont les formes linéaires $\beta_{i,j} = e_i^*+e_j^*$, avec $1\leqslant i<j\leqslant n$. Nous avons à nouveau $\beta_{i,j} = \beta_{j,i}$.

Le calcul des sous-groupes à un paramètre dominants $\got{p}^-$-admissibles de $\TC$ est très similaire au cas du groupe $Sp(2n,\R)$. Nous gardons les mêmes notations que celles introduites dans le paragraphe \ref{subsection:ssgpe1param_Sp(R,n)}.

Soit $\lambda$ un sous-groupe à un paramètre dominant de $\TC$. La seule différence par rapport au paragraphe précédent apparaît dans la dimension de l'espace vectoriel
\[
\mathrm{Vect}\{\beta_{i,j}; i\in I(0),j\in I(0), i<j\}.
\]

\begin{lemm}
Si $|I(0)|\in\{1,2\}$, alors $\dim\mathrm{Vect}\{\beta_{i,j}; i\in I(0),j\in I(0), i<j\} = |I(0)|-1$. Si $|I(0)|\geqslant 3$, alors $\dim\mathrm{Vect}\{\beta_{i,j}; i\in I(0),j\in I(0), i<j\} = |I(0)|$.
\end{lemm}

\begin{proof}
Le cas $|I(0)| = 1$ est trivial puisque $\{\beta_{i,j}; i\in I(0),j\in I(0), i<j\} = \emptyset$ dans ce cas-là. Ensuite, si $|I(0)| = 2$, l'ensemble $\{\beta_{i,j}; i\in I(0),j\in I(0), i<j\}$ aura un seul élément, qui est non nul, donc engendrera un espace vectoriel de dimension $1$.

Supposons maintenant que $|I(0)|\geqslant 3$. Sans perdre de généralité, on peut supposer que $I(0) = \{1,\ldots,p\}$, avec $p\geqslant 3$. Nous avons
\[
\mathrm{Vect}\{\beta_{i,j}; i\in I(0),j\in I(0), i<j\} \subset \mathrm{Vect}(e_k^*)_{k=1,\ldots,p}
\]
et, par conséquent, la dimension de $\mathrm{Vect}\{\beta_{i,j}; i\in I(0),j\in I(0), i<j\}$ est inférieure ou égale à $|I(0)|=p$. Montrons que la famille $(\beta_{2,3},\beta_{1,2},\beta_{1,3},\ldots,\beta_{1,p})$ est libre. Pour cela, considérons des scalaires $\mu_1,\ldots,\mu_p$ tels que
\[
\mu_1\beta_{2,3}+\mu_2\beta_{1,2}+\mu_3\beta_{1,3}+\ldots+\mu_p\beta_{1,p} = 0.
\]
Appliquant le vecteur $e_k$ à cette combinaison linéaire de formes linéaires, pour $k\geqslant 4$, cela nous donne $\mu_k = 0$ pour tout $k\geqslant 4$. Il ne reste donc plus que la combinaison linéaire $\mu_1\beta_{2,3}+\mu_2\beta_{1,2}+\mu_3\beta_{1,3}=0$. Autrement dit, tout se passe dans le cas $|I(0)| = 3$. Or on peut facilement vérifier que la famille $(\beta_{2,3},\beta_{1,2},\beta_{1,3})$ est libre. On en déduit que tous les scalaires $\mu_k$ sont nuls et que la famille $(\beta_{2,3},\beta_{1,2},\beta_{1,3},\ldots,\beta_{1,p})$ est libre. Cette famille est constituée de $p$ vecteurs de $\mathrm{Vect}\{\beta_{i,j}; i\in I(0),j\in I(0), i<j\}$, ce qui permet de conclure la preuve du lemme.
\end{proof}

On déduit du lemme précédent que lorsque $|I(0)| = 0$ ou $|I(0)|\geqslant 3$, alors
\[
\dim\mathrm{Vect}(\WT(\got{p}^+_{\lambda,0})
) = n - \sum_{l\in\Lasym}|I(l)| - |\Lsym\setminus\{0\}|,
\]
comme dans le cas de $Sp(2n,\R)$. On en déduit que les sous-groupes à un paramètre $\lambda$ dominants indivisibles $\got{p}^-$-admissibles sont les sous-groupes à un paramètre $\lambda_{k,l}$ définis en $\eqref{eq:lambda_k,l}$, pour $1\leqslant k\leqslant n-4$ et $l=1,\ldots,n-k-3$ (pour $|I(0)|\geqslant 3$), ainsi que les sous-groupes à un paramètre $\lambda_{k,n-k}$ avec $k=1,\ldots,n-1$ (pour $|I(0)|=0$), et enfin $(1,0,\ldots,0)$ et $(0,\ldots,0,-1)$ lorsque $n\geqslant 4$.

Pour $|I(0)| =1,2$, nous aurons
\[
\dim\mathrm{Vect}(\WT(\got{p}^+_{\lambda,0})
) = n -1 - \sum_{l\in\Lasym}|I(l)| - |\Lsym\setminus\{0\}|,
\]
ce qui implique que l'on doit nécessairement avoir $\Lasym=\emptyset$ et $\Lsym=\{0\}$. Mais alors $I(0) = \{1,\ldots,n\}$, avec $n\geqslant 3$, et $|I(0)|\leqslant 2$, ce qui est impossible. Donc le cas $|I(0)|=1,2$ n'apparaît pas pour $\lambda$ dominant $\got{p}^-$-admissible.

\begin{theo}
\label{theo:ssgpe1param_SO*(2n)}
L'ensemble des sous-groupes à un paramètre dominants indivisibles $\got{p}^-$-admissibles de $U(n)$ associés à $G=SO^*(2n)$ est
\begin{align*}
\mathcal{A}dm^+_{SO^*(2n)} = &\ \{(1,0,\ldots,0),(0,\ldots,0,-1)\} \cup\{\lambda_{k,n-k}; k=1,\ldots, n-1\} \\
& \quad \cup\{\lambda_{k,l}; k=1,\ldots, n-4, l=1,\ldots,n-k-3\}
\end{align*}
lorsque $n\geqslant 4$, et
\[
\mathcal{A}dm^+_{SO^*(6)} = \{(1,-1,-1),(1,1,-1)\}.
\]
\end{theo}


\subsection{Le groupe $SU(p,q)$, $p\geqslant q\geqslant 2$}
\label{subsection:ssgpe1param_SU(p,q)}

Nous allons maintenant donner la liste des sous-groupes à un paramètre dominants indivisibles $\got{p}^-$-admissibles associés au groupe $G=SU(p,q)$, lorsque $p\geqslant q\geqslant 2$. Le cas $SU(p,1)$ a été fait dans le paragraphe \ref{subsection:ssgpes1param_SU(n,1)}.

Rappelons que les racines non compactes positives de $\got{su}(p,q)$ sont les formes linéaires $\beta_{i,p+j}=e_i^*-e_{p+j}^*$ de $\got{t}^*$, pour $1\leq i\leq p$ et $1\leq j\leq q$, cf paragraphe \ref{subsection:Notations_SU(p,q)}.

Commençons par démontrer le lemme suivant.

\begin{lemm}
\label{lemm:SU(p,q)_dimension_espaceengendrépar_sousensemble_racinesnoncompactespositives}
Soit $I\subset\{1,\ldots,p\}$ et $J\subset\{1,\ldots,q\}$ deux ensembles non vides. Alors
\[
\dim\mathrm{Vect}(\{\beta_{i,p+j}\}_{i\in I,j\in J}) = |I| + |J| -1.
\]
\end{lemm}

\begin{proof}
Ce lemme découle du Lemme \ref{lemm:dimension_sousespace_engendré_par_famille_uij}, où $E=\mathrm{Vect}(e_i^*)_{i=1,\ldots,p}$, $F=\mathrm{Vect}(e_j^*)_{j=p+1,\ldots,p+q}$, et on prend pour bases de $E$ et $F$ respectivement $(e_1^*,\ldots,e_p^*)$ et $(-e_{p+1}^*,\ldots,-e_{p+q}^*)$. 
\end{proof}

Soit $\lambda$ un sous-groupe à un paramètre dominant de $\TC$. Nous noterons
\[
L:=\{\lambda_k; k=1,\ldots,p\}\cap\{\lambda_{p+k'}; k'=1,\ldots,q\}\subset\Z
\]
l'ensemble des valeurs qui apparaissent à la fois parmi les $p$ premières composantes et les $q$ dernières composantes du $(p+q)$-uplet $\lambda$
. Nous définissons également les deux sous-ensembles de $\Z$ suivants, pour tout $n\in\Z$,
\[
I(n):=\{1\leqslant i\leqslant p; \lambda_i=n\} \quad \text{et} \quad J(n):=\{1\leqslant j\leqslant q; \lambda_{p+j}=n\}.
\]
Les parties $I(n)$ (resp. $J(n)$), pour $n$ parcourant $\Z$, sont deux à deux disjointes. De plus, pour tout $\ell\in\Z$, l'entier $\ell$ appartient à $L$ si et seulement si $I(\ell)$ et $J(\ell)$ sont tous les deux non vides.

Nous rappelons que $\WT(\got{p}^+_{\lambda,0})$ est égal à l'ensemble des racines non compactes positives $\beta$ vérifiant $\langle\lambda,\beta\rangle=0$. Nous pouvons travailler sur les racines positives en lieu et place des racines négatives, car on a $\WT(\got{p}^+) = -\WT(\got{p}^-)$ et $\WT(\got{p}^+_{\lambda,0}) = - \WT(\got{p}^-_{\lambda,0})$, ce qui nous permet d'écrire
\[
\bigcap_{\beta\in\WT(\got{p}^+_{\lambda,0})}\ker\beta = \bigcap_{\beta'\in\WT(\got{p}^-_{\lambda,0})}\ker\beta'.
\]
Soit $(i,j)\in\{1,\ldots,p\}\times\{1,\ldots,q\}$. Les assertions sont clairement équivalentes :
\begin{enumerate}
\item la racine non compacte positive $\beta_{i,p+j}$ appartient à $\WT(\got{p}^+_{\lambda,0})$;
\item $j\in J(\lambda_i)$;
\item $i\in I(\lambda_j)$;
\item il existe un entier $\ell\in L$ tel que $i\in I(\ell)$ et $j\in J(\ell)$.
\end{enumerate}
On a donc
\begin{equation}
\label{eq:SU(p,q)_partitiondespoids_annulant_lambda}
\WT(\got{p}^+_{\lambda,0}) = \bigcup_{\ell\in L}\bigl\{\beta_{i,j}; (i,j)\in I(\ell)\times J(\ell)\bigr\}.
\end{equation}


\begin{lemm}
\label{lemm:sg1padmissibles_SU(p,q)_sommedirecte_espacesengendréspardifférentsI(l)J(l)}
Les sous-espaces vectoriels $\mathrm{Vect}\left(\{\beta_{i,p+j}\}_{i\in I(\ell), j\in J(\ell)}\right)$, pour $\ell$ parcourant $L$, sont en somme directe dans $\got{t}^*$ et, plus exactement,
\[
\mathrm{Vect}\left(\WT(\got{p}^+_{\lambda,0})\right) = \bigoplus_{\ell\in L}\mathrm{Vect}\left(\{\beta_{i,p+j}\}_{i\in I(\ell), j\in J(\ell)}\right).
\]
\end{lemm}

\begin{proof}
La somme directe découle directement de la deuxième assertion du Lemme \ref{lemm:dimension_sousespace_engendré_par_famille_uij}, en utilisant la propriété que si $\ell,\ell'\in L$ sont distincts, alors $I(\ell)\cap I(\ell') = \emptyset = J(\ell)\cap J(\ell')$. Et d'après l'égalité \eqref{eq:SU(p,q)_partitiondespoids_annulant_lambda}, cette somme directe est forcément égale au sous-espace de $\got{t}^*$ engendré par les éléments de $\WT(\got{p}^+_{\lambda,0})$.
\end{proof}


Notons maintenant $I = \cup_{\ell\in L}I(\ell)$ et $J = \cup_{\ell\in L}J(\ell)$. On a $\cup_{\ell\in L}I(\ell)\times J(\ell)\subset I\times J$, ce qui donne, d'après le Lemme \ref{lemm:sg1padmissibles_SU(p,q)_sommedirecte_espacesengendréspardifférentsI(l)J(l)},
\[
\mathrm{Vect}\left(\WT(\got{p}^+_{\lambda,0})\right) \subset \mathrm{Vect}\left(\{\beta_{i,p+j}\}_{(i,j)\in I\times J}\right).
\]
On en déduit que la dimension du sous-espace vectoriel $\mathrm{Vect}\left(\WT(\got{p}^+_{\lambda,0})\right)$ de $\got{t}^*$ est majorée par la dimension de $\mathrm{Vect}\left(\{\beta_{i,p+j}\}_{(i,j)\in I\times J}\right)$, et, par conséquent,
\begin{equation}
\label{eq:SU(p,q)_inégalités_dimensiondeVect(WTp-)}
\dim\mathrm{Vect}\left(\WT(\got{p}^+_{\lambda,0})\right) \leqslant |I|+|J|-1 \leqslant p+q-1,
\end{equation}
par le Lemme \ref{lemm:SU(p,q)_dimension_espaceengendrépar_sousensemble_racinesnoncompactespositives}. De plus, une autre application du Lemme \ref{lemm:SU(p,q)_dimension_espaceengendrépar_sousensemble_racinesnoncompactespositives} donne la dimension
\[
\dim\mathrm{Vect}(\{\beta_{i,p+j}\}_{i\in I(\ell),j\in J(\ell)}) = |I(\ell)| + |J(\ell)| - 1,
\]
pour tout $\ell\in L$. Ceci est vrai puisque, par définition de $L$, $I(\cdot)$ et $J(\cdot)$, les ensembles $I(\ell)$ et $J(\ell)$ sont des parties respectivement de $\{1,\ldots,p\}$ et $\{1,\ldots,q\}$, toutes les deux non vides dès que $\ell$ appartient à $L$. D'après le Lemme \ref{lemm:sg1padmissibles_SU(p,q)_sommedirecte_espacesengendréspardifférentsI(l)J(l)}, nous pouvons calculer la dimension suivante,
\begin{align}
\dim\mathrm{Vect}\left(\WT(\got{p}^+_{\lambda,0})\right) & = \sum_{\ell\in L} \dim \mathrm{Vect}\left(\{\beta_{i,p+j}\}_{i\in I(\ell),j\in J(\ell)}\right) \notag\\
& = \sum_{\ell\in L} (|I(\ell)| + |J(\ell)| - 1) \notag\\
& = |I| + |J| - |L|.\label{eq:SU(p,q)_égalité_dimensiondeVect(WTp-)}
\end{align}

Supposons maintenant que $\lambda$ est $\got{p}^-$-admissible. Cela signifie que 
\[
\dim\mathrm{Vect}\left(\WT(\got{p}^+_{\lambda,0})\right) = \dim\got{t}^*-1 = p+q-2.
\]
Des équations \eqref{eq:SU(p,q)_inégalités_dimensiondeVect(WTp-)} et \eqref{eq:SU(p,q)_égalité_dimensiondeVect(WTp-)}, on obtient
\begin{equation}
\label{eq:SU(p,q)_égalité+inégalité_cardinauxdeIetJ}
p+q-2 = |I| + |J| - |L| \leqslant |I|+|J|-1 \leqslant p+q-1.
\end{equation}
Or, par les définitions de $I$ et $J$, on doit avoir $|I|\leqslant p$ et $|J|\leqslant q$. Ainsi, on doit forcément avoir $(|I|,|J|)\in\{(p,q-1),(p-1,q),(p,q)\}$.

\bigskip

Déterminons les sous-groupes à un paramètre dominants indivisibles $\got{p}^-$-admissibles de $\TC$ dans chacun de ces cas. Commençons par le cas $|I|=p-1$ et $|J|=q$. On a donc $I=\{1,\ldots,p\}\setminus\{i_0\}$, pour un certain $i_0\in\{1,\ldots,p\}$, et $J=\{1,\ldots,q\}$. En appliquant ces valeurs à l'équation \eqref{eq:SU(p,q)_égalité+inégalité_cardinauxdeIetJ}, on a $p+q-2 = (p-1)+q-|L|$, donc $|L|=1$, c'est-à-dire $L=\{\ell\}$. Donc,
\begin{itemize}
\item pour tout $j\in\{1,\ldots,q\}$, $\lambda_{p+j} = \ell$,
\item pour tout $i\in\{1,\ldots,p\}\setminus\{i_0\}$, $\lambda_i=\ell$,
\item et $\lambda_{i_0} = (p+q-1)\ell$.
\end{itemize}
Cette dernière égalité vient du fait que $\lambda$ est un élément diagonal de $\got{sl}_{p+q}(\C)$, donc la somme de ses composantes doit être nulle. Pour que $\lambda$ soit indivisible, il faut alors prendre $\ell\in\{\pm 1\}$. Pour que $\lambda$ soit en plus dominant, il faut $\ell=1$ et $i_0=p$, ou $\ell=-1$ et $i_0=1$. Donc on obtient les deux sous-groupes à un paramètre suivants :
\[
\lambda = (1,\ldots,1,1-p-q;1,\ldots,1) \quad \text{ou} \quad \lambda = (p+q-1,-1,\ldots,-1;-1,\ldots,-1).
\]

Pour $|I|=p$ et $|J|=q-1$, le raisonnement est similaire. On obtient
\[
\lambda = (1,\ldots,1;1,\ldots,1-p-q) \quad \text{ou} \quad \lambda = (-1,\ldots,-1;p+q-1,-1,\ldots,-1).
\]

Enfin, prenons maintenant $|I|=p$ et $|J|=q$, c'est-à-dire $I=\{1,\ldots,p\}$ et $J=\{1,\ldots,q\}$. L'équation \eqref{eq:SU(p,q)_égalité+inégalité_cardinauxdeIetJ} implique $|L|=2$. On note $L=\{a,b\}$ avec $a\geqslant b$. On a ainsi les deux partitions
\[
\{1,\ldots,p\} = I(a)\cup I(b) \quad \text{et} \quad \{1,\ldots,q\} = J(a)\cup J(b).
\]
Pour que $\lambda$ soit dominant, comme on a pris $a\geqslant b$, il faut avoir que tout indice dans $I(a)$ (resp. dans $J(a)$) soit plus petit que tout indice dans $I(b)$ (resp. $J(b)$). Il existe donc deux entiers $k\in\{1,\ldots,p-1\}$ et $l\in\{1,\ldots,q-1\}$ tels que
\begin{itemize}
\item $I(a) = \{1,\ldots,k\}$ et $I(b) = \{k+1,\ldots,p\}$;
\item $J(a) = \{1,\ldots,l\}$ et $I(b) = \{l+1,\ldots,q\}$.
\end{itemize}
Le sous-groupe à un paramètre dominant $\got{p}^-$-admissible correspondant est
\begin{equation}
\label{eq:defi_sg1p_lambdakl}
\lambda_{k,l} = (\underbrace{a,\ldots,a}_{k\ \mbox{termes}},\underbrace{b,\ldots,b}_{p-k\ \mbox{termes}};\underbrace{a,\ldots,a}_{l\ \mbox{termes}},\underbrace{b,\ldots,b}_{q-l\ \mbox{termes}}).
\end{equation}

\begin{theo}
\label{theo:sg1pdominantsadmissibles_de_SU(p,q)}
Les sous-groupes à un paramètre dominants indivisibles $\got{p}^-$-admissibles de $\TC$ pour $G = SU(p,q)$ sont les sous-groupes à un paramètre $\lambda_{k,l}$ définis en \eqref{eq:defi_sg1p_lambdakl}, pour tous $k=1,\ldots,p-1$ et $l=1,\ldots,q-1$, avec 
\[
a=\frac{p+q-k-l}{\mathrm{pgcd}(p+q-k-l,k+l)} \quad\text{et} \quad b=\frac{-(k+l)}{\mathrm{pgcd}(p+q-k-l,k+l)},
\]
et les quatre sous-groupes à un paramètre suivants,
\begin{align*}
\lambda_{p,q-1} & := (1,\ldots,1;1,\ldots,1,1-p-q), \\
\lambda_{p-1,q} & := (1,\ldots,1,1-p-q;1,\ldots,1), \\
\lambda_{0,1} & := (-1,\ldots,-1;p+q-1,-1,\ldots,-1), \\
\lambda_{1,0} & := (p+q-1,-1,\ldots,-1;-1,\ldots,-1).
\end{align*}
\end{theo}

\subsection{Le groupe $SO(2p,2)$, $p\geqslant 2$}
\label{subsection:ssgpe1param_SO(2p,2)}

Rappelons que les notations pour le groupe $SO(2p,2)$ ont été posées dans le paragraphe \ref{subsection:Notations_SO(2p,2)}.

Etudions maintenant les sous-groupes à un paramètre dominants $\got{p}^-$-admissibles de $\TC$, complexifié du tore maximal $T$ de $K$ dont l'algèbre de Lie est $\got{t}$. Un sous-groupe à un paramètre dominant de $\TC$ peut s'écrire $\lambda=(\lambda_1,\ldots,\lambda_{p+1})$ dans la base $(e_1,\ldots,e_{p+1})$ du complexifié $\got{t}_{\C}$ de $\got{t}$, avec $\lambda_1,\ldots,\lambda_{p+1}$ des entiers vérifiant
\[
\lambda_1\geqslant\lambda_2\geqslant\ldots\geqslant\lambda_p \quad \text{et} \quad \lambda_{p-1}+\lambda_p\geqslant 0.
\]
En particulier, si $\lambda$ est dominant, on doit avoir $\lambda_1\geqslant\ldots\geqslant\lambda_{p-1}\geqslant 0$. Remarquons tout d'abord que, pour tout $i\in\{1,\ldots,p\}$, $\langle\lambda,\beta_i^{\pm}\rangle = 0$ si et seulement si $\lambda_i=\mp\lambda_{p+1}$. Nous définissons donc les deux sous-ensembles
\[
I^{\pm} := \{1\leqslant i\leqslant p; \lambda_i=\mp \lambda_{p+1}\} = \{1\leqslant i\leqslant p; \langle\lambda,\beta_i^{\pm}\rangle=0\}.
\]
On a donc
\begin{equation}
\label{eq:SO(2p,2)_écriture_WT(p-0)}
\WT(\got{p}_{\lambda,0}) = \{\beta_i^+; i\in I^+\}\cup\{\beta_j^-; j\in I^-\}.
\end{equation}

\begin{lemm}
\label{lemm:SO(2p,2)_I+capI-neq0_équivalent_I+=I-}
Supposons que l'un des ensembles $I^+$ ou $I^-$ soit non vide. Alors, $I^+\cap I^-\neq\emptyset$ si et seulement si $I^+=I^-$.
\end{lemm}

\begin{proof}
En effet, supposons que $I^+\cap I^-\neq\emptyset$. Fixons $i_0\in I^+\cap I^-$. On a donc $\lambda_{i_0} = \lambda_{p+1} = -\lambda_{i_0}$. Donc $\lambda_{p+1} = 0$. Soit $i\in\{1,\ldots,p\}$. On a donc les équivalences
\[
i\in I^+ \Longleftrightarrow \lambda_i=-\lambda_{p+1}=0=\lambda_{p+1}=\lambda_i \Longleftrightarrow i\in I^-.
\]
D'où $I^+=I^-$.

Réciproquement, supposons que $I^+=I^-$. Par hypothèse, on a $I^+$ ou $I^-$ non vide. Ils sont donc tous les deux non vides, et $I^+\cap I^-=I^+=I^-\neq\emptyset$.
\end{proof}

\begin{rema}
\label{rema:SO(2p,2)_I+capI-neq0_équivalent_lambdap+1=0}
Conservant l'hypothèse de l'énoncé du Lemme \ref{lemm:SO(2p,2)_I+capI-neq0_équivalent_I+=I-}, on peut voir également que l'assertion $I^+\cap I^-\neq\emptyset$ est équivalente à $\lambda_{p+1}=0$. En effet, la preuve du Lemme \ref{lemm:SO(2p,2)_I+capI-neq0_équivalent_I+=I-} donne la première implication et si on suppose $\lambda_{p+1}=0$, alors pour un $i\in I^{\pm}\neq\emptyset$ (on est sûr qu'au moins l'un des deux est non vide par hypothèse), on doit avoir $\lambda_i=\mp\lambda_{p+1}=0$. On aura donc aussi $\lambda_i = \pm 0 = \pm\lambda_{p+1}$, donc $i$ appartient à $I^+\cap I^-$.
\end{rema}

\begin{lemm}
\label{lemm:SO(2p,2)_dimensionVect(WTp-0)}
Lorsque $I^+=I^-=\emptyset$, alors $\dim\mathrm{Vect}(\WT(\got{p}^-_{\lambda,0})) = 0$. Sinon, on a
\begin{itemize}
\item soit $I^+\cap I^-=\emptyset$, et alors $\dim\mathrm{Vect}(\WT(\got{p}^-_{\lambda,0})) = |I^+|+|I^-|$;
\item sinon $I^+=I^-$ et $\dim\mathrm{Vect}(\WT(\got{p}^-_{\lambda,0})) = |I^+|+1 = |I^-|+1$.
\end{itemize}
\end{lemm}

\begin{proof}
La première égalité est directe, d'après \eqref{eq:SO(2p,2)_écriture_WT(p-0)}, car alors $\WT(\got{p}^-_{\lambda,0})$ est vide. Supposons donc maintenant que $I^+$ ou $I^-$ soit non vide. D'après le Lemme \ref{lemm:SO(2p,2)_I+capI-neq0_équivalent_I+=I-}, on a deux possibilités : soit $I^+\cap I^-=\emptyset$, soit $I^+=I^-$.

Si $I^+\cap I^-=\emptyset$, alors les éléments de $\WT(\got{p}_{\lambda,0})$ forment clairement une famille libre. Par conséquent, l'espace engendré par $\WT(\got{p}_{\lambda,0})$ a pour dimension le cardinal de $\WT(\got{p}_{\lambda,0})$, qui est égal à $|I^+|+|I^-|$, d'après \eqref{eq:SO(2p,2)_écriture_WT(p-0)} et le fait que $I^+\cap I^-$ soit vide.

Enfin, si $I^+=I^-$, puisque $\beta_i^{\pm} = \pm e_i^*+e_{p+1}^*$, on a
\[
\mathrm{Vect}(\WT(\got{p}_{\lambda,0})) \subset \mathrm{Vect}(\{e_i^*; i\in I^+=I^-\}\cup \{e_{p+1}^*\}).
\]
Or, pour tout $i\in I^+$, on a aussi $i\in I^-$, donc $\beta_i^+,\beta_i^-\in\WT(\got{p}_{\lambda,0})$ et $(\beta_i^+ +\beta_i^-)/2 = e_{p+1}^*\in\mathrm{Vect}(\WT(\got{p}_{\lambda,0}))$. Alors, $e_i^* = \beta_i^+-e_{p+1}^*$ est aussi dans $\mathrm{Vect}(\WT(\got{p}_{\lambda,0}))$, d'où l'égalité
\[
\mathrm{Vect}(\WT(\got{p}_{\lambda,0})) = \mathrm{Vect}(\{e_i^*; i\in I^+=I^-\}\cup \{e_{p+1}^*\}),
\]
ce qui donne le résultat attendu.
\end{proof}

Soit $\lambda$ un sous-groupe à un paramètre dominant indivisible $\got{p}^-$-admissible de $\TC$. Alors nous devons avoir $\dim\mathrm{Vect}(\WT(\got{p}_{\lambda,0})) = p \geqslant 2$. D'après le Lemme \ref{lemm:SO(2p,2)_dimensionVect(WTp-0)}, on a deux possibilités,
\begin{itemize}
\item soit $I^+=I^-$ et alors $|I^+|=|I^-|=p-1$, c'est-à-dire $I^+=I^-=\{1,\ldots,p\}\setminus\{i_0\}$.
\item soit $I^+\cap I^-=\emptyset$ et alors $|I^+|+|I^-| = p$, c'est-à-dire $I^+\cup I^- = \{1,\ldots,p\}$.
\end{itemize}
Pour que $\lambda$ soit dominant, il faut dans le premier cas que $i_0=1$, et dans le deuxième cas $I^{\pm} = \{1,\ldots,p-1\}$ et $I^{\mp}=\{p\}$, ce qui donne les cinq sous-groupes à un paramètre
\begin{equation}
\label{eq:SO(2p,2)_défi_des_sg1pdomindivp-admissibles}
\lambda_0:=(1,0,\ldots,0;0), \quad \lambda_1^{\pm}:=(1,\ldots,1;\pm 1) \quad \text{et} \quad \lambda_{-1}^{\pm}:=(1,\ldots,1,-1;\pm 1).
\end{equation}

\begin{theo}
L'ensemble des sous-groupes à un paramètre dominants indivisibles $\got{p}^-$-admissibles de $\TC$ associés à $G=SO(2p,2)$ est
\[
\mathcal{A}dm^+_{SO(2p,2)} = \{\lambda_0,\lambda_1^+,\lambda_1^-,\lambda_{-1}^+,\lambda_{-1}^-\}, 
\]
où $\lambda_0$, $\lambda_1^{\pm}$ et $\lambda_{-1}^{\pm}$ sont définis en \eqref{eq:SO(2p,2)_défi_des_sg1pdomindivp-admissibles}.
\end{theo}

\subsection{Le groupe $SO(2p+1,2)$, $p\geqslant 1$}
\label{subsection:ssgpe1param_SO(2p+1,2)}

Nous terminons l'étude des sous-groupes à un paramètre $\got{p}^-$-admissibles par ceux du groupe $SO(2p+1,2)$. Rappelons que les notations pour le groupe $SO(2p+1,2)$ et son algèbre de Lie ont été posées dans le paragraphe \ref{subsection:Notations_SO(2p+1,2)}. Les racines compactes positives sont
\[
\got{R}_c^+=\{e_i^*\pm e_j^*; 1\leqslant i<j\leqslant p\}\cup\{e_k^*; k=1,\ldots,p\}
\]
et les racines non compactes positives
\[
\got{R}_n^+=\{\pm e_i^*+e_{p+1}^*; i=1,\ldots,p\}\cup\{e_{p+1}^*\}.
\]
On notera à nouveau $\beta^{\pm}_i := \pm e_i^* + e_{p+1}^*$ pour tout $i\in\{1,\ldots,p\}$. On notera également $\beta_{p+1}=e_{p+1}^*$. On définit de manière identique au paragraphe \ref{subsection:ssgpe1param_SO(2p,2)} les sous-ensembles $I^+$ et $I^-$ de $\{1,\ldots,p\}$

L'étude des sous-groupes à un paramètre dominants indivisibles $\got{p}^-$-admissibles est sensiblement la même que celle du groupe $SO(2p,2)$. Il y a essentiellement deux différences. La première est qu'un sous-groupe à un paramètre $\lambda=(\lambda_1,\ldots,\lambda_{p+1})$ est dominant si et seulement si $\lambda_1\geqslant\ldots\geqslant\lambda_{p-1}\geqslant\lambda_p\geqslant 0$. La deuxième est le fait que l'on a rajouté la racine non compacte positive $\beta_{p+1}$ par rapport à $SO(2p,2)$. Cependant, cela n'apporte de changement que pour le cas $I^+=I^-=\emptyset$. En effet, ici, lorsque $I^+=I^-=\emptyset$, on a
\begin{itemize}
\item soit $\lambda_{p+1}\neq 0$, donc $\WT(\got{p}^-_{\lambda,0}) = \emptyset$ et $\dim\mathrm{Vect}(\WT(\got{p}^-_{\lambda,0})) = 0$;
\item sinon $\lambda_{p+1}= 0$, donc $\WT(\got{p}^-_{\lambda,0}) = \{\beta_{p+1}\}$ et $\dim\mathrm{Vect}(\WT(\got{p}^-_{\lambda,0})) = 1$.
\end{itemize}
Les autres propriétés sont identiques. On les regroupe dans le lemme suivant.

\begin{lemm}
\label{lemm:SO(2p+1,2)_propriétésimportantes}
Supposons que l'un des ensembles $I^+$ ou $I^-$ soit non vide. Alors, les trois assertions suivantes sont équivalentes :
\begin{enumerate}
\item $I^+\cap I^-\neq\emptyset$;
\item $I^+=I^-$;
\item $\lambda_{p+1}=0$.
\end{enumerate}
De plus, on a
\begin{itemize}
\item si $I^+\cap I^-=\emptyset$, alors $\dim\mathrm{Vect}(\WT(\got{p}^-_{\lambda,0})) = |I^+|+|I^-|$;
\item sinon $I^+=I^-$ et $\dim\mathrm{Vect}(\WT(\got{p}^-_{\lambda,0})) = |I^+|+1 = |I^-|+1$.
\end{itemize}
\end{lemm}

\begin{proof}
La preuve de ce résultat est identique à celles des Lemmes \ref{lemm:SO(2p,2)_I+capI-neq0_équivalent_I+=I-} et \ref{lemm:SO(2p,2)_dimensionVect(WTp-0)} et de la Remarque \ref{rema:SO(2p,2)_I+capI-neq0_équivalent_lambdap+1=0}.
\end{proof}

Soit $\lambda$ un sous-groupe à un paramètre dominant indivisible $\got{p}^-$-admissible de $\TC$. Alors nous devons avoir $\dim\mathrm{Vect}(\WT(\got{p}_{\lambda,0})) = p \geqslant 1$. Nous devons distinguer le cas $p=1$, car on peut alors avoir $I^+=I^-=\emptyset$. On doit alors avoir $\lambda_{p+1}=\lambda_2=0$ et $\lambda_1>0$ pour que $\lambda$ soit dominant. Donc $\lambda = (1;0)$. Sinon, $I^{\pm}\neq\emptyset$, et alors la seule possibilité est $I^+\cap I^-=\emptyset$ et $\lambda_{p+1}=\lambda_2\neq 0$, d'après le Lemme \ref{lemm:SO(2p+1,2)_propriétésimportantes}. Ceci nous donne $\lambda=(1;\pm 1)$.

Lorsque $p\geqslant 2$, un raisonnement analogue au paragraphe \ref{subsection:ssgpe1param_SO(2p,2)} nous donne que les seuls sous-groupes à un paramètre de $\TC$ dominants indivisibles et $\got{p}^-$-admissibles sont $\lambda_0$ et $\lambda_1^{\pm}$, qui ont été définis en \eqref{eq:SO(2p,2)_défi_des_sg1pdomindivp-admissibles}.

\begin{theo}
L'ensemble des sous-groupes à un paramètre dominants indivisibles $\got{p}^-$-admissibles de $\TC$ associés à $G=SO(2p+1,2)$ est
\[
\mathcal{A}dm^+_{SO(2p+1,2)} = \{\lambda_0,\lambda_1^+,\lambda_1^-\}.
\]
\end{theo}

\section[Exemples de calculs explicites de projection d'orbites]{Exemples de calculs explicites de projections d'orbites holomorphes}
\label{subsection:exemplespolyèdresmoment}

Cette section est consacrée au calcul des polyèdres moments $\Delta_K(\Orb_{\Lambda})$ de la projection d'une orbite coadjointe holomorphe $\Orb_{\Lambda} = G\cdot\Lambda$, pour un certain nombre d'exemples. 

Ce travail a pour principal objectif de vérifier que la partie théorique effectuée dans les Chapitres \ref{chap:projectiondorbitecoadjointe+GIT} et \ref{chap:PairesBienCouvrantes} est correcte, en faisant le lien avec les résultats obtenus lors des premières tentatives de calculs, cf Chapitre \ref{chap:premières_tentatives}.

Nous continuons à utiliser les notations générales posées au Chapitre \ref{chap:projectiondorbitecoadjointe+GIT}. Les exemples porteront sur les groupes classiques $Sp(2n,\R)$, $SU(n,1)$, $SO^*(6)$, $SO^*(8)$ et $SU(2,2)$, dont les définitions et les notations associées ont été faites dans le Chapitre \ref{chap:espaces_symétriques_hermitiens}.

%
%

\subsection{Projections d'orbites coadjointes holomorphes de $Sp(2n,\R)$, $n\geqslant 2$}
\label{subsection:exemplespolyèdresmoment_Sp(R2n)}

Toutes les notations pour les calculs concernant le groupe $G = Sp(2n,\R)$ sont rassemblées dans le paragraphe \ref{subsection:Notations_Sp(R2n)}. Les racines non compactes positives sont les $\beta_{i,j} = e_i^*+e_j^*$, pour $1\leqslant i\leqslant j\leqslant n$. La plus petite racine non compacte négative est $-\beta_{1,1}=\bmin$. Toute racine non compacte négative $\beta$ s'écrit
\[
\beta = \bmin + \sum_{\alpha\in\got{R}_c^+} n_{\alpha}\alpha,
\]
avec $n_{\alpha}\in\N$ pour tout $\alpha\in\got{R}_c^+$. Par conséquent, pour tout sous-groupe à un paramètre $\lambda$ dominant, on a $\langle\lambda,\bmin\rangle \leqslant \langle\lambda,\beta\rangle$, pour toute racine $\beta\in\got{R}_n^-$.

Nous allons calculer $\Theta(-\bmin) = \Theta(\beta_{1,1})$. La formule de Chevalley (Théorème \ref{theo:Chevalley}) nous donne
\[
\Theta(\beta_{1,1})  = \Theta(\beta_{1,1}).\sigma^B_{\id}  =  \sum_{\stackrel{\alpha\in\got{R}_c^+}{l(s_{\alpha}) = 1}}\beta_{1,1}(\alpha^{\vee})\sigma^B_{s_{\alpha}} = \sum_{i=1}^{n-1}\beta_{1,1}(\alpha_{i,i+1}^{\vee})\sigma^B_{s_{\alpha_{i,i+1}}}.
\]
Or, on a $\alpha_{i,i+1}^{\vee} = e_i-e_{i+1}$, donc $\beta_{1,1}(\alpha_{1,2}^{\vee}) = 2$ et $\beta_{1,1}(\alpha_{i,i+1}^{\vee}) = 0$ si $2\leqslant i \leqslant n-1$. On en déduit que $\Theta(-\bmin) = \Theta(\beta_{1,1}) = 2\sigma^B_{s_{\alpha_{1,2}}}$. Par conséquent, pour tout $m>\langle\lambda,\bmin\rangle$, il existe un entier positif $p$ tel que
\[
\sigma_{w_0w}^B\,.\,\sigma_{w_0w'}^B\,.\,\prod_{\beta\in\WT(\got{p}^-_{< m})}\Theta(-\beta)^{n_{\beta}}= 2p\sigma_{w_0w_{\lambda}}^B+\text{ termes de degré $< 2l(w_0w_{\lambda})$}.
\]
On en déduit que, pour un sous-groupe à un paramètre dominant $\lambda$ de $\TC$ fixé, si $\got{p}^-_{< 0}\neq 0$, alors il n'existe aucun couple $(w,w')\in(W^{\lambda})^2$ vérifiant l'équation \[
\sigma_{w_0w}^B\,.\,\sigma_{w_0w'}^B\,.\,\prod_{\beta\in\WT(\got{p}^-_{<0})}\Theta(-\beta)^{n_{\beta}}=\sigma_{w_0w_{\lambda}}.
\]
Une application du Théorème \ref{theo:cns_pairebiencouvrante} nous permet alors d'affirmer que les seules paires bien couvrantes de $X_{\got{p}^-\oplus\C}$ du type $(C(w,w',0),\lambda)$, avec $\lambda$ dominant de $\TC$, sont les paires $(C(w,w_0ww_{\lambda},0),\lambda)$ vérifiant $\got{p}^-_{<0}=0$, c'est-à-dire, $\langle\lambda,\bmin\rangle \geqslant 0$. De plus, si on demande que $\lambda$ soit aussi $\got{p}^-$-admissible, on doit avoir nécessairement $\langle\lambda,\bmin\rangle\leqslant 0$. Il doit donc vérifier $\langle\lambda,\bmin\rangle = -2\lambda_1 = 0$, donc $\lambda_1=0$.

L'ensemble des sous-groupes à un paramètre dominants indivisibles $\got{p}^-$-admissibles de $\TC$ a été calculé dans le paragraphe \ref{subsection:ssgpe1param_Sp(R,n)}. D'après le Théorème \ref{theo:ssgpe1param_Sp(R,n)}, les sous-groupes à un paramètre dominants indivisibles $\got{p}^-$-admissibles sont $(1,0,\ldots,0)$, $(0,\ldots,0,-1)$ et
\[
\lambda_{k,l} = (\underbrace{1,\ldots,1}_{k\ \mbox{termes}},\underbrace{0,\ldots,0}_{n-(k+l)\ \mbox{termes}},\underbrace{-1,\ldots,-1}_{l\ \mbox{termes}}),
\]
pour tout $k=1,\ldots, n-1$ et tout $l=1,\ldots,n-k$. Or, parmi ces sous-groupes à un paramètre, le seul qui vérifie $\lambda_1=0$ est $\lambda=(0,\ldots,0,-1)$.

De la Proposition \ref{prop:faces_ConeWtT(E)_et_pairesbiencouvrantes} et de sa preuve, on déduit que $\Delta_K(\Orb_{\Lambda}) = (\Lambda + \CR(\got{R}_n^+))\cap \got{t}^*_+$. Remarquons que l'on retrouve bien le résultat annoncé au Théorème \ref{theo:Sp(R2n)_polyèdremomentSpn}.

Le Théorème \ref{theo:équations_DeltaK_G0Lambda} nous permet de donner la description concrète suivante du polyèdre moment de la variété $K$-hamiltonienne $\Orb_{\Lambda}$ :
\[
\Delta_K(\Orb_{\Lambda}) = \{\xi=(\xi_1,\ldots,\xi_n)\in\got{t}^*_+;
\mbox{ pour tout } i=1,\ldots,n,\ \xi_i\geqslant\Lambda_i\}.
\]
La figure $1$ représente ce polyèdre moment dans le cas du groupe $Sp(4,\R)$. Nous résumons ceci dans le théorème suivant.

\begin{theo}
Pour $G = Sp(2n,\R)$, le polyèdre moment est égal au polyèdre convexe
\[
\Delta_K(\Orb_{\Lambda}) = (\Lambda + \CR(\got{R}_n^+))\cap\got{t}^*_+ = \{(\xi_1,\ldots,\xi_n)\in\got{t}^*_+; \xi_i\geqslant\Lambda_i\ \forall i=1,\ldots,n\}.
\]
\end{theo}

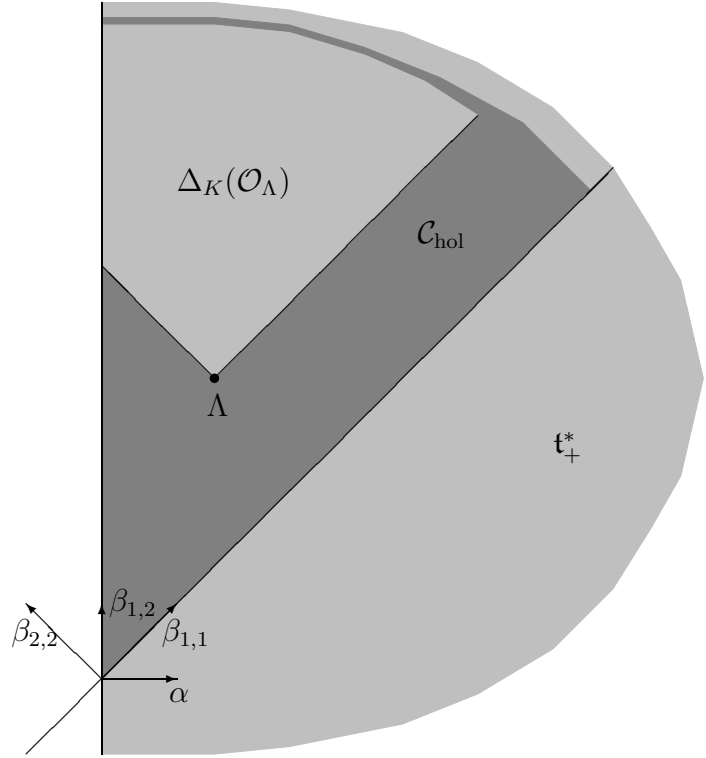
\begin{figure}[htb]
\label{figure:polyèdremoment_sp(r,n)}
\begin{center}
\setlength{\unitlength}{1cm}
\begin{pspicture}(10,10)
\definecolor{Couleur1}{rgb}{0,0.5,0.7}
\definecolor{Couleur2}{rgb}{0.6,0.4,0.2}

\pspolygon[linecolor=lightgray,fillstyle=solid,linewidth=0pt,fillcolor=lightgray](2,0)(2,10)(3.5,10)(4.5,9.9)(5,9.8)(6,9.6)(7,9.2)(8,8.6)(8.8,7.8)(9.3,7)(9.7,6.3)(10,5)(9.7,3.7)(9.3,3)(8.8,2.2)(8,1.4)(7,0.8)(6,0.4)(5,0.2)(4.5,0.1)(3.5,0)
\pspolygon[linecolor=gray,fillstyle=solid,linewidth=0pt,fillcolor=gray](2,1)(2,9.8)(3.5,9.8)(4.5,9.7)(5.5,9.4)(6.5,9)(7.6,8.4)(8.5,7.5)
\pspolygon[linecolor=lightgray,fillstyle=solid,linewidth=0pt,fillcolor=lightgray](3.5,5)(2,6.5)(2,9.7)(3.5,9.7)(4.5,9.6)(5.5,9.3)(6.3,8.95)(7,8.5)

\psdot(3.5,5)

\put(2, 0){\line(0, 1){10}}
\put(1, 0){\line(1, 1){7.8}}
\put(3.5,5){\line(-1, 1){1.5}}
\put(3.5,5){\line(1, 1){3.5}}

\put(2, 1){\vector(0, 1){1}}
\put(2, 1){\vector(1, 0){1}}
\put(2, 1){\vector(1, 1){1}}
\put(2, 1){\vector(-1, 1){1}}

\put(3.4,4.5){$\Lambda$}
\put(2.9,0.7){$\alpha$}
\put(2.8,1.5){$\beta_{1,1}$}
\put(2.1,1.9){$\beta_{1,2}$}
\put(0.8,1.5){$\beta_{2,2}$}
\put(8,4){$\got{t}^*_+$}
\put(6.2,6.8){$\Chol$}
\put(3,7.5){$\Delta_K(\Orb_{\Lambda})$}

\end{pspicture}
\end{center}
\caption{Polyèdre de la projection d'une orbite coadjointe holomorphe de $Sp(4,\R)$}
\end{figure}

\subsection{Projections d'orbites coadjointes holomorphes de $SU(n,1)$, $n\geqslant 2$}
\label{subsection:exemplespolyèdresmoment_SU(n,1)}

Rappelons que les notations concernant le groupe $G = SU(n,1)$ ont été données dans le paragraphe \ref{subsection:Notations_SU(p,q)}. Les racines non compactes sont les $\pm\beta_k = \pm(e_k^* + \sum_{j=1}^n e_j^*)$ pour $k=1,..,n$, avec $\got{R}_n^+ =\{\beta_k; k=1,\ldots,n\}$ et $\got{R}_n^- =\{-\beta_k; k=1,\ldots,n\}$. La plus petite racine négative est $\bmin = -\beta_1$ et, pour tout $k=2,\ldots,n$, on a
\[
-\beta_k = \bmin + \alpha_{1,2} + \ldots + \alpha_{k-1,k}.
\]
On pose, pour $k=1,\ldots,n$, le sous-groupe à un paramètre $\lambda_k = (n+1)e_k - \sum_{j=1}^ne_j$. D'après le Théorème \ref{theo:ssgpes1param_SU(n,1)}, les seuls sous-groupes à un paramètre dominants indivisibles $\got{p}^-$-admissibles de $\TC$ sont $\lambda_1 = (n,-1,\ldots,-1)$ et $-\lambda_n = (1,\ldots,1,-n)$.

Le second, $-\lambda_n$, donne $M_{<0} = 0$, car $\langle-\lambda_{n},\bmin\rangle = 0$, avec $\bmin$ le plus petit poids de l'action de $T$ sur $\got{p}^-$. Les équations provenant de $-\lambda_n$ donnent, par conséquent, les équations du cône convexe $\Lambda + \CR(\got{p}^+)$. Ces équations sont les suivantes :
\[
\langle-\lambda_k,\xi-\Lambda\rangle \leqslant 0,
\]
pour $k$ parcourant l'ensemble $\{1,\ldots,n\}$, ou, en notant $\xi = (\xi_1,\ldots,\xi_n)$, ainsi que $\Lambda = (\Lambda_1,\ldots,\Lambda_n)$, dans la base $(e_1^*,\ldots,e_n^*)$,
\[
\begin{array}{cc}
(A_k): & n\xi_k - \sum_{j\neq k}\xi_j \geqslant n\Lambda_k - \sum_{j\neq k}\Lambda_j,
\end{array}
\]
pour $k=1,\ldots,n$.

Nous allons voir que, contrairement à $Sp(2n,\R)$, ici apparaissent d'autres équations, celles pour $\lambda_1$. Maintenant, nous avons $\langle\lambda_1,\bmin\rangle = -(n+1)$ et $\langle\lambda_1,-\beta_k\rangle = 0$, pour $k=2,\ldots,n$. De plus, $\langle\lambda_1,\alpha_{1,2}\rangle=n+1 > 0$ et $\langle\lambda_1,\alpha_{k,k+1}\rangle = 0$ si $k\in\{2,\ldots,n-1\}$. Par conséquent, les sous-groupes paraboliques associés à $\lambda_1$ sont
\[
P(\lambda_1) = \left(\begin{array}{c|ccc}
* & * & \ldots & * \\ \hline
0 & * & \ldots & * \\
\vdots & \vdots & \ddots & \vdots \\
0 & * & \ldots & *
\end{array}\right)\subset GL_n(\C) \text{ et }  
\hat{P}(\lambda_1) = \left(\begin{array}{c|ccc}
* & * & \ldots & * \\ \hline
0 & * & \ldots & * \\
\vdots & \vdots & \ddots & \vdots \\
0 & * & \ldots & *
\end{array}\right)\subset GL_{n+1}(\C).
\]
Dans cette situation, nous n'avons que deux $m$ possibles pour les paires bien couvrantes $(C(w,w',m),\lambda_1)$. Ce sont les entiers $m=0$ et $m=-(n+1)<0$. Mais seul $m=0$ nous intéresse, d'après le Théorème \ref{theo:équations_DeltaK_G0Lambda}. 
Donc les paires bien couvrantes avec $\lambda_1$ et $m=0$ sont les paires $(C(w,w',0),\lambda_1)$, avec $(w,w')\in (W^{\lambda_1})^2$, telles que
\[
\sigma_{w_0w}^B\,.\,\sigma_{w_0w'}^B\,.\,\Theta(-\bmin) = \sigma_{w_0w_{\lambda_1}}^B
\]
et
\begin{equation}
\label{eq:calulSU(n,1)_équationlinéaireàvérifier}
\langle w\lambda_1 + w'\lambda_1,\rho\rangle - \langle\lambda_1,\bmin\rangle = 0.
\end{equation}
d'après le Théorème \ref{theo:cns_pairebiencouvrante}. De plus, le Corollaire \ref{coro:info_longueurs_éléments_pour_pairebiencouvrante} montre que, nécessairement, nous aurons l'équation suivante
\[
l(w) + l(w') = l(w_0) + l(w_{\lambda_1}) + 1.
\]
Nous nous trouvons ici dans le cas où $\KC = GL_n(\C)$ et $W^{\lambda_1}$ est connu. En effet, nous savons d'après le Lemme \ref{lemm:éléments_les_plus_courts_de_W/WQ}, que $W/W_{\lambda_1}$ a $n$ classes, et que les éléments $\hat{w}_k^{-1}$, pour $k=1,\ldots,n$, forment un système de représentants de $W/W_{\lambda_1}$ de plus petite longueur dans leurs classes respectives. Les éléments les plus longs de chaque classe seront donc les $\hat{w}_k^{-1}w_{\lambda_1}$ et, toujours grâce au Lemme \ref{lemm:éléments_les_plus_courts_de_W/WQ}, la longueur d'un tel élément est $l(\hat{w}_k^{-1}w_{\lambda_1}) = l(w_{\lambda_1}) + k-1$, pour $k=1,\ldots,n$. Donc, si la paire $(C(\hat{w}_k^{-1}w_{\lambda_1},\hat{w}_{k'}^{-1}w_{\lambda_1},0),\lambda_1)$ est bien couvrante, nous obtenons une équation vérifiée par $k$ et $k'$,
\[
k-1 + k'-1 = l(w_0) - l(w_{\lambda_1}) + 1 = n-1+1 = n,
\]
c'est-à-dire $k' = n-k+2$. Il reste donc à trouver, parmi les paires de la forme $(C(\hat{w}_k^{-1}w_{\lambda_1},\hat{w}_{n-k+2}^{-1}w_{\lambda_1},0),\lambda_1)$ avec $k\in\{2,\ldots,n\}$, celles qui sont bien couvrantes. Il faut donc commencer par trouver celles qui vérifient
\[
\sigma_{w_0\hat{w}_k^{-1}w_{\lambda_1}}^B\,.\,\sigma_{w_0\hat{w}_{n-k+2}^{-1}w_{\lambda_1}}^B\,.\,\Theta(-\bmin) = \sigma_{w_0w_{\lambda_1}}^B.
\]
Le Lemme \ref{lemm:éléments_les_plus_courts_de_W/WQ} montre que $w_0\hat{w}_k^{-1}w_{\lambda_1} = \hat{w}_{n-k+1}^{-1} = s_{n-k}\ldots s_1$, pour tout $k=1,\ldots,n$ (pour $k=n$, $w_0\hat{w}_n^{-1}w_{\lambda_1} = \id$), et $w_0w_{\lambda_1} = \hat{w}_{n}^{-1}=s_{n-1}\ldots s_1$. Ensuite, nous pouvons calculer $\Theta(-\bmin)$ par la formule de Chevalley,
\[
\Theta(-\bmin) = \sum_{\stackrel{\alpha\in\got{R}_c^+}{l(s_{\alpha}) = 1}}\beta_1(\alpha^{\vee})\sigma_{s_{\alpha}}^B = \sum_{i=1}^{n-1}\beta_1(\alpha_{i,i+1}^{\vee})\sigma_{s_i}^B = \sigma_{s_1}^B,
\]
car $\beta_1(\alpha_{1,2}^{\vee}) = \beta_1(e_1-e_2) = 1$ et $\beta_1(\alpha_{i,i+1}^{\vee}) = \beta_1(e_i-e_{i+1}) = 0$ si $i=2,\ldots,n-1$. Il ne reste plus qu'à calculer le produit cup $\sigma_{\hat{w}_{n-k+1}^{-1}}^B\,.\,\sigma_{\hat{w}_{k-1}^{-1}}^B\,.\,\sigma_{s_1}^B$ pour $k=1,\ldots,n$.

\begin{lemm}
\label{lemm:formule_puissancede_sigma_s_1}
Pour tout $k=1,\ldots,n-1$, on a $(\sigma_{s_1}^B)^k = \sigma_{s_k\ldots s_1}^B = \sigma_{\hat{w}_{k+1}^{-1}}^B$. De plus, $(\sigma_{s_1}^B)^n = 0$. Et, pour $a,b\in\{1,\ldots,n\}$, nous aurons $\sigma_{\hat{w}_{a}^{-1}}^B\,.\,\sigma_{\hat{w}_{b}^{-1}}^B = \sigma_{\hat{w}_{a+b-1}^{-1}}^B$ si $a+b \leqslant n+1$ et $\sigma_{\hat{w}_{a}^{-1}}^B\,.\,\sigma_{\hat{w}_{b}^{-1}}^B = 0$ sinon.
\end{lemm}

\begin{proof}
Nous nous inspirons des Lemmes \ref{lemm:produitcup_sk_wtildek} et \ref{lemm:produitcup_skmoins1_wtildek}. Nous allons montrer que l'on a $\sigma_{s_1}^B\,.\,\sigma_{s_k\ldots s_1}^B = \sigma_{s_{k+1}\ldots s_1}$ si $k\in\{1,\ldots, n-2\}$ et réaliser une récurrence simple sur $k$. Nous utilisons la formule de Chevalley (Théorème \ref{theo:Chevalley}),
\[
\sigma_{s_1}^B\,.\,\sigma_{s_k\ldots s_1}^B = \Theta(\pi_{\alpha_{1,2}})\,.\,\sigma_{s_k\ldots s_1} = \sum_{\stackrel{j=3,\ldots,n}{l(s_k\ldots s_1s_{\alpha_{1,j}}) = l(s_k\ldots s_1)+1}} \sigma_{s_k\ldots s_1 s_{\alpha_{1,j}}}^B.
\]
Du Lemme \ref{lemm:wtilde-1_plus1_jégalk+2}, nous avons $l(s_k\ldots s_1s_{\alpha_{1,j}}) = l(s_k\ldots s_1)+1$ si et seulement si $j = k+2$, ce qui est possible car $k\leqslant n-2$, donc $j\leqslant n$, et, dans ce cas-là, $s_k\ldots s_1s_{\alpha_{1,j}} = s_{k+1}\ldots s_1$. Nous obtenons bien $\sigma_{s_1}^B\,.\,\sigma_{s_k\ldots s_1}^B = \sigma_{s_{k+1}\ldots s_1}^B$. Pour $k=1$, nous avons évidemment $(\sigma_{s_1}^B)^1 = \sigma_{s_1}^B$, et, par récurrence sur $k$, nous avons $(\sigma_{s_1}^B)^k = \sigma_{s_k\ldots s_1}^B$, pour $k\leqslant n-1$.

Pour $k=n-1$, nous n'aurons jamais $j = k+2$ dans $\{3,\ldots,n\}$, donc $(\sigma_{s_1}^B)^n = \sigma_{s_1}^B\,.\,\sigma_{s_{n-1}\ldots s_1}^B = 0$, d'après le Lemme \ref{lemm:wtilde-1_plus1_jégalk+2}.

Enfin, pour $a,b\in\{2,\ldots,n\}$, nous pouvons calculer $\sigma_{\hat{w}_a^{-1}}^B\,.\,\sigma_{\hat{w}_b^{-1}}^B$, car $\hat{w}_a^{-1} = s_{a-1}\ldots s_1$ (resp. $\hat{w}_b^{-1} = s_{b-1}\ldots s_1$). On obtient
\[
\sigma_{\hat{w}_a^{-1}}^B\,.\,\sigma_{\hat{w}_b^{-1}}^B = \sigma_{s_{a-1}\ldots s_1}^B\,.\,\sigma_{s_{b-1}\ldots s_1}^B = (\sigma_{s_1}^B)^{a-1}\,.\,(\sigma_{s_1}^B)^{b-1} = (\sigma_{s_1}^B)^{a+b-2},
\]
et ce dernier est égal à $\sigma_{s_{a+b-2}\ldots s_1}^B = \sigma_{\hat{w}_{a+b-1}^{-1}}^B$ si $a+b-2\leqslant n-1$, c'est-à-dire $a+b\leqslant n+1$. Et pour $a+b>n+1$, nous avons $a+b-2 \geqslant n$, donc $\sigma_{\hat{w}_a^{-1}}^B\,.\,\sigma_{\hat{w}_b^{-1}}^B = 0$. Pour terminer, si $a$ ou $b$ vaut $1$, quitte à permuter les notations, on peut supposer que $a=1$, alors $\hat{w}_1^{-1} = \id$, donc $\sigma_{\hat{w}_1^{-1}}^B\,.\,\sigma_{\hat{w}_b^{-1}}^B = \sigma_{\hat{w}_b^{-1}}^B$, avec bien $b= b+a-1\leqslant n$.
\end{proof}

Revenons à nos produits cup $\sigma_{\hat{w}_{n-k+1}^{-1}}^B\,.\,\sigma_{\hat{w}_{k-1}^{-1}}^B\,.\,\sigma_{s_1}^B$ pour $k=2,\ldots,n$. Du lemme ci-dessus, nous obtenons
\[
\sigma_{\hat{w}_{n-k+1}^{-1}}^B\,.\,\sigma_{\hat{w}_{k-1}^{-1}}^B\,.\,\sigma_{s_1}^B = (\sigma_{s_1}^B)^{n-k} \,.\, (\sigma_{s_1}^B)^{k-2} \,.\, \sigma_{s_1}^B = (\sigma_{s_1}^B)^{n-1} = \sigma_{\hat{w}_n^{-1}}^B,
\]
ce qui redonne bien
\[
\sigma_{w_0\hat{w}_k^{-1}w_{\lambda_1}}^B\,.\,\sigma_{w_0\hat{w}_{n-k+2}^{-1}w_{\lambda_1}}^B\,.\,\Theta(-\bmin) = \sigma_{w_0w_P}^B,
\]
pour tout $k=2,\ldots,n$. Il reste encore à montrer que
\[
\langle \hat{w}_k^{-1}w_{\lambda_1}\lambda_1 + \hat{w}_{n-k+2}^{-1}w_{\lambda_1}\lambda_1,\rho\rangle - \langle\lambda_1,\bmin\rangle = 0.
\]
pour tout $k=2,\ldots,n$.

\begin{lemm}
\label{lemm:crochetdedualité_lambda1_gamma_SU(n,1)}
Pour tout $k=1,\ldots,n$, on a $\langle\lambda_1,\hat{w}_k\rho+\hat{w}_{n-k+2}\rho\rangle = \langle\lambda_1,\bmin\rangle$.
\end{lemm}

\begin{proof}
Nous utilisons ici le fait que $\hat{w}_{n-k+2} = w_{\lambda_1}\hat{w}_{k-1}w_0$. Ceci implique les égalités
\begin{align*}
\langle\lambda_1,\hat{w}_k\rho+\hat{w}_{n-k+2}\rho\rangle & = \langle\lambda_1,\hat{w}_k\rho+\hat{w}_{k-1}w_0\rho\rangle \\
& = \langle\hat{w}_{k-1}\lambda_1,s_k\rho-\rho\rangle \\
& = -\langle\hat{w}_{k-1}\lambda_1,\alpha_{k,k+1}\rangle \\
& = -\langle\lambda_1,\alpha_{1,k+1}\rangle.
\end{align*}
On en conclut l'égalité $\langle\lambda_1,\hat{w}_k\rho+\hat{w}_{n-k+2}\rho\rangle = -(n+1) = \langle\lambda_1,\bmin\rangle$.
\end{proof}

Puisque $w_{\lambda_1}$ appartient à $W_{\lambda_1}$, on a $\langle \hat{w}_k^{-1}w_{\lambda_1}\lambda_1 + \hat{w}_{n-k+2}^{-1}w_{\lambda_1}\lambda_1,\rho\rangle = \langle\lambda_1,\hat{w}_k\rho+\hat{w}_{n-k+2}\rho\rangle$, ce qui implique que l'équation \eqref{eq:calulSU(n,1)_équationlinéaireàvérifier} est vraie pour les couples $(w,w') = (\hat{w}_k^{-1}w_{\lambda_1},\hat{w}_{n-k+2}^{-1}w_{\lambda_1})$. Les paires $(C(\hat{w}_k^{-1},\hat{w}_{n-k+2}^{-1},0),\lambda_1)$ sont donc toutes bien couvrantes pour $k=2,\ldots,n$, et ce sont les seules pour $\lambda_1$.


La paire $(C(\hat{w}_k^{-1},\hat{w}_{n-k+2}^{-1},0),\lambda_1)$ apporte l'équation suivante,
\begin{equation}
\label{eq:équation_pairesavechatwkmoins1}
\langle \hat{w}_k^{-1}\lambda_1,x\rangle \leqslant \langle w_0\hat{w}_{n-k+2}^{-1}\lambda_1,\Lambda\rangle.
\end{equation}
De la définition $\hat{w}_{k}^{-1} = s_{k-1}\ldots s_1$ pour $2\leqslant k\leqslant n$, on a donc $\hat{w}_{k}^{-1}\lambda_1 = s_{k-1}\ldots s_1\lambda_1 = \lambda_{k}$. Remarquons, de plus, que l'on a
\[
\langle w_0\hat{w}_{n-k+2}^{-1}\lambda_1,\Lambda\rangle = \langle w_0\hat{w}_{n-k+2}^{-1}w_{\lambda_1}\lambda_1,\Lambda\rangle = \langle \hat{w}_{k-1}^{-1}\lambda_1,\Lambda\rangle = \langle\lambda_{k-1},\Lambda\rangle.
\]
L'équation \eqref{eq:équation_pairesavechatwkmoins1} devient maintenant
\[
\langle \lambda_k,x\rangle\leqslant\langle\lambda_{k-1},\Lambda\rangle,
\]
pour $k\in\{2,\ldots,n\}$. D'où, en faisant le changement de variable $k'= k-1$, on obtient
\[
\langle \lambda_{k+1},x\rangle\leqslant\langle\lambda_{k},\Lambda\rangle
\]
pour tout $k=1,\ldots,n-1$. Finalement, on peut remplacer $\lambda_k$ par sa valeur $ne_k - \sum_{j\neq k}e_j$, ce qui nous donne
\[
\begin{array}{cc}
(B_k): & \sum_{j\neq k+1}\xi_j - n\xi_{k+1} \geqslant  \sum_{j\neq k}\Lambda_j - n\Lambda_k,
\end{array}
\]
pour $k=1,\ldots,n-1$.

En mettant ensemble les équations obtenues pour $-\lambda_n$ puis $\lambda_1$, nous obtenons le théorème qui suit. Il découle évidemment du Théorème \ref{theo:équations_DeltaK_G0Lambda}.

\begin{theo}
Pour $G = SU(n,1)$, le polyèdre moment est égal au polyèdre convexe
\begin{align*}
\Delta_K(\Orb_{\Lambda}) & = \{\xi\in\got{t}^*_+; \, \xi \mbox{ vérifie } (A_k),\,k=1,\ldots,n,\,\mbox{et } \xi \mbox{ vérifie } (B_{k'}),\,k'=1,\ldots,n-1\} \\
& = \{\xi\in\got{t}^*_+; \ \langle\lambda_1,\xi\rangle \geqslant \langle\lambda_1,\Lambda\rangle\geqslant\langle\lambda_2,\xi\rangle\geqslant\ldots\geqslant\langle\lambda_n,\xi\rangle\geqslant\langle\lambda_n,\Lambda\rangle\}.
\end{align*}
\end{theo}

On retrouve bien l'énoncé du Théorème \ref{theo:SU(n,1)_PolyèdreMoment}, obtenu grâce aux calculs appliquant le Théorème de Horn-Klyachko. C'est aussi le même résultat que ce qui a été annoncé au paragraphe \ref{subsection:exemple_utilisation_formule_DHV} pour le groupe $SU(2,1)$.

\begin{figure}[htb]
\begin{center}
\setlength{\unitlength}{1cm}
\begin{pspicture}(10,10)
\definecolor{Couleur1}{rgb}{0,0.5,0.7}
\definecolor{Couleur2}{rgb}{0.6,0.4,0.2}

\pspolygon[linecolor=lightgray,fillstyle=solid,linewidth=0pt,fillcolor=lightgray](2,0)(2,10)(3.5,10)(4.5,9.9)(5,9.8)(6,9.6)(7,9.2)(8,8.6)(8.8,7.8)(9.3,7)(9.7,6.3)(10,5)(9.7,3.7)(9.3,3)(8.8,2.2)(8,1.4)(7,0.8)(6,0.4)(5,0.2)(4.5,0.1)(3.5,0)
\pspolygon[linecolor=gray,fillstyle=solid,linewidth=0pt,fillcolor=gray](2,1)(2,9.8)(3.5,9.8)(4.5,9.7)(5.5,9.4)(6.5,9)(7.6,8.4)(8.5,7.5)(9.1,6.4)(9.4,5.272)
\pspolygon[linecolor=lightgray,fillstyle=solid,linewidth=0pt,fillcolor=lightgray](4,3.5)(2,6.964)(3.6,9.735)(4.4,9.65)(5,9.5)(5.8,9.25)(6.3,9.05)(6.7,8.85)(7,8.696)


\psdot(4,3.5)

\psline[linewidth=1pt](2,0)(2,10)
\psline[linewidth=1pt](0.268,0)(9.9,5.561)
\psline[linewidth=0.5pt](7,8.69)(4,3.5)(2,6.96)(3.6,9.73)

\psline[linewidth=1pt]{->}(2,1)(3,1)
\psline[linewidth=1pt]{->}(2,1)(2.5,1.87)
\psline[linewidth=1pt]{->}(2,1)(1.5,1.87)

\rput(4,3.2){$\Lambda$}
\rput(2.9,0.7){$\alpha$}
\rput(2.8,2){$\beta_{1}$}
\rput(1.2,2){$\beta_{2}$}
\rput(7,2){$\got{t}^*_+$}
\rput(8,5.5){$\Chol$}
\rput(4,7){$\Delta_K(\Orb_{\Lambda})$}

\end{pspicture}
\end{center}
\caption{Polyèdre de la projection d'une orbite régulière holomorphe de $SU(2,1)$}
\end{figure}
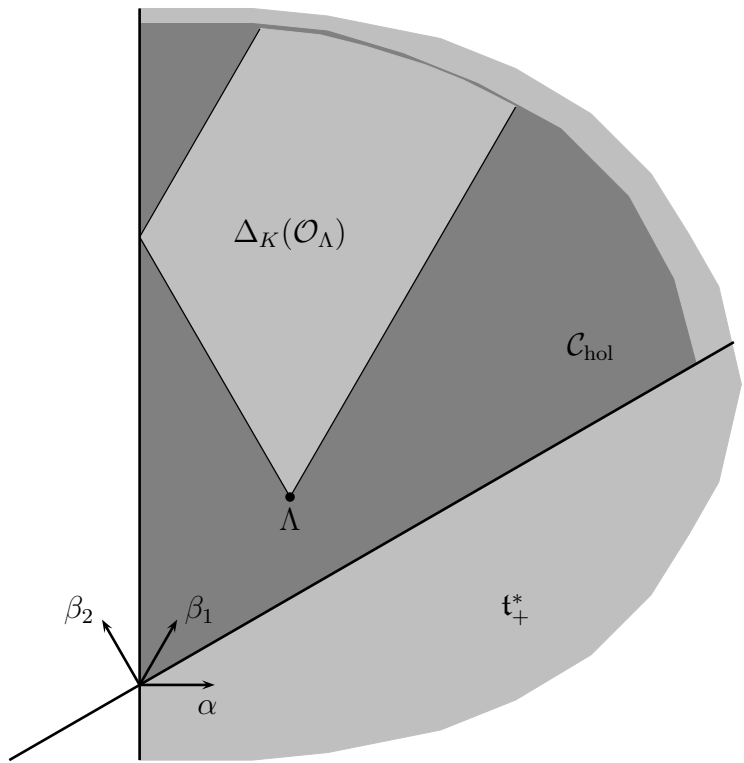

\begin{rema}
\label{rema:repirr_dans_VLambda_otimes_polynomes_sur_Cn}
Presque tous les calculs réalisés dans ce paragraphe sont identiques quand on considère le polyèdre moment $\Delta_{U(n)}(U(n)\cdot\Lambda\times(\C^n)^*)$, où $U(n)$ agit canoniquement sur $(\C^n)^*$ et $\Lambda$ est un élément de  $\wedge_{\Q,+}^*$. La seule différence apparaît dans l'ensemble des sous-groupes à un paramètre dominants indivisibles $(\C^n)^*$-admissibles de $\TC$, puisqu'ici cet ensemble est $\{(1,0,\ldots,0),(0,\ldots,-1)\}$. Par conséquent, le Corollaire \ref{coro:équationsgénérales_DGIT} nous donne
\[
\Delta_{U(n)}\bigl(U(n)\cdot\Lambda\times(\C^n)^*\bigr) = \bigl\{\xi\in\got{t}^*_+; \  \xi_1\geqslant\Lambda_1\geqslant\xi_2\geqslant\ldots\xi_n\geqslant\Lambda_n\bigl\}.
\]
Ceci donne une autre preuve que la représentation irréductible $V_{\mu}$ de $GL_n(\C)$ de plus haut poids $\mu=(\mu_1\geqslant\ldots\geqslant\mu_n)$, apparaît dans la décomposition en somme directe de représentations irréductibles du $GL_n(\C)$-module $V_{\Lambda}\otimes\mathrm{S}(\C^n)$ si et seulement si
\[
\mu_1\geqslant\Lambda_1\geqslant\mu_2\geqslant\ldots\mu_n\geqslant\Lambda_n,
\]
résultat prouvé de manière différente dans \cite{macdonald,brion}, mais qui est bien antérieur et serait dû à Pieri. Nous avons utilisé ici la propriété de saturation de $GL_n(\C)$ relativement au produit tensoriel de ses représentations irréductibles. Nous renvoyons le lecteur à \cite[9.3]{woodward} pour plus de détails concernant ce résultat.
\end{rema}

\subsection{Projections d'orbites coadjointes holomorphes de $SO^*(2n)$, $n\geqslant 3$}
\label{subsection:polyèdremoment_SO*(2n)}

Le groupe $G = SO^*(2n)$ a été défini dans le paragraphe \ref{subsection:Notations_SO*(2n)}. On a vu que les racines compactes seront les $\alpha_{i,j}=e_i^*-e_j^*$, $1\leqslant i<j\leqslant n$, pour la base définie dans ce même paragraphe.

Les racines non compactes sont les $\pm\beta_{i,j} = \pm(e_i^* + e_j^*)$ pour $1\leqslant i< j \leqslant n$, avec $\got{R}_n^+ =\{\beta_{i,j}\,; \, 1\leqslant i< j \leqslant n\}$ et $\got{R}_n^- =\{-\beta_{i,j}\,; \, 1\leqslant i< j \leqslant n\}$. La plus petite racine non compacte négative est $\bmin = -\beta_{1,2} = -e_1^*-e_2^*$.

D'après le Théorème \ref{theo:ssgpe1param_SO*(2n)}, les sous-groupes à un paramètre dominants indivisibles et $\got{p}^-$-admissibles de $\TC$ sont les éléments de
\begin{align*}
\mathcal{A}dm^+_{SO^*(2n)} & = \{(1,0,\ldots,0),(0,\ldots,0,-1)\}\cup\{\lambda_{k,n-k}; k=1,\ldots, n-1\} \\
& \quad \cup\{\lambda_{k,l}; k=1,\ldots, n-4, l=1,\ldots,n-k-3\}.
\end{align*}
lorsque $n\geqslant 4$, et
\[
\mathcal{A}dm^+_{SO^*(6)} = \{(1,-1,-1),(1,1,-1)\}.
\]
Rappelons que les sous-groupes à un paramètre $\lambda_{k,l}$, pour $k=1,\ldots, n-1$, $l=1,\ldots,n-k$, ont été définis par
\[
\lambda_{k,l} = (\underbrace{1,\ldots,1}_{k\ \mbox{termes}},\underbrace{0,\ldots,0}_{n-(k+l)\ \mbox{termes}},\underbrace{-1,\ldots,-1}_{l\ \mbox{termes}}).
\]


Le sous-groupe à un paramètre $\lambda_{1,n-1}=(1,-1,\ldots,-1)$ vérifie $\langle\lambda_{1,n-1},\bmin\rangle = 0$, donc $\got{p}^-_{<0} = 0$ pour $\lambda_{1,n-1}$ et les équations obtenues par les paires couvrantes associées à $\lambda_{1,n-1}$ sont les équations de $\Lambda + \CR(\got{R}_n^+)$. Ce sont les équations suivantes :
\[
\langle w\lambda_1,\xi-\Lambda\rangle \leqslant 0, \mbox{ pour tout } w\in W^{\lambda_{1,n-1}},
\]
ce qui donne
\begin{equation}
\label{eq:SO*(2n)_équations_cônededépart}
\xi_k - \Lambda_k \leqslant \sum_{i\neq k}\xi_i-\Lambda_i, \mbox{ pour tout } k=1,\ldots,n.
\end{equation}

Pour $n\geqslant 4$, un autre sous-groupe à un paramètre vérifie $\langle\lambda,\bmin\rangle = 0$, il s'agit de $\lambda_{0,1}:=(0,0,0,-1)$. Il apporte les équations
\begin{equation}
\label{eq:SO*(2n)_équations_cônededépart2}
\xi_k \geqslant \Lambda_k,\quad\text{pour tout $k=1,\ldots,n$}.
\end{equation}

Pour les sous-groupes à un paramètre $\lambda_{k,n-k}$, cela devient beaucoup plus compliqué. En effet, on voit facilement que, pour $k\in\{1,\ldots,n-1\}$, on a
\[
\langle\lambda_{k,n-k}, -\beta_{i,j}\rangle = \left\{\begin{array}{rl}
-2 & \mbox{si } i<j\leqslant k, \\
0 & \mbox{si } i\leqslant k < j, \\
2 & \mbox{si } k< i< j.
\end{array}\right.
\]
Ceci nous donne que $\dim_{\C}(\got{p}^-_{<0}) = \sharp\{(i,j); 1\leqslant i<j \leqslant k\} = \frac{k(k-1)}{2}$, pour le $\got{p}^-_{<0}$ associé à $\lambda_{k,n-k}$. De plus, le sous-groupe parabolique $P(\lambda_{k,n-k})$ de $GL_n(\C)$ sera le sous-groupe parabolique maximal associé à la racine simple $\alpha_{k,k+1}$, puisque $\langle\lambda_{k,n-k},\alpha_{k,k+1}\rangle = 2$, et $\langle\lambda_{k,n-k},\alpha_{j,j+1}\rangle = 0$ si $j\neq k$.

Nous sommes donc amenés à calculer les valeurs de $\Theta(\prod_{1\leqslant i<j\leqslant k}\beta_{i,j})$, pour $k=2,\ldots, n-1$. Nous pouvons utiliser ici la formule de Chevalley. Pour le cas le plus simple, $k=2$, nous avons simplement $\Theta(\beta_{1,2}) = \sigma_{s_2}^B$. Nous conjecturons, grâce à des calculs numériques en petites dimensions, que, pour $k$ quelconque dans $\{2,\ldots,n-1\}$, on a
\[
\Theta\left(\prod_{1\leqslant i<j\leqslant k}\beta_{i,j}\right) = \sigma^B_{s_2s_4s_3s_6s_5s_4\ldots s_{k-1}s_{2k-2}s_{2k-3}\ldots s_{k+1}s_k}, \mbox{ si } 2k-2 \leqslant n-1,
\]
c'est-à-dire $k\leqslant (n+1)/2$, et $\Theta(\prod_{1\leqslant i<j\leqslant k}\beta_{i,j}) = 0$ sinon.

Il reste aussi les $\lambda_{k,l}$ avec $k=1,\ldots,n-2$ et $l=1,\ldots,n-k-1$. On a pour ces éléments les valeurs
\[
\langle\lambda_{k,n-k}, -\beta_{i,j}\rangle = \left\{\begin{array}{rl}
-2 & \text{si $i<j\leqslant k$}, \\
-1 & \text{si $i\leqslant k <j\leqslant n-l$}, \\
0 & \text{si $i\leqslant k<n-l<j$ ou $k<i<j\leqslant n-l$}, \\
1 & \text{si $k<i\leqslant n-l<j$}, \\
2 & \text{si $n-l< i< j$}.
\end{array}\right.
\]
L'espace vectoriel complexe $\got{p}^-_{<0}$ est dans ce cas de dimension
\begin{align*}
\dim_{\C}(\got{p}^-_{<0}) & = \sharp\{(i,j); 1\leqslant i<j \leqslant k\}+\sharp\{(i,j); 1\leqslant i\leqslant k<j \leqslant n-l\} \\
& = \frac{k(k-1)}{2} + k(n-l-k).
\end{align*}
Il faudrait également connaître $\Theta(\prod_{1\leqslant i<j\leqslant k}\beta_{i,j})\cdot\Theta(\prod_{1\leqslant i'\leqslant k<j'\leqslant n-l}\beta_{i',j'})$, pour tout $k=1,\ldots,n-2$ et tout $l=1,\ldots,n-k-1$.

\bigskip

Ne pouvant pas trouver de formule générale pour tout $n\geqslant 3$ et tout $k=2,\ldots,n-1$, nous allons présenter les résultats pour les dimensions $n=3$ et $n=4$.

\paragraph*{Equations pour le groupe $SO^*(6)$}
Ici, $n=3$. Il ne reste alors qu'à étudier les équations fournies par le sous-groupe à un paramètre $\lambda_{2,1}$, qui apporte les équations $\langle w\lambda_{2,1},\xi\rangle\leqslant\langle w_0w'\lambda_{2,1},\Lambda\rangle$ pour tous $w,w'\in W^{\alpha_{2,1}}$ tels que la paire $(C(w,w',0),\lambda_{2,1})$ soit dans $\mathcal{P}_0(\got{p}^-)$. Une telle paire sera dans $\mathcal{P}_0(\got{p}^-)$ si et seulement si
\[
\sigma_{w_0w}^B\,.\,\sigma_{w_0w'}^B\,.\,\Theta(\beta_{1,2}) = \sigma_{w_0w_{\lambda}}^B,
\]
où $\Theta(\beta_{1,2}) = \sigma_{s_2}^B$, et
\[
\langle w\lambda_{2,1} + w'\lambda_{2,1},\rho\rangle - \sum_{k<0}k\dim_{\C}(\got{p}^-_{\lambda_{2,1},k}) = 0.
\]
Pour $n=3$, nous avons pour sous-groupe parabolique maximal
\[
P(\lambda_{2,1}) = P^{\alpha_{2,3}} = \left(\begin{array}{cc|c}
* & * & * \\
* & * & * \\\hline
0 & 0 & *
\end{array}\right),
\]
et $w_{\lambda} = s_1$. De plus, $w_0 = s_1s_2s_1 = s_2s_1s_2$. Ceci nous ramène à déterminer les $(u,v)\in (W^{\lambda_{2,1}})^2$ tels que $\sigma_{u}^B\,.\,\sigma_v^B\,.\,\sigma_{s_2}^B = \sigma_{s_1s_2}^B$. Cette dernière équation n'est réalisée que pour les couples $(u,v)\in\{(\id,s_2),(s_2,\id)\}$. Nous devons donc prendre $(w,w')\in\{(w_0,s_2s_1),(s_2s_1,w_0)\}$. 
La vérification de l'égalité
\[
\langle u\lambda_{2,1} + v\lambda_{2,1},\rho\rangle - \sum_{k<0}k\dim_{\C}(\got{p}^-_{\lambda_{2,1},k}) = 0.
\]
est aisée pour ces deux couples. Par conséquent, on a deux paires bien couvrantes $(C(w_0,s_2s_1,0),\lambda_{2,1})$ et $(C(s_2s_1,w_0,0),\lambda_{2,1})$, soit deux équations :
\[
\langle w_0\lambda_{2,1},\xi\rangle\leqslant\langle s_2\lambda_{2,1},\Lambda\rangle \quad \mbox{et}\quad  \langle s_2s_1\lambda_{2,1},\xi\rangle\leqslant\langle\lambda_{2,1},\Lambda\rangle.
\]
On remplace $\lambda_2$ par $(1,1,-1)$, $\xi$ par $(\xi_1,\xi_2,\xi_3)$, et $\Lambda$ par $(\Lambda_1,\Lambda_2,\Lambda_3)$, dans les équations ci-dessus. Rajoutées aux équations \eqref{eq:SO*(2n)_équations_cônededépart}, on obtient le théorème suivant.

\begin{theo}
Le polyèdre moment $\Delta_K(\Orb_{\Lambda})$, pour $G = SO^*(6)$, est le polyèdre de $\got{t}^*$ défini par les équations
\[
\left\{\begin{aligned}
-\xi_1+\xi_2+\xi_3 & \geqslant -\Lambda_1+\Lambda_2+\Lambda_3 \\
\xi_1-\xi_2+\xi_3 & \geqslant \Lambda_1-\Lambda_2+\Lambda_3 \\
\xi_1+\xi_2-\xi_3 & \geqslant \Lambda_1+\Lambda_2-\Lambda_3 \\
\xi_1-\xi_2-\xi_3 & \geqslant -\Lambda_1+\Lambda_2-\Lambda_3 \\
-\xi_1+\xi_2-\xi_3 & \geqslant -\Lambda_1-\Lambda_2+\Lambda_3 \\
\xi_1 \geqslant \xi_2 \geqslant \xi_3 &
\end{aligned}\right.
\]
\end{theo}

\begin{rema}
On retrouve bien les équations du polyèdre $\Delta_K(\Orb_{\Lambda})$ obtenues dans le Chapitre \ref{subsection:polyèdremoment_SO*(6)} en utilisant directement le problème de Horn-Klyachko, cf Théorème \ref{theo:SO*(2n)_équationsdupolyèdremoment}.
\end{rema}

\paragraph*{Equations pour le groupe $SO^*(8)$}
Les équations \eqref{eq:SO*(2n)_équations_cônededépart} (resp. équations \eqref{eq:SO*(2n)_équations_cônededépart2}) correspondent à $\lambda_{1,3}$ (resp. à $\lambda_{0,1}=(0,0,0,-1)$). Ici, il faut encore regarder les équations pour $\lambda_{2,2}$, $\lambda_{3,1}$ et $\lambda_{0,1}:=(0,0,0,-1)$, d'après le Théorème \ref{theo:ssgpe1param_SO*(2n)}. On peut vérifier que $\Theta(\beta_{1,2}\beta_{1,3}\beta_{2,3}) = 0$, comme cela a été conjecturé précédemment. De plus, on a également $\Theta(\beta_{1,2}\beta_{1,3}\beta_{1,4})=0$ par le calcul, en utilisant la Formule de Chevalley, ce qui élimine $\lambda_{1,0}$. Donc finalement, seules les équations de $\lambda_{2,2}$ manquent. Le sous-groupe parabolique $P(\lambda_{2,2})$ est le sous-groupe parabolique maximal de $GL_4(\C)$
\[
P(\lambda_{2,2}) = P^{\alpha_{2,3}} = \left(\begin{array}{cc|cc}
* & * & * & * \\
* & * & * & * \\\hline
0 & 0 & * & * \\
0 & 0 & * & *
\end{array}\right), 
\]
le groupe de Weyl $W_{\lambda_{2,2}} = W_{P^{\alpha_{2,3}}}$ est isomorphe au produit direct $S_2\times S_2$ de deux copies du groupe symétrique sur l'ensemble à deux éléments. Son plus long élément est $w_{\lambda_{2,2}} = s_1s_3 = s_3s_1$. L'ensemble $W/W_{\lambda_{2,2}}$ a $4!/4 = 6$ classes. Dans le tableau suivant, $w$ parcourt l'ensemble des représentants de longueur maximale de $W/W_{\lambda_{2,2}}$, c'est-à-dire $w$ parcourt $W^{\lambda_{2,2}}$.
\[
\begin{array}{|c|c|c|c|c|}
\hline
w & l(w) & w_0w & w\lambda_{2,2} & \langle w\lambda_{2,2},\rho\rangle \\ \hline
w_{\lambda_{2,2}} = s_1s_3 & 2 & s_2s_1s_3s_2 & \scriptstyle(1,1,-1,-1) & 4 \\
s_2s_1s_3 & 3 & s_1s_3s_2 & \scriptstyle (1,-1,1,-1) & 2 \\ 
s_1s_2s_1s_3 & 4 & s_1s_2 & \scriptstyle (-1,1,1,-1) & 0 \\ 
s_3s_2s_1s_3 & 4 & s_3s_2 & \scriptstyle (1,-1,-1,1) & 0 \\ 
s_3s_1s_2s_1s_3 & 5 & s_2 & \scriptstyle (-1,1,-1,1) & -2 \\ 
s_2s_1s_3s_2s_1s_3 & 6 & \id & \scriptstyle (-1,-1,1,1) & -4 \\ \hline
\end{array}
\]
Ci-dessus, $\rho$ désigne la demi-somme des racines positives de $U(4)$, il est donc égal ici à $\rho = \frac{1}{2}(3e_1^*+e_2^*-e_3^*-3e_4^*)$.

De plus, nous devons choisir les couples $(w,w')\in (W^{\lambda_{2,2}})^2$ qui vérifient l'équation $\sigma_{w_0w}^B\,.\,\sigma_{w_0w'}^B\,.\,\Theta(\beta_{1,2}) = \sigma_{w_0w_{\lambda_{2,2}}}^B$. Ils doivent donc vérifier l'équation de longueurs suivante
\[
l(w) + l(w') = l(w_0) + l(w_{\lambda_{2,2}}) + \dim(\got{p}^-_{<0}) = 6 + 2 + 1 = 9,
\]
d'après le Corollaire \ref{coro:info_longueurs_éléments_pour_pairebiencouvrante}. Les couples vérifiant cette équation de longueurs sont rangés dans le tableau qui suit.
\[
\begin{array}{|c|c|c|c|c|c|c|}
\hline
w & w' & l(w) & l(w') & w_0w & w_0w' & \langle w\lambda_2+w'\lambda_{2,2},\rho\rangle \\ \hline
s_2s_1s_3 & s_2s_1s_3s_2s_1s_3 & 3 & 6 & s_1s_3s_2 & \id & 2+-4=-2 \\
s_1s_2s_1s_3 & s_3s_1s_2s_1s_3 & 4 & 5 & s_1s_2 & s_2 & 0-2=-2 \\ 
s_3s_2s_1s_3 & s_3s_1s_2s_1s_3 & 4 & 5 & s_3s_2 & s_2 & 0-2=-2 \\ 
s_3s_1s_2s_1s_3 & s_1s_2s_1s_3 & 5 & 4 & s_2 & s_1s_2 & -2+0=-2  \\ 
s_3s_1s_2s_1s_3 & s_3s_2s_1s_3 & 5 & 4 & s_2 & s_3s_2 & -2+0=-2 \\ 
s_2s_1s_3s_2s_1s_3 & s_2s_1s_3 & 6 & 3 & \id & s_1s_3s_2 & -4+2=-2 \\ \hline
\end{array}
\]
Remarquant que $\got{p}^-_{<0} = \got{p}^-_{-\beta_{1,2}}$ est de dimension $1$ et que l'on a également $\langle\lambda_{2,2},-\beta_{1,2}\rangle=-2$, nous aurons, pour tous les couples $(w,w')$ du tableau ci-dessus,
\[
\langle w\lambda + w'\lambda,\rho\rangle = \sum_{k<0}k\dim_{\C}(\got{p}^-_{\lambda_{2,2},k}) = -2.
\]
Nous pouvons calculer plusieurs produits cup par la formule de Chevalley,
\[
\left\{\begin{array}{l}
\sigma_{s_2}\,.\,\sigma_{s_1s_2} = \sigma_{s_1s_3s_2}, \\
\sigma_{s_2}\,.\,\sigma_{s_3s_2} = \sigma_{s_1s_3s_2}, \\
\sigma_{s_2}\,.\,\sigma_{s_1s_3s_2} = \sigma_{s_2s_1s_3s_2}.
\end{array}\right.
\]
On en déduit que, pour tous les couples $(w,w')$ qui apparaissent dans le tableau ci-dessus, nous avons $\sigma_{w_0w}^B\,.\,\sigma_{w_0w'}^B = \sigma_{s_1s_3s_2}^B$ et, par conséquent,
\[
\sigma_{w_0w}^B\,.\,\sigma_{w_0w'}^B\,.\,\Theta(\beta_{1,2}) = \sigma_{s_2}^B\,.\,\sigma_{s_1s_3s_2}^B = \sigma_{s_2s_1s_3s_2}^B = \sigma_{w_0w_{\lambda_{2,2}}}^B.
\]
Finalement, tous les couples cités dans le tableau ci-dessus donnent une paire bien couvrante et une équation pour le polyèdre moment $\Delta_K(\Orb_{\Lambda})$. D'après le Théorème \ref{theo:équations_DeltaK_G0Lambda}, une telle équation est de la forme $\langle w\lambda_{2,2},\xi\rangle\leqslant\langle w_0w'\lambda_{2,2},\Lambda\rangle$. Le tableau qui suit donne les informations nécessaires pour chaque couple $(w,w')$ pour obtenir les équations par rapport aux coordonnées d'un vecteur $\xi=(\xi_1,\xi_2,\xi_3,\xi_4)$ de $\got{t}^*$ dans la base $(e_1^*,e_2^*,e_3^*,e_4^*)$.
\[
\begin{array}{|c|c|c|c|c|c|c|}
\hline
w & w' & w\lambda_2 & w_0w' & w_0w'\lambda_2 \\ \hline
s_2s_1s_3 & s_2s_1s_3s_2s_1s_3 & (1,-1,1,-1) & \id & (1,1,-1,-1) \\
s_1s_2s_1s_3 & s_3s_1s_2s_1s_3 & (-1,1,1,-1) & s_2 & (1,-1,1,-1) \\
s_3s_2s_1s_3 & s_3s_1s_2s_1s_3 & (1,-1,-1,1) & s_2 & (1,-1,1,-1) \\
s_3s_1s_2s_1s_3 & s_1s_2s_1s_3 & (-1,1,-1,1) & s_1s_2 & (-1,1,1,-1) \\
s_3s_1s_2s_1s_3 & s_3s_2s_1s_3 & (-1,1,-1,1) & s_3s_2 & (1,-1,-1,1) \\
s_2s_1s_3s_2s_1s_3 & s_2s_1s_3 & (-1,-1,1,1) & s_1s_3s_2 & (-1,1,-1,1) \\ \hline
\end{array}
\]
Nous obtenons maintenant les équations suivantes,
\begin{equation}
\label{eq:SO*(2n)_équations_pourlambda2,2}
\left\{\begin{aligned}
\xi_1-\xi_2+\xi_3-\xi_4 & \leqslant \Lambda_1+\Lambda_2-\Lambda_3-\Lambda_4 \\
-\xi_1+\xi_2+\xi_3-\xi_4 & \leqslant \Lambda_1-\Lambda_2+\Lambda_3-\Lambda_4 \\
\xi_1-\xi_2-\xi_3+\xi_4 & \leqslant \Lambda_1-\Lambda_2+\Lambda_3-\Lambda_4 \\
-\xi_1+\xi_2-\xi_3+\xi_4 & \leqslant -\Lambda_1+\Lambda_2+\Lambda_3-\Lambda_4 \\
-\xi_1+\xi_2-\xi_3+\xi_4 & \leqslant \Lambda_1-\Lambda_2-\Lambda_3+\Lambda_4 \\
-\xi_1-\xi_2+\xi_3+\xi_4 & \leqslant -\Lambda_1+\Lambda_2-\Lambda_3+\Lambda_4 
\end{aligned}\right.
\end{equation}

D'après \eqref{eq:SO*(2n)_équations_cônededépart} et \eqref{eq:SO*(2n)_équations_cônededépart2}, les équations du cône affine $\Lambda+\CR(\got{R}_n^+)$ sont les équations
\begin{equation}
\label{eq:SO*(2n)_équations_du_cone_des_racines}
\left\{\begin{aligned}
-\xi_1+\xi_2+\xi_3+\xi_4 & \geqslant -\Lambda_1+\Lambda_2+\Lambda_3+\Lambda_4, \\
\xi_1-\xi_2+\xi_3+\xi_4 & \geqslant \Lambda_1-\Lambda_2+\Lambda_3+\Lambda_4, \\
\xi_1+\xi_2-\xi_3+\xi_4 & \geqslant \Lambda_1+\Lambda_2-\Lambda_3+\Lambda_4, \\
\xi_1+\xi_2+\xi_3-\xi_4 & \geqslant \Lambda_1+\Lambda_2+\Lambda_3-\Lambda_4, \\
\xi_1\geqslant \Lambda_1, \ \xi_2\geqslant\Lambda_2, &\  \xi_3\geqslant\Lambda_3,\ \xi_4\geqslant\Lambda_4,
\end{aligned}\right.
\end{equation}

Nous pouvons maintenant rassembler toutes les informations obtenues dans l'énoncé du théorème suivant.

\begin{theo}
Le polyèdre moment $\Delta_K(\Orb_{\Lambda})$, pour $G = SO^*(8)$, est le polyèdre de $\got{t}^*$ défini par les huit équations 
de \eqref{eq:SO*(2n)_équations_du_cone_des_racines} déterminant le cône $\Lambda+\CR(\got{R}_n^+)$, les six équations de \eqref{eq:SO*(2n)_équations_pourlambda2,2} 
et les inégalités $\xi_1\geqslant \xi_2\geqslant\xi_3\geqslant\xi_4$ provenant du choix de la chambre de Weyl.
\end{theo}

\begin{rema}
Parmi les équations obtenues en \eqref{eq:SO*(2n)_équations_pourlambda2,2}, deux sont semblables, à apparition d'un signe \og{}$-$\fg{} près. En effet, nous avons les deux équations
\begin{align*}
-\xi_1+\xi_2-\xi_3+\xi_4 & \leqslant -\Lambda_1+\Lambda_2+\Lambda_3-\Lambda_4 \\
-\xi_1+\xi_2-\xi_3+\xi_4 & \leqslant \Lambda_1-\Lambda_2-\Lambda_3+\Lambda_4 = -(-\Lambda_1+\Lambda_2+\Lambda_3-\Lambda_4)
\end{align*}
En particulier, pour tout $\Lambda$ dominant, l'une de ces deux équations sera toujours redondante par rapport à l'autre, puisque seule l'équation avec le terme de droite négatif apportera une réelle information. Ceci nous indique que l'ensemble des équations obtenues dans le Théorème \ref{theo:équations_DeltaK_G0Lambda} n'est pas forcément minimal, comme cela a pu être le cas pour les groupes $Sp(2n,\R)$ et $SU(n,1)$.
\end{rema}

\begin{rema}
Cet exemple apporte une seconde réponse à une question qui est apparue après le calcul des exemples précédents. Pour les groupes $Sp(2n,\R)$ et $SU(n,1)$, seules les directions des faces du cône $\CR(\got{R}_n^+)$ apparaissaient au final, que ce soit comme équation de face du cône $\Lambda+\CR(\got{R}_n^+)$ ou comme équation de face de réflexion sur un mur de la chambre de Weyl. Mais ici ce n'est plus le cas.

Regardons un exemple avec $\Lambda$ régulier. On choisit $\Lambda = (3,2,1,0)$. Soit $\xi$ un élément central, c'est-à-dire $\xi_1 = \xi_2=\xi_3=\xi_4=x$. Alors $\xi$ est dans $\Lambda+\CR(\got{R}_n^+)$ si et seulement si $x\geqslant 3$. De plus, par hypothèse sur $\xi$, on a $\xi_{i_1}+\xi_{i_2}-\xi_{i_3}-\xi_{i_4} = 0$, quel que soit l'ordre des indices choisis dans $\{1,2,3,4\}$. La dernière équation vaut alors
\[
0\leqslant -\Lambda_1+\Lambda_2-\Lambda_3+\Lambda_4 = -3+2-1=-2,
\]
ce qui est évidemment faux, donc les onze équations du cône $(\Lambda+\CR(\got{R}_n^+))\cap\got{t}_+^*$ ne suffisent pas à décrire $\Delta_K(\Orb_{\Lambda})$. Cela signifie que des équations de \eqref{eq:SO*(2n)_équations_pourlambda2,2} apparaissent nécessairement pour décrire le polyèdre $\Delta_K(\Orb_{\Lambda})$.
\end{rema}

\subsection{Projections d'orbites coadjointes holomorphes de $SU(2,2)$}

Nous terminons ce chapitre en donnant la description du polyèdre moment $\Delta_K(\Orb_{\Lambda})$ pour le cas du groupe $SU(2,2)$.

\subsubsection{Quelques informations pour l'étude du cas général $SU(p,q)$, $p\geqslant q\geqslant 2$}
\label{subsubsection:SU(p,q)_sg1pdominantsgotp-admissibles}

D'après le Théorème \ref{theo:sg1pdominantsadmissibles_de_SU(p,q)}, les sous-groupes à un paramètre dominants indivisibles $\got{p}^-$-admissibles de $\TC$ pour $G = SU(p,q)$ sont les sous-groupes à un paramètre $\lambda_{k,l}$ définis par
\[
\lambda_{k,l} = (\underbrace{a,\ldots,a}_{k\ \mbox{termes}},\underbrace{b,\ldots,b}_{p-k\ \mbox{termes}};\underbrace{a,\ldots,a}_{l\ \mbox{termes}},\underbrace{b,\ldots,b}_{q-l\ \mbox{termes}}),
\]
pour tous $k=1,\ldots,p-1$ et $l=1,\ldots,q-1$, avec $a=\frac{p+q-k-l}{\mathrm{pgcd}(p+q-k-l,k+l)}$ et $b=\frac{-(k+l)}{\mathrm{pgcd}(p+q-k-l,k+l)}$, et les quatre sous-groupes à un paramètre suivants,
\begin{align*}
\lambda_{p,q-1} & := (1,\ldots,1;1,\ldots,1,1-p-q), \\
\lambda_{p-1,q} & := (1,\ldots,1,1-p-q;1,\ldots,1), \\
\lambda_{0,1} & := (-1,\ldots,-1;p+q-1,-1,\ldots,-1), \\
\lambda_{1,0} & := (p+q-1,-1,\ldots,-1;-1,\ldots,-1).
\end{align*}

Les deux sous-groupes à un paramètre $\lambda_{p-1,q}$ et $\lambda_{0,1}$ donnent les équations du cône $\Lambda+\CR(\got{R}_n^+)$, puisque $\langle\lambda_{p-1,q},\bmin\rangle = \langle\lambda_{0,1},\bmin\rangle = 0$. Les équations sont obtenues facilement,
\[
\sum_{i=1}^{p+q}\xi_i - (p+q)\xi_k \leqslant \sum_{i=1}^{p+q}\Lambda_i - (p+q)\Lambda_k, \mbox{ pour tout } k=1,\ldots p,
\]
et
\[
-\sum_{i=1}^{p+q}\xi_i + (p+q)\xi_{p+l} \leqslant -\sum_{i=1}^{p+q}\Lambda_i + (p+q)\Lambda_{p+l}, \mbox{ pour tout } l=1,\ldots q,
\]
si on note $\xi=(\xi_1,\ldots,\xi_{p+q})$ et $\Lambda=(\Lambda_1,\ldots,\Lambda_{p+q})$ les deux éléments de $\got{t}^*$, les composantes étant écrites dans la base $(e_1^*,\ldots,e_{p+q}^*)$ de $(\got{t}\oplus\got{z}(\got{u}(p,q)))^*$. On peut noter que le premier ensemble d'inégalités est exactement le même que celles apparaissant pour $SU(p,1)$, ce sont les équations $(A_k)$ du paragraphe \ref{subsection:exemplespolyèdresmoment_SU(n,1)}.

\bigskip

%
%

Les entiers $\langle\lambda,-\beta_{i,p+j}\rangle$ peuvent se déterminer facilement, pour tout $\lambda$ sous-groupe à un paramètre dominant indivisible $\got{p}^-$-admissible de $\TC$ et $(i,j)\in\{1,\ldots,p\}\times\{1,\ldots,q\}$. En effet,
lorsque l'on considère $\lambda_{k,l}$, avec $k=1,\ldots,p-1$ et $l=1,\ldots,q-1$, on a
\[
\langle\lambda_{k,l},-\beta_{i,p+j}\rangle = \left\{\begin{array}{ll}
b-a<0 & \text{si $i\leqslant k$ et $j>l$}, \\
a-b>0 & \text{si $i>k$ et $j\leqslant l$}, \\
0 & \text{sinon}.
\end{array}\right.
\]
Puis, pour les quatre sous-groupes à un paramètre restants,
\[
\langle\lambda_{p,q-1},-\beta_{i,p+j}\rangle = \left\{\begin{array}{ll}
-p-q<0 & \ \text{si $j=q$}, \\
0 & \ \text{sinon},
\end{array}\right.
\]
\begin{equation}
\label{eq:valeurs_lambda1,0_sur_-betaip+j}
\langle\lambda_{1,0},-\beta_{i,p+j}\rangle = \left\{\begin{array}{ll}
-p-q<0 & \ \text{si $i=1$}, \\
0 & \ \text{sinon},
\end{array}\right.
\end{equation}
\[
\langle\lambda_{0,1},-\beta_{i,p+j}\rangle = \left\{\begin{array}{ll}
p+q>0 & \ \text{si $j=1$}, \\
0 & \ \text{sinon},
\end{array}\right.
\]
\[
\langle\lambda_{p-1,q},-\beta_{i,p+j}\rangle = \left\{\begin{array}{ll}
p+q>0 & \ \text{si $i=p$}, \\
0 & \ \text{sinon}.
\end{array}\right.
\]
On connaît déjà les équations apportées par les deux derniers, $\lambda_{0,1}$ et $\lambda_{p-1,q}$. Pour tous les autres, on voit que $\langle\lambda,\bmin\rangle<0$ et $\got{p}^-_{<0} = (\got{p}^-_{<0})^{\langle\lambda,\bmin\rangle}$. Les sous-groupes paraboliques $P(\lambda_{p,q-1})$ et $P(\lambda_{1,0})$ sont clairement maximaux (ils sont de la forme $P_{\alpha_{i,i+1}}$ pour une certaines racine simple $\alpha_{i,i+1}$ de $S(U(p)\times U(q))$.

\bigskip

L'étude des paires bien couvrantes demande de calculer certaines expressions dans l'anneau de cohomologie $\coh{S(GL_p(\C)\times GL_q(\C))/B_{p,q}}$, où $B_{p,q}$ est le sous-groupe de Borel des matrices triangulaires supérieures de $S(GL_p(\C)\times GL_q(\C))$. Puisque la variété des drapeaux $S(GL_p(\C)\times GL_q(\C))/B_{p,q}$ s'identifie à $SL_p(\C)/B_p\times SL_q(\C)/B_q$, où $B_p$ (resp. $B_q$) est le sous-groupe de Borel des matrices triangulaires supérieures de $SL_p(\C)$ (resp. $SL_q(\C)$), nous travaillerons dans la suite sur l'anneau de cohomologie
\[
\coh{SL_p(\C)/B_p\times SL_q(\C)/B_q}\cong \coh{SL_p(\C)/B_p}\otimes\coh{SL_q(\C)/B_q}.
\]
Evidemment, pour $(w,w')\in W_p\times W_q$, la classe de Schubert $\sigma_{(w,w')}^{B_{p,q}}$ s'identifiera au produit tensoriel $\sigma_w^{B_p}\otimes\sigma_{w'}^{B_q}$, et la famille $(\sigma_w^{B_p}\otimes\sigma_{w'}^{B_q})_{(w,w')\in W_p\times W_q}$ formera une base du $\Z$-module $\coh{SL_p(\C)/B_p\times SL_q(\C)/B_q}$. Ici, $W_n$ désignera le groupe de Weyl de $SL_n(\C)$ pour le tore maximale des matrices diagonales. C'est aussi celui de $GL_n(\C)$ pour le tore maximal des matrices diagonales de $GL_n(\C)$, cette fois-ci. Ensuite, pour deux couples $(w_1,w_1'),(w_2,w_2')$ de $W_p\times W_q$, le produit cup vérifiera la propriété suivante,
\[
(\sigma_{w_1}^{B_p}\otimes\sigma_{w_1'}^{B_q})\,.\,(\sigma_{w_2}^{B_p}\otimes\sigma_{w_2'}^{B_q}) = (\sigma_{w_1}^{B_p}\,.\,\sigma_{w_2}^{B_p})\otimes(\sigma_{w_1'}^{B_q}\,.\,\sigma_{w_2'}^{B_q}).
\]

En réalisant le produit cup de $\Theta(\beta_{i,p+j})$ avec $\sigma_{(\id,\id)}^{B_{p,q}}$, la Formule de Chevalley (Théorème \ref{theo:Chevalley} (\emph{iii})) donne l'égalité
\begin{eqnarray}
\Theta(\beta_{i,p+j}) & = & \sum_{k=1}^{p-1}\beta_{i,j}(e_k-e_{k+1})\sigma_{s_k}^{B_p}\otimes\sigma_{\id}^{B_q} + \sum_{l=1}^{q-1}\beta_{i,p+j}(e_{p+l}-e_{p+l+1})\sigma_{\id}^{B_p}\otimes\sigma_{s_l}^{B_q} \nonumber \\
& = & -\sigma_{s_{i-1}}^{B_p}\otimes\sigma_{\id}^{B_q}+\sigma_{s_{i}}^{B_p}\otimes\sigma_{\id}^{B_q} + \sigma_{\id}^{B_p}\otimes\sigma_{s_{j-1}}^{B_q} - \sigma_{\id}^{B_p}\otimes\sigma_{s_j}^{B_q} \nonumber \\
& = & \bigl(-\sigma_{s_{i-1}}^{B_p}+\sigma_{s_{i}}^{B_p}\bigr)\otimes\sigma_{\id}^{B_q} + \sigma_{\id}^{B_p}\otimes\bigl(\sigma_{s_{j-1}}^{B_q} - \sigma_{s_j}^{B_q}\bigr), \label{eq:Theta(beta_i,p+j)} 
\end{eqnarray}
où on fixe les conventions suivantes:
\[
\sigma_{s_{0}}^{B_p}\otimes\sigma_{\id}^{B_q} = 0, \qquad \sigma_{s_{p}}^{B_p}\otimes\sigma_{\id}^{B_q} = 0, \qquad \sigma_{\id}^{B_p}\otimes\sigma_{s_{0}}^{B_q} = 0 \qquad\text{et}\qquad \sigma_{\id}^{B_p}\otimes\sigma_{s_{q}}^{B_q} = 0.
\]
On aura donc, par exemple, $\Theta(\beta_{1,p+1}) = \sigma_{s_1}^{B_p}\otimes\sigma_{\id}^{B_q} - \sigma_{\id}^{B_p}\otimes\sigma_{s_1}^{B_q}$ et $\Theta(\beta_{1,p+q}) = \sigma_{s_1}^{B_p}\otimes\sigma_{\id}^{B_q} + \sigma_{\id}^{B_p}\otimes\sigma_{s_{q-1}}^{B_q}$.

\subsubsection{Calcul des équations du polyèdre moment $\Delta_K(SU(2,2)\cdot\Lambda)$}

Pour le groupe $SU(2,2)$, les sous-groupes à un paramètre dominants indivisibles $\got{p}^-$-admissibles sont $\lambda_{1,2}$, $\lambda_{0,1}$, $\lambda_{2,1}$, $\lambda_{1,0}$ et $\lambda_{1,1}$.

On sait déjà que les deux premiers donnent les équations des faces du cône $\Lambda+\CR(\got{R}_n^+)$. Les deux sous-groupes à un paramètre suivants ne vont faire intervenir aucune équation. En effet, $\lambda_{2,1}$ (resp. $\lambda_{1,0}$) fait apparaître le terme $\Theta(\beta_{1,4}\beta_{2,4})$ (resp. $\Theta(\beta_{1,3}\beta_{1,4})$) qui est nul. Ceci se calcule facilement en utilisant \eqref{eq:Theta(beta_i,p+j)} et le fait que $(\sigma_{s_1}^{B_2})^2 = 0$.

Il ne reste plus qu'à chercher les équations associées au sous-groupe à un paramètre $\lambda_{1,1}$. Le sous-groupe parabolique associé dans $SL_2(\C)\times SL_2(\C)$ est $P(\lambda_{1,1}) = B_2\times B_2$, le produit de deux copies du Borel des matrices triangulaires supérieures de $SL_2(\C)$.

%
Par le calcul, on montre que les quatre couples
\[
\left\{\bigl((s_1,\id),(s_1,s_1)\bigr),\bigl((\id,s_1),(s_1,s_1)\bigr),\bigl((s_1,s_1),(s_1,\id)\bigr),\bigl((s_1,s_1),(\id,s_1)\bigr)\right\}
\]
donnent toutes les paires bien couvrantes associées à $\lambda_{1,1}$. On obtient les équations suivantes, données dans l'ordre d'apparition des couples dans l'ensemble ci-dessus,
\[
\left\{\begin{aligned}
\bigl\langle(s_1,\id)\lambda_{1,1},\xi-(s_1,\id)\Lambda\bigr\rangle & \leqslant 0, \\
\bigl\langle(\id,s_1)\lambda_{1,1},\xi-(\id,s_1)\Lambda\bigr\rangle & \leqslant 0, \\
\bigl\langle(s_1,s_1)\lambda_{1,1},\xi-(s_1,\id)\Lambda\bigr\rangle & \leqslant 0, \\
\bigl\langle(s_1,s_1)\lambda_{1,1},\xi-(\id,s_1)\Lambda\bigr\rangle & \leqslant 0.
\end{aligned}\right.
\]
Ceci est équivalent au sytème d'équations
\[
\left\{\begin{aligned}
|\xi_1 - \xi_2 - \xi_3 + \xi_4| & \leqslant \Lambda_1 - \Lambda_2 + \Lambda_3 - \Lambda_4, \\
-\xi_1 + \xi_2 - \xi_3 + \xi_4 & \leqslant -|\Lambda_1 - \Lambda_2 - \Lambda_3 + \Lambda_4|.
\end{aligned}\right.
\]

Les équations apportées par $\lambda_{p,1}$ et $\lambda_{1,0}$ sont quant à elles
\[
\left\{\begin{aligned}
-3\xi_1 + \xi_2 + \xi_3 + \xi_4 & \leqslant -3\Lambda_1 + \Lambda_2 + \Lambda_3 + \Lambda_4, \\
\xi_1 - 3\xi_2 + \xi_3 + \xi_4 & \leqslant \Lambda_1 - 3\Lambda_2 + \Lambda_3 + \Lambda_4, \\
-\xi_1 - \xi_2 + 3\xi_3 - \xi_4 & \leqslant -\Lambda_1 - \Lambda_2 + 3\Lambda_3 - \Lambda_4, \\
-\xi_1 - \xi_2 - \xi_3 + 3\xi_4 & \leqslant -\Lambda_1 - \Lambda_2 - \Lambda_3 + 3\Lambda_4, \\
-\xi_1 + \xi_2 + \xi_3 - \xi_4 & \leqslant \Lambda_1 - \Lambda_2 + \Lambda_3 - \Lambda_4.
\end{aligned}\right.
\]
Or, un élément $\xi=(\xi_1,\xi_2,\xi_3,\xi_4)$ de $\got{t}$ doit vérifier la condition de trace $\sum\xi_i=0$. Donc les quatre équations ci-dessus sont équivalentes aux quatre équations
\[
\xi_1\geqslant\Lambda_1,\quad \xi_2\geqslant\Lambda_2, \quad \xi_3\leqslant\Lambda_3, \quad \xi_4\leqslant\Lambda_4.
\]

Nous pouvons maintenant résumer le cas de $SU(2,2)$ dans le théorème qui suit.

\begin{theo}
\label{theo:SU(2,2)_équations_polyèdremoment}
Soit $\Lambda\in\wedge_{\Q}^*\cap\Chol$. Le polyèdre moment $\Delta_{S(U(2)\times U(2))}(SU(2,2)\cdot\Lambda)$ est le polyèdre de $\got{t}^*$ défini par les équations
\[
\left\{\begin{aligned}
\xi_1\geqslant\Lambda_1,\ \xi_2\geqslant\Lambda_2, & \quad \Lambda_3\geqslant\xi_3, \ \Lambda_4\geqslant\xi_4 \\
|\xi_1 - \xi_2 - \xi_3 + \xi_4| & \leqslant \Lambda_1 - \Lambda_2 + \Lambda_3 - \Lambda_4, \\
-\xi_1 + \xi_2 - \xi_3 + \xi_4 & \leqslant -|\Lambda_1 - \Lambda_2 - \Lambda_3 + \Lambda_4|, \\
\xi_1\geqslant\xi_2,& \quad \xi_3\geqslant \xi_4.
\end{aligned}\right.
\]
\end{theo}

Dans cette configuration, la chambre holomorphe est
\[
\Chol = \{\xi=(\xi_1,\xi_2,\xi_3,\xi_4)\in\got{t}^*;\ \xi_1\geqslant\xi_2\geqslant\xi_3\geqslant\xi_4\}.
\]
Lorsque $\Lambda$ appartient à $\wedge_{\Q}^*\cap\Chol$, on vérifie aisément grâce à la première ligne des équations de l'énoncé du Théorème \ref{theo:SU(2,2)_équations_polyèdremoment}, que $\Delta_{S(U(2)\times U(2))}(SU(2,2)\cdot\Lambda)$ est bien contenu dans $\Chol$, comme l'indique la Proposition \ref{prop:polyèdremoment_contenudans_Chol}.

\begin{rema}
Nous pouvons également calculer les équations du cône polyédral $\Delta_{S(U(2)\times U(2))}(\got{p}^-)$. Il suffit en effet de prendre $\Lambda=0$ dans les équations de l'énoncé du Théorème \ref{theo:SU(2,2)_équations_polyèdremoment}. Ce procédé nous donne
\[
\left\{\begin{aligned}
\xi_1\geqslant 0,\ \xi_2\geqslant 0, \ \xi_3&\leqslant 0, \ \xi_4\leqslant 0, \\
\xi_1-\xi_2-\xi_3+\xi_4 & = 0, \\
-\xi_1+\xi_2-\xi_3+\xi_4 & \leqslant 0.
\end{aligned}\right.
\]
Or, on doit avoir $\sum\xi_i=0$, donc $\xi_1+\xi_4=\xi_2+\xi_3=0$. On en déduit que le cône polyédral $\Delta_{S(U(2)\times U(2))}(\got{p}^-)$ est
\[
\Delta_{S(U(2)\times U(2))}(\got{p}^-)=\{(x_1,x_2,-x_2,-x_1); x_1\geqslant x_2\geqslant 0\},
\]
comme indiqué dans \cite[Section 5]{Paradan}.
\end{rema}

\appendix


\chapter{Combinatoire dans le groupe symétrique}
\label{chap:combinatoire_gpdeWeyl_de_GL_r(C)}

Nous rassemblons dans cet appendice une liste de résultats nécessaires pour la démonstration du Théorème \ref{theo:cns_pairebiencouvrante}. La première partie regroupe des lemmes techniques sur la combinatoire de certains éléments du groupe de Weyl de $GL_{r}(\C)$, ce groupe de Weyl n'étant autre que le groupe symétrique $S_r$. La seconde partie est consacrée à la formule de Chevalley et présente le calcul des images par $(f_{\lambda}^B)^*$ des classes de cohomologie $\check{w}_k$. La troisième donne des produits cup généralisant ceux utilisés pour la preuve du Théorème \ref{theo:cns_pairebiencouvrante}.

\section{Quelques propriétés du groupe de Weyl de $GL_{r}(\C)$}
\label{section:propriétés_groupeWeyl_GLr}

Cette section contient plusieurs propriétés sur certains éléments du groupe de Weyl $W_r$ de $GL_r(\C)$ qui nous seront utiles dans la suite de cet appendice. Le groupe de Weyl $W_r$ n'est autre que le groupe symétrique $S_r$. Toutes les propriétés qui seront montrées dans ce paragraphe, seront aussi valables pour le groupe $SL_r(\C)$, car il a le même groupe de Weyl que $GL_r(\C)$.

Nous avons fixé dans les chapitres précédents un tore maximal de $GL_{r}(\C)$, l'ensemble $T_r$ de ses matrices diagonales. L'algèbre de Lie de $GL_{r}(\C)$ est l'ensemble $\glr$ de toutes les matrices de taille $r\times r$. Celle de $T_r$ est l'algèbre de Lie $\got{t}_r$ des matrices diagonales de $\glr$. Les racines de $\glr$ sont les formes linéaires
\[
\alpha_{i,j}\left(\begin{array}{ccc}
h_1 & & \\
& \ddots & \\
& & h_{r}
\end{array}\right) = h_i - h_j,
\]
où $1\leq i,j\leq r$, avec $i\neq j$. L'ensemble des racines de $\glr$ sera noté $\got{R}_r$. Les racines positives sont les racines $\alpha_{i,j}$ avec $i<j$. Le Borel associé à ce système de racines positives est $B_r$ l'ensemble des matrices triangulaires supérieures de $GL_r(\C)$. Les racines simples sont alors les racines $\alpha_{i,i+1}$, où $i\in\{1,\ldots,r-1\}$. L'élément de $W_r$ associé à $\alpha_{i,i+1}$ est $s_i$, la matrice de permutation définie dans le paragraphe \ref{section:notations_et_paramétrage}. Les $s_i$ engendrent $W_r$. Pour $w\in W_r$, on définit la longueur de $w$ comme le nombre de racines positives $\alpha$ telles que $w\alpha$ soit négative. C'est aussi le plus petit entier $k$ tel que $w$ s'écrive comme composée de $k$ permutations simples. Plus généralement, nous noterons $s_{\alpha_{i,j}}$ l'élément de $W_r$ associé à la racine $\alpha_{i,j}$.

Si on parle en termes de groupe symétrique, l'élément $s_{\alpha_{i,j}}$ est la transposition $t_{i,j}$ qui échange les entiers $i$ et $j$, et l'élément $s_i=t_{i,i}$ est la transposition simple qui échange $i$ avec $i+1$. La longueur d'un élément $w\in S_r$ correspond dans ce cadre au nombre d'inversions de $w$, c'est-à-dire, c'est le cardinal de l'ensemble des inversions
\[
I(w) = \{(i,j); \ 1\leq i<j\leq n \text{ et } w(i)>w(j)\}
\]
de $w$. Nous utiliserons à plusieurs reprises la propriété suivante de la longueur.

\begin{lemm}
\label{lemm:l(ws_i)=l(w)+-1}
Soit $w\in S_r$ et $1\leq i\leq r-1$. Alors on a
\[
l(ws_i) = \left\{\begin{array}{lll}
l(w) + 1 & \quad & \text{si $w(i)<w(i+1)$} \\
l(w) - 1 & \quad & \text{si $w(i)>w(i+1)$.}
\end{array}\right.
\]
\end{lemm}

\begin{proof}
Il suffit de voir que l'ensemble $\{(k,l); \ 1\leq k<l\leq n\}\setminus\{(i,i+1)\}$ est stable par $s_i$. Par conséquent, nous avons deux cas possibles. Soit $w(i)<w(i+1)$, et donc $(i,i+1)\notin I(w)$. Dans ce cas, on aura $ws_i(i) = w(i+1)>w(i)=ws_i(i+1)$. D'où $I(ws_i) = I(w)\cup\{(i,i+1)\}$ et $l(ws_i) = l(w)+1$. Sinon $w(i)>w(i+1)$ et $(i,i+1)\in I(w)$, ce qui donne $I(ws_i) = I(w)\setminus\{(i,i+1)\}$ et $l(ws_i) = l(w)-1$. Le second cas est donné l'opposé, d'où le résultat.
\end{proof}

\begin{rema}
Ce résultat est valable de manière plus générale pour un groupe de Weyl quelconque lorsqu'on remplace $s_i$ par la symétrie associée à une racine simple de l'algèbre de Lie, cf \cite[Lemme 2.71]{knapp}.
\end{rema}

%

La démonstration du Théorème \ref{theo:cns_pairebiencouvrante} nécessite l'utilisation de la Formule de Chevalley aux classes de cohomologie $\sigma_w^{B_r}$ associées à des éléments spécifiques de $w$ du groupe de Weyl $W_r$ de $GL_r(\C)$. Ces éléments sont définis de la manière suivante:
\[
\check{w}_k = s_{r-1}s_{r-2}\ldots s_{k+1}s_k,
\]
\index{$\check{w}_k$}pour $k\in\{1,\ldots,r-1\}$ et $\check{w}_{r} = \id$. Sont intervenus également dans la preuve du Théorème \ref{theo:cns_pairebiencouvrante}, mais dans une moindre mesure, les éléments $\hat{w}_k$, définis par $\hat{w}_1 = \id$ et $\hat{w}_k = s_1\ldots s_{k-1}$\index{$\hat{w}_k$} si $2\leq k\leq r$.

\begin{lemm}
\label{lemm:longueur_wcheck_what}
Pour tout $k=1,\ldots,r$, on a $l(\check{w}_k) = r-k$.
\end{lemm}

\begin{proof}
Le cas de $\check{w}_{r} = \id$ est évident. Supposons que $k$ appartienne à $\{1,\ldots,r-1\}$. L'élément $\check{w}_k$ est le cycle $(r\ r-1\ \cdots\ k+1\ k)$. Il est bien connu qu'un tel cycle est de longueur $r-k$.
\end{proof}

%
%

Nous avons un résultat similaire avec les $\hat{w}_k$.

\begin{lemm}
\label{lemm:longueur_what}
Pour tout $k=1,\ldots,r$, on a $l(\hat{w}_k) = k-1$.
\end{lemm}

\begin{proof}
C'est évident pour $\hat{w}_1 = \id$. Pour $k\geq 2$, la permutation $\hat{w}_k$ est le cycle $(1\ 2\ \cdots \ k-1 \ k)$ de $S_r$. Sa longueur est donc $k-1$.
\end{proof}

Dans la section \ref{section:notations_et_paramétrage}, nous avons introduit le sous-groupe parabolique maximal $\hat{Q}$ de $GL_r(\C)$ suivant:
\[
\hat{Q} = \left(\begin{array}{c|ccc}
* & * & \ldots & * \\\hline
0 & * & \ldots & * \\
\vdots & \vdots & \ddots & \vdots \\
0 & * & \ldots & *
\end{array}\right).
\]
On notera $W_{\hat{Q}}$ le groupe de Weyl du sous-groupe de Levi de $\hat{Q}$ associé au tore maximal $T_r$, et $w_{\hat{Q}}$ le plus long élément de $W_{\hat{Q}}$.

\begin{lemm}
\label{lemm:éléments_les_plus_courts_de_W_mod_WQ}
L'ensemble $W_{\hat{Q}}\backslash W_r$ possède $r$ classes, et les $\hat{w}_k$, pour $k$ parcourant l'ensemble $\{1,\ldots,r\}$, sont les éléments les plus courts de chaque classe. Plus exactement, pour tout $w\in W_{\hat{Q}}$, on a $l(w\hat{w}_k) = l(w) + k-1$.
\end{lemm}

\begin{proof}
Le groupe de Weyl $W_r$ est le groupe de Weyl de $GL_{r}(\C)$, il est isomorphe au groupe symétrique $S_{r}$. Il a donc $r!$ éléments. Le groupe $W_{\hat{Q}}$ est le groupe de Weyl du Levi de $\hat{Q}$, il est donc isomorphe à $S_{r-1}$ et son cardinal est $(r-1)!$. L'ensemble $W_r/W_{\hat{Q}}$ a par conséquent $r$ éléments.

Il est clair que $\hat{w}_1 = \id$ est le plus court élément de sa classe. Montrons par récurrence sur $k$ que, pour tout $k\in\{1,\ldots,r-1\}$, on a $l(ws_1\ldots s_k) = l(w) + k$ pour tout $w\in W_{\hat{Q}}$.
%

Posons maintenant $w\in W_{\hat{Q}}$. On voit clairement que $W_{\hat{Q}}$ est engendré par les éléments $s_2,\ldots,s_{r-1}$. Donc $w(1)=1$. Or $ws_1\ldots s_{k-1}(k) = w(1) = 1$ et $ws_1\ldots s_{k-1}(k+1) = w(k+1) > 1$. D'après le Lemme \ref{lemm:l(ws_i)=l(w)+-1}, on en déduit que $l((ws_1\ldots s_{k-1})s_k) = l(ws_1\ldots s_{k-1})+1$ pour $k=2,\ldots,r-1$. De plus, comme $w(2)>1$, on a également $l(ws_1) = l(w) +1$. Une récurrence évidente nous donne le résultat escompté, c'est-à-dire $l(ws_1\ldots s_k) = l(w)+k$ pour tout $k\in\{1,\ldots,r-1\}$.

On en déduit que $\hat{w}_k$ est le plus petit élément de sa classe modulo $W_{\hat{Q}}$, puisque pour tout $w\in W_{\hat{Q}}$, $l(w\hat{w}_k) \geqslant l(w)$. De plus, pour tout $0\leq i<j\leq r-1$, nous avons $l(\hat{w}_i) = i-1 < j-1 = l(\hat{w}_j) \leqslant l(w\hat{w}_j)$ pour tout $w\in W_{\hat{Q}}$. Donc $\hat{w}_i$ et $\hat{w}_j$ sont dans la même classe modulo $W_{\hat{Q}}$ si et seulement si $\hat{w}_i = \hat{w}_j$, c'est-à-dire $i=j$.
\end{proof}

Le plus long élément de $W_{\hat{Q}}\hat{w}_k$ est $w_{\hat{Q}}\hat{w}_k$, car $w_{\hat{Q}}$ est le plus long élément de $W_{\hat{Q}}$. Ceci est clair d'après l'égalité $l(w\hat{w}_k) = l(w)+(k-1)$ obtenue dans la démonstation du lemme ci-dessus. De plus, toujours grâce à cette égalité, nous voyons que
\[
l(w_{\hat{Q}}\hat{w}_{r}) = l(w_{\hat{Q}}s_1\ldots s_{r-1}) = l(w_{\hat{Q}}) + r-1 = l(w_0).
\]
La dernière égalité vient du fait que $l(w_0)$ (resp. $l(w_{\hat{Q}})$) est égal au cardinal de l'ensemble $I(w_0) = \{(i,j); \ 1\leq i<j \leqslant n\}$ (resp. au cardinal de l'ensemble $I(w_{\hat{Q}}) = \{(i,j); \ 2\leq i<j \leqslant n\}$). Par unicité de l'élément le plus long dans $W_r$, on a $w_{\hat{Q}}s_1\ldots s_{r-1} = w_0$ et, par conséquent, $w_0w_{\hat{Q}} = s_{r-1}\ldots s_1$.

\begin{lemm}
\label{lemm:relation_wtilde_wchapeau}
Pour tout $k\in\{1,\ldots,r\}$, on a $w_0w_{\hat{Q}}\hat{w}_k = \check{w}_k$.
\end{lemm}

\begin{proof}
Nous venons de montrer que $w_0w_{\hat{Q}} = s_{r-1}\ldots s_1$. Il suffit alors de voir que
\[
w_0w_{\hat{Q}}\hat{w}_k = w_0w_{\hat{Q}}s_1\ldots s_{k-1} = s_{r-1}\ldots s_1 s_1\ldots s_{k-1} = s_{r-1}\ldots s_k = \check{w}_k,
\]
quel que soit $k\in\{2,\ldots,r-1\}$. De plus, $w_0w_{\hat{Q}}\hat{w}_1 = w_0w_{\hat{Q}} = s_{r-1}\ldots s_1 = \check{w}_1$, et enfin  $w_0w_{\hat{Q}}\hat{w}_{r} = \hat{w}_{r}^{-1}\hat{w}_{r} = \id = \check{w}_{r}$.
\end{proof}

Le lemme suivant découle du Lemme \ref{lemm:éléments_les_plus_courts_de_W_mod_WQ}. Il est utilisé dans le calcul du cas $SU(n,1)$. Ci-dessous, l'élément $\hat{w}_k^{-1}$ est égal à $\hat{w}_k^{-1}=s_{k-1}\ldots s_1$, pour $k=2,\ldots,r$, et $\hat{w}_1^{-1} = \id$.

\begin{lemm}
\label{lemm:éléments_les_plus_courts_de_W/WQ}
L'ensemble $W_r/W_{\hat{Q}}$ possède $r$ classes, et les $\hat{w}_k^{-1}$, pour $k$ parcourant l'ensemble $\{1,\ldots,r\}$, sont les éléments les plus courts de chaque classe. Plus exactement, on a $l(\hat{w}_k^{-1}w) = l(w) + k$, pour tout $w\in W_{\hat{Q}}$. De plus, pour tout $k=1,\ldots,r$, on a $w_0\hat{w}_k^{-1}w_{\hat{Q}} = \hat{w}_{r-k+1}^{-1}$. En particulier, on retrouve $w_0w_{\hat{Q}} = \hat{w}_r^{-1} = s_{r-1}\ldots s_1$.
\end{lemm}

\begin{proof}
Ce résultat découlera directement du Lemme \ref{lemm:éléments_les_plus_courts_de_W_mod_WQ}, en remarquant les deux faits suivants. Tout d'abord, $ W_{\hat{Q}}w =  W_{\hat{Q}}w'$ si et seulement si $w^{-1} W_{\hat{Q}} = w'^{-1}W_{\hat{Q}}$. Ensuite, pour tout $w\in W_r$, on a $l(w^{-1}) = l(w)$.

Le premier est évident, car $W_{\hat{Q}}$ est un groupe. Pour le second, il suffit de voir que si $s_{i_1}\ldots s_{i_k}$ est un décomposition réduite de $w\in W_r$, alors $w^{-1}=s_{i_k}\ldots s_{i_1}$ et donc $l(w)\leqslant l(w^{-1})$. De même, on a $l(w^{-1})\leqslant l(w)$ en inversant les rôles de $w$ et $w^{-1}$, d'où l'égalité $l(w^{-1}) = l(w)$. Ceci est vrai de manière générale pour les groupes de Coxeter.

Prenons maintenant $w\in W_{\hat{Q}}$ et $k\in\{1,\ldots,r-1\}$. Nous avons donc 
\[
l(s_k\ldots s_1w) = l(w^{-1}s_1\ldots s_k) = l(w^{-1}) + k = l(w) + k,
\]
la deuxième égalité provenant du Lemme \ref{lemm:éléments_les_plus_courts_de_W_mod_WQ}, avec $w^{-1}\in W_{\hat{Q}}$. Et comme $\hat{w}_1^{-1} = \id$, les $\hat{w}_k^{-1}$, pour $k=1,\ldots,r$, sont les éléments les plus courts de chaque classe.

Montrons la dernière assertion. Soit $k\in\{1,\ldots,r\}$. Nous regardons ici les classes à droite modulo $ W_{\hat{Q}}$, donc la multiplication à gauche par $w_0$ dans $W_r$ passe au quotient $W_r/W_{\hat{Q}}$, c'est-à-dire $w_0wW_{\hat{Q}}$ est encore une classe et, du fait que $l(w_0\hat{w}_k^{-1}w) = l(w_0) - l(\hat{w}_k^{-1}w) \leqslant l(w_0) - l(\hat{w}_k^{-1}) = l(w_0\hat{w}_k^{-1})$ pour tout $w\in W_{\hat{Q}}$, $w_0\hat{w}_k^{-1}$ sera le plus grand élément de sa classe dans $W_r/W_{\hat{Q}}$. Par conséquent, $w_0\hat{w}_k^{-1}w_{\hat{Q}}$ sera le plus petit élément de sa classe, donc $w_0\hat{w}_k^{-1}w_{\hat{Q}} = \hat{w}_{k'}^{-1}$ pour un certain $k'\in\{1,\ldots,n\}$. La longueur de cet élément doit vérifier l'équation suivante :
\begin{align*}
k'-1 & = l(\hat{w}_{k'}^{-1}) = l(w_0\hat{w}_k^{-1}w_{\hat{Q}}) = l(w_0) - l(\hat{w}_k^{-1}w_{\hat{Q}}) \\
& = (r-1) - (k-1).
\end{align*}
Finalement $k' = r-k+1$. Donc $w_0\hat{w}_k^{-1}w_{\hat{Q}} = \hat{w}_{r-k+1}^{-1}$.
\end{proof}

Les lemmes suivants vont nous permettre de calculer des classes de cohomologies comme produits cup de classes de degré $2$.


\begin{lemm}
\label{lemm:wtilde_plus2}
Si $1\leq i\leq k<j\leq r$ et $(i,j)\neq(k-1,k+1)$, alors on a $l(\check{w}_kt_{i,j}) \geqslant l(\check{w}_k)+2$. De plus, 
$\check{w}_kt_{k-1,k+1} = \check{w}_{k+1}s_{k-1}s_k = s_{r-1}\ldots s_{k+1}s_{k-1}s_k$ est de longueur $l(\check{w}_k t_{k-1,k+1}) = l(\check{w}_k)+1$.
\end{lemm}


\begin{proof}
Pour $l\geq k+1$, nous avons $\check{w}_kt_{i,j}(l) = \check{w}_k(l) = l-1 < r$ si $l\neq r$, et $\check{w}_kt_{i,j}(j) = \check{w}_k(i) = i < r$ également. De plus, $\check{w}_kt_{i,j}(k) = r$, d'où le couple $(k,l)$ appartient à l'ensemble $I(\check{w}_kt_{i,j})$ pour tout $l\in\{k+1,\ldots,r\}$. Nous pouvons voir également que $(i,j)\in I(\check{w}_kt_{i,j})$, car $\check{w}_kt_{i,j}(i) = \check{w}_k(j) = j-1\geq k$, tandis que $\check{w}_kt_{i,j}(j) = \check{w}_k(i) = i < k$.

Dans le cas où $(i,j)\neq(k-1,k+1)$, alors on a soit $i\leq k-2 < k-1$, et on peut regarder le couple $(k-1,j)$, qui est dans $I(\check{w}_kt_{i,j})$ car
\[
\check{w}_kt_{i,j}(k-1) = \check{w}_k(k-1) = k > i = \check{w}_kt_{i,j}(j),
\]
soit $k+1 < k+2 \leqslant j$, et ici on regarde le couple $(i,k+1)$, donnant
\[
\check{w}_kt_{i,j}(i) = \check{w}_k(j) = j-1 > k = \check{w}_k(k+1) = \check{w}_kt_{i,j}(k+1),
\]
et qui est donc dans $I(\check{w}_kt_{i,j})$ dans ce cas-là. Précisons que dans chacun de ces deux cas, les couples donnés ne sont pas apparus dans le paragraphe précédent.

Nous avons donc montré que, dans le cas où $(i,j)\neq (k-1,k+1)$, le cardinal de $I(\check{w}_kt_{i,j})$ est supérieur ou égal à $\sharp\{k+1,\ldots,r\}+2 = l(\check{w}_k) + 2$. On en conclut que $l(\check{w}_k t_{i,j}) \geqslant l(\check{w}_k) + 2$ si $(i,j)\neq (k-1,k+1)$.

Pour $i=k-1$ et $j=k+1$, nous avons $\check{w}_k t_{k-1,k+1} = (\check{w}_{k+1}s_k). (s_k s_{k-1} s_k) = \check{w}_{k+1}s_{k-1}s_k$. On peut facilement vérifier en utilisant deux fois successivement le Lemme \ref{lemm:l(ws_i)=l(w)+-1} que $l(\check{w}_{k+1}s_{k-1}s_k) = l(\check{w}_{k+1})+2 = l(\check{w}_k)+1$.
\end{proof}

\begin{lemm}
\label{lemm:wtilde_moins1}
Si $j\geq k+1$, alors $l(\check{w}_k t_{k,j}) \leqslant l(\check{w}_k) - 1$.
\end{lemm}

\begin{proof}
En effet, on peut aisément vérifier que dans $W_r$, l'élément $s_{\alpha_{k,j}}$ peut se décomposer de la manière suivante :
\[
t_{k,j} = s_k s_{k+1}\ldots s_{j-2} s_{j-1} s_{j-2}\ldots s_{k+1}s_{k},
\]
car $(k,j) = s_ks_{k-1}\ldots s_{j-2}(j-1,j)$. On compose ceci avec $\check{w}_k$ à gauche, ce qui nous donne, en utilisant la définition de $\check{w}_k$,
\[
\check{w}_k t_{k,j} = s_{r-1}\ldots s_k s_k s_{k+1}\ldots s_{j-2} s_{j-1} s_{j-2}\ldots s_{k+1}s_{k} = s_{r-1} \ldots s_j s_{j-2} s_{j-3}\ldots s_k.
\]
Il s'agit d'une décomposition de $\check{w}_k t_{k,j}$ en produit de transpositions simples, et cette décomposition est formée de $l(\check{w}_k) - 1$ facteurs. D'où le résultat. 
\end{proof}

\begin{lemm}
\label{lemm:wtilde_plus2_jégalk}
Soit $k\in\{2,\ldots,r-1\}$. Si $i<k-1$, alors $l(\check{w}_k t_{i,k}) \geqslant l(\check{w}_k) + 2$. Pour $i=k-1$, on a $\check{w}_k t_{k-1,k} = \check{w}_{k-1}$ et $l(\check{w}_k t_{k-1,k}) = l(\check{w}_k) + 1$
\end{lemm}

\begin{proof}
Comme dans les lemmes précédents, nous allons montrer que l'élément $\check{w}_k t_{i,k}$ envoie au moins $l(\check{w}_k) + 2$ racines positives sur des racines négatives.

Pour $l > k$, on a $\check{w}_k t_{i,k}(l) = \check{w}_k(l)=l-1$. Donc $\check{w}_k t_{i,k}(i) = \check{w}_k(k) = r > l-1 = \check{w}_k t_{i,k}(l)$. On en déduit que, pour tout $l\in\{k+1,\ldots,r\}$, le couple $(i,l)$ appartient à l'ensemble $I(\check{w}_k t_{i,k})$. Ceci nous donne déjà $l(\check{w}_k)$ couples dans $I(\check{w}_k t_{i,k})$.

Il reste deux autres couples à trouver. Les deux candidats sont les couples $(i,k)$ et $(k-1,k)$, qui ne font pas partie des couples donnés dans le paragraphe précédent. Tout d'abord, on a
\[
\check{w}_k t_{i,k}(k) = \check{w}_k(i) = i < r = \check{w}_k t_{i,k}(i),
\]
donc $(k,i)\in I(\check{w}_k t_{i,k})$. De plus,
\[
\check{w}_k t_{i,k}(k-1) = \check{w}_k(k-1) = k-1 > i = \check{w}_k t_{i,k}(k),
\]
d'où $(k-1,k)\in I(\check{w}_k t_{i,k})$. Par conséquent, la longueur de $\check{w}_k t_{i,k}$ est supérieure ou égale à $l(\check{w}_k) + 2$.

La deuxième assertion est triviale, puisque l'on a $t_{k-1,k} = s_{k-1}$ ainsi que $\check{w}_k t_{k-1,k} = s_{r-1}\ldots s_{k+1}s_k s_{k-1} =\check{w}_{k-1}$. Le Lemme \ref{lemm:longueur_wcheck_what} permet de conclure.
\end{proof}

%
%

Ce dernier lemme intervient dans la preuve du Lemme \ref{lemm:formule_puissancede_sigma_s_1}.

\begin{lemm}
\label{lemm:wtilde-1_plus1_jégalk+2}
Soit $k\in\{1,\ldots,n-2\}$. Alors, pour tout $j\in\{3,\ldots,n\}$, on a $l(s_k\ldots s_1 t_{1,j}) = l(s_k\ldots s_1) + 1$ si et seulement si $j=k+2$, et alors $s_k\ldots s_1 t_{1,k+2} = s_{k+1}s_k\ldots s_1$.
\end{lemm}

\begin{proof}
Vérifions tout d'abord que $l(s_k\ldots s_1 t_{1,k+2}) = l(s_k\ldots s_1)+1$. Mais ceci est vrai, car $t_{1,k+2} = s_1\ldots s_{k}s_{k+1}s_k\ldots s_1$, et on voit aisément que $s_k\ldots s_1t_{1,k+2} = s_{k+1}\ldots s_1$ est de longueur $k+1 = l(s_k\ldots s_1)+1$.

Montrons maintenant la réciproque. Plusieurs cas sont à distinguer. Il y a tout d'abord un cas spécial, qui vérifie $l(s_k\ldots s_1 t_{1,j}) \leqslant l(s_k\ldots s_1)-1$, c'est le cas lorsque $2\leq j\leq k+1$. En effet, dans ce cas-là, on aura 
\[
s_k\ldots s_1t_{1,j} = (s_k\ldots s_j\ldots s_1)(s_1\ldots s_{j-2}s_{j-1}s_{j-2}\ldots s_1) = s_k\ldots s_{j}s_{j-2}\ldots s_1,
\]
lorsque $j\leq k$, ou $s_k\ldots s_1t_{1,k+1} = s_{k-1}\ldots s_1$ si $j=k+1$. Dans les deux cas, on aura toujours $l(s_k\ldots s_1t_{1,j}) \leqslant l(s_k\ldots s_1)-1$, puisque le nombre de réflexions apparaissant dans ces décompositions est inférieur ou égal à $k-l = l(s_k\ldots s_1)-1$.

Le deuxième cas que l'on regarde concerne les $j>k+2$. Le couple $j=k+2$ a en effet déjà été considéré au début de la preuve. Alors dans ce cas, en itérant le Lemme \ref{lemm:l(ws_i)=l(w)+-1} au terme de droite de l'égalité $s_k\ldots s_1t_{1,j} = s_{k+1}\ldots s_{j-2}s_{j-1}s_{j-2}\ldots s_1$, on voit que la longueur de $s_k\ldots s_1t_{1,j}$ est égale à $2j-k-3\geqslant j$, puisque $j\geqslant k+3$. D'où $l(s_k\ldots s_1t_{1,j})\geqslant k+3 > l(s_k\ldots s_1) + 1$.
\end{proof}

\section{La Formule de Chevalley}
\label{subsection:formule_Chevalley}

\subsection{\'Enoncé de la Formule de Chevalley}
\label{subsubsection:énoncé_formule_Chevalley}

On conserve les notations définies dans le Chapitre \ref{chap:PairesBienCouvrantes}. Nous supposerons ici que $\KC$ est semi-simple et simplement connexe.

L'énoncé de la Formule de Chevalley est l'objet du théorème ci-dessous. Cet énoncé est tiré de \cite{berenstein_sjamaar} (Théorème A.2.1). Voir aussi \cite{chevalley} et \cite{demazure} pour une preuve. Dans l'énoncé suivant, $\alpha^{\vee}$ désigne la coracine de $\alpha$ dans $\got{t}$. De plus, si $\alpha$ est simple, $\pi_{\alpha}$ désigne le poids fondamental associé à $\alpha$. On notera $\Theta : \wedge^*\rightarrow\mathrm{H}^2(\KC/B,\Z)$ le morphisme qui, à un poids $\mu$ du réseau des poids $\wedge^*$ de $\TC$, associe $\Theta(\mu) = c_1(\mathcal{L}_{\mu})$, la première classe de Chern du fibré en droite $\mathcal{L}_{\mu}$ avec poids $\mu$.

\begin{theo}[Formule de Chevalley]
\label{theo:Chevalley}
\begin{enumerate}
\item $\Theta$ est un isomorphisme,
\item $\Theta(\pi_{\alpha}) = \sigma_{s_{\alpha}}$ pour toute racine simple $\alpha$.
\item Pour tout poids $\mu$ de $\wedge^*$,
\begin{equation}
\label{eq:formule_de_Chevalley}
\Theta(\mu).\sigma^{B}_{w} = \sum_{\stackrel{\alpha\in\got{R}^+}{l(w s_{\alpha}) = l(w)+1}}\mu(\alpha^{\vee})\sigma^{B}_{ws_{\alpha}}.
\end{equation}
\end{enumerate}
\end{theo}

Soit $\hKC$ un autre groupe algébrique complexe semi-simple connexe et simplement connexe, et $f:\KC\rightarrow\hKC$ un morphisme de groupes algébriques de noyau fini. L'application $f$ passe au quotient en une immersion fermée $f^B:\KC/B\rightarrow \hKC/\hat{B}$, induisant une application en cohomologie $(f^B)^*:\coh{\hKC/\hat{B}}\rightarrow\coh{\KC/B}$. Par propriété de la première classe de Chern et par la Formule de Chevalley, nous obtenons une description complète de l'application $(f^B)^*$ restreinte au degré $2$. En effet, nous avons le diagramme commutatif
\[
\begin{CD}
   \hat{\wedge}^* @>f^*>> \wedge^* \\
   @V\hat{\Theta}VV @VV\Theta V\\
   \mathrm{H}^2(\hKC/\hat{B},\Z) @>(f^B)^*>> \mathrm{H}^2(\KC/B,\Z)
\end{CD}
\]
Par conséquent, en degré $2$, on obtient la matrice de $(f^B)^*$ dans la base de classes de Schubert, c'est tout simplement la matrice de $f^*$ relativement à la base des poids fondamentaux.

De plus, on considère l'algèbre graduée $\R[\got{t}]$ des fonctions polynômiales sur $\got{t}$. Chaque élément de $\got{t}^*$ est défini comme étant de degré $2$. Notons $J$ l'idéal engendré par les polynômes $W$-invariants de degré strictement positif. Nous avons l'énoncé très connu suivant (cf. \cite{BGG}, \cite{berenstein_sjamaar}).

\begin{theo}[Borel]
\label{theo:Borel_cohomologie}
L'application $\Theta$ s'étend à un morphisme surjectif d'algèbres graduées $\R[\got{t}]\rightarrow \mathrm{H}^*(\KC/B,\R)$, dont le noyau est égal à $J$.
\end{theo}

Ce résultat est valable aussi pour $\hKC$. Nous avons $f^*(\hat{J}) \subset J$, où $\hat{J}$ est l'idéal engendré par les polynômes $\hat{W}$-invariants de degré strictement positif dans $\R[\hat{\got{t}}]$. Par conséquent, $(f^B)^*$ est complètement déterminée par les valeurs de $f^*$ sur $\hat{\wedge}^*$, du moins pour la cohomologie à coefficients dans $\R$.

\bigskip

Dans le cas où $\KC$ n'est plus semi-simple mais seulement réductif, nous pouvons toujours utiliser la formule ci-dessus, du fait que nous avons $\KC/B = [\KC,\KC]/(B\cap [\KC,\KC])$. De plus, $f$ envoie $[\KC,\KC]$ sur $[\hKC,\hKC]$, donc $f^B$ peut être aussi vu comme le plongement injectif induit par le morphisme de groupes $[\KC,\KC]\rightarrow[\hKC,\hKC]$.

\begin{rema}
\label{rema:KC_simplementconnexe}
Lorsque le groupe $\KC$ est semi-simple mais n'est pas simplement connexe, le Théorème \ref{theo:Chevalley} n'est plus vrai. Cependant, on peut tout de même déterminer le morphisme $(f^B)^*$. En effet, d'après \cite[Proposition 7]{chevalley}, il existe un groupe semi-simple simplement connexe $\tilde{K}_{\C}$ du même type que $\KC$ et une isogénie $i$ de $\tKC$ dans $\KC$ (c'est-à-dire $i:\tKC\rightarrow\KC$ est un morphisme surjectif de noyau fini) telle qu'il existe un sous-groupe de Borel $\tilde{B}$ de $\tKC$ vérifiant $i(\tilde{B})=B$, et que $i$ induise par passage au quotient un isomorphisme $\bar{i}$ de la variété $\tKC/\tilde{B}$ sur $\KC/B$. De plus, cet isomorphisme $\bar{i}$ envoie les variétés de Schubert de $\tKC/\tilde{B}$ sur les variétés de Schubert de $\KC/B$.

Posons $g := f\circ i : \tKC\rightarrow \hKC$, morphisme de groupes de noyau fini. On peut facilement voir que $g$ donne par passage au quotient une application $g^{\tilde{B}}:\tKC/\tilde{B}\rightarrow\hKC/\hat{B}$ qui est tout simplement $g^{\tilde{B}} = f^B\circ \bar{i}$. Comme $\tKC$ et $\hKC$ sont simplement connexes, on peut appliquer la méthode précédente pour déterminer $(g^{\tilde{B}})^*$, et donc déterminer $(f^B)^*$.
\end{rema}

\subsection{Calcul de $(f_{\lambda}^B)^*$}
\label{subsection:calcul_flambdaB*}

Soit $\KC$ un groupe algébrique affine réductif complexe. D'après la Remarque \ref{rema:KC_simplementconnexe}, on peut supposer $\KC$ simplement connexe. Comme il a été vu dans le paragraphe \ref{subsubsection:énoncé_formule_Chevalley}, le calcul de l'image par $(f_{\lambda}^B)^*$ de la classe $\sigma_w$, pour un élément $w$ du groupe de Weyl de $[\hKC,\hKC]$ (qui s'identifie au groupe de Weyl de $\hKC$), revient à connaître les images par $f^*$ des poids fondamentaux de $\hKC$, ainsi que la décomposition de $\sigma_w$ en sommes et produits de termes de degré $2$. C'est cette deuxième opération qui est difficile à calculer en général, même si nous avons une manière effective de le réaliser pour chaque groupe.

Cependant, ici, nous travaillons avec le groupe $GL_r(\C)$, ou plus exactement sur sa partie semi-simple $SL_r(\C)$ qui est simplement connexe, ce qui diminue la difficulté. De plus, nous ne nous intéressons pas au calcul de tous les $\sigma_w$. Ceux qui nous intéressent, les éléments $\check{w}_k$, ont une décomposition finalement très simple.

\begin{lemm}
\label{lemm:produitcup_sk_wtildek}
Pour tout $k=2,\ldots,r-1$, on a $\sigma_{s_k}^{B_r}.\sigma_{\check{w}_k}^{B_r} = \sigma^{B_r}_{s_{r-1}\ldots s_{k+2}s_{k+1}s_{k-1}s_k}$. De plus, $\sigma_{s_1}^{B_r}.\sigma_{\check{w}_1}^{B_r} = 0$, et $\sigma_{s_k}^{B_r}.\sigma_{\check{w}_{r}}^{B_r} = \sigma_{s_k}^{B_r}$.
\end{lemm}

\begin{proof}
Les poids fondamentaux de $\got{sl}_{r}$ sont bien connus. Si on note $\hat{e}_i^*$ la forme linéaire sur $\got{t}_r\cap \got{sl}_{r}$ définie par
\[
\hat{e}_i^*\left(\begin{array}{ccc}
h_1 & & \\
& \ddots & \\
& & h_{r}
\end{array}\right) = h_i,
\]
on aura $\pi_{\alpha_{k,k+1}} = \sum_{i=1}^{k}\hat{e}_i^*$. De plus, si $\alpha = \alpha_{i,j}$ est une racine quelconque, alors la coracine $\alpha_{i,j}^{\vee}$ est la matrice diagonale dont tous les éléments diagonaux sont nuls sauf le $i$-ième qui vaut $1$ et le $j$-ième qui vaut $-1$. Par conséquent, nous avons les valeurs suivantes, dans le cas où $\alpha_{i,j}$ est positive, c'est-à-dire $i<j$,
\[
\pi_{\alpha_{k,k+1}}(\alpha_{i,j}^{\vee}) = \left\{\begin{array}{l}
0 \mbox{ si } k<i, \\
1 \mbox{ si } i\leq k <j, \\
0 \mbox{ si } j\leq k. \\
\end{array}\right.
\]
On en déduit que
\begin{equation}
\label{eq:produitcup_lemme1}
\sigma^{B_r}_{s_k}.\sigma^{B_r}_{\check{w}_k} = \sum_{\stackrel{i\leq k< j}{l(\check{w}_k s_{\alpha_{i,j}}) = l(\check{w}_k)+1}}\sigma^{B_r}_{\check{w}_k s_{\alpha_{i,j}}},
\end{equation}
d'après la formule \eqref{eq:formule_de_Chevalley}. Les Lemmes \ref{lemm:wtilde_plus2} et \ref{lemm:wtilde_moins1} nous indiquent que le seul couple $(i,j)$ vérifiant $1\leq i\leq k<j\leq r$ et $l(\check{w}_k s_{\alpha_{i,j}}) = l(\check{w}_k)+1$ est le couple $(k-1,k+1)$. Dans le cas où $2\leq k\leq r-1$, l'équation \eqref{eq:produitcup_lemme1} et la formule $\check{w}_ks_{\alpha_{i,j}} = s_{r-1}\ldots s_{k+1}s_{k-1}s_k$, toujours du Lemme \ref{lemm:wtilde_plus2}, nous permettent de conclure la première assertion du lemme. Pour $k=1$, on a $k-1=0$ n'entre pas en compte, donc $\sigma_{s_1}^{B_r}.\sigma_{\check{w}_1}^{B_r} = 0$. Et la dernière assertion est évidente puisque $\check{w}_{r} = \id$.
\end{proof}

Nous avons un résultat similaire, en décalant d'un cran le $k$ du $\sigma_{s_k}$.

\begin{lemm}
\label{lemm:produitcup_skmoins1_wtildek}
Pour tout $k=2,\ldots,r-1$, on a $\sigma_{s_{k-1}}^{B_r}.\sigma_{\check{w}_k}^{B_r} = \sigma^{B_r}_{s_{r-1}\ldots s_{k+2}s_{k+1}s_{k-1}s_k} + \sigma_{\check{w}_{k-1}}^{B_r}$.
\end{lemm}

\begin{proof}
Par un raisonnement identique à la preuve du Lemme \ref{lemm:produitcup_sk_wtildek}, nous obtenons la formule suivante,
\begin{equation}
\label{eq:produitcup_lemme2}
\sigma^{B_r}_{s_{k-1}}.\sigma^{B_r}_{\check{w}_k} = \sum_{\stackrel{i\leq k-1< j}{l(\check{w}_k s_{\alpha_{i,j}}) = l(\check{w}_k)+1}}\sigma^{B_r}_{\check{w}_k s_{\alpha_{i,j}}}.
\end{equation}
De la même manière, le seul couple $(i,j)$ vérifiant $1\leq i< k<j\leq r$ et $l(\check{w}_k s_{\alpha_{i,j}}) = l(\check{w}_k)+1$ est le couple $(k-1,k+1)$, grâce au Lemme \ref{lemm:wtilde_plus2}. Puisque $2\leq k \leqslant r-1$, ce couple-là sera toujours présent. Il reste à voir les couples $(i,k)$. Le Lemme \ref{lemm:wtilde_plus2_jégalk} nous affirme que le seul $i\leq k-1$ qui vérifie $l(\check{w}_k s_{\alpha_{i,k}}) = l(\check{w}_k)+1$ est $i=k-1$. L'équation \eqref{eq:produitcup_lemme2} devient $\sigma_{s_{k-1}}^{B_r}.\sigma_{\check{w}_k}^{B_r} = \sigma^{B_r}_{s_{r-1}\ldots s_{k+2}s_{k+1}s_{k-1}s_k} + \sigma_{\check{w}_{k-1}}^{B_r}$, ce que nous voulions démontrer.
\end{proof}

\begin{theo}
\label{theo:formume_sigmawtildek}
Pour tout entier $k\in\{2,\ldots,r-1\}$, on a
\[
\sigma_{\check{w}_{k-1}}^{B_r} = (\sigma_{s_{k-1}}^{B_r}-\sigma_{s_k}^{B_r}).\sigma_{\check{w}_k}^{B_r}.
\]
On peut écrire $\sigma_{\check{w}_{k-1}}^{B_r}$ uniquement en terme des $\sigma_{s_i}^{B_r}$, pour $i=k,\ldots,r-1$, de la manière suivante,
\[
\sigma_{\check{w}_{k-1}}^{B_r} = (\sigma_{s_{k-1}}^{B_r}-\sigma_{s_k}^{B_r}).(\sigma_{s_k}^{B_r}-\sigma_{s_{k+1}}^{B_r}).\cdots.(\sigma_{s_{r-2}}^{B_r}-\sigma_{s_{r-1}}^{B_r}).\sigma_{s_{r-1}}^{B_r}.
\]
\end{theo}

\begin{proof}
Des formules obtenues dans les Lemmes \ref{lemm:produitcup_sk_wtildek} et \ref{lemm:produitcup_skmoins1_wtildek}, l'égalité $\sigma_{\check{w}_{k-1}}^{B_r} = (\sigma_{s_{k-1}}^{B_r}-\sigma_{s_k}^{B_r}).\sigma_{\check{w}_k}^{B_r}$ est claire. La deuxième assertion est une simple récurrence sur $k$, en utilisant le fait que $\check{w}_{r-1} = s_{r-1}$.
\end{proof}

Nous sommes maintenant en mesure de calculer les images par $(f_{\lambda}^B)^*$ des classes de cohomologie $\sigma_{\check{w}_{k-1}}^{B_r}\in\coh{GL_r(\C)/B_r}$, pour $k\in\{2,\ldots,r\}$. Nous regardons ici l'application
\[
\begin{array}{cccc}
f_{\lambda}: & \got{t}\cap[\got{k},\got{k}] & \longrightarrow & \got{t}_r\cap\got{sl}_r \\
& H & \longmapsto & \left(\begin{array}{ccc}
\beta_1(H) & & \\
& \ddots & \\
& & \beta_r(H)
\end{array}\right)
\end{array}
\]
Nous noterons dorénavant $\got{t}_{ss} = \got{t}\cap[\got{k},\got{k}]$. Il est évident que $f_{\lambda}^*(\hat{e}_i) = \beta_i|_{\got{t}_{ss}}$, par conséquent, $f_{\lambda}^*(\pi_{\alpha_{i,i+1}}) = \sum_{j=1}^i \beta_j|_{\got{t}_{ss}}$ et $f_{\lambda}^*(\pi_{\alpha_{i-1,i}}) - f_{\lambda}^*(\pi_{\alpha_{i,i+1}}) = -\beta_i|_{\got{t}_{ss}}$. En appliquant ceci à $\sigma_{\check{w}_{k-1}}^{B_r} = (\sigma_{s_{k-1}}^{B_r} - \sigma_{s_k}^{B_r}).\cdots.(\sigma_{s_{r-2}}^{B_r} - \sigma_{s_{r-1}}^{B_r}).\sigma_{s_{r-1}}^{B_r}$, pour $k\leq r-1$, nous obtenons l'égalité
\[
(f_{\lambda}^B)^*(\sigma_{\check{w}_{k-1}}^{B_r}) = \Theta\left(f_{\lambda}(\pi_{\alpha_{r-1,r}}).\prod_{i=k}^{r-1}\bigl(f_{\lambda}(\pi_{\alpha_{i-1,i}}) - f_{\lambda}(\pi_{\alpha_{i,i+1}})\bigr)\right),
\]
qui nous donne la formule
\begin{equation}
\label{eq:image_par_flambdatilde_de_sigmawtilde}
(f_{\lambda}^B)^*(\sigma_{\check{w}_{k-1}}^{B_r}) = \Theta\left((-\beta_k|_{\got{t}_{ss}}) \ldots (-\beta_{r-1}|_{\got{t}_{ss}}) (\sum_{i=1}^{r-1}\beta_i|_{\got{t}_{ss}})\right).
\end{equation}
Nous avons aussi évidemment
\begin{equation}
\label{eq:image_par_flambdatilde_de_sigmawtilde_rmoins1}
(f_{\lambda}^B)^*(\sigma_{\check{w}_{r-1}}^{B_r}) = \Theta\left(\sum_{i=1}^{r-1}\beta_i|_{\got{t}_{ss}}\right).
\end{equation}
Le lemme suivant va nous permettre de transformer la somme $\sum_{i=1}^{r-1}\beta_i|_{\got{t}_{ss}}$ apparaissant ci-dessus.

\begin{lemm}
Soit $M$ une représentation complexe de $\KC$ et $\delta=\sum_{\beta\in\WT(M)}\beta$ la somme de tous les poids de l'action de $\TC$ sur $M$. Alors $\delta|_{\got{t}_{ss}} = 0$.
\end{lemm}

\begin{proof}
La représentation $M$ de $\KC$ induit une représentation de $\KC$ sur le déterminant $\det M = \bigwedge^{\dim_{\C} M} M$ de $M$. Ce dernier est une droite vectorielle, donc on obtient un caractère $\chi:\KC\rightarrow\C^*$. Sa dérivée $d\chi_e:\got{k}_{\C}\rightarrow\C$ est un morphisme d'algèbres de Lie. Or, ici, $\C$ est une algèbre de Lie abélienne, donc, pour tout $X,Y\in\got{k}_{\C}$, nous aurons $d\chi_e([X,Y]) = 0$. Et il est clair que $\delta = \sum_{\beta\in\got{Wt}_T(E)}\beta = id\chi_e|_{\got{t}_{\C}^*}$. Puisque $\got{t}_{ss} = \got{t}\cap[\got{k},\got{k}]$, on en conclut que $\delta|_{\got{t}_{ss}} = 0$.
\end{proof}

Dans notre cas, nous avons numéroté, avec redondance, les poids de l'action de $\TC$ sur $M$, de sorte que $\delta = \sum_{i=1}^r\beta_i$. Ainsi, $-\beta_r|_{\got{t}_{ss}} = \sum_{i=1}^{r-1}\beta_i|_{\got{t}_{ss}}$. Nous pouvons remplacer cette valeur dans les équations \eqref{eq:image_par_flambdatilde_de_sigmawtilde} et \eqref{eq:image_par_flambdatilde_de_sigmawtilde_rmoins1}, ce qui donne
\begin{equation}
\label{eq:image_par_flambdatilde_de_sigmawtilde_finale}
(f_{\lambda}^B)^*\left(\sigma_{\check{w}_{k-1}}^{\tilde{B}}\right) = \Theta\left((-\beta_k|_{\got{t}_{ss}}) \ldots (-\beta_{r}|_{\got{t}_{ss}})\right).
\end{equation}
pour tout entier $k\in\{2,\ldots,r\}$.

\chapter{Classe fondamentale des variétés de Schubert}
\label{chap:classefondamentale}

Pour définir les classes de Schubert, le passage délicat est de bien définir la classe fondamentale d'une variété de Schubert, même quand celle-ci n'est pas lisse. Cette définition doit bien sûr coïncider avec la définition standard de classe fondamentale lorsque la variété est lisse.

Pour ce faire, nous présentons ici la méthode très élégante utilisant l'homologie de Borel-Moore. Tous les détails de cette construction se trouvent dans \cite[Appendice B]{fulton_youngtableaux} et \cite[Appendice]{manivel}, dont les paragraphes suivants s'inspirent largement.

\section{Classe fondamentale d'une variété projective irréductible lisse}

Toute variété projective irréductible lisse $X$ de dimension $n$ peut être vue comme une variété différentiable réelle compacte de dimension $2n$. De plus, la structure complexe confère à $X$ une orientation. Il est bien connu que le groupe d'homologie de degré maximal $H_{2n}(X)$ est alors canoniquement isomorphe à $\Z$, où $H_{i}(X):=H_{i}(X,\Z)$ désigne le $i$\ieme{} groupe d'homologie singulière sur $X$. Le choix d'une orientation revient ici à choisir un générateur de $H_{2n}(X)$. La variété lisse $X$ possède donc une \emph{classe fondamentale}, notée $[X]$, telle que
\[
H_{2n}(X) = \Z\cdot[X].
\]
Cette classe fondamentale permet de définir un isomorphisme entre les groupes d'homologie et de cohomologie de la variété $X$, la \emph{dualité de Poincaré}, donné par l'application
\[
\begin{array}{cccc}
\bullet\cap [X] : & H^i(X) & \longrightarrow & H_{2n-i}(X) \\
& \alpha & \longmapsto & \alpha \cap [X],
\end{array}
\]
réalisant le produit cap avec la classe fondamentale de $X$. La dualité de Poincaré est évidemment valable pour toute variété différentiable compacte connexe orientée.

\section{Homologie de Borel-Moore}

Soit $X$ un espace topologique et $A$ une partie de $X$ (qui sera le plus souvent ouvert dans $X$). Les groupes de cohomologie $H^i(X,A)$ sont définis comme groupes de cohomologie du complexe $C^*(X,A)$ des cochaînes singulières (c'est-à-dire fonctions sur les chaînes singulières à valeurs dans $\Z$) sur $X$ qui s'annulent sur $A$.

Si $X$ est homéomorphe à un fermé d'un espace affine $\R^n$, alors on définit les groupes d'homologie de Borel-Moore
\[
\overline{H}_i(X) = H^{n-i}(\R^n,\R^n\setminus X).
\]
Le premier point à noter est que cette définition est indépendante du choix de l'homéomorphisme dans un espace affine \cite[Appendice B Lemme 1]{fulton_youngtableaux}.

De plus, on peut remplacer l'espace affine par une variété différentiable orientée quelconque. En effet, d'après \cite[Appendice B Lemme 1]{fulton_youngtableaux}, si $X$ est homéomorphe à une partie fermée d'une variété différentiable orientée $M$, alors il existe un isomorphisme canonique entre $\overline{H}_i(X)$ et $H^{m-i}(M,M\setminus X)$, où $m$ désigne la dimension de la variété $M$. En particulier, si $X$ est une variété différentiable orientée connexe de dimension $n$, on a
\[
\overline{H}_i(M) = H^{n-i}(M,M\setminus M) = H^{n-i}(M).
\]
Ainsi, $\overline{H}_i(M) = 0$ dès que $i>n$, et $\overline{H}_n(X) = H^0(M) = \Z$. Il existe donc une classe fondamentale $[X]$ de sorte que
\[
\overline{H}_n(X) = \Z\cdot [X].
\]
Si $X$ est également compacte, alors par dualité de Poincaré, les groupes d'homologie $\overline{H}_i(X)$ sont égaux aux groupes d'homologie singulière $H_i(X)$ ordinaires.

\bigskip

Soit $U$ un ouvert d'un espace topologique $X$ qui se plonge dans une variété orientée $M$. Alors il existe une \emph{application de restriction} canonique de $\overline{H}_i(X)$ dans $\overline{H}_i(U)$. En effet, $U$ est un fermé dans la variété orientée $N = M\setminus(X\setminus U)$, d'où le morphisme suivant, donné par l'application de restriction en cohomologie,
\[
\overline{H}_i(X) = H^{n-i}(M,M\setminus X) \longrightarrow H^{n-i}(N,N\setminus U) = \overline{H}_i(U),
\]
où $n=\dim M = \dim N$. En découle la suite exacte longue (\cite[Appendice B Lemme 3]{fulton_youngtableaux})
\[
\cdots \rightarrow \overline{H}_i(X\setminus U) \rightarrow \overline{H}_i(X) \rightarrow \overline{H}_i(U) \rightarrow \overline{H}_{i-1}(X\setminus U) \rightarrow \cdots .
\]

Enfin, même si l'image directe entre les groupes d'homologie $\overline{H}_i$ n'est pas définie pour toute application continue, elle l'est pour les applications continues propres.

\section{Classe fondamentale des sous-variétés algébriques irréductibles}

Soient $X$ une variété projective irréductible lisse de dimension complexe $n$, et $Y$ une sous-variété algébrique irréductible fermée de $X$. Alors l'ensemble de ses points singuliers $\mathrm{Sing}(Y)$ est une sous-variété algébrique de $Y$ de codimension complexe au moins un, donc de codimension réelle au moins deux. De plus, comme $Y$ est irréductible, $Y\setminus\mathrm{Sing}(Y)$ est une variété lisse connexe et, par une récurrence sur la dimension de $Y$ \cite[Appendice B Lemme 4]{fulton_youngtableaux}, on montre que $\overline{H}_q(Y) = 0$ si $q>2\dim_{\C} Y$, et
\[
\overline{H}_{2\dim_{\C}Y}(Y) \simeq \overline{H}_{2\dim_{\C}Y}(Y\setminus\mathrm{Sing}(Y)) \simeq \Z.
\]
Ceci permet de définir une classe fondamentale $[Y]$ dans $\overline{H}_{2\dim_{\C}Y}(Y)$, image de celle de $Y\setminus\mathrm{Sing}(Y)$ par l'isomorphisme ci-dessus. L'injection $i:Y\hookrightarrow X$, étant continue et propre, permet de définir l'image de la classe fondamentale de $Y$ dans le groupe d'homologie $H_{2(\dim_{\C}X-\dim_{\C}Y)}(X)$, que l'on notera également $[Y]$.

\bigskip

Supposons que $X$ admet une décomposition cellulaire finie $X = \coprod_{i\in I}Y_i$, où les cellules $Y_i$ sont isomorphes à des espaces affines $\C^{n_i}$, et vérifient des conditions de bord
\[
\overline{Y_i}\setminus Y_i = \coprod_{j \in J_i} Y_j,
\]
avec $n_j<n_i$ si $j\in J_i$. Alors, chaque $\overline{Y_i}$ a sa classe fondamentale dans $H^{2(\dim_{\C}X - n_i)}(X)$, et on a
\[
H^{2q}(X) = \bigoplus_{n_i = n - q}\Z\cdot [\overline{Y_i}]
\]
et $H^{2q+1}(X) = 0$, pour tout $q=0,\ldots,n$.

\bigskip

Soit $f:X\rightarrow Z$ un morphisme entre deux variétés projectives irréductibles lisses. Pour toute sous-variété irréductible fermée $Y$ de $X$ de dimension $k$, $f(Y)$ est une sous-variété fermée irréductible de $Z$ de dimension au plus $k$. De plus, si $f(Y)$ est de dimension $k$, alors il existe un ouvert de Zariski $U$ de $f(Y)$ tel que l'application de $Y$ dans $f(Y)$ détermine un revêtement fini de $f^{-1}(U)\cap Y$ sur $U$. Le nombre de feuillets de ce revêtement est appelé \emph{degré} de $Y$ sur $f(Y)$. On a alors la propriété suivante :
\[
f_*([Y]) = \left\{\begin{array}{ll}
0 & \text{si $\dim f(Y) < \dim Y$} \\
d[f(Y)] & \text{si $Y$ est de degré $d$ sur $f(Y)$.}
\end{array}\right.
\]

\bigskip

Pour terminer, soient $Y$ et $Y'$ deux sous-variétés irréductibles fermées de la variété projective irréductible lisse $X$. Supposons que l'intersection $Y\cap Y'$ est l'union de sous-variétés irréductibles $Z_1,\ldots,Z_r$ de $X$.

On dira que l'intersection $Y\cap Y'$ est \emph{propre} si la codimension de chaque $Z_i$ dans $X$ est égale à la somme des codimensions de $Y$ et $Y'$ dans $X$. On dira également que $Y$ et $Y'$ s'intersectent transversalement si pour tout point $z$ dans un ouvert de Zariski de chaque $Z_i$, on a l'égalité suivante sur les espaces tangents
\[
T_zZ_i = T_zY\cap T_zY' \subset T_zX.
\]
Alors, si $Y$ et $Y'$ s'intersectent proprement et transversalement, le produit cup des classes fondamentales $[Y]$ et $[Y']$ est donné par
\[
[Y]\,.\,[Y'] = [Z_1] + \ldots + [Z_r].
\]

\chapter[Exemples de triplets de $T_r^n$]{Exemples de triplets de $T_r^n$ non triviaux intervenant dans le problème de Horn}
\label{chap:exemples_de_triplets_pour_Horn}

Dans ce court appendice, nous explicitons des triplets non triviaux de l'ensemble $T_r^n$ défini au paragraphe \ref{subsection:problème_de_Horn}. Ces triplets interviennent dans la preuve du Théorème \ref{theo:SU(n,1)_PolyèdreMoment} donnant les équations du polyèdre $\Delta_K(\Orb_{\Lambda})$ pour une orbite coadjointe holomorphe pour le groupe $SU(n,1)$.

\bigskip

\section{Les triplets du type $(I,\overline{I},L^n_r)$}

Soit $I \subset\{1,\ldots,n\}$ avec $|I| = r\leqslant n-1$. Notons $\overline{I} = \{n-i+1; \ i\in I\}$ et $L^n_r = \{n-r+1< n-r+2 < \ldots < n-1 < n\}$. Les trois sous-ensembles $I$, $\overline{I}$ et $L^n_r$ de $\{1,\ldots,n\}$ sont de même cardinal $r$.

\begin{prop}
\label{prop:SU(n,1)_I_Ibar_Lnr}
Le triplet $(I,\overline{I},L^n_r)$ appartient à $T^n_r$.
\end{prop}

\begin{proof}
ici, $I\subset\{1,\ldots,n\}$ est quelconque, de cardinal $r$, $J = \overline{I}$ et $L = L^n_r$. Numérotons les éléments de $I$ par ordre croissant: $I=\{i_1<\cdots<i_r\}$. On a alors $\overline{I} = \{n-i_r+1<\ldots<n-i_1+1\}$. La partition $\lambda(I)$ est définie au paragraphe \ref{subsection:problème_de_Horn} par $\lambda(I) = (i_r-r\geqslant i_{r-1}-(r-1)\geqslant\ldots\geqslant i_2-2\geqslant i_1-1)$, et les autres sont
\[
\lambda(\overline{I})= (n-i_1-r+1\geqslant n-i_2-r+2\geqslant \ldots\geqslant n-i_k-r+k\geqslant \ldots\geqslant n-i_{r-1}-1\geqslant n-i_r)
\]
et
\[
\lambda(L^n_r)= (n-r\geqslant n-r\geqslant \ldots \geqslant n-r\geqslant n-r).
\]
Prenons maintenant les trois matrices diagonales réelles, donc hermitiennes, de spectres respectifs $\lambda(I)$, $\lambda(\overline{I})$ et $\lambda(L^n_r)$:
\[
\begin{array}{l}
A = \diag(i_r-r,\ldots,i_{r-k+1}-(r-k+1),\ldots,i_1-1), \\
B = \diag(n-i_r,\ldots,n-i_{r-k+1}-(k-1),\ldots,n-i_1-(r-1)), \\
C = \diag(n-r,\ldots,n-r\ldots,n-r).
\end{array}
\]
On a clairement $A+B = C$. Le Théorème \ref{theo:Fultonthm2} permet alors de conclure que le triplet  $(I,\overline{I},L^n_r)$ appartient à $T^n_r$.
\end{proof}

\section{Les triplets du type $(I,\overline{I^*},\tilde{L}^n_r)$}

Soit à nouveau $I\subset\{1,\ldots,n\}$ avec $|I| = r\leqslant n-1$. Supposons maintenant $1\in I$. Nous notons $I^* = \left\{i-1; \ i\in I\backslash\{1\}\right\} \cup \{n\}$ et $\tilde{L}^n_r = \{1<n-r+2<n-r+3<\ldots<n-1<n\}$. On indexe les éléments de $I$ par ordre croissant, $I=\{1<i_2<\ldots<i_r\}$, ce qui nous donne pour les autres ensembles, $I^* = \{i_2-1<\ldots<i_r-1<n\}$ et $\overline{I^*} = \{1<n-i_r+2<\ldots<n-i_2+2\}$.

\begin{prop}
\label{prop:SU(n,1)_IJL_1dansL}
Le triplet $(I,\overline{I^*},\tilde{L}^n_r)$ appartient à $T^n_r$.
\end{prop}

\begin{proof}
On utilise à nouveau le Théorème \ref{theo:Fultonthm2}. Cette fois-ci, on a $\lambda(I) = (i_r-r\geqslant \ldots\geqslant i_2-2\geqslant 0)$, ainsi que
\[
\lambda(\overline{I^*})= (n-i_2+2-r\geqslant \ldots\geqslant n-i_k+2-(r-k+1)\geqslant \ldots\geqslant n-i_r\geqslant 0)
\]
et
\[
\lambda(\overline{L^n_r})= (n-r\geqslant n-r\geqslant \ldots \geqslant n-r\geqslant 0).
\]
On peut prendre comme matrices hermitiennes, de spectres respectifs ces trois $r$-uplets, les matrices diagonales réelles suivantes:
\[
\begin{array}{l}
A=\diag(i_r-r,\ldots,i_{r-k+1}-(r-k+1),\ldots,i_2-2,0), \\
B=\diag(n-i_r,\ldots,n-i_{r-k+1}-(k-1),\ldots,n-i_2-r+2,0), \\
C=\diag(n-r,\ldots,n-r\ldots,n-r,0).
\end{array}
\]
On obtient l'égalité attendue: $A+B=C$. Donc le triplet $(I,\overline{I^*},\tilde{L}^n_r)$ appartient bien à $T^n_r$ d'après le Théorème \ref{theo:Fultonthm2}.
\end{proof}

\backmatter

\bibliographystyle{smfalpha}
\bibliography{biblio}

\printindex


\newpage
~
\thispagestyle{empty}
\newpage
~
\thispagestyle{empty}


\newpage
~
\thispagestyle{empty}
\vfill
\begin{center}
{\bf R\'ESUM\'E}
\end{center}

\vspace{8pt}

L'objet de cette thèse est l'étude de la structure symplectique des orbites coadjointes holomorphes, et de leurs projections. 

Une orbite coadjointe holomorphe $\Orb_{\Lambda}$ est une orbite coadjointe elliptique d'un groupe de Lie $G$ réel semi-simple connexe non compact à centre fini provenant d'un espace symétrique hermitien $G/K$, telle que $\Orb_{\Lambda}$ puisse être naturellement munie d'une structure kählérienne $G$-invariante canonique. Ces orbites coadjointes sont une généralisation de l'espace symétrique hermitien $G/K$.

Dans cette thèse, nous prouvons que le symplectomorphisme de McDuff se généralise aux orbites coadjointes holomorphes. Ce symplectomorphisme est ensuite utilisé pour déterminer les équations de la projection de l'orbite $\Orb_{\Lambda}$ relative au sous-groupe compact maximal $K$ de $G$, en faisant intervenir des résultats récents de Ressayre en Théorie Géométrique des Invariants. On donne également un critère cohomologique pour calculer l'ensemble des équations du polyèdre moment.

\vspace{16pt}

\begin{center}
{\bf MOTS-CL\'ES}
\end{center}

\vspace{8pt}

\noindent Orbite coadjointe, projection d'orbite, symplectomorphisme de McDuff, GIT, cône ample.

\vspace{16pt}

\begin{center}
{\bf ABSTRACT}
\end{center}

\vspace{8pt}

This thesis studies the symplectic structure of holomorphic coadjoint orbits, and their projections.

A holomorphic coadjoint orbit $\Orb_{\Lambda}$ is an elliptic coadjoint orbit which is endowed with a natural invariant Kählerian structure. These coadjoint orbits are defined for real semi-simple connected non compact Lie group $G$ with finite center such that $G/K$ is a Hermitian symmetric space, where $K$ is a maximal compact subgroup of $G$. Holomorphic coadjoint orbits are a generalization of the Hermitian symmetric space $G/K$.

In this thesis, we prove that the McDuff's symplectomorphism, available for Hermitian symmetric spaces, has an analogous for holomorphic coadjoint orbits. Then, using this symplectomorphism and recent GIT arguments from Ressayre, we compute the equations of the projection of the orbit $\Orb_{\Lambda}$, relatively to the maximal compact subgroup $K$.

\vspace{16pt}

\begin{center}
{\bf KEYWORDS}
\end{center}

\vspace{8pt}

\noindent Coadjoint orbit, orbit projection, McDuff's symplectomorphism, GIT, ample cone.

\vspace{16pt}

\begin{center}
{\bf CLASSIFICATION MATH\'EMATIQUE}
\end{center}

\vspace{8pt}

\noindent 53D20, 53C55, 53C57, 58F05, 14L24.
\vfill

\end{document}